\documentclass[english,11pt]{article}%
\usepackage[ total={6in, 8in}, voffset=0.45in]{geometry}
\usepackage[T1]{fontenc}
\usepackage[latin9]{inputenc}
\usepackage{amsmath}
\usepackage{amsthm}
\usepackage{amssymb}
\usepackage{amsfonts}
\usepackage{babel}
\usepackage{fancyhdr}
\usepackage{enumitem}
\usepackage{color}
\usepackage[toc]{appendix}
\usepackage{graphicx}%
\usepackage{url}
\providecommand{\U}[1]{\protect\rule{.1in}{.1in}}
\makeatletter
\newtheorem{theorem}{Theorem}

\newtheorem{condition}[theorem]{Condition}
\newtheorem{conjecture}[theorem]{Conjecture}
\newtheorem{corollary}[theorem]{Corollary}

\newtheorem{definition}[theorem]{Definition}

\newtheorem{lemma}[theorem]{Lemma}

\theoremstyle{remark}

\theoremstyle{definition}
\newtheorem{example}[theorem]{Example}
\newtheorem{remark}[theorem]{Remark}

\numberwithin{equation}{section}
\numberwithin{theorem}{section}

\makeatother
\begin{document}

\title{Large deviation properties of the empirical measure of a metastable small noise diffusion}
\author{Paul Dupuis\thanks{Division of Applied Mathematics, Brown University,
Providence, USA. Research supported in part by the National Science Foundation (DMS-1904992) and the AFOSR (FA-9550-18-1-0214).} \, and Guo-Jhen Wu\thanks{Department of
Mathematics, KTH Royal Institute of Technology, Stockholm, Sweden. Research supported in part by the AFOSR (FA-9550-18-1-0214); gjwu@kth.se (Corresponding author).} }
\maketitle

\begin{abstract}
The aim of this paper is to develop tractable large deviation approximations for the empirical measure of a small noise diffusion.  The starting point is the Freidlin-Wentzell theory, which shows how to approximate via a large deviation principle the invariant distribution of such a diffusion.  The rate function of the invariant measure is formulated in terms of quasipotentials, quantities that measure the difficulty of a transition from the neighborhood of one metastable set to another. The theory provides an intuitive and useful approximation for the invariant measure, and along the way many useful related results (e.g., transition rates between metastable states) are also developed.

With the specific goal of design of Monte Carlo schemes in mind, we prove large deviation limits for integrals with respect to  the empirical measure, where the process is considered over a time interval whose length grows as the noise decreases to zero. In particular, we show how the first and second moments of these integrals can be expressed in terms of quasipotentials. When the dynamics of the process depend on parameters, these approximations can be used for algorithm design, and applications of this sort will appear elsewhere.  The use of a small noise limit is well motivated, since in this limit good sampling of the state space becomes most challenging.  The proof exploits a regenerative structure, and  a number of  new techniques are needed to turn large deviation estimates over a regenerative cycle into estimates for the empirical measure and its moments.

\end{abstract}


\section{Introduction}

\label{sec:introduction}

Among the many interesting results proved by Freidlin and Wentzell in the 70's
and 80's concerning small random perturbations of dynamical systems, one of
particular note is the
large deviation principle for the invariant measure of such a system. Consider
the small noise diffusion
\[
dX_{t}^{\varepsilon}=b(X_{t}^{\varepsilon})dt+\sqrt{\varepsilon}\sigma
(X_{t}^{\varepsilon})dW_{t},\quad X_{0}^{\varepsilon}=x,
\]
where $X_{t}^{\varepsilon}\in\mathbb{R}^{d}$, $b:\mathbb{R}^{d}\rightarrow
\mathbb{R}^{d}$, $\sigma:\mathbb{R}^{d}\rightarrow\mathbb{R}^{d}%
\times\mathbb{R}^{k}$ (the $d\times k$ matrices) and $W_{t}\in\mathbb{R}^{k}$
is a standard Brownian motion. Under mild regularity conditions on $b$ and
$\sigma$, one has that for any $T\in(0,\infty)$ the processes $\{X_{\cdot
}^{\varepsilon}\}_{\varepsilon>0}$ satisfy a large deviation principle on
$C([0,T]:\mathbb{R}^{d})$ with rate function
\[
I_{T}(\phi)\doteq\int_{0}^{T}\sup_{\alpha\in\mathbb{R}^{d}}\left[
\langle \dot{\phi}_{t},\alpha\rangle -\left\langle b(\phi
_{t}),\alpha\right\rangle -\frac{1}{2}\left\Vert \sigma(\phi_{t}%
)\alpha\right\Vert ^{2}\right]  dt
\]
when $\phi$ is absolutely continuous and $\phi(0)=x$, and $I_{T}(\phi)=\infty$
otherwise. If $\sigma(x)\sigma(x)^{\prime}>0$ (in the sense of symmetric
square matrices) for all $x\in\mathbb{R}^{d}$, then one can evaluate the
supremum and find
\begin{equation}\label{eqn:ratefn}
I_{T}(\phi)=\int_{0}^{T}\frac{1}{2}\left\langle \dot{\phi}_{t}-b(\phi
_{t}),\left[  \sigma(\phi_{t})\sigma(\phi_{t})^{\prime}\right]  ^{-1}%
(\dot{\phi}_{t}-b(\phi_{t}))\right\rangle dt.
\end{equation}
To simplify the discussion we will assume this non-degeneracy condition. It is
also assumed by Freidlin and Wentzell in \cite{frewen2}, but can be weakened.

Define the \textbf{quasipotential} $V(x,y)$ for $x,y \in \mathbb{R}^{d}$ by
\[
V(x,y)\doteq\inf\left\{  I_{T}(\phi):\phi(0)=x,\phi(T)=y,T<\infty\right\}  .
\]
Suppose that $\{X^{\varepsilon}\}$ is ergodic on a compact manifold
$M\subset\mathbb{R}^{d}$ with invariant measure $\mu^{\varepsilon}%
\in\mathcal{P}(M)$. Then under a number of additional assumptions, including
assumptions on the structure of the dynamical system $\dot{X}_{t}^{0}%
=b(X_{t}^{0})$, Freidlin and Wentzell \cite[Chapter 6]{frewen2} show how to
construct a function $J:M\rightarrow\lbrack0,\infty]$ in terms of $V$, such
that $J$ is the large deviation rate function for $\{\mu^{\varepsilon
}\}_{\varepsilon>0}$: $J$ has compact level sets, and
\[
\liminf_{\varepsilon\rightarrow0}\varepsilon\log\mu^{\varepsilon}(G)\geq
-\inf_{y\in G}J(y)\text{ for open }G\subset M,
\]%
\[
\limsup_{\varepsilon\rightarrow0}\varepsilon\log\mu^{\varepsilon}(F)\leq
-\inf_{y\in F}J(y)\text{ for closed }F\subset M.
\]
This gives a very useful approximation to $\mu^{\varepsilon}$, and along the
way many interesting related results (e.g., transition rates between
metastable states) are also developed.

The aim of this paper is to develop large deviation type estimates for a
quantity that is closely related to $\mu^{\varepsilon}$, which is the empirical
measure over an interval $[0,T^{\varepsilon}]$. This is defined by
\begin{equation}\label{eqn:defofrho}
 \rho^{\varepsilon}(A)\doteq \frac{1}{T^{\varepsilon}}\int_{0}^{T^{\varepsilon}}%
1_{A}(X_{s}^{\varepsilon})ds   
\end{equation}
for $A\in\mathcal{B}(M)$. For reasons that will be made precise later on, we
will assume $T^{\varepsilon}\rightarrow\infty$ as $\varepsilon\rightarrow0$,
and typically $T^{\varepsilon}$ will grow exponentially in the form
$e^{c/\varepsilon}$ for some $c>0$.

There is of course a large deviation theory for the empirical measure when
$\varepsilon>0$ is held fixed and the length of the time interval tends to
infinity (see e.g., \cite{donvar1,donvar3}). However, it can be hard to extract information from the corresponding rate
function. Our interest in proving large deviations estimates when
$\varepsilon\rightarrow0$ and $T^{\varepsilon}\rightarrow\infty$ is 
in the hope that 
one will find it easier to extract information in this double limit,
analogous to the simplified approximation to $\mu^{\varepsilon}$ just
mentioned. These results will be applied in \cite{dupwu} to analyze and optimize a Monte Carlo
method known as infinite swapping \cite{dupliupladol,doldupnyq} when the noise
is small.
Small noise models are common in applications, 
and are also the setting
in which Monte Carlo methods can have the greatest difficulty. We expect
that the general set of results will be useful for other purposes as well.

We note that while developed in the context of small noise diffusions, the
collection of results due to Freidlin and Wentzell that are discussed in
\cite{frewen2} also hold for other classes of processes, such as scaled
stochastic networks, when appropriate conditions are assumed and the finite
time sample path large deviation results are available (see, e.g.,
\cite{shwwei}). We expect that such generalizations are possible for the
results we prove as well.  

The outline of the paper is as follows. In Section \ref{sec:quantities_of_interest} we explain our motivation and the relevance 
for studying the particular quantities that are the topic of the paper. In Section \ref{sec:setting_of_the_problem} we provide definitions
and assumptions that are used throughout the paper, and Section \ref{sec:results_and_conjectures} states
the main asymptotic results  as well as a related conjecture. 
Examples that illustrate the results are given in Section \ref{sec:examples}.
In Section \ref{sec:wald's_identities} we introduce an important tool
for our analysis --- the regenerative structure, and with this concept, we decompose the original asymptotic problem into two sub-problems 
that require very different forms  of analysis. These two types of asymptotic problems are then analyzed separately in Sections \ref{sec:asymptotics_of_moments_of_S} and
Section \ref{sec:moments_of_the_number_of_renewals}. In Section \ref{subsec:lower_bound_for_performance} we combine the partial asymptotic results
from Section \ref{sec:asymptotics_of_moments_of_S} and
Section \ref{sec:moments_of_the_number_of_renewals} to prove the main large deviation type results
that were stated in Section \ref{sec:results_and_conjectures}. 
Section \ref{sec:exponential__returning_law_and_tail_behavior} gives the proof of a key theorem from  Section \ref{sec:moments_of_the_number_of_renewals},
which asserts an approximately exponential distribution for 
return times that arise in the decomposition based on 
regenerative structure,
as well as a tail bound needed for some integrability arguments.
The last section of the paper, Section \ref{sec:upper_bound_for_performance}, presents the proof of an upper bound for the rate of decay of the variance per unit time in the context of a special case,
thereby showing for the case that the lower bounds of Section 
\ref{sec:results_and_conjectures} are in a sense tight.
To focus on the main discussion, proofs of some lemmas are collected in an Appendix.

\begin{remark}
There are certain time-scaling parameters that play key roles throughout this paper. 
For the reader's convenience, we record here where they are first described: $h_1$ and $w$ are defined in \eqref{eqn:defofh} and \eqref{eqn:defofw}; $c$ is introduced and its relation to $h_1$ and $w$ are given in Theorem \ref{Thm:4.1};  $m$ is introduced at the beginning of Subsection \ref{subsection:multicycle}.
\end{remark}

\section{Quantities of Interest}

\label{sec:quantities_of_interest}

The quantities we are interested in are the higher order moments, and in
particular second moments, of an integral of a risk-sensitive functional with
respect to the empirical measure $\rho^{\varepsilon}$
defined in \eqref{eqn:defofrho}. 
To be more precise, the
integral is of the form
\begin{equation}
\int_{M}e^{-\frac{1}{\varepsilon}f\left(  x\right)  }1_{A}\left(  x\right)
\rho^{\varepsilon}\left(  dx\right)  \label{integral_1}%
\end{equation}
for some nice (e.g., bounded and continuous) function $f:M\rightarrow
\mathbb{R}$ and a closed set $A\in\mathcal{B}(M)$. Note that this integral can
also be expressed as
\begin{equation}
\frac{1}{T^{\varepsilon}}\int_{0}^{T^{\varepsilon}}e^{-\frac{1}{\varepsilon
}f\left(  X_{t}^{\varepsilon}\right)  }1_{A}\left(  X_{t}^{\varepsilon
}\right)  dt. \label{integral_2}%
\end{equation}

In order to understand the large deviation behavior of moments of such an
integral, we must identify the correct scaling to extract meaningful
information. Moreover, as will be shown, there is an important difference
between centered moments and ordinary (non-centered) moments.

By the use of the regenerative structure of $\{X_{t}^{\varepsilon}\}_{t\geq0}%
$, we can decompose (\ref{integral_2}) [equivalently (\ref{integral_1})] into
the sum of a random number of independent and identically distributed (iid) random variables,
plus a residual term which here we will ignore. 
To
simplify the notation, we temporarily drop the $\varepsilon$, 
and without being precise about how the regenerative structure is
introduced, let $Y_{j}$ denote the integral of $e^{-\frac{1}{\varepsilon
}f\left(  X_{t}^{\varepsilon}\right)  }1_{A}\left(  X_{t}^{\varepsilon
}\right)  $ over a regenerative cycle. (The specific regenerative structure we
use will be identified later on.)

Thus we consider a sequence $\{Y_{j}\}_{j\in%
\mathbb{N}
}$ of iid\ random variables with finite second moments, and want to compare
the scaling properties of, for example, the second moment and the second
centered moment of $\frac{1}{n}\sum_{j=1}^{n}Y_{j}$. When used for the small
noise system, both $n$ and moments of $Y_{i}$ will scale exponentially in
$1/\varepsilon$, and $n$ will be random, but for now we assume $n$ is
deterministic. The second moment is%
\begin{align*}
E\left(  \frac{1}{n}\sum_{k=1}^{n}Y_{k}\right)  ^{2}    =\frac{1}{n^{2}}%
\sum_{k=1}^{n}E\left(  Y_{k}\right)  ^{2}+\frac{1}{n^{2}}\sum_{i,j:i\neq
j}E\left(  Y_{i}Y_{j}\right) 
  =\left(  EY_{1}\right)  ^{2}+\frac{1}{n}\mathrm{Var}\left(  Y_{1}\right)  ,
\end{align*}
and the second centered moment is
\[
E\left(  \frac{1}{n}\sum_{k=1}^{n}\left(  Y_{k}-EY_{1}\right)  \right)
^{2}=\mathrm{Var}\left(  \frac{1}{n}\sum_{k=1}^{n}Y_{k}\right)  =\frac{1}%
{n}\mathrm{Var}\left(  Y_{1}\right)  .
\]

When analyzing the performance of the Monte Carlo schemes one is concerned of
course with both bias and variance, but in situations where we would
like to apply the results of this paper one  assumes $T^{\varepsilon}$ is
large enough that the bias term is unimportant, so that all we are concerned
with is the variance. However some care will be needed to determine a suitable
measure of quality of the algorithm, since as noted $Y_{i}$ could scale
exponentially with in $1/\varepsilon$ with a negative coefficient
(exponentially small), while $n$ will be exponentially large. 

In the analysis
of unbiased accelerated Monte Carlo methods for small noise systems over bounded
time intervals (e.g., to estimate escape probabilities), it is standard to use the second moment, which is often easier
to analyze, in lieu of the variance \cite[Chapter
VI]{asmgly}, \cite[Chapter
14]{buddup4}. This situation corresponds to $n=1$.
The alternative criterion is more convenient since by Jensen's inequality one can easily
establish a best possible rate of decay of the second moment, and estimators
are deemed efficient if they possess the optimal rate of decay  \cite{asmgly,buddup4}. However with $n$
exponentially large this is no longer true. Using the previous calculations,
we see that the second moment of $\frac{1}{n}\sum_{j=1}^{n}Y_{j}$ can be
completely dominated by $\left(  EY_{1}\right)  ^{2}$, and therefore using
this quantity to compare algorithms may be misleading, since our true concern
is the variance of $\frac{1}{n}\sum_{j=1}^{n}Y_{j}$.

This observation suggests that our study of moments of the empirical measure
we should consider only centered moments, and in particular quantities like
\[
T^{\varepsilon}\mathrm{Var}\left(  \int_{M}e^{-\frac{1}{\varepsilon
}f\left(  x\right)  }1_{A}\left(  x\right)  \rho^{\varepsilon}\left(
dx\right)  \right)  =T^{\varepsilon}\mathrm{Var}\left(  \frac
{1}{T^{\varepsilon}}\int_{0}^{T^{\varepsilon}}e^{-\frac{1}{\varepsilon
}f\left(  X_{t}^{\varepsilon}\right)  }1_{A}\left(  X_{t}^{\varepsilon
}\right)  dt\right)  ,
\]
which is the variance per unit
time. For Monte Carlo one wants to minimize the variance per unit time,
and to make the problem more tractable we instead try to maximize the decay rate of the variance per unit time.
Assuming the limit exists, 
this is defined by 
\[
\lim_{\varepsilon\rightarrow 0}-\varepsilon \log \left[ T^{\varepsilon}\mathrm{Var}\left(  \frac
{1}{T^{\varepsilon}}\int_{0}^{T^{\varepsilon}}e^{-\frac{1}{\varepsilon
}f\left(  X_{t}^{\varepsilon}\right)  }1_{A}\left(  X_{t}^{\varepsilon
}\right)  dt\right) \right]
\]
and so we are
especially interested in lower bounds on this decay rate.

Thus our goal is to develop methods that allow the approximation of at least first and
second moments of \eqref{integral_2}. In fact, the methods we introduce can be developed
further to obtain large deviation estimates of higher moments if
that were needed or desired.

\section{Setting of the Problem, Assumptions and Definitions}

\label{sec:setting_of_the_problem}

The process model we would like to consider is an $\mathbb{R}^{d}$-valued
solution to an It\^{o} stochastic differential equation (SDE), where the drift so
strongly returns the process to some compact set that events involving exit of
the process from some larger compact set are so rare that they can effectively be ignored when analyzing the
empirical measure. However, to simplify the
analysis we follow the convention of \cite[Chapter 6]{frewen2}, and work with
a small noise diffusion that takes values in a compact and connected
manifold $M\subset\mathbb{R}^{d}$ of dimension $r$ and with smooth boundary.
The precise regularity assumptions for $M$ are given on \cite[page
135]{frewen2}.
With this convention in mind, we consider a family of 
diffusion processes $\{X^{\varepsilon}\}_{\varepsilon\in(0,\infty
)},X^{\varepsilon}\in C([0,\infty):M)$,
that satisfy the following condition.

\begin{condition}
\label{Con:3.1} Consider continuous $b:M\rightarrow\mathbb{R}^{d}$ and
$\sigma:M\rightarrow\mathbb{R}^{d}\times\mathbb{R}^{d}$ (the $d\times d$
matrices), and assume that $\sigma$ is uniformly nondegenerate, in that there
is $c>0$ such that for any $x$ and any $v$ in the tangent space of $M$ at $x$,
$\langle v,\sigma(x)\sigma(x)^{\prime}v\rangle\geq c\langle v,v\rangle$. For
absolutely continuous $\phi\in C([0,T]:M)$ define $I_{T}(\phi)$ by
\eqref{eqn:ratefn},
where the inverse $\left[  \sigma(x)\sigma(x)^{\prime}\right]  ^{-1}$ is
relative to the tangent space of $M$ at $x$. Let $I_{T}(\phi)=\infty$ for all
other $\phi\in C([0,T]:M)$. Then we assume that for each $T<\infty$,
$\{X^{\varepsilon}_t\}_{0\leq t\leq T}$ satisfies the large deviation
principle with rate function $I_{T}$, uniformly with respect to the initial condition \cite[Definition 1.13]{buddup4}.
\end{condition}

We note that for such diffusion processes nondegeneracy of the diffusion
matrix implies there is a unique invariant measure $\mu^{\varepsilon}%
\in\mathcal{P}(M)$.
A discussion of weak sufficient conditions under which Condition \ref{Con:3.1} holds appears in \cite[Section 3, Chapter 5]{frewen2}.

\begin{remark}
There are several ways one can approximate a diffusion of the sort described
at the beginning of this section by a diffusion on a smooth compact manifold.
One such ``compactification'' of the state space can be obtained by assuming
that for some bounded but large enough rectangle trajectories that exit the
rectangle do not affect the large deviation behavior of quantities of
interest, and then extend the coefficients of the process periodically and
smoothly off an even larger rectangle to all of $\mathbb{R}^{d}$ (a technique
sometimes used to bound the state space for purposes of numerical
approximation). One can then map $\mathbb{R}^{d}$ to a manifold that is
topologically equivalent to a torus, and even arrange that the metric
structure on the part of the manifold corresponding to the smaller rectangle
coincides with a Euclidean metric.
\end{remark}

Define the \textbf{quasipotential} $V(x,y):M\times M\rightarrow\lbrack
0,\infty)$ by
\begin{equation}\label{eqn:QP}
    V(x,y)\doteq\inf\left\{  I_{T}(\phi):\phi(0)=x,\phi(T)=y,T<\infty\right\}  .
\end{equation}

For a given set $A\subset M,$ define $V(x,A)\doteq\inf_{y\in A}V(x,y)$ and
$V(A,y)\doteq\inf_{x\in A}V(x,y).$

\begin{remark}
\label{Rmk:3.3}For any fixed $y$ and set $A,$ $V(x,y)$ and $V(x,A)$ are both
continuous functions of $x$. Similarly, for any given $x$ and any set $A,$
$V(x,y)$ and $V(A,y)$ are also continuous in $y.$
\end{remark}

\begin{definition}
We say that a set $N\subset M$ is \textbf{stable} if for any $x\in N,y\notin
N$ we have $V(x,y)>0.$ A set which is not stable is called \textbf{unstable}.
\end{definition}

\begin{definition}
We say that $O\in M$ is an \textbf{equilibrium point} of the
ordinary differential equation (ODE)
$\dot{x}_{t}=b(x_{t})$ if $b(O)=0.$
Moreover, we say that this equilibrium point $O$ is \textbf{asymptotically
stable} if for every neighborhood $\mathcal{E}_{1}$ of $O$ (relative to $M$)
there exists a smaller neighborhood $\mathcal{E}_{2}$ such that the
trajectories of system $\dot{x}_{t}=b(x_{t})$ starting in $\mathcal{E}_{2}$
converge to $O$ without leaving $\mathcal{E}_{1}$ as $t\rightarrow\infty.$
\end{definition}

\begin{remark}
An asymptotically stable equilibrium point is a stable set, but a stable set
might contain no asymptotically stable equilibrium point.
\end{remark}

The following restrictions on the structure of the
dynamical system in $M$ will be used.
These restrictions include the assumption that the equilibrium points 
are a finite collection.
This is a more restrictive framework than that of \cite{frewen2},
which allows, e.g., limit cycles. 
In a remark at the end of this section we comment on what would be needed to extend to the general setup of \cite{frewen2}.

\begin{condition}
\label{Con:3.2}There exists a finite number of points  $\{O_{j}%
\}_{j\in L} \subset M$ with $L\doteq\{1,2,\ldots,l\}$ for some $l\in%
\mathbb{N}
$, such that
$\cup_{j\in L}\{O_j\}$ coincides with the $\omega$-limit set of the ODE $\dot{x}_{t}=b(x_{t})$.
\end{condition}

Without loss of generality, we may assume that $O_{j}$ is stable if and only
if $j\in L_{\rm{s}}$ where $L_{\rm{s}}\doteq\{1,\ldots,l_{\rm{s}}\}$ for some $l_{\rm{s}}\leq l.$

\vspace{\baselineskip}
Next we give a definition from graph theory which will be used in the
statement of the main results.

\begin{definition}
\label{Def:3.3}Given a subset $W\subset L=\{1,\ldots,l\},$ a directed graph
consisting of arrows $i\rightarrow j$ $(i\in L\setminus W,j\in L,i\neq j)$ is
called a $W$\textbf{-graph on }$L$ if it satisfies the following conditions.

\begin{enumerate}
\item Every point $i$ $\in L\setminus W$ is the initial point of exactly one arrow.

\item For any point $i$ $\in L\setminus W,$ there exists a sequence of
arrows leading from $i$ to some point in $W.$
\end{enumerate}
\end{definition}

We note that we could replace the second condition by the requirement that 
there are no closed cycles in the graph.
We denote by $G(W)$ the set of $W$-graphs; we shall use the letter $g$ to
denote graphs. Moreover, if $p_{ij}$ ($i,j\in L,j\neq i$) are numbers, then
$\prod_{(i\rightarrow j)\in g}p_{ij}$ will be denoted by $\pi(g).$

\begin{remark}
\label{Rmk:3.1}We mostly consider the set of $\{i\}$-graphs,
i.e., $G(\{i\})$ for some $i\in$ $L$, and also use $G(i)$ to denote
$G(\{i\}).$ We occasionally consider the set of $\{i,j\}$-graphs, i.e.,
$G(\{i,j\})$ for some $i,j\in$ $L$ with $i\neq j.$ Again, we also use $G(i,j)$
to denote $G(\{i,j\}).$
\end{remark}

\begin{definition}\label{def:defofWs}
For all $j\in L$, define
\begin{equation}
W\left(  O_{j}\right)  \doteq\min_{g\in G\left(  j\right)  }\left[
{\textstyle\sum_{\left(  m\rightarrow n\right)  \in g}}
V\left(  O_{m},O_{n}\right)  \right]   \label{eqn:Wtwarg}%
\end{equation}
and 
\begin{equation}
W\left( O_1\cup O_{j}\right)  \doteq\min_{g\in G\left(  1,j\right)  }\left[
{\textstyle\sum_{\left(  m\rightarrow n\right)  \in g}}
V\left(  O_{m},O_{n}\right)  \right]  . 
\label{eqn:Wtwarg_2}
\end{equation}
\end{definition}

\begin{remark}
\label{rmk:roleofW}
Heuristically, if we interpret
$V\left(  O_{m},O_{n}\right)  $ as the \textquotedblleft
cost\textquotedblright\ of moving from $O_{m}$ to $O_{n},$ then $W\left(
O_{j}\right)  $ is the \textquotedblleft least total cost\textquotedblright%
\ of reaching $O_{j}$ from every $O_{i}$ with $i\in L\setminus\{j\}.$
According to \cite[Theorem 4.1, Chapter 6]{frewen2},
one can interpret $W(O_i)-\min_{j \in L}W(O_j)$ as the decay rate of $\mu^\varepsilon(B_\delta (O_i))$,
where $B_\delta (O_i)$ is a small open neighborhood of $O_i$.
\end{remark}

\begin{definition}
\label{Rmk:3.2}We use $G_{\rm{s}}\left(  W\right)  $ to denote the collection of
all $W$-graphs on $L_{\rm{s}}=\{1,\ldots,l_{\rm{s}}\}$ with $W\subset L_{\rm{s}}.$
\end{definition}

W make the following technical assumptions on
the structure of the SDE. Let $B_{\delta}(K)$ denote the $\delta
$-neighborhood of a set $K\subset M.$ Recall that $\mu^{\varepsilon}$
is the unique invariant probability measure of the diffusion process $\{X^{\varepsilon}_t\}_{t}.$
The existence of the limits appearing in the first part of the condition 
is ensured by Theorem 4.1
in \cite[Chapter 6]{frewen2}.

\begin{condition}
\label{Con:3.3}

\begin{enumerate}

\item There exists a unique asymptotically stable equilibrium point $O_{1}$
of the system $\dot{x}_{t}=b(x_{t})$ such that
\[
\lim_{\delta\rightarrow0}\lim_{\varepsilon\rightarrow0}-\varepsilon\log
\mu^{\varepsilon}(B_{\delta}(O_{1}))=0,
\text{ and } 
\lim_{\delta\rightarrow0}\lim_{\varepsilon\rightarrow0}-\varepsilon\log
\mu^{\varepsilon}(B_{\delta}(O_{j}))>0 \mbox{ for any }j\in L\setminus\{1\}.
\]

\item 
All of the
eigenvalues of the matrix of partial derivatives of $b$ at $O_\ell$ relative to $M$ have negative real parts for 
$\ell \in L_{\rm{s}}$.

\item 
$b:M\rightarrow\mathbb{R}^{d}$ and
$\sigma:M\rightarrow\mathbb{R}^{d}\times\mathbb{R}^{d}$ are $C^{1}$.

\end{enumerate}
\end{condition}

\begin{remark}
\label{Rmk:4.1}According to \cite[Theorem 4.1, Chapter 6]{frewen2} and the first part of Condition \ref{Con:3.3}, we know that
$W(O_{j})>W(O_{1})$ for all $j\in L\setminus\{1\}.$
\end{remark}

\begin{remark}
\label{rk:onconds}
We comment on the use of the various parts of the condition. 
Part 1 means that neighborhoods of $O_{1}$ capture more of the mass as $\varepsilon\rightarrow 0$ 
than neighborhoods of any other equilibrium point. 
It simplifies the analysis greatly,
but we expect it could be weakened if desired.
Parts 2 and 3 are assumed in \cite{day4}, which gives an
explicit exponential bound on the tail probability of the exit time from the domain of attraction.
It is largely because of our reliance on the results of \cite{day4} that we must assume that equilibrium sets are points in Condition \ref{Con:3.2}, 
rather than the more general compacta as considered in 
\cite{frewen2}.
Both Condition \ref{Con:3.2} and Condition \ref{Con:3.3} could be weakened if the corresponding versions
of the results we use from \cite{day4} were available.
\end{remark}

\begin{remark}
The quantities $V(O_i,O_j)$ determine various key transition probabilities and time scales in the analysis of the empirical
measure.
The more general framework of \cite{frewen2}, as well as the one dimensional case (i.e., $r=1$) in the present setting, require some closely related but slightly more complicated quantities.
These are essentially the analogues of $V(O_i,O_j)$ under the assumption that trajectories used in the definition 
are not allowed to pass through equilibrium compacta (such as a limit cycle) when traveling from 
$O_i$ to $O_j$.
The related quantities,
which are designated using notation of the form $\tilde{V}(O_i,O_j)$ in \cite{frewen2},
are needed since the probability of a direct transition from $O_i$ to $O_j$ without passing though another equilibrium structure may be zero, 
which means that transitions from $O_i$ to $O_j$ must be decomposed according to these intermediate  transitions.
To simplify the presentation we do not provide the details of the one dimensional case in our setup,
but simply note that it can be handled by the introduction of these additional quantities.
\end{remark}

Consider the filtration $\{\mathcal{F}_{t}\}_{t\geq0}$ defined by
$\mathcal{F}_{t}\doteq\sigma(X_{s}^{\varepsilon},s\leq t)$ for any $t\geq0.$
For any $\delta>0$ smaller than a quarter of the minimum of the distances
between $O_{i}$ and $O_{j}$ for all $i\neq j$, we consider two types of
stopping times with respect to the filtration $\{\mathcal{F}_{t}\}_{t}$. The
first type are the hitting times of $\{X^{\varepsilon}_t\}_{t}$ at the
$\delta$-neighborhood of all equilibrium points $\{O_{j}\}_{j\in L}$ after
traveling a reasonable distance away from those neighborhoods. More precisely,
we define stopping times by $\tau_{0}\doteq0,$
\[
\sigma_{n}\doteq\inf\{t>\tau_{n}:X_{t}^{\varepsilon}\in
{\cup
_{j\in L}}\partial B_{2\delta}(O_{j})\}
\text{ and }\tau_{n}\doteq\inf\{t>\sigma_{n-1}:X_{t}^{\varepsilon}\in{\cup
_{j\in L}}
\partial B_{\delta}(O_{j})\}.
\]
The second type of stopping times are the return times of $\{X^{\varepsilon
}_t\}_{t}$ to the $\delta$-neighborhood of $O_{1}$, where as noted previously
$O_{1}$ is in some sense the most important equilibrium point. The exact
definitions are $\tau_{0}^{\varepsilon}\doteq 0,$
\begin{equation}
\label{eqn:sigma}
\sigma_{n}^{\varepsilon}\doteq\inf\{t>\tau_{n}^{\varepsilon}:X_{t}%
^{\varepsilon}\in%
{\textstyle\cup
_{j\in L\setminus\{1\}}}
\partial B_{\delta}(O_{j})\}
\text{ and }
\tau_{n}^{\varepsilon}\doteq\inf\left\{  t>\sigma_{n-1}^{\varepsilon}%
:X_{t}^{\varepsilon}\in\partial B_{\delta}(O_{1})\right\}.
\end{equation}
We then define two embedded Markov chains $\{Z_{n}\}_{n\in%
\mathbb{N}
_{0}}\doteq\{X_{\tau_{n}}^{\varepsilon}{}\}_{n\in%
\mathbb{N}
_{0}}$ with state space $%
{\textstyle\cup
_{j\in L}}
\partial B_{\delta}(O_{j})$, and $\{Z_{n}^{\varepsilon}\}_{n\in%
\mathbb{N}
_{0}}\doteq\{X_{\tau_{n}^{\varepsilon}}^{\varepsilon}{}\}_{n\in%
\mathbb{N}
_{0}}$ with state space $\partial B_{\delta}(O_{1}).$

Let $p(x,\partial B_{\delta}(O_{j}))$ denote the one-step transition
probabilities of $\{Z_{n}\}_{n\in%
\mathbb{N}
_{0}}$ starting from a point $x\in%
{\textstyle\cup
_{i\in L}}
\partial B_{\delta}(O_{i}),$ namely,%
\[
p(x,\partial B_{\delta}(O_{j}))\doteq P_{x}(Z_{1}\in\partial B_{\delta}%
(O_{j})).
\]
We have the following estimates on $p(x,\partial B_{\delta}(O_{j}))$ in terms
of $V$. The lemma is a consequence of \cite[Lemma 2.1, Chapter
6]{frewen2} and the fact that under our conditions 
$V(O_i,O_j)$ and $\tilde{V}(O_i,O_j)$ as defined in \cite{frewen2} coincide.

\begin{lemma}
\label{Lem:3.3}For any $\eta>0,$ there exists $\delta_{0}\in(0,1)$ and
$\varepsilon_{0}\in(0,1),$ such that for any $\delta\in(0,\delta_{0})$ and
$\varepsilon\in(0,\varepsilon_{0}),$ for all $x\in\partial B_{\delta}(O_{i}),$
the one-step transition probability of the Markov chain $\{Z_{n}\}_{n\in%
\mathbb{N}
}$ on $\partial B_{\delta}(O_{j})$ satisfies%
\[
e^{-\frac{1}{\varepsilon}\left(  V\left(  O_{i},O_{j}\right)  +\eta\right)
}\leq p(x,\partial B_{\delta}(O_{j}))\leq e^{-\frac{1}{\varepsilon}\left(
V\left(  O_{i},O_{j}\right)  -\eta\right)  }%
\]
for any $i,j\in L.$
\end{lemma}

\begin{remark}
\label{Rmk:3.5}According to Lemma 4.6 in \cite{kha2}, Condition \ref{Con:3.1}
guarantees the existence and uniqueness of invariant measures for
$\{Z_{n}\}_{n}$ and $\{Z_{n}^{\varepsilon}\}_{n}.$ We use $\nu^{\varepsilon
}\in\mathcal{P}(\cup_{i\in L}\partial B_{\delta}(O_{i}))$ and $\lambda
^{\varepsilon}\in\mathcal{P}(\partial B_{\delta}(O_{1}))$ to denote the
associated invariant measures.
\end{remark}

\section{Results and a Conjecture}

\label{sec:results_and_conjectures}

The following main results of this paper assume Conditions
\ref{Con:3.1}, \ref{Con:3.2} and \ref{Con:3.3}. 
Although moments
higher than the second moment are not considered in this paper, as noted previously one can
use arguments such as those used here to identify and prove the analogous
results.

Recall that $\{O_{j}\}_{j\in L}$ are the equilibrium points and that they 
satisfy Condition \ref{Con:3.2} and Condition \ref{Con:3.3}. In addition,
$O_{j}$ is stable if and only if $j\in L_{\rm{s}}$, where $L_{\rm{s}}\doteq \{1,\ldots,l_{\rm{s}}\}$
for some $l_{\rm{s}}\leq l=\left\vert L\right\vert$, 
and $\tau^\varepsilon_1$ is the first return time to the $\delta$-neighborhood of $O_1$ after having first visited the $\delta$-neighborhood of any other equilibrium point.

\begin{lemma}
\label{Lem:4.1}For any $\delta\in(0,1)$ smaller than a quarter of the minimum
of the distances between $O_{i}$ and $O_{j}$ for all $i\neq j$, any
$\varepsilon>0$ and any nonnegative measurable function $g:M\rightarrow%
\mathbb{R}
$ %
\[
E_{\lambda^{\varepsilon}}\left(  \int_{0}^{\tau_{1}^{\varepsilon}}g\left(
X_{s}^{\varepsilon}\right)  ds\right)  =E_{\lambda^{\varepsilon}}\tau
_{1}^{\varepsilon}\cdot\int_M g\left(  x\right)  \mu^{\varepsilon}\left(
dx\right)  ,
\]
where $\lambda^{\varepsilon}\in\mathcal{P}(\partial B_{\delta}(O_{1}))$ is the
unique invariant measure of $\{Z_{n}^{\varepsilon}\}_{n}=\{X_{\tau
_{n}^{\varepsilon}}^{\varepsilon}\}_{n}$ and $\mu^{\varepsilon}\in
\mathcal{P}(M)$ is the unique invariant measure of $\{X_{t}^{\varepsilon
}\}_{t}.$
\end{lemma}

\begin{proof}
We define a measure on $M$ by
\[
\hat{\mu}^{\varepsilon}\left(  B\right)  \doteq E_{\lambda^{\varepsilon}%
}\left(  \int_{0}^{\tau_{1}^{\varepsilon}}1_{B}\left(  X^{\varepsilon}_t  \right)  dt\right)
\]
for $B\in\mathcal{B}(M),$ so that for any nonnegative measurable function
$g:M\rightarrow%
\mathbb{R}
$%
\[
\int_{M}g\left(  x\right)  \hat{\mu}^{\varepsilon}\left(  dx\right)
=E_{\lambda^{\varepsilon}}\left(  \int_{0}^{\tau_{1}^{\varepsilon}}g\left(
X^{\varepsilon}_t  \right)  dt\right)  .
\]
According to the proof of Theorem 4.1 in \cite{kha2}, the measure given by
$\hat{\mu}^{\varepsilon}\left(  B\right)  /\hat{\mu}^{\varepsilon}\left(
M\right)  $ is an invariant measure of $\{X^{\varepsilon}_t\}_{t}.$ Since we
already know that $\mu^{\varepsilon}$ is the unique invariant measure of
$\{X^{\varepsilon}_t\}_{t},$ this means that $\mu^{\varepsilon}(B)=\hat{\mu
}^{\varepsilon}\left(  B\right)  /\hat{\mu}^{\varepsilon}\left(  M\right)  $
for any $B\in\mathcal{B}(M).$ Therefore for any nonnegative measurable
function $g:M\rightarrow%
\mathbb{R}
$%
\begin{align*}
E_{\lambda^{\varepsilon}}\left(  \int_{0}^{\tau_{1}^{\varepsilon}}g\left(
X^{\varepsilon}_t  \right)  dt\right)   &  =\int_{M}g\left(
x\right)  \mu^{\varepsilon}\left(  dx\right)  \cdot\hat{\mu}^{\varepsilon
}\left(  M\right) 
=\int_{M}g\left(  x\right)  \mu^{\varepsilon}\left(  dx\right)  \cdot
E_{\lambda^{\varepsilon}}\tau_{1}^{\varepsilon}.
\end{align*}
\end{proof}

\vspace{\baselineskip}
Recall the definitions of $W(O_j)$ and $W(O_1\cup O_j)$ in Definition \ref{def:defofWs}, as well as the definition of the quasipotential $V(x,y)$ in \eqref{eqn:QP}.
For any $k\in L$, we define 
\begin{equation}
\label{eqn:defofh}
h_k\doteq \min_{j\in L\setminus\{k\}}V(O_{k},O_{j}). 
\end{equation}
In addition, define 
\begin{equation}
\label{eqn:defofw}
    w\doteq W(O_1)-\min_{j\in L\setminus\{1\}}W(O_1\cup O_j).
\end{equation}

\begin{remark}
The quantity $h_k$ is related to the time that it takes for the process to leave a neighborhood of $O_k$, and $W(O_1)-W(O_1\cup O_j)$ is related to the transition time from a neighborhood of $O_j$ to one of $O_1$. It turns out that our results and arguments depend on which of $h_1$ or $w$ is larger. Throughout the paper, the constructions used in the case when $h_1>w$ will be in terms of what we call a \textbf{single cycle}, and those for the case when $h_1\leq w$ in terms of a \textbf{multicycle}.
\end{remark}

\begin{theorem}
\label{Thm:4.1}Let $T^{\varepsilon}=e^{\frac{1}{\varepsilon}c}$ for some
$c>h_1\vee w$. Given  $\eta>0,$ a continuous
function $f:M\rightarrow%
\mathbb{R}
$ and any compact set $A\subset M,$ there exists $\delta_{0}\in(0,1)$ such
that for any $\delta\in(0,\delta_{0})$
\begin{align*}
&  \liminf_{\varepsilon\rightarrow0}-\varepsilon\log\left\vert E_{\lambda
^{\varepsilon}}\left(  \frac{1}{T^{\varepsilon}}\int_{0}^{T^{\varepsilon}%
}e^{-\frac{1}{\varepsilon}f\left(  X_{t}^{\varepsilon}\right)  }1_{A}\left(
X_{t}^{\varepsilon}\right)  dt\right)  -\int_M e^{-\frac{1}{\varepsilon}f\left(
x\right)  }1_{A}\left(  x\right)  \mu^{\varepsilon}\left(  dx\right)
\right\vert \\
&  \qquad\geq\inf_{x\in A}\left[  f\left(  x\right)  +W\left(  x\right)
\right]  -W\left(  O_{1}\right)  +c-(h_1\vee w)-\eta,
\end{align*}
where $W(x)\doteq \min_{j\in L}[W(O_{j})+V(O_{j},x)]$.
\end{theorem}

\begin{remark}
\label{Rmk:4.2}Since $W(x)=\min_{j\in L}[W(O_{j})+V(O_{j},x)],$ the lower
bound appearing in Theorem \ref{Thm:4.1} is equivalent to
\[
\min_{j\in L}\left(  \inf_{x\in A}\left[  f\left(  x\right)  +V(O_{j}%
,x)\right]  +W(O_{j})-W\left(  O_{1}\right)  \right)  +c-(h_1\vee w)-\eta.
\]

\end{remark}

The next result gives an upper bound on the variance per unit time, or
equivalently a lower bound on its rate of decay. In the design of a Markov
chain Monte Carlo method, one would maximize this rate of decay to improve the
method's performance.

\begin{theorem}
\label{Thm:4.2}
Let $T^{\varepsilon}=e^{\frac{1}{\varepsilon}c}$ for
some $c>h_1\vee w$. Given  $\eta>0,$
a continuous function $f:M\rightarrow%
\mathbb{R}
$ and any compact set $A\subset M,$ there exists $\delta_{0}\in(0,1)$ such
that for any $\delta\in(0,\delta_{0})$
\begin{align*}
&  \liminf_{\varepsilon\rightarrow0}-\varepsilon\log\left(  T^{\varepsilon
}\cdot\mathrm{Var}_{\lambda^{\varepsilon}}\left(  \frac{1}{T^{\varepsilon}}%
\int_{0}^{T^{\varepsilon}}e^{-\frac{1}{\varepsilon}f\left(  X_{t}%
^{\varepsilon}\right)  }1_{A}\left(  X_{t}^{\varepsilon}\right)  dt\right)
\right) \\
&  \qquad\geq
\begin{cases}
\min_{j\in L}\left(  R_{j}^{(1)}\wedge R_{j}^{(2)}\right)  -\eta,& \text{if } h_1>w\\
\min_{j\in L}\left(  R_{j}^{(1)}\wedge R_{j}^{(2)}\wedge R_{j}^{(3)}\right)  -\eta,& \text{otherwise}
\end{cases},
\end{align*}
where 
\[
R_{j}^{(1)}\doteq\inf_{x\in A}\left[  2f\left(  x\right)  +V\left(
O_{j},x\right)  \right]  +W\left(  O_{j}\right)  -W\left(  O_{1}\right)  ,
\]%
\[
R_{1}^{(2)}\doteq2\inf_{x\in A}\left[  f\left(  x\right)  +V\left(O_{1},x\right)  \right]  -h_1,
\]
and for $j\in L\setminus\{1\}$%
\begin{align*}
R_{j}^{(2)}  &  \doteq2\inf_{x\in A}\left[  f\left(  x\right)  +V\left(O_{j},x\right)  \right]  +W\left(  O_{j}\right)  -2W\left(  O_{1}\right) 
  +W(O_{1}\cup O_{j}),
\end{align*}
\[
R_{j}^{(3)}  \doteq2\inf_{x\in A}\left[  f\left(  x\right)  +V\left(O_{j},x\right)  \right]  +2W\left(  O_{j}\right)  -2W\left(  O_{1}\right)-w .
\]
 \end{theorem}

\begin{remark}
If one mistakenly treated a single cycle case as a multicycle case in the application of Theorem $\ref{Thm:4.2}$, then the result is the same since with $h_1>w$, \eqref{eqn:defofw} implies that $R_{j}^{(3)}\geq R_{j}^{(2)}$ for any $j\in L$.
\end{remark}

\begin{remark}
Although Theorems \ref{Thm:4.1} and \ref{Thm:4.2} as stated assume the starting distribution $\lambda^\varepsilon$,
they can be extended to general initial distributions by using results from Section 
\ref{sec:exponential__returning_law_and_tail_behavior}, which show that the process essentially forgets the initial distribution before leaving the neighborhood of $O_1$.
\end{remark}

\begin{remark}
In this remark we interpret the use of Theorems \ref{Thm:4.1} and \ref{Thm:4.2} in the context of Monte Carlo, and also explain the role of the time scaling $T^\varepsilon$.

There is a minimum amount of time that must elapse before the process can visit all stable equilibrium points often enough that good estimation of risk-sensitive integrals is possible. 
As is well known, this time scales exponentially in the form of $T^\varepsilon = e^{c/\varepsilon}$,
and the issue is the selection of the constant $c>0$,
which motivates the assumptions on $T^\varepsilon$ for the two cases.
However,
when designing a scheme there typically will be parameters available for selection. 
The growth constant in $T^\varepsilon$ will then depend on these parameters, which will then be chosen to (either directly or indirectly, depending on the criteria used) reduce the size of $T^\varepsilon$. 
For a compelling example  we refer to \cite{dupwu},
which shows how for a system with fixed well depths a scheme known as infinite swapping can be designed so that given any $a>0$ one can design a scheme so that an interval of length  $e^{a/\varepsilon}$ suffices.

Theorem \ref{Thm:4.1} is concerned with bias, and for $T^\varepsilon$ as above will give a negligible contribution to the total error in comparison to the variance. 
Thus it is Theorem \ref{Thm:4.2}
that determines the performance of the scheme and serves as a criteria for optimization.
Of particular note is that the value of $c$ does {\em not} appear in the variational problem appearing in Theorem \ref{Thm:4.2}.

Theorem \ref{Thm:4.2} gives a lower bound on the rate of decay of
variance per unit time.
For applications to the design of Monte Carlo schemes as in \cite{dupwu} there is an a priori bound on the best possible
performance, and so this lower bound (which yields an upper bound on variances) is sufficient to determine if a scheme is nearly optimal.
However,
for other purposes an upper bound on the decay rate could be useful, 
and we expect the other direction holds as
well.
\end{remark}

The proofs of Theorems \ref{Thm:4.1} and \ref{Thm:4.2} for single cycles and multicycles are almost identical with a few key differences. We focus on providing proofs in the single cycle case, and then point out the required modifications in the proofs for 
the multicycle case.

\begin{theorem}
\label{Thm:4.3}The bound in Theorem \ref{Thm:4.1} can be calculated using only stable equilibrium points. Specifically,

\begin{enumerate}
\item $W(x)=\min_{j\in L_{\rm{s}}}[W(O_{j})+V(O_{j},x)]$

\item $W\left(  O_{j}\right)  =\min_{g\in G_{\rm{s}}\left(  j\right)  }\left[
{\textstyle\sum_{\left(  m\rightarrow n\right)  \in g}}
V\left(  O_{m},O_{n}\right)  \right]  $

\item $W(O_{1}\cup O_{j})=\min_{g\in G_{\rm{s}}\left(  1,j\right)  }\left[
{\textstyle\sum_{\left(  m\rightarrow n\right)  \in g}}
V\left(  O_{m},O_{n}\right)  \right]  $

\item 
 $ \min_{j\in L}(  \inf_{x\in A}[  f(  x)  +V(O_{j},x)]  +W(O_{j}) ) 
 =\min_{j\in L_{\rm{s}}}(  \inf_{x\in A}[ f(x)+V(O_{j},x)]  +W(O_{j}) ) $.
\end{enumerate}
\end{theorem}

\begin{remark}
\label{Rmk:4.3}
Theorem \ref{Thm:4.3} says that the bound appearing in Theorem
\ref{Thm:4.1} depends on the set of indices of
only stable equilibrium points. This is not surprising, since in \cite[Chapter 6]{frewen2}, it has been shown that the logarithmic asymptotics of the
invariant measure of a Markov process in this framework can be characterized in terms of graphs on the set of indices
of just stable equilibrium points.
It is natural to ask if the same property holds for the lower bound appearing in Theorem \ref{Thm:4.2}. 
Notice that part 4 of Theorem \ref{Thm:4.3} implies $\min_{j\in L}R_{j}^{(1)}=\min_{j\in L_{\rm{s}}}R_{j}^{(1)}$, so if one can prove (possibly under extra conditions, for example, by considering a double-well model as in Section \ref{sec:upper_bound_for_performance}) that $\min_{j\in L}R_{j}^{(2)}=\min_{j\in L_{\rm{s}}}R_{j}^{(2)}$, then these two equations assert the property we want for the single cycle case, namely, 
$
\min_{j\in L}(  R_{j}^{(1)}\wedge
R_{j}^{(2)})  =\min_{j\in L_{\rm{s}}}(  R_{j}^{(1)}\wedge R_{j}%
^{(2)}).
$ 
An analogous comment applies for the multicycle case.
\end{remark}

\begin{conjecture}
\label{Conj:4.1}Let $T^{\varepsilon}=e^{\frac{1}{\varepsilon}c}$ for some
$c>h_1\vee w$.
Let $f$ be continuous and suppose that $A$ is the closure of its interior.
Then for any $\eta>0,$ there exists $\delta_{0}\in(0,1)$ such that for any
$\delta\in(0,\delta_{0})$ 
\begin{align*}
&  \liminf_{\varepsilon\rightarrow0}-\varepsilon\log\left(  T^{\varepsilon
}\cdot\mathrm{Var}_{\lambda^{\varepsilon}}\left(  \frac{1}{T^{\varepsilon}}%
\int_{0}^{T^{\varepsilon}}e^{-\frac{1}{\varepsilon}f\left(  X_{t}%
^{\varepsilon}\right)  }1_{A}\left(  X_{t}^{\varepsilon}\right)  dt\right)
\right) \\
&  \qquad\leq
\begin{cases}
\min_{j\in L}\left(  R_{j}^{(1)}\wedge R_{j}^{(2)}\right)  +\eta,& \text{if } h_1>w\\
\min_{j\in L}\left(  R_{j}^{(1)}\wedge R_{j}^{(2)}\wedge R_{j}^{(3)}\right)  +\eta,& \text{otherwise}
\end{cases}.
\end{align*}
\end{conjecture}

In Section \ref{sec:upper_bound_for_performance} we outline the proof of  Conjecture \ref{Conj:4.1} for a special case.

\section{Examples}

\label{sec:examples}

\begin{example}
We first consider the situation depicted in Figure \ref{fig:1}. Values of
$W(O_{j})$ are given in the figure. If one interprets the figure as a 
potential with minimum zero then the corresponding heights of the equilibrium points are given
by $W(O_{j})-W(O_{1})$. We take $f=0$ and $A$ to be a small closed interval
about $O_{5}$. As we will see and should be clear from the figure, this
example can be analyzed using single~regenerative cycles.

Recall that
\[
R_{j}^{(1)}\doteq \inf_{x\in A}[2f(x)+V(O_{j},x)]+W(O_{j})-W(O_{1})
\]%
\[
R_{1}^{(2)}\doteq2\inf_{x\in A}[f(x)+V(O_{1},x)]-h_{1}%
\]
and for $j>1$%
\[
R_{j}^{(2)}\doteq2\inf_{x\in A}[f(x)+V(O_{j},x)]+\left(  W(O_{j})-W(O_{1})\right)
-W(O_{1})+W(O_{1}\cup O_{j})
\]

If one traces through the proof of Theorem \ref{Thm:4.2} for the case of a
single cycle, then one finds that the constraining bound is  given in
Lemma \ref{Lem:6.18}, which is in turn based on Lemma \ref{Lem:6.7}. As we
will see, in the minimization problem $\min_{j\in L}(R_{j}^{(1)}\wedge
R_{j}^{(2)})$ the min on $j$ turns out to be achieved at $j=5$. This is of
course not surprising, since $A$ is an interval about $O_{5}$. It is then the
minimum of $R_{5}^{(1)}$ and $R_{5}^{(2)}$ which determines the dominant
source of the variance of the estimator.

\begin{figure}[h]
    \centering
    \includegraphics[width=0.9\textwidth]{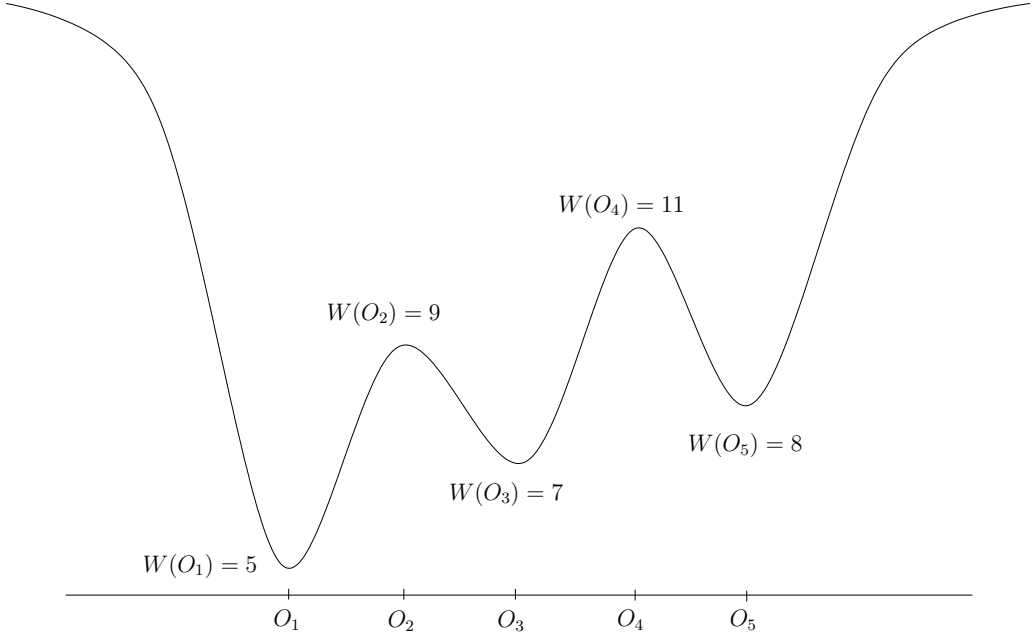}
    \caption{Single cycle example}
    \label{fig:1}
\end{figure}

We recall that $\tau_{1}^{\varepsilon}$ is the time for a full regenerative
cycle, and that $\tau_{1}$ is the time to first reach the $2\delta$
neighborhood of an equilibrium point and then reach a $\delta$ neighborhood of
a (perhaps the same) equilibrium point. The quantities that are relevant in Lemma \ref{Lem:6.7}
are
\[
\sup_{y\in\partial B_{\delta}(O_{5})}E_{y}\left(  \int_{0}^{\tau_{1}}%
1_{A}(X_{t}^{\varepsilon})dt\right)  ^{2}\text{ and }E_{x}N_{5}%
\]
for $R_{j}^{(1)}$ and
\[
\left[  \sup_{y\in\partial B_{\delta}(O_{5})}E_{y}\int_{0}^{\tau_{1}}%
1_{A}(X_{t}^{\varepsilon})dt\right]  ^{2}\text{ , }E_{x}N_{5}\text{, and
essentially }\sup_{y\in\partial B_{\delta}(O_{5})}E_{y}N_{5}%
\]
for $R_{j}^{(2)}$. Decay rates are in turn determined by (see the proof of
Lemma \ref{Lem:6.18})%
\[
0\text{ and }W(O_{1})-W(O_{5})+h_{1}%
\]
and
\[
0,\text{ }W(O_{1})-W(O_{5})+h_1\text{ and }W(O_{1})-W(O_{1}\cup O_{5}),
\]
respectively. Thus for this example it is only the term $W(O_{1})-W(O_{1}\cup
O_{5})$ that distinguishes between the two. Since this is always greater than
zero and it appears in $R_{j}^{(2)}$ in the form $-(W(O_{1})-W(O_{1}\cup
O_{5}))$ it must be the case that $R_{5}^{(2)}<$ $R_{5}^{(1)}$. 

The numerical values for the example are %
\[
(W(O_{1}\cup O_{j}),j=2,\ldots,5)=(5,3,5,2)
\]%
\[
(V(O_{j},O_{5}),j=1,\ldots,5)=(8,4,4,0,0)
\]%
\[
(W(O_{j})-W(O_{1}),j=1,\ldots,5)=(0,4,2,6,3)
\]%
\[
(R_{j}^{(1)},j=1,\ldots,5)=(8,8,6,6,3)
\]%
\[
(R_{j}^{(2)},j=2,\ldots,5)=(12,8,6,0)
\]
and $R_{1}^{(2)}=16-4=12,h_{1}=4$ and $w=5-2=3$. Since $w<h_{1}$ it falls into
the single cycle case. We therefore find $\min_{j}R_{j}^{(1)}\wedge R_{j}^{(2)}$ equals
to $0$ and occurs with superscript $2$ and at $j=5$. 

For an example where the dominant contribution to the variance is through the
quantities associated with $R_{j}^{(1)}$, we move the set $A$ further to the
right of $O_{5}$. All other quantities are unchanged save%
\[
\sup_{y\in\partial B_{\delta}(O_{5})}E_{y}\left(  \int_{0}^{\tau_{1}}%
1_{A}(X_{t}^{\varepsilon})dt\right)  ^{2}\text{ and }\left[  \sup
_{y\in\partial B_{\delta}(O_{5})}E_{y}\int_{0}^{\tau_{1}}1_{A}(X_{t}%
^{\varepsilon})dt\right]  ^{2},
\]
whose decay rates are governed (for $j=5$) by
$
\inf_{x\in A}[V(O_{5},x)]\text{ and }2\inf_{x\in A}[V(O_{5},x)],
$
respectively. Choosing $A$ so that $\inf_{x\in A}[V(O_{5},x)]>3$, it is now
the case that $R_{5}^{(1)}<$ $R_{5}^{(2)}$.

\end{example}

\begin{example}
In this example we again take $f=0$ and $A$ to be a small closed interval
about $O_{3}$.
Since the well at $O_5$ is deeper than that at $O_1$ we expect that multicycles will be needed,
and so recall
\[
R_{j}^{(3)}  \doteq2\inf_{x\in A}\left[  f\left(  x\right)  +V\left(O_{j},x\right)  \right]  +2W\left(  O_{j}\right)  -2W\left(  O_{1}\right)-w .
\]
\begin{figure}[h]
    \centering
    \includegraphics[width=0.9\textwidth]{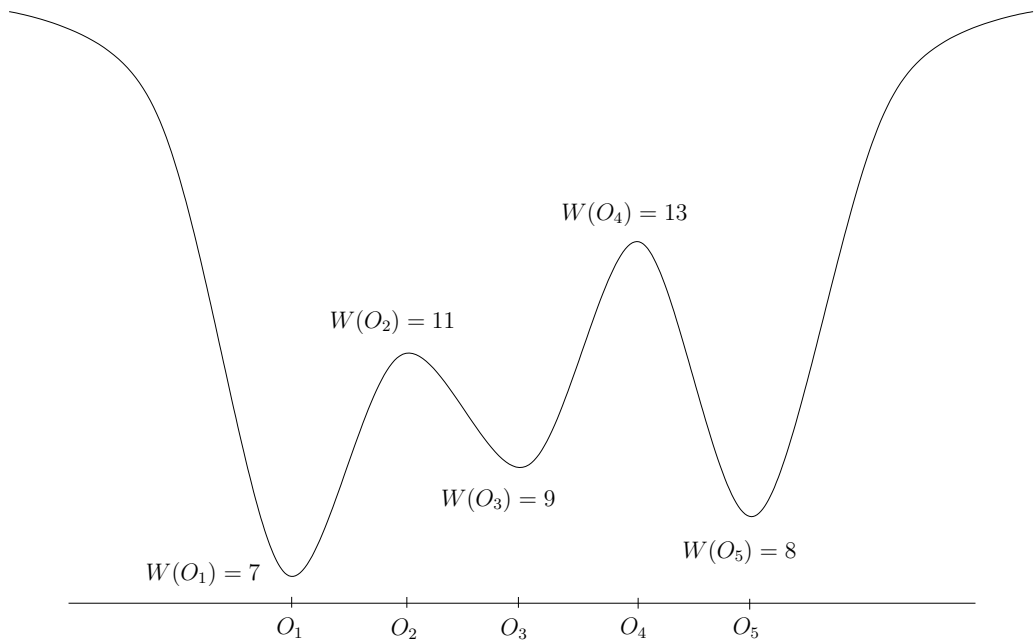}
    \caption{Multicycle example}
    \label{fig:2}
\end{figure}

 The needed values are
\[
(W(O_{1}\cup O_{j}),j=2,\ldots,5)=(7,5,7,2)
\]%
\[
(V(O_{j},O_{3}),j=1,\ldots,5)=(4,0,0,0,5)
\]%
\[
(W(O_{j})-W(O_{1}),j=1,\ldots,5)=(0,4,2,6,1)
\]%
\[
(R_{j}^{(1)},j=1,\ldots,5)=(4,4,2,6,6)
\]%
\[
(R_{j}^{(2)},j=2,\ldots,5)=(4,0,6,6)
\]
\[
(R_{j}^{(3)},j=1,\ldots,5)=(3,3,-1,7,7)
\]
and $R_{1}^{(2)}=8-4=4,h_{1}=4$ and $w=7-2=5$. Since $w>h_{1}$ a single
cycle cannot be used for the analysis of the variance, and we need to use multicycles.
We find $\min_{j}%
R_{j}^{(1)}\wedge R_{j}^{(2)}\wedge R_{j}^{(3)}$ is equal to $-1$ and occurs with superscript $3$
and $j=3$.

\end{example}

\section{Wald's Identities and Regenerative Structure}

\label{sec:wald's_identities}

To prove Theorems \ref{Thm:4.1} and \ref{Thm:4.2}, we will use the
regenerative structure to analyze the system over the interval
$[0,T^{\varepsilon}]$. Since the number of regenerative cycles will be random,
Wald's identities will be useful.

Recall that $\tau_{n}^{\varepsilon}$ is the $n$-th return time to $\partial
B_{\delta}\left(  O_{1}\right)  $ after having visited the neighborhood of a
different equilibrium point, and $\lambda^{\varepsilon}\in\mathcal{P}(\partial B_{\delta}(O_{1}))$ is the invariant measure of the Markov process $\{X_{\tau_{n}^{\varepsilon}}^{\varepsilon}\}_{n\in\mathbb{N}_{0}}$ with state space $\partial B_{\delta}(O_{1}).$ If we let the process $\{X^{\varepsilon}_t%
\}_{t}$ start with $\lambda^{\varepsilon}$ at time
$0,$ that is, assume the distribution of 
$X^{\varepsilon}_0$ is
$\lambda^{\varepsilon},$ then by the strong Markov property of
$\{X^{\varepsilon}_t  \}_{t},$ we find that $\{X^{\varepsilon}_t  \}_{t}$ is a regenerative process and the cycles
$\{\{X^{\varepsilon}_{\tau_{n-1}^{\varepsilon}+t}:0\leq t<\tau
_{n}^{\varepsilon}-\tau_{n-1}^{\varepsilon}\},\tau_{n}^{\varepsilon}%
-\tau_{n-1}^{\varepsilon}\}$ are iid\ objects. Moreover, $\{\tau
_{n}^{\varepsilon}\}_{n\in%
\mathbb{N}
_{0}}$ is a sequence of renewal times under $\lambda^{\varepsilon}.$
\subsection{Single cycle}
Define the filtration $\{\mathcal{H}_{n}\}_{n\in%
\mathbb{N}
},$ where $\mathcal{H}_{n}\doteq\mathcal{F}_{\tau_{n}^{\varepsilon}}$ and
$\mathcal{F}_{t}\doteq$ $\sigma(\{X^{\varepsilon}_s$; $s\leq t\})$. With
respect to this filtration, in the single cycle case (i.e., when $h_1>w$), we consider the stopping times
$
N^{\varepsilon}\left(  T\right)  \doteq\inf\left\{  n\in%
\mathbb{N}:\tau_{n}^{\varepsilon}>T\right\}  .
$
Note that $N^{\varepsilon}\left(  T\right)  -1$ is the number of complete single renewal intervals contained in  $[0,T]$.

With this notation, we can bound $\frac{1}{T^{\varepsilon}}\int_{0}%
^{T^{\varepsilon}}e^{-\frac{1}{\varepsilon}f\left(  X_{t}^{\varepsilon
}\right)  }1_{A}\left(  X_{t}^{\varepsilon}\right)  dt$ from above and below
by%
\begin{equation}
\frac{1}{T^{\varepsilon}}\sum\nolimits_{n=1}^{N^{\varepsilon}\left(  T^{\varepsilon
}\right)  -1}S_{n}^{\varepsilon}\leq\frac{1}{T^{\varepsilon}}\int
_{0}^{T^{\varepsilon}}e^{-\frac{1}{\varepsilon}f\left(  X_{t}^{\varepsilon
}\right)  }1_{A}\left(  X_{t}^{\varepsilon}\right)  dt\leq\frac{1}%
{T^{\varepsilon}}\sum\nolimits_{n=1}^{N^{\varepsilon}\left(  T^{\varepsilon}\right)
}S_{n}^{\varepsilon}, \label{eqn:bounds}%
\end{equation}
where
\[
S_{n}^{\varepsilon}\doteq \int_{\tau_{n-1}^{\varepsilon}}^{\tau_{n}^{\varepsilon}%
}e^{-\frac{1}{\varepsilon}f\left(  X_{t}^{\varepsilon}\right)  }1_{A}\left(
X_{t}^{\varepsilon}\right)  dt.
\]
Applying Wald's first identity shows
\begin{align} \label{eqn:mean}
E_{\lambda^{\varepsilon}}\left(  \frac{1}{T^{\varepsilon}}\sum\nolimits_{n=1}%
^{N^{\varepsilon}\left(  T^{\varepsilon}\right)  }S_{n}^{\varepsilon}\right)
=\frac{1}{T^{\varepsilon}}E_{\lambda^{\varepsilon}}\left(  N^{\varepsilon
}\left(  T^{\varepsilon}\right)  \right)  E_{\lambda^{\varepsilon}}%
S_{1}^{\varepsilon}.
\end{align}
Therefore, the logarithmic asymptotics of $E_{\lambda^{\varepsilon}}(\int
_{0}^{T^{\varepsilon}}e^{-\frac{1}{\varepsilon}f\left(  X_{t}^{\varepsilon
}\right)  }1_{A}\left(  X_{t}^{\varepsilon}\right)  dt/T^{\varepsilon})$ are
determined by those of $E_{\lambda^{\varepsilon}}\left(  N^{\varepsilon
}\left(  T^{\varepsilon}\right)  \right)  /T^{\varepsilon}$ and $E_{\lambda
^{\varepsilon}}S_{1}^{\varepsilon}.$ Likewise, to understand the logarithmic
asymptotics of $T^{\varepsilon}\cdot$Var$_{\lambda^{\varepsilon}}(\int
_{0}^{T^{\varepsilon}}e^{-\frac{1}{\varepsilon}f\left(  X_{t}^{\varepsilon
}\right)  }1_{A}\left(  X_{t}^{\varepsilon}\right)  dt/T^{\varepsilon}),$ it
is sufficient to identify the corresponding logarithmic asymptotics of
Var$_{\lambda^{\varepsilon}}\left(  N^{\varepsilon}\left(
T^{\varepsilon}\right)  \right)  /T^{\varepsilon}$, Var$_{\lambda
^{\varepsilon}}(S_{1}^{\varepsilon}),$ $E_{\lambda^{\varepsilon}}\left(
N^{\varepsilon}\left(  T^{\varepsilon}\right)  \right)  /T^{\varepsilon}$ and
$E_{\lambda^{\varepsilon}}S_{1}^{\varepsilon}$. This can be done with the help
of Wald's second identity, since%
\begin{align}
T^{\varepsilon}  &  \cdot\mathrm{Var}_{\lambda^{\varepsilon}}\left(  \frac
{1}{T^{\varepsilon}}%
{\textstyle\sum_{n=1}^{N^{\varepsilon}\left(  T^{\varepsilon}\right)  }}
S_{n}^{\varepsilon}\right) \label{eqn:variance}\\
 &  \leq2T^{\varepsilon}\cdot E_{\lambda^{\varepsilon}}\left(  \frac
 {1}{T^{\varepsilon}}%
 {\textstyle\sum_{n=1}^{N^{\varepsilon}\left(  T^{\varepsilon}\right)  }}
 S_{n}^{\varepsilon}-\frac{1}{T^{\varepsilon}}N^{\varepsilon}\left(
 T^{\varepsilon}\right)  E_{\lambda^{\varepsilon}}S_{1}^{\varepsilon}\right)
 ^{2}\nonumber\\
&  \quad+2T^{\varepsilon}\cdot E_{\lambda^{\varepsilon}}\left(  \frac
{1}{T^{\varepsilon}}N^{\varepsilon}\left(  T^{\varepsilon}\right)
E_{\lambda^{\varepsilon}}S_{1}^{\varepsilon}-\frac{1}{T^{\varepsilon}%
}E_{\lambda^{\varepsilon}}\left(  N^{\varepsilon}\left(  T^{\varepsilon
}\right)  \right)  E_{\lambda^{\varepsilon}}S_{1}^{\varepsilon}\right)
^{2}\nonumber\\
&  =2\frac{E_{\lambda^{\varepsilon}}\left(  N^{\varepsilon}\left(
T^{\varepsilon}\right)  \right)  }{T^{\varepsilon}}\mathrm{Var}_{\lambda
^{\varepsilon}}S_{1}^{\varepsilon}+2\frac{\mathrm{Var}_{\lambda^{\varepsilon}%
}\left(  N^{\varepsilon}\left(  T^{\varepsilon}\right)  \right)
}{T^{\varepsilon}}\left(  E_{\lambda^{\varepsilon}}S_{1}^{\varepsilon}\right)
^{2}.\nonumber
\end{align}

In the next two sections we derive bounds on $E_{\lambda^{\varepsilon}}%
S_{1}^{\varepsilon}$, Var$_{\lambda^{\varepsilon}}(S_{1}^{\varepsilon})$ and
$E_{\lambda^{\varepsilon}}\left(  N^{\varepsilon}\left(  T^{\varepsilon
}\right)  \right)  $, Var$_{\lambda^{\varepsilon}}\left(  N^{\varepsilon
}\left(  T^{\varepsilon}\right)  \right)$, respectively.
\subsection{Multicycle}
\label{subsection:multicycle}
Recall that in the case of a multicycle we have $w\geq h_1$. 
For any $m>0$ such that $h_1+m>w$ and for any $\varepsilon>0$, on the same probability space as $\{\tau^{\varepsilon}_n\}$, one can define a sequence of independent and geometrically distributed random variables $\{\mathbf{M}^{\varepsilon}_i\}_{i\in\mathbb{N}}$ with parameter $e^{-m/\varepsilon}$
that are independent of $\{\tau^{\varepsilon}_n\}$. 
We then define multicycles according to 
\begin{equation}\label{eqn:defofMC}
    \mathbf{K}^\varepsilon_i\doteq \sum_{j=1}^i \mathbf{M}^{\varepsilon}_j,
\quad 
    \hat{\tau}^\varepsilon_i\doteq \sum_{n=\mathbf{K}^\varepsilon_{i-1}+1}^{\mathbf{K}^{\varepsilon}_i}\tau^\varepsilon_n,
    \quad i\in \mathbb{N}.
\end{equation}
Consider the stopping times
$
\hat{N}^{\varepsilon}\left(  T\right)  \doteq\inf\left\{  n\in \mathbb{N}
:\hat{\tau}_{n}^{\varepsilon}>T\right\}  .
$
Note that $\hat{N}^{\varepsilon}\left(  T\right)  -1$ is the number of complete multicycles contained in  $[0,T]$.
With this notation and by following the same idea as in the single cycle case, we can bound $\frac{1}{T^{\varepsilon}}\int_{0}%
^{T^{\varepsilon}}e^{-\frac{1}{\varepsilon}f\left(  X_{t}^{\varepsilon
}\right)  }1_{A}\left(  X_{t}^{\varepsilon}\right)  dt$ from above and below
by%
\begin{equation}
\frac{1}{T^{\varepsilon}}\sum\nolimits_{n=1}^{\hat{N}^{\varepsilon}\left(  T^{\varepsilon
}\right)  -1}\hat{S}_{n}^{\varepsilon}\leq\frac{1}{T^{\varepsilon}}\int
_{0}^{T^{\varepsilon}}e^{-\frac{1}{\varepsilon}f\left(  X_{t}^{\varepsilon
}\right)  }1_{A}\left(  X_{t}^{\varepsilon}\right)  dt\leq\frac{1}%
{T^{\varepsilon}}\sum\nolimits_{n=1}^{\hat{N}^{\varepsilon}\left(  T^{\varepsilon}\right)
}\hat{S}_{n}^{\varepsilon}, \label{eqn:mega_bounds}%
\end{equation}
where
\[
\hat{S}_{n}^{\varepsilon}\doteq \int_{\hat{\tau}_{n-1}^{\varepsilon}}^{\hat{\tau}_{n}^{\varepsilon}%
}e^{-\frac{1}{\varepsilon}f\left(  X_{t}^{\varepsilon}\right)  }1_{A}\left(
X_{t}^{\varepsilon}\right) dt.
\]
Therefore, by applying Wald's first and second identities, we know that the logarithmic asymptotics of $E_{\lambda^{\varepsilon}}(\int
_{0}^{T^{\varepsilon}}e^{-\frac{1}{\varepsilon}f\left(  X_{t}^{\varepsilon
}\right)  }1_{A}\left(  X_{t}^{\varepsilon}\right)  dt/T^{\varepsilon})$ are
determined by those of $E_{\lambda^{\varepsilon}}(  \hat{N}^{\varepsilon
}\left(  T^{\varepsilon}\right) )  /T^{\varepsilon}$ and $E_{\lambda
^{\varepsilon}}\hat{S}_{1}^{\varepsilon}$,
and the asymptotics of $T^{\varepsilon}\cdot$Var$_{\lambda^{\varepsilon}}(\int
_{0}^{T^{\varepsilon}}e^{-\frac{1}{\varepsilon}f\left(  X_{t}^{\varepsilon
}\right)  }1_{A}\left(  X_{t}^{\varepsilon}\right)  dt/T^{\varepsilon})$ 
by those of 
Var$_{\lambda^{\varepsilon}}(  \hat{N}^{\varepsilon}\left(
T^{\varepsilon}\right)  )  /T^{\varepsilon}$, Var$_{\lambda
^{\varepsilon}}(\hat{S}_{1}^{\varepsilon})$, $E_{\lambda^{\varepsilon}}(
\hat{N}^{\varepsilon}\left(  T^{\varepsilon}\right)  )  /T^{\varepsilon}$ and
$E_{\lambda^{\varepsilon}}\hat{S}_{1}^{\varepsilon}$. 
In particular, we have
\begin{align} \label{eqn:mega_mean}
E_{\lambda^{\varepsilon}}\left(  \frac{1}{T^{\varepsilon}}\sum\nolimits_{n=1}%
^{\hat{N}^{\varepsilon}\left(  T^{\varepsilon}\right)  }\hat{S}_{n}^{\varepsilon}\right)
=\frac{1}{T^{\varepsilon}}E_{\lambda^{\varepsilon}}\left(  \hat{N}^{\varepsilon
}\left(  T^{\varepsilon}\right)  \right)  E_{\lambda^{\varepsilon}}%
\hat{S}_{1}^{\varepsilon}.
\end{align}
and
\begin{align}
T^{\varepsilon}   \cdot\mathrm{Var}_{\lambda^{\varepsilon}}\left(  \frac
{1}{T^{\varepsilon}}%
{\sum\nolimits_{n=1}^{\hat{N}^{\varepsilon}\left(  T^{\varepsilon}\right)  }}
\hat{S}_{n}^{\varepsilon}\right)
  \leq 2\frac{E_{\lambda^{\varepsilon}}(  \hat{N}^{\varepsilon}\left(
T^{\varepsilon}\right)  )  }{T^{\varepsilon}}\mathrm{Var}_{\lambda
^{\varepsilon}}\hat{S}_{1}^{\varepsilon}+2\frac{\mathrm{Var}_{\lambda^{\varepsilon}%
}(  \hat{N}^{\varepsilon}\left(  T^{\varepsilon}\right)  )
}{T^{\varepsilon}}\left(  E_{\lambda^{\varepsilon}}\hat{S}_{1}^{\varepsilon}\right)
^{2}. \label{eqn:mega_variance}
\end{align}

In the next two sections we derive bounds on $E_{\lambda^{\varepsilon}}%
\hat{S}_{1}^{\varepsilon}$, Var$_{\lambda^{\varepsilon}}(\hat{S}_{1}^{\varepsilon})$ and
$E_{\lambda^{\varepsilon}}(  \hat{N}^{\varepsilon}\left(  T^{\varepsilon
}\right)  )  $, Var$_{\lambda^{\varepsilon}}(  \hat{N}^{\varepsilon
}\left(  T^{\varepsilon}\right)  )$, respectively.

\begin{remark} It should be kept in mind that
    $\hat{\tau}^{\varepsilon}_n, \hat{N}^{\varepsilon}\left(  T^{\varepsilon
}\right)$ and $\hat{S}_{n}^{\varepsilon}$ all depend on $m$, although this dependence is not explicit in the notation.
\end{remark}

\begin{remark}
    In general, for any quantity in the single cycle case, we use analogous notation with a ``hat'' on it to represent the corresponding quantity in the multicycle version. For instance, we use $\tau^{\varepsilon}_n$ for a single regenerative cycle, and $\hat{\tau}^{\varepsilon}_n$ for a multi regenerative cycle.
\end{remark}

\section{Asymptotics of Moments of $S_{1}^{\varepsilon}$ and $\hat{S}_{1}^{\varepsilon}$}

\label{sec:asymptotics_of_moments_of_S}

In this section we will first introduce the elementary theory of an irreducible
finite state Markov chain $\{Z_{n}\}_{n\in%
\mathbb{N}
_{0}}$ with state space $L$, and then state and prove bounds for
the asymptotics of moments of $S_{1}^{\varepsilon}$ and $\hat{S}_{1}^{\varepsilon}$.

For the asymptotic analysis, the following useful facts will be used repeatedly.

\begin{lemma}
\label{Lem:6.1}For any nonnegative sequences $\left\{  a_{\varepsilon
}\right\}  _{\varepsilon>0}$ and $\left\{  b_{\varepsilon}\right\}
_{\varepsilon>0}$, we have%
\begin{equation}
\liminf_{\varepsilon\rightarrow0}-\varepsilon\log\left(  a_{\varepsilon
}b_{\varepsilon}\right)  \geq\liminf_{\varepsilon\rightarrow0}-\varepsilon\log
a_{\varepsilon}+\liminf_{\varepsilon\rightarrow0}-\varepsilon\log
b_{\varepsilon}, \label{eqn:product}%
\end{equation}
\[
\limsup_{\varepsilon\rightarrow0}-\varepsilon\log\left(  a_{\varepsilon
}+b_{\varepsilon}\right)  \leq\min\left\{  \limsup_{\varepsilon\rightarrow
0}-\varepsilon\log a_{\varepsilon},\limsup_{\varepsilon\rightarrow
0}-\varepsilon\log b_{\varepsilon}\right\},
\]
\begin{equation}
\liminf_{\varepsilon\rightarrow0}-\varepsilon\log\left(  a_{\varepsilon
}+b_{\varepsilon}\right)  =\min\left\{  \liminf_{\varepsilon\rightarrow
0}-\varepsilon\log a_{\varepsilon},\liminf_{\varepsilon\rightarrow
0}-\varepsilon\log b_{\varepsilon}\right\}  . \label{eqn:sum}%
\end{equation}

\end{lemma}

\subsection{Markov chains and graph theory}

\label{subsec:markov_chain_and_graph_theory}

In this subsection we state some elementary theory for finite state Markov
chains taken from \cite[Chapter 2]{aldfil}. For a finite state Markov chain,
the invariant measure, the mean exit time, etc., can be expressed explicitly
as the ratio of certain determinants, i.e., sums of products consisting of
transition probabilities, and these sums only contain terms with a plus sign.
Which products should appear in the various sums can be described conveniently
by means of graphs on the set of states of the chain. This method of linking
graphs and quantities associated with a finite state Markov chain was
introduced by Freidlin and Wentzell in \cite[Chapter 6]{frewen2}.

Consider an irreducible finite state Markov chain $\{Z_{n}\}_{n\in%
\mathbb{N}
_{0}}$ with state space $L.$ For any $i,j\in L,$ let $p_{ij}$ be the one-step
transition probability of $\{Z_{n}\}_{n}$ from state $i$ to state $j.$ Write
$P_{i}(\cdot)$ and $E_{i}(\cdot)$ for probabilities and expectations of the
chain started at state $i$ at time $0.$ Recall the notation $\pi(g)\doteq
\prod_{(i\rightarrow j)\in g}p_{ij}$.

\begin{lemma}
\label{Lem:6.2}The unique invariant measure of $\{Z_{n}\}_{n\in%
\mathbb{N}
}$ can be expressed
\[
\lambda_i  =\frac{\sum_{g\in G\left(  i\right)  }\pi\left(
g\right)  }{\sum_{j\in L}\left(  \sum_{g\in G\left(  j\right)  }\pi\left(
g\right)  \right)  }.
\]
\end{lemma}

\begin{proof}
See Lemma 3.1, Chapter 6 in \cite{frewen2}.
\end{proof}

\vspace{\baselineskip}
To analyze the empirical measure we will need additional results,
including representations for the number of visits to a state during a 
regenerative cycle.
Write
\[
T_{i}\doteq\inf\left\{  n\geq0:Z_{n}=i\right\}
\]
for the \textit{first hitting time} of state $i,$ and write
\[
T_{i}^{+}\doteq\inf\left\{  n\geq1:Z_{n}=i\right\}  .
\]
Observe that $T_{i}^{+}=T_{i}$ unless $Z_{0}=i,$ in which case we call
$T_{i}^{+}$ the \textit{first return time} to state $i.$

Let $\hat{N}\doteq \inf\{n\in%
\mathbb{N}
_{0}:Z_{n}\in L\setminus\{1\}\}$ and $N\doteq \inf\{n\in%
\mathbb{N}
:Z_{n}=1,n\geq\hat{N}\}.$ $\hat{N}$ is the first time of visiting a state
other than state $1$ and $N$ is the first time of visiting state $1$ after
$\hat{N}.$ For any $j\in L,$ let $N_{j}$ be the number of visits (including
time $0$) of state $j$ before $N,$ i.e., $N_{j}=\left\vert \{n\in%
\mathbb{N}
_{0}:n<N\text{ and }Z_{n}=j\}\right\vert .$ We would like to understand
$E_{1}N_{j}$ and $E_{j}N_{j}$ for any $j\in L.$ These quantities will appear
later on in Subsection \ref{sec:proof_of_asymptotics_of_moments_of_S}. The
next lemma shows how they can be related to the invariant measure of
$\{Z_{n}\}_{n}$.

\begin{lemma}
\label{Lem:6.3}

\begin{enumerate}
\item For any $j\in L\setminus\{1\}$%
\[
E_{j}N_{j}=\frac{\sum_{g\in G\left(  1,j\right)  }\pi\left(  g\right)  }%
{\sum_{g\in G\left(  1\right)  }\pi\left(  g\right)  }\text{ and }E_{j}%
N_{j}=\lambda_{j}\left(  E_{j}T_{1}+E_{1}T_{j}\right)  .
\]

\item For any $i,j\in L,$ $j\neq i$%
\[
P_{i}\left(  T_{j}<T_{i}^{+}\right)  =\frac{1}{\lambda_{i}\left(  E_{j}%
T_{i}+E_{i}T_{j}\right)  }.
\]

\item For any $j\in L$
\[
E_{1}N_{j}=\frac{1}{1-p_{11}}\frac{\lambda_{j}}{\lambda_{1}}.
\]

\end{enumerate}
\end{lemma}

\begin{proof}
See Lemma 3.4 in \cite[Chapter 6]{frewen2} for the first assertion of part 1
and see Lemma 2.7 in \cite[Chapter 2]{aldfil} for the second assertion of part
1. For part 2, see Corollary 2.8 in \cite[Chapter 2]{aldfil}.
For part 3, since
$
E_{1}N_{j}=\sum_{\ell=1}^{\infty}P_{1}\left(  N_{j}\geq\ell\right)  ,
$ 
we need to understand $P_{1}\left(  N_{j}\geq\ell\right)  $, which means we need
to know how to count all the ways to get $N_{j}\geq\ell$ before returning to
state $1.$

We first have to move away from state $1$, so the types of sequences are of
the form%
\[
\underset{i\text{ times}}{\underbrace{1,1,\ldots,1}},k_{1},k_{2},\ldots
,k_{q},1
\]
for some $i,q\in%
\mathbb{N}
$ and $k_{1}\neq1,\cdots,k_{q}\neq1$. When $j=1,$ we do not care about
$k_{1},k_{2},\ldots,k_{q},$ and therefore%
\[
P_{1}\left(  N_{1}\geq i\right)  =p_{11}^{i-1}\text{ and }E_{1}N_{1}%
=\sum\nolimits_{i=1}^{\infty}P_{1}\left(  N_{1}\geq i\right)  =\frac{1}{1-p_{11}}.
\]
For $j\in L\setminus\{1\},$ the event $\{N_{j}\geq\ell\}$ requires that
within $k_{1},k_{2},\ldots,k_{q},$ we

\begin{enumerate}
\item first visit state $j$ before returning to state $1,$ which has
corresponding probability $P_{1}(T_{j}<T_{1}^{+})$,

\item then start from state $j$ and again visit state $j$ before returning to
state $1,$ which has corresponding probability $P_{j}(T_{j}^{+}<T_{1}).$
\end{enumerate}

Step 2 needs to happen at least $\ell-1$ times in a row, and after that we do
not care. Thus,%
\begin{align*}
P_{1}\left(  N_{j}\geq\ell\right)   &  =\sum\nolimits_{i=1}^{\infty}\left(
p_{11}\right)  ^{i-1}P_{1}\left(  T_{j}<T_{1}^{+}\right)  (  P_{j}(
T_{j}^{+}<T_{1})  )  ^{\ell-1}\\
&  =\frac{1}{1-p_{11}}P_{1}\left(  T_{j}<T_{1}^{+}\right)  (
P_{j}(  T_{j}^{+}<T_{1})  )  ^{\ell-1}%
\end{align*}
and%
\begin{align*}
\sum\nolimits_{\ell=1}^{\infty}P_{1}\left(  N_{j}\geq\ell\right)   
&  =\frac{1}{1-p_{11}}\frac{P_{1}\left(  T_{j}<T_{1}^{+}\right)  }{P_{j}%
(T_{1}<T_{j}^{+})}
  =\frac{1}{1-p_{11}}\frac{\lambda_{j}\left(  E_{1}T_{j}+E_{j}T_{1}\right)
}{\lambda_{1}\left(  E_{1}T_{j}+E_{j}T_{1}\right)  }\\
&  =\frac{1}{1-p_{11}}\frac{\lambda_{j}}{\lambda_{1}}.
\end{align*}
The third equality comes from part 2.
\end{proof}

\vspace{\baselineskip} To apply the preceding results using the machinery developed by
Freidlin and Wentzell, one must have analogues that allow for small
perturbations of the transition probabilities due to the fact that initial
conditions are to be taken in small neighborhoods of the equilibrium points.
The addition of a tilde will be used to identify the corresponding objects,
such as hitting and return times. Take as given a Markov chain $\{\tilde
{Z}_{n}\}_{n\in%
\mathbb{N}
_{0}}$ on a state space $\mathcal{X}=%
{\textstyle\cup_{i\in L}}
\mathcal{X}_{i},$ with $\mathcal{X}_{i}\cap\mathcal{X}_{j}=\emptyset$ $(i\neq
j),$ and assume there is $a\in\lbrack1,\infty)$ such that for any $i,j\in L$ and $j\neq i,$ the transition
probability of the chain from $x\in\mathcal{X}_{i}$ to $\mathcal{X}_{j}$
(denoted by $p\left(  x,\mathcal{X}_{j}\right)  $) satisfies the inequalities%
\begin{equation}
a^{-1}p_{ij}\leq p\left(  x,\mathcal{X}_{j}\right)  \leq ap_{ij}
\label{eqn:prob_bounds}%
\end{equation}
for any $x\in\mathcal{X}_{i}$. Write
$P_{x}(\cdot)$ and $E_{x}(\cdot)$ for probabilities and expectations of the
chain started at $x\in\mathcal{X}$ at time $0.$ Write
\[
\tilde{T}_{i}\doteq\inf\{  n\geq0:\tilde{Z}_{n}\in\mathcal{X}%
_{i}\}
\]
for the \textit{first hitting time} of $\mathcal{X}_{i},$ and write
\[
\tilde{T}_{i}^{+}\doteq\inf\{  n\geq1:\tilde{Z}_{n}\in\mathcal{X}%
_{i}\}  .
\]
Observe that $\tilde{T}_{i}^{+}=\tilde{T}_{i}$ unless $\tilde{Z}_{0}%
\in\mathcal{X}_{i},$ in which case we call $\tilde{T}_{i}^{+}$ the
\textit{first return time} to $\mathcal{X}_{i}.$ Recall that $l=|L|$.

\begin{remark}
\label{Rmk:6.1}Observe that given $j\in L$ and for any $x\in\mathcal{X}_{j}$,
$
1-p\left(  x,\mathcal{X}_{j}\right)  =\textstyle\sum_{k\in L\setminus\left\{  j\right\}
}p\left(  x,\mathcal{X}_{k}\right)  .
$
Therefore, we can apply (\ref{eqn:prob_bounds}) to obtain%
\[
a^{-1}\textstyle\sum_{k\in L\setminus\left\{  j\right\}  }p_{jk}\leq1-p\left(
x,\mathcal{X}_{j}\right)  \leq a\textstyle\sum_{k\in L\setminus\left\{  j\right\}
}p_{jk}.
\]

\end{remark}

\begin{lemma}
\label{Lem:6.4}

\begin{enumerate}
\item Consider distinct $i,j,k\in L$. Then for $x\in\mathcal{X}_{k},$%
\[
a^{-4^{l-2}}P_{k}\left(  T_{j}<T_{i}\right)  \leq P_{x}(  \tilde{T}%
_{j}<\tilde{T}_{i})  \leq a^{4^{l-2}}P_{k}\left(  T_{j}<T_{i}\right)  .
\]

\item For any $i\in L$, $j\in L\setminus\{i\}$ and $x\in\mathcal{X}_{i},$%
\[
a^{-4^{l-2}-1}P_{i}\left(  T_{j}<T_{i}^{+}\right)  \leq P_{x}(  \tilde
{T}_{j}<\tilde{T}_{i}^{+})  \leq a^{4^{l-2}+1}P_{i}\left(  T_{j}%
<T_{i}^{+}\right)  .
\]

\end{enumerate}
\end{lemma}

\begin{proof}
For part 1, see Lemma 3.3 in \cite[Chapter 6]{frewen2}. We only need to prove
part $2.$ Note that by a first step analysis on $\{\tilde{Z}_{n}\}_{n\in%
\mathbb{N}
_{0}}$, for any $i\in L$, $j\in L\setminus\{i\}$ and $x\in\mathcal{X}_{i},$%
\begin{align*}
P_{x}(  \tilde{T}_{j}<\tilde{T}_{i}^{+})   &  =p\left(
x,\mathcal{X}_{j}\right)  +\sum\nolimits_{k\in L\setminus\{i,j\}}\int_{\mathcal{X}%
_{k}}P_{y}(  \tilde{T}_{j}<\tilde{T}_{i})  p\left(  x,dy\right) \\
&  \leq ap_{ij}+\sum\nolimits_{k\in L\setminus\{i,j\}}\left(  a^{4^{l-2}}P_{k}\left(
T_{j}<T_{i}\right)  \right)  \left(  ap_{ik}\right) \\
&  \leq a^{4^{l-2}+1}\left(  p_{ij}+\sum\nolimits_{k\in L\setminus\{i,j\}}P_{k}\left(
T_{j}<T_{i}\right)  p_{ik}\right) \\
&  =a^{4^{l-2}+1}P_{i}\left(  T_{j}<T_{i}^{+}\right)  .
\end{align*}
The first inequality comes from the use of (\ref{eqn:prob_bounds}) and part 1;
the last equality holds since we can do a first step analysis on
$\{Z_{n}\}_{n}.$ Similarly, we can show the lower bound.
\end{proof}

\vspace{\baselineskip} Let $\check{N}\doteq \inf\{n\in%
\mathbb{N}
_{0}:\tilde{Z}_{n}\in\cup_{j\in L\setminus\{1\}}\mathcal{X}_{j}\}$ and
$\tilde{N}\doteq \inf\{n\in%
\mathbb{N}
:Z_{n}\in\mathcal{X}_{1},n\geq\check{N}\}.$ For any $j\in L,$ let $\tilde
{N}_{j}$ be the number of visits (including time $0$) of state $\mathcal{X}%
_{j}$ before $\tilde{N},$ i.e. $\tilde{N}_{j}=| \{n\in%
\mathbb{N}
_{0}:n<\tilde{N}\text{ and }Z_{n}\in\mathcal{X}_{j}\}| .$ We would
like to understand $E_{x}\tilde{N}_{j}$ for any $j\in L$ and $x\in
\mathcal{X}_{1}$ or $\mathcal{X}_{j}.$

\begin{lemma}
\label{Lem:6.5}For any $j\in L$ and $x\in\mathcal{X}_{1}$
\[
E_{x}\tilde{N}_{j}\leq\frac{a^{4^{l-1}}}{\sum_{\ell\in L\setminus
\{1\}}p_{1\ell}}\frac{\sum_{g\in G\left(  j\right)  }\pi\left(  g\right)
}{\sum_{g\in G\left(  1\right)  }\pi\left(  g\right)  }.
\]
Moreover, for any $j\in L\setminus\{1\}$
\[
\sum_{\ell=1}^{\infty}\sup_{x\in\mathcal{X}_{j}}P_{x}\left(  \tilde{N}_{j}%
\geq\ell\right)  \leq a^{4^{l-1}}\frac{\sum_{g\in G\left(  1,j\right)  }%
\pi\left(  g\right)  }{\sum_{g\in G\left(  1\right)  }\pi\left(  g\right)  }%
\text{ and }
\sum_{\ell=1}^{\infty}\sup_{x\in\mathcal{X}_{1}}P_{x}\left(  \tilde{N}_{1}%
\geq\ell\right)  \leq\frac{a}{\sum_{\ell\in L\setminus\{1\}}p_{1\ell}}.
\]

\end{lemma}

\begin{proof}
For any $x\in\mathcal{X}_{1},$ note that for any $\ell\in%
\mathbb{N}
,$ by a conditioning argument as in the proof of Lemma \ref{Lem:6.3} (3), we
find that for $j\in L\setminus\{1\}$%
\[
P_{x}(  \tilde{N}_{j}\geq\ell)  \leq\frac{\sup_{y\in\mathcal{X}%
_{1}}P_{y}(  \tilde{T}_{j}<\tilde{T}_{1}^{+})  }{1-\sup
_{y\in\mathcal{X}_{1}}p\left(  y,\mathcal{X}_{1}\right)  }\left(  \sup\nolimits
_{y\in\mathcal{X}_{j}}P_{y}(  \tilde{T}_{j}^{+}<\tilde{T}_{1})
\right)  ^{\ell-1}%
\]
and
\[
P_{x}(  \tilde{N}_{1}\geq\ell)  \leq\left(  \sup\nolimits_{y\in
\mathcal{X}_{1}}p\left(  y,\mathcal{X}_{1}\right)  \right)  ^{\ell-1}.
\]
Thus, for any $x\in\mathcal{X}_{1}$ and for $j\in L\setminus\{1\}$%
\begin{align*}
E_{x}\tilde{N}_{j}  &  =\sum_{\ell=1}^{\infty}P_{x}(  \tilde{N}_{j}%
\geq\ell) 
  \leq\frac{\sup_{y\in\mathcal{X}_{1}}P_{y}(  \tilde{T}_{j}<\tilde
{T}_{1}^{+})  }{1-\sup_{y\in\mathcal{X}_{1}}p\left(  y,\mathcal{X}%
_{1}\right)  }\cdot\frac{1}{1-\sup_{y\in\mathcal{X}_{j}}P_{y}(  \tilde
{T}_{j}^{+}<\tilde{T}_{1})  }\\
&  =\frac{\sup_{y\in\mathcal{X}_{1}}P_{y}(  \tilde{T}_{j}<\tilde{T}%
_{1}^{+})  }{\left(  \inf_{y\in\mathcal{X}_{j}}\left(  1-p\left(
y,\mathcal{X}_{1}\right)  \right)  \right)  (  \inf_{y\in\mathcal{X}_{j}%
}P_{y}(  \tilde{T}_{1}<\tilde{T}_{j}^{+})  )  }\\
&  \leq a^{4^{l-1}}\frac{P_{1}(  T_{j}<T_{1}^{+})  }{(
\sum_{\ell\in L\setminus\{1\}}p_{1\ell})  P_{j}(  T_{1}<T_{j}%
^{+})  }\\
&  =\frac{a^{4^{l-1}}}{\sum_{\ell\in L\setminus\{1\}}p_{1\ell}}\frac
{\lambda_{j}}{\lambda_{1}}
=\frac{a^{4^{l-1}}}{\sum_{\ell\in L\setminus\{1\}}p_{1\ell}}\frac
{\sum_{g\in G\left(  j\right)  }\pi\left(  g\right)  }{\sum_{g\in G\left(
1\right)  }\pi\left(  g\right)  }.
\end{align*}
The second inequality is from Remark \ref{Rmk:6.1} and Lemma \ref{Lem:6.4}
(2); the third equality comes from Lemma \ref{Lem:6.3} (2); the last equality
holds due to Lemma \ref{Lem:6.2}.
Also 
\begin{align*}
E_{x}\tilde{N}_{1}    =\sum_{\ell=1}^{\infty}P_{x}(  \tilde{N}_{1}%
\geq\ell)  \leq\frac{1}{1-\sup_{y\in\mathcal{X}_{1}}p\left(
y,\mathcal{X}_{1}\right)  }
  =\frac{1}{\inf_{y\in\mathcal{X}_{1}}\left(  1-p\left(  y,\mathcal{X}%
_{1}\right)  \right)  }\leq\frac{a}{\sum_{\ell\in L\setminus\{1\}}p_{1\ell}}.
\end{align*}
The last inequality is from Remark \ref{Rmk:6.1}. This completes the proof of
part 1.

Turning to part 2, since for any $\ell\in%
\mathbb{N}
$%
\[
\sup\nolimits_{x\in\mathcal{X}_{1}}P_{x}(  \tilde{N}_{1}\geq\ell)
\leq\left(  \sup\nolimits_{y\in\mathcal{X}_{1}}p\left(  y,\mathcal{X}_{1}\right)
\right)  ^{\ell-1},
\]
we have%
\[
\sum_{\ell=1}^{\infty}\sup_{x\in\mathcal{X}_{1}}P_{x}(  \tilde{N}_{1}%
\geq\ell)  \leq\frac{1}{1-\sup_{y\in\mathcal{X}_{1}}p\left(
y,\mathcal{X}_{1}\right)  }\leq\frac{a}{\sum_{\ell\in L\setminus
\{1\}}p_{1\ell}}.
\]
Furthermore, we use the conditioning argument again to find that for any $j\in
L\setminus\{1\}$ and $\ell\in%
\mathbb{N}
$
\[
\sup\nolimits_{x\in\mathcal{X}_{j}}P_{x}(  \tilde{N}_{j}\geq\ell)
\leq(  \sup\nolimits_{y\in\mathcal{X}_{j}}P_{y}(  \tilde{T}_{j}^{+}<\tilde
{T}_{1})  )  ^{\ell-1}.
\]
This implies that%
\begin{align*}
&\sum_{\ell=1}^{\infty}\sup\nolimits_{x\in\mathcal{X}_{j}}P_{x}(  \tilde{N}_{j}%
\geq\ell) \\  
&\qquad  \leq\sum_{\ell=1}^{\infty}(  \sup\nolimits_{y\in
\mathcal{X}_{j}}P_{y}(  \tilde{T}_{j}^{+}<\tilde{T}_{1})  )
^{\ell-1}
 =\frac{1}{1-\sup_{y\in\mathcal{X}_{j}}P_{y}(  \tilde{T}_{j}^{+}%
<\tilde{T}_{1})  }\\
&\qquad   =\frac{1}{\inf_{y\in\mathcal{X}_{j}}  P_{y}(  \tilde{T}%
_{1}<\tilde{T}_{j}^{+})    }
  \leq a^{4^{l-1}}\frac{1}{P_{j}(  T_{1}<T_{j}^{+})  }\\
& \qquad  =a^{4^{l-1}}\lambda_{j}(E_{1}T_{j}+E_{j}T_{1}) =a^{4^{l-1}}\frac{\sum_{g\in
G\left(  1,j\right)  }\pi\left(  g\right)  }{\sum_{g\in G\left(  1\right)
}\pi\left(  g\right)  }.
\end{align*}
We use Lemma \ref{Lem:6.4} (2) to obtain the second inequality and Lemma
\ref{Lem:6.3}, parts (2) and (1), for the penultimate and last equalities.
\end{proof}

\subsection{Asymptotics of moments of $S_{1}^{\varepsilon}$}

\label{sec:proof_of_asymptotics_of_moments_of_S}

Recall that $\{X^{\varepsilon}\}_{\varepsilon\in(0,\infty)}\subset
C([0,\infty):M)$ is a sequence of stochastic processes satisfying Condition
\ref{Con:3.1}, Condition \ref{Con:3.2} and Condition \ref{Con:3.3}. Moreover,
recall that $S_{1}^{\varepsilon}$ is defined by
\begin{equation}
\label{eqn:se1}S_{1}^{\varepsilon}\doteq\int_{0}^{\tau_{1}^{\varepsilon}%
}e^{-\frac{1}{\varepsilon}f\left(  X_{t}^{\varepsilon}\right)  }1_{A}\left(
X_{t}^{\varepsilon}\right)  dt.
\end{equation}
As mentioned in Section \ref{sec:wald's_identities}, we are interested in the
logarithmic asymptotics of $E_{\lambda^{\varepsilon}}S_{1}^{\varepsilon}$ and
$E_{\lambda^{\varepsilon}}(S_{1}^{\varepsilon})^{2}.$ To find these
asymptotics, the main tool we will use is Freidlin-Wentzell theory
\cite{frewen2}. In fact, we will generalize the results of Freidlin-Wentzell
to the following: For any given continuous function $f:M\rightarrow%
\mathbb{R}
$ and any compact set $A\subset M,$ we will provide lower bounds for
\begin{equation}
\liminf_{\varepsilon\rightarrow0}-\varepsilon\log\left(  \sup_{z\in\partial
B_{\delta}(O_{1})}E_{z}\left(  \int_{0}^{\tau_{1}^{\varepsilon}}e^{-\frac
{1}{\varepsilon}f\left(  X_{s}^{\varepsilon}\right)  }1_{A}\left(
X_{s}^{\varepsilon}\right)  ds\right)  \right)  \label{eqn:liminf_mean}%
\end{equation}
and
\begin{equation}
\liminf_{\varepsilon\rightarrow0}-\varepsilon\log\left(  \sup_{z\in\partial
B_{\delta}(O_{1})}E_{z}\left(  \int_{0}^{\tau_{1}^{\varepsilon}}e^{-\frac
{1}{\varepsilon}f\left(  X_{s}^{\varepsilon}\right)  }1_{A}\left(
X_{s}^{\varepsilon}\right)  ds\right)  ^{2}\right)  .
\label{eqn:liminf_2nd_moment}%
\end{equation}
As will be shown, these two bounds can be expressed in terms of the
quasipotentials $V(O_{i},O_{j})$ and $V(O_{i},x).$

\begin{remark}
In the Freidlin-Wentzell theory as presented in \cite{frewen2}, they only
consider bounds for%
\[
\liminf_{\varepsilon\rightarrow0}-\varepsilon\log\left(  \sup\nolimits_{z\in
\partial B_{\delta}(O_{1})}E_{z}\tau_{1}^{\varepsilon}\right)  .
\]
Thus, their result is a special case of (\ref{eqn:liminf_mean}) with
$f\equiv0$ and $A=M$. Moreover, we generalize their result further by
considering the logarithmic asymptotics of higher moment quantities such as
(\ref{eqn:liminf_2nd_moment}).
\end{remark}

Before proceeding, we recall that $L=\{1,\ldots,l\}$ and for any $\delta>0,$
we define $\tau_{0}\doteq0,$
\[
\sigma_{n}\doteq\inf\{t>\tau_{n}:X_{t}^{\varepsilon}\in%
{\textstyle\bigcup\nolimits_{j\in L}}
\partial B_{2\delta}(O_{j})\} \text{ and } \tau_{n}\doteq\inf\{t>\sigma
_{n-1}:X_{t}^{\varepsilon}\in%
{\textstyle\bigcup\nolimits_{j\in L}}
\partial B_{\delta}(O_{j})\}.
\]
Moreover, $\tau_{0}^{\varepsilon}\doteq0,$
\[
\sigma_{n}^{\varepsilon}\doteq\inf\{t>\tau_{n}^{\varepsilon}:X_{t}%
^{\varepsilon}\in%
{\textstyle\bigcup\nolimits_{j\in L\setminus\{1\}}}
\partial B_{\delta}(O_{j})\} \text{ and }\tau_{n}^{\varepsilon}\doteq
\inf\left\{  t>\sigma_{n-1}^{\varepsilon}:X_{t}^{\varepsilon}\in\partial
B_{\delta}(O_{1})\right\}  .
\]
In addition, $\{Z_{n}\}_{n\in%
\mathbb{N}
_{0}}\doteq\{X_{\tau_{n}}^{\varepsilon}\}_{n\in%
\mathbb{N}
_{0}}$ is a Markov chain on $%
{\textstyle\bigcup\nolimits_{j\in L}}
\partial B_{\delta}(O_{j})$ and $\{Z_{n}^{\varepsilon}\}_{n\in%
\mathbb{N}
_{0}}$ $\doteq\{X_{\tau_{n}^{\varepsilon}}^{\varepsilon}\}_{n\in%
\mathbb{N}
_{0}}$ is a Markov chain on $\partial B_{\delta}(O_{1}).$ It is essential to
keep the distinction clear: when there is an $\varepsilon$ superscript the
chain makes transitions between neighborhoods of distinct equilibria, while if
absent such transitions are possible, but for stable equilibria there will be
many more transitions between the $\delta$ and $2\delta$ neighborhoods.

Following the notation of Subsection
\ref{subsec:markov_chain_and_graph_theory}, let $\hat{N}\doteq\inf\{n\in%
\mathbb{N}
_{0}:Z_{n}\in%
{\textstyle\bigcup\nolimits_{j\in L\setminus\{1\}}}
\partial B_{\delta}(O_{j})\}$, $N\doteq\inf\{n\geq\hat{N}:Z_{n}\in\partial
B_{\delta}(O_{1})\}$, and recall $\mathcal{F}_{t}\doteq\sigma(\{X_{s}%
^{\varepsilon};s\leq t\})$. Then since $\{\tau_{n}\}_{n\in%
\mathbb{N}
_{0}}$ are stopping times with respect to the filtration $\{\mathcal{F}%
_{t}\}_{t\geq0},$ $\mathcal{F}_{\tau_{n}}$ are well-defined for any $n\in%
\mathbb{N}
_{0}$ and we use $\mathcal{G}_{n}$ to denote $\mathcal{F}_{\tau_{n}}.$ One can
prove that $\hat{N}$ and $N$ are stopping times with respect to $\{\mathcal{G}%
_{n}\}_{n\in%
\mathbb{N}
}.$ For any $j\in L,$ let $N_{j}$ be the number of visits of $\{Z_{n}\}_{n\in%
\mathbb{N}
_{0}}$ to $\partial B_{\delta}(O_{j})$ (including time $0$) before $N.$

The proofs of the following two lemmas are given in the Appendix.

\begin{lemma}
\label{Lem:6.6}Given $\delta>0$ sufficiently small, for any $x\in\partial
B_{\delta}(O_{1})$ and any nonnegative measurable function $g$ $:M\rightarrow%
\mathbb{R}
$,
\[
E_{x}\left(  \int_{0}^{\tau_{1}^{\varepsilon}}g\left(  X_{s}^{\varepsilon
}\right)  ds\right)  \leq\sum_{j\in L}\left[  \sup_{y\in\partial B_{\delta
}(O_{j})}E_{y}\left(  \int_{0}^{\tau_{1}}g\left(  X_{s}^{\varepsilon}\right)
ds\right)  \right]  \cdot E_{x}N_{j}.
\]

\end{lemma}

\begin{lemma}
\label{Lem:6.7}Given $\delta>0$ sufficiently small$,$ for any $x\in\partial
B_{\delta}(O_{1})$ and any nonnegative measurable function $g$ $:M\rightarrow%
\mathbb{R}
$,%
\begin{align}
E_{x}\left(  \int_{0}^{\tau_{1}^{\varepsilon}}g\left(  X_{s}^{\varepsilon
}\right)  ds\right)  ^{2}  &  \leq l\sum_{j\in L}\left[  \sup_{y\in\partial
B_{\delta}(O_{j})}E_{y}\left(  \int_{0}^{\tau_{1}}g\left(  X_{s}^{\varepsilon
}\right)  ds\right)  ^{2}\right]  \cdot E_{x}N_{j}\nonumber\\
&  \qquad+2l\sum_{j\in L}\left[  \sup_{y\in\partial B_{\delta}(O_{j})}%
E_{y}\left(  \int_{0}^{\tau_{1}}g\left(  X_{s}^{\varepsilon}\right)
ds\right)  \right]  ^{2}\cdot E_{x}N_{j}\nonumber\\
&  \quad\quad\qquad\cdot\sum_{k=1}^{\infty}\sup_{y\in\partial B_{\delta}%
(O_{j})}P_{y}\left(  k\leq N_{j}\right)  . \label{eqn:7.11}%
\end{align}

\end{lemma}

Although as noted the proofs are given in the Appendix, these results follow
in a straightforward way by decomposing the excursion away from $O_{1}$ during
$[0,\tau_{1}^{\varepsilon}]$, which only stops when returning to a
neighborhood of $O_{1}$, into excursions between any pair of equilibrium
points, counting the number of such excursions that start near a particular
equilibrium point, and using the strong Markov property.

\begin{remark}
\label{Rmk:6.3}Following an analogous argument as in the proof of Lemma
\ref{Lem:6.6} and Lemma \ref{Lem:6.7}, we can prove the following: Given
$\delta>0$ sufficiently small, for any $x\in\partial B_{\delta}(O_{1})$ and
any nonnegative measurable function $g$ $:M\rightarrow%
\mathbb{R}
$,
\[
E_{x}\left(  \int_{\sigma_{0}^{\varepsilon}}^{\tau_{1}^{\varepsilon}}g\left(
X_{s}^{\varepsilon}\right)  ds\right)  \leq%
{\textstyle\sum_{j\in L\setminus\{1\}}}
\left[  \sup_{y\in\partial B_{\delta}(O_{j})}E_{y}\left(  \int_{0}^{\tau_{1}%
}g\left(  X_{s}^{\varepsilon}\right)  ds\right)  \right]  \cdot E_{x}N_{j}%
\]
and
\begin{align*}
E_{x}\left(  \int_{\sigma_{0}^{\varepsilon}}^{\tau_{1}^{\varepsilon}}g\left(
X_{s}^{\varepsilon}\right)  ds\right)  ^{2}  &  \leq l%
{\textstyle\sum_{j\in L\setminus\{1\}}}
\left[  \sup_{y\in\partial B_{\delta}(O_{j})}E_{y}\left(  \int_{0}^{\tau_{1}%
}g\left(  X_{s}^{\varepsilon}\right)  ds\right)  ^{2}\right]  \cdot E_{x}%
N_{j}\\
&  \quad+2l%
{\textstyle\sum_{j\in L\setminus\{1\}}}
\left[  \sup_{y\in\partial B_{\delta}(O_{j})}E_{y}\left(  \int_{0}^{\tau_{1}%
}g\left(  X_{s}^{\varepsilon}\right)  ds\right)  \right]  ^{2}\cdot E_{x}%
N_{j}\\
&  \quad\qquad\cdot%
{\textstyle\sum_{\ell=1}^{\infty}}
\sup_{y\in\partial B_{\delta}(O_{j})}P_{y}\left(  k\leq N_{j}\right)  .
\end{align*}
The main difference is that if the integration starts from $\sigma
_{0}^{\varepsilon}$ (the first visiting time of $%
{\textstyle\bigcup\nolimits_{j\in L\setminus\{1\}}}
\partial B_{\delta}(O_{j})$), then any summation appearing in the upper bounds
should sum over all indices in $L\setminus\{1\}$ instead of $L.$
\end{remark}

Owing to its frequent appearance but with varying arguments, we introduce the
notation%
\begin{equation}
I^{\varepsilon}(t_{1},t_{2};f,A)\doteq\int_{t_{1}}^{t_{2}}e^{-\frac
{1}{\varepsilon}f(X_{s}^{\varepsilon})}1_{A}(X_{s}^{\varepsilon})ds,
\label{eqn:defofI}%
\end{equation}
and write $I^{\varepsilon}(t;f,A)$ if $t_{1}=0$ and $t_{2}=t$ so that, e.g.,
$S_{1}^{\varepsilon}=I^{\varepsilon}(\tau_{1}^{\varepsilon};f,A)$.

\begin{corollary}
\label{Cor:6.1} 
Given any
measurable set $A\subset M$, a measurable function $f:M\rightarrow%
\mathbb{R}
,$ $j\in L$ and $\delta>0,$ we have%
\begin{align*}
&  \liminf_{\varepsilon\rightarrow0}-\varepsilon\log\left(  \sup_{z\in\partial
B_{\delta}(O_{1})}E_{z}I^{\varepsilon}(\tau_{1}^{\varepsilon};f,A)\right) \\
&  \quad\geq\min_{j\in L}\left\{  \liminf_{\varepsilon\rightarrow
0}-\varepsilon\log\left(  \sup_{z\in\partial B_{\delta}(O_{1})}E_{z}%
N_{j}\right)  +\liminf_{\varepsilon\rightarrow0}-\varepsilon\log\left(
\sup_{z\in\partial B_{\delta}(O_{j})}E_{z}I^{\varepsilon}(\tau_{1}%
;f,A)\right)  \right\}  ,
\end{align*}
and
\[
\liminf_{\varepsilon\rightarrow0}-\varepsilon\log\left(  \sup_{z\in\partial
B_{\delta}(O_{1})}E_{z}I^{\varepsilon}(\tau_{1};f,A)^{2}\right)  \geq
\min_{j\in L}\left(  \hat{R}_{j}^{(1)}\wedge\hat{R}_{j}^{(2)}\right)  ,
\]
where
\[
\hat{R}_{j}^{(1)}\doteq\liminf_{\varepsilon\rightarrow0}-\varepsilon
\log\left(  \sup_{z\in\partial B_{\delta}(O_{j})}E_{z}I^{\varepsilon}(\tau
_{1};f,A)^{2}\right)  +\liminf_{\varepsilon\rightarrow0}-\varepsilon
\log\left(  \sup_{z\in\partial B_{\delta}(O_{1})}E_{z}N_{j}\right)
\]
and
\begin{align*}
\hat{R}_{j}^{(2)}  &  \doteq2\liminf_{\varepsilon\rightarrow0}-\varepsilon
\log\left(  \sup_{z\in\partial B_{\delta}(O_{j})}E_{z}I^{\varepsilon}(\tau
_{1};f,A)\right)  +\liminf_{\varepsilon\rightarrow0}-\varepsilon\log\left(
\sup_{z\in\partial B_{\delta}(O_{1})}E_{z}N_{j}\right) \\
&  \qquad+\liminf_{\varepsilon\rightarrow0}-\varepsilon\log\left(  \sum\nolimits
_{\ell=1}^{\infty}\sup_{z\in\partial B_{\delta}(O_{j})}P_{z}\left(  \ell\leq
N_{j}\right)  \right)  .
\end{align*}

\end{corollary}

\begin{proof}
For the first part, applying Lemma \ref{Lem:6.6} with $g(x)=e^{-\frac
{1}{\varepsilon}f\left(  x\right)  }1_{A}\left(  x\right)  $ and using
(\ref{eqn:product}) and (\ref{eqn:sum}) completes the proof. For the second
part, using Lemma \ref{Lem:6.7} with $g(x)=e^{-\frac{1}{\varepsilon}f\left(
x\right)  }1_{A}\left(  x\right)  $ and using (\ref{eqn:product}) and
(\ref{eqn:sum}) again completes the proof.
\end{proof}

\begin{remark}
\label{Rmk:6.4}Owing to Remark \ref{Rmk:6.3}, we can modify the proof of
Corollary \ref{Cor:6.1} and show that given any set $A\subset M,$ a measurable
function $f:M\rightarrow%
\mathbb{R}
,$ $j\in L$ and $\delta>0,$%
\begin{align*}
&  \liminf_{\varepsilon\rightarrow0}-\varepsilon\log\left(  \sup_{z\in\partial
B_{\delta}(O_{1})}E_{z}I^{\varepsilon}(\sigma_{0}^{\varepsilon},\tau
_{1}^{\varepsilon};f,A)\right) \\
&  \quad\geq\min_{j\in L\setminus\{1\}}\left\{  \liminf_{\varepsilon
\rightarrow0}-\varepsilon\log\left(  \sup_{z\in\partial B_{\delta}(O_{1}%
)}E_{z}N_{j}\right)  +\liminf_{\varepsilon\rightarrow0}-\varepsilon\log\left(
\sup_{z\in\partial B_{\delta}(O_{j})}E_{z}I^{\varepsilon}(\tau_{1}%
;f,A)\right)  \right\}  .
\end{align*}
Moreover,%
\[
\liminf_{\varepsilon\rightarrow0}-\varepsilon\log\left(  \sup_{z\in\partial
B_{\delta}(O_{1})}E_{z}I^{\varepsilon}(\sigma_{0}^{\varepsilon},\tau
_{1}^{\varepsilon};f,A)^{2}\right)  \geq\min_{j\in L\setminus\{1\}}\left(
\hat{R}_{j}^{(1)}\wedge\hat{R}_{j}^{(2)}\right)  ,
\]
where the definitions of $\hat{R}_{j}^{(1)}$ and $\hat{R}_{j}^{(2)}$ can be
found in Corollary \ref{Cor:6.1}.
\end{remark}

We next consider lower bounds on
\[
\liminf_{\varepsilon\rightarrow0}-\varepsilon\log\left(  \sup_{z\in\partial
B_{\delta}(O_{j})}E_{z}I^{\varepsilon}(\tau_{1};f,A)\right)  \quad
\mbox{and}\quad\liminf_{\varepsilon\rightarrow0}-\varepsilon\log\left(
\sup_{z\in\partial B_{\delta}(O_{j})}E_{z}I^{\varepsilon}(\tau_{1}%
;f,A)^{2}\right)
\]
for $j\in L$.
We state some
useful results before studying the lower bounds. Recall also that $\tau_{1}$
is the time to reach the $\delta$-neighborhood of any of
the equilibrium points after leaving the $2\delta$-neighborhood of one of the
equilibrium points.

\begin{lemma}
\label{Lem:6.9}For any $\eta>0,$ there exists $\delta_{0}\in(0,1)$ and
$\varepsilon_{0}\in(0,1)$, such that for all $\delta\in(0,\delta_{0})$ and
$\varepsilon\in(0,\varepsilon_{0})$
\[
\sup_{x\in M}E_{x}\tau_{1}\leq e^{\frac{\eta}{\varepsilon}}\text{ and }%
\sup_{x\in M}E_{x}\left(  \tau_{1}\right)  ^{2}\leq e^{\frac{\eta}%
{\varepsilon}}.
\]

\end{lemma}

\begin{proof}
If $x$ is not in $\cup_{j\in L}B_{2\delta}(O_{j})$ then a uniform (in $x$ and
small $\varepsilon$) upper bound on these expected values follows from the
corollary to \cite[Lemma 1.9, Chapter 6]{frewen2}.

If $x\in\cup_{j\in L}B_{2\delta}(O_{j})$ then we must wait till the process
reaches $\cup_{j\in L}\partial B_{2\delta}(O_{j})$, after which we can use the
uniform bound (and the strong Markov property). Since there exists $\delta>0$
such the lower bound $P_{x}(\inf\{t\geq0:X_{t}^{\varepsilon}\in\cup_{j\in
L}\partial B_{2\delta}(O_{j})\leq1)\geq e^{-\eta/2\varepsilon}$ is valid for
all $x\in\cup_{j\in L}B_{2\delta}(O_{j})$ and small $\varepsilon>0$, upper
bounds of the desired form follow from the Markov property and standard calculations.
\end{proof}

\vspace{\baselineskip} For any compact set $A\subset M$, we use $\vartheta
_{A}$ to denote the first hitting time
\[
\vartheta_{A}\doteq\inf\left\{  t\geq0:X_{t}^{\varepsilon}\in A\right\}  .
\]
Note that $\vartheta_{A}$ is a stopping time with respect to filtration
$\{\mathcal{F}_{t}\}_{t\geq0}.$ The following result is relatively
straightforward given the just discussed bound on the distribution of
$\tau_{1}$, and follows by partitioning according to $\tau_{1} \geq T$ and
$\tau_{1} < T$ for large but fixed $T$.

\begin{lemma}
\label{Lem:6.10}For any compact set $A\subset M,$ $j\in L$ and any $\eta>0,$
there exists $\delta_{0}\in(0,1)$ and $\varepsilon_{0}\in(0,1)$, such that for
all $\varepsilon\in(0,\varepsilon_{0})$ and $\delta\in(0,\delta_{0})$
\[
\sup_{z\in\partial B_{\delta}(O_{j})}P_{z}\left(  \vartheta_{A}\leq\tau
_{1}\right)  \leq e^{-\frac{1}{\varepsilon}\left(  \inf_{x\in A}\left[
V\left(  O_{j},x\right)  \right]  -\eta\right)  }.
\]

\end{lemma}

\begin{lemma}
\label{Lem:6.11}Given a compact set $A\subset M$, any $j\in L$ and $\eta>0,$
there exists $\delta_{0}\in(0,1),$ such that for any $\delta\in(0,\delta
_{0})$
\[
\liminf_{\varepsilon\rightarrow0}-\varepsilon\log\left(  \sup_{z\in\partial
B_{\delta}(O_{j})}E_{z}\left[  \int_{0}^{\tau_{1}}1_{A}\left(  X_{s}%
^{\varepsilon}\right)  ds\right]  \right)  \geq\inf_{x\in A}V\left(
O_{j},x\right)  -\eta
\]
and
\[
\liminf_{\varepsilon\rightarrow0}-\varepsilon\log\left(  \sup_{z\in\partial
B_{\delta}(O_{j})}E_{z}\left(  \int_{0}^{\tau_{1}}1_{A}\left(  X_{s}%
^{\varepsilon}\right)  ds\right)  ^{2}\right)  \geq\inf_{x\in A}V\left(
O_{j},x\right)  -\eta.
\]

\end{lemma}

\begin{proof}
The idea of this proof follows from the proof of Theorem 4.3 in \cite[Chapter
4]{frewen2}. Since $I^{\varepsilon}(\tau_{1};0,A)=\int_{0}^{\tau_{1}}%
1_{A}\left(  X_{s}^{\varepsilon}\right)  ds$, for any $x\in\partial B_{\delta
}(O_{j}),$%
\begin{align*}
&  E_{x}I^{\varepsilon}(\tau_{1};0,A)\\
&  =E_{x}\left[  I^{\varepsilon}(\tau_{1};0,A)1_{\left\{  \vartheta_{A}%
\leq\tau_{1}\right\}  }\right]  =E_{x}\left[  E_{x}\left[  \left.
I^{\varepsilon}(\tau_{1};0,A)\right\vert \mathcal{F}_{\vartheta_{A}}\right]
1_{\left\{  \vartheta_{A}\leq\tau_{1}\right\}  }\right]  \\
&  =E_{x}\left[  (  E_{X_{\vartheta_{A}}^{\varepsilon}}I^{\varepsilon
}(\tau_{1};0,A))  1_{\left\{  \vartheta_{A}\leq\tau_{1}\right\}
}\right]  \leq \sup\nolimits_{y\in\partial A}E_{y}\tau_{1} \cdot  \sup\nolimits
_{z\in\partial B_{\delta}(O_{j})}P_{z}\left(  \vartheta_{A}\leq\tau
_{1}\right)  .
\end{align*}
The inequality is due to%
$
E_{X_{\vartheta_{A}}^{\varepsilon}}I^{\varepsilon}(\tau_{1};0,A)\leq
E_{X_{\vartheta_{A}}^{\varepsilon}}\tau_{1}\leq\sup_{y\in\partial A}E_{y}%
\tau_{1}.
$ 
We then apply Lemma \ref{Lem:6.9} and Lemma \ref{Lem:6.10} to find that for
the given $\eta>0,$ there exists $\delta_{0}\in(0,1)$ and $\varepsilon_{0}%
\in(0,1)$, such that for all $\varepsilon\in(0,\varepsilon_{0})$ and
$\delta\in(0,\delta_{0}),$%
\[
E_{x}I^{\varepsilon}(\tau_{1};0,A)\leq\sup_{y\in\partial A}E_{y}\tau_{1}%
\cdot\sup_{z\in\partial B_{\delta}(O_{j})}P_{z}\left(  \vartheta_{A}\leq
\tau_{1}\right)  \leq e^{\frac{\eta/2}{\varepsilon}}e^{-\frac{1}{\varepsilon
}\left(  \inf_{y\in A}V\left(  O_{j},y\right)  -\eta/2\right)  }.
\]
Thus,
\begin{align*}
&  \liminf_{\varepsilon\rightarrow0}-\varepsilon\log\left(  \sup\nolimits_{z\in\partial
B_{\delta}(O_{j})}E_{z}I^{\varepsilon}(\tau_{1};0,A)\right)  
\geq \inf_{x\in A}V\left(  O_{j},x\right)  -\eta.
\end{align*}
This completes the proof of part 1.

For part 2, following the same conditioning argument as for part 1 with the
use of Lemma \ref{Lem:6.9} and Lemma \ref{Lem:6.10} gives that for the given
$\eta>0,$ there exists $\delta_{0}\in(0,1)$ and $\varepsilon_{0}\in(0,1)$,
such that for all $\varepsilon\in(0,\varepsilon_{0})$ and $\delta\in
(0,\delta_{0}),$%
\[
E_{x}I^{\varepsilon}(\tau_{1};0,A)^{2}\leq\sup_{y\in\partial A}E_{y}\left(
\tau_{1}\right)  ^{2}\cdot\sup_{z\in\partial B_{\delta}(O_{j})}P_{z}\left(
\vartheta_{A}\leq\tau_{1}\right)  \leq e^{\frac{\eta/2}{\varepsilon}}%
e^{-\frac{1}{\varepsilon}\left(  \inf_{x\in A}V\left(  O_{j},x\right)
-\eta/2\right)  }.
\]
Therefore,%
\begin{align*}
 \liminf_{\varepsilon\rightarrow0}-\varepsilon\log\left(  \sup\nolimits_{z\in\partial
B_{\delta}(O_{j})}E_{z}I^{\varepsilon}(\tau_{1};0,A)^{2}\right)  
\geq\inf_{x\in A}V\left(  O_{j},x\right)  -\eta.
\end{align*}

\end{proof}

\begin{lemma}
\label{Lem:6.12}Given compact sets $A_{1},A_{2}\subset M$, $j\in L$ and
$\eta>0,$ there exists $\delta_{0}\in(0,1),$ such that for any $\delta
\in(0,\delta_{0})$
\begin{align*}
&  \liminf_{\varepsilon\rightarrow0}-\varepsilon\log\left(  \sup_{z\in\partial
B_{\delta}(O_{j})}E_{z}\left[  \left(  \int_{0}^{\tau_{1}}1_{A_{1}}\left(
X_{s}^{\varepsilon}\right)  ds\right)  \left(  \int_{0}^{\tau_{1}}1_{A_{2}%
}\left(  X_{s}^{\varepsilon}\right)  ds\right)  \right]  \right) \\
&  \qquad\geq\max\left\{  \inf_{x\in A_{1}}V\left(  O_{j},x\right)
,\inf_{x\in A_{2}}V\left(  O_{j},x\right)  \right\}  -\eta.
\end{align*}

\end{lemma}

\begin{proof}
We set $\vartheta_{A_{i}}\doteq\inf\left\{  t\geq0:X_{t}^{\varepsilon}\in
A_{i}\right\}  $ for $i=1,2.$ For any $x\in\partial B_{\delta}(O_{j}),$ using
a conditioning argument as in the proof of Lemma \ref{Lem:6.11} we obtain that
for any $\eta>0,$ there exists $\delta_{0}\in(0,1)$ and $\varepsilon_{0}%
\in(0,1)$, such that for all $\varepsilon\in(0,\varepsilon_{0})$ and
$\delta\in(0,\delta_{0}),$%
\begin{align}
&  E_{x}\left[  \left(  \int_{0}^{\tau_{1}}1_{{A}_{1}}\left(  X_{s}%
^{\varepsilon}\right)  ds\right)  \left(  \int_{0}^{\tau_{1}}1_{{A}_{2}%
}\left(  X_{s}^{\varepsilon}\right)  ds\right)  \right]  \label{eqn:Lem:6.12}%
\\
&  =E_{x}\left[  \left(  \int_{0}^{\tau_{1}}\int_{0}^{\tau_{1}}1_{{A}_{1}%
}\left(  X_{s}^{\varepsilon}\right)  1_{{A}_{2}}\left(  X_{t}^{\varepsilon
}\right)  dsdt\right)  1_{\left\{  \vartheta_{{A}_{1}}\vee\vartheta_{{A}_{2}%
}\leq\tau_{1}\right\}  }\right] \nonumber\\
&  =E_{x}\left[  \left(  E_{X_{\vartheta_{{A}_{1}}\vee\vartheta_{{A}_{2}}%
}^{\varepsilon}}\left[  \int_{0}^{\tau_{1}}\int_{0}^{\tau_{1}}1_{{A}_{1}%
}\left(  X_{s}^{\varepsilon}\right)  1_{{A}_{2}}\left(  X_{t}^{\varepsilon
}\right)  dsdt\right]  \right)  1_{\left\{  \vartheta_{{A}_{1}}\vee
\vartheta_{{A}_{2}}\leq\tau_{1}\right\}  }\right] \nonumber\\
&  \leq\sup\nolimits_{y\in\partial{A}_{1}\cup\partial{A}_{2}}E_{y}\left(  \tau
_{1}\right)  ^{2}\cdot\sup\nolimits_{z\in\partial B_{\delta}(O_{j})}P_{z}\left(
\vartheta_{{A}_{1}}\leq\tau_{1},\vartheta_{{A}_{2}}\leq\tau_{1}\right)
\nonumber\\
&  \leq e^{\frac{\eta/2}{\varepsilon}}\cdot\min\left\{  \sup\nolimits_{z\in
\partial B_{\delta}(O_{j})}P_{z}\left(  \vartheta_{{A}_{1}}\leq\tau
_{1}\right)  ,\sup\nolimits_{z\in\partial B_{\delta}(O_{j})}P_{z}\left(
\vartheta_{{A}_{2}}\leq\tau_{1}\right)  \right\}  ,\nonumber
\end{align}
The last inequality holds since for $i=1,2$%
\[
\sup\nolimits_{z\in\partial B_{\delta}(O_{j})}P_{z}\left(  \vartheta_{A_{1}%
}\leq\tau_{1},\vartheta_{A_{2}}\leq\tau_{1}\right)  \leq\sup\nolimits_{z\in
\partial B_{\delta}(O_{j})}P_{z}\left(  \vartheta_{A_{i}}\leq\tau_{1}\right)
\]
and owing to Lemma \ref{Lem:6.9}, for all $\varepsilon\in(0,\varepsilon_{0})$
\[
\sup\nolimits_{y\in\partial A_{1}}E_{y}\left(  \tau_{1}\right)  ^{2}\leq
e^{\frac{\eta/2}{\varepsilon}}\text{ and }\sup\nolimits_{y\in\partial A_{2}%
}E_{y}\left(  \tau_{1}\right)  ^{2}\leq e^{\frac{\eta/2}{\varepsilon}}.
\]

Furthermore, for the given $\eta>0,$ by Lemma \ref{Lem:6.10}, there exists
$\delta_{i}\in(0,1)$ such that for any $\delta\in(0,\delta_{i})$%
\[
\liminf_{\varepsilon\rightarrow0}-\varepsilon\log\left(  \sup\nolimits_{z\in
\partial B_{\delta}(O_{j})}P_{z}\left(  \vartheta_{A_{i}}\leq\tau_{1}\right)
\right)  \geq\inf_{x\in A_{i}}V\left(  O_{j},x\right)  -\eta/2
\]
for $i=1,2.$ Hence, letting $\delta_{0}=\delta_{1}\wedge\delta_{2},$ for any
$\delta\in(0,\delta_{0})$%
\begin{align*}
&  \liminf_{\varepsilon\rightarrow0}-\varepsilon\log\left(  \sup
\nolimits_{z\in\partial B_{\delta}(O_{j})}E_{z}\left[  \left(  \int_{0}%
^{\tau_{1}}1_{A_{1}}\left(  X_{s}^{\varepsilon}\right)  ds\right)  \left(
\int_{0}^{\tau_{1}}1_{A_{2}}\left(  X_{s}^{\varepsilon}\right)  ds\right)
\right]  \right) \\
&  \geq\liminf_{\varepsilon\rightarrow0}-\varepsilon\log\left(  e^{\frac{\eta
}{2\varepsilon}}\min\left\{  \sup\nolimits_{z\in\partial B_{\delta}(O_{j}%
)}P_{z}\left(  \vartheta_{A_{1}}\leq\tau_{1}\right)  ,\sup\nolimits_{z\in
\partial B_{\delta}(O_{j})}P_{z}\left(  \vartheta_{A_{2}}\leq\tau_{1}\right)
\right\}  \right) \\
&  \geq-\eta/2+\max\left\{  \liminf_{\varepsilon\rightarrow0}-\varepsilon
\log\left(  \sup\nolimits_{z\in\partial B_{\delta}(O_{j})}P_{z}\left(
\vartheta_{A_{1}}\leq\tau_{1}\right)  \right)  ,\right. \\
&  \left.  \qquad\qquad\qquad\qquad\qquad\liminf_{\varepsilon\rightarrow
0}-\varepsilon\log\left(  \sup\nolimits_{z\in\partial B_{\delta}(O_{j})}%
P_{z}\left(  \vartheta_{A_{2}}\leq\tau_{1}\right)  \right)  \right\} \\
&  \geq\max\left\{  \inf\nolimits_{x\in A_{1}}V\left(  O_{j},x\right)
,\inf\nolimits_{x\in A_{2}}V\left(  O_{j},x\right)  \right\}  -\eta.
\end{align*}
The first inequality is from (\ref{eqn:Lem:6.12}).
\end{proof}

\vspace{0.5pt}

\begin{remark}
The next lemma considers asymptotics of the first and second moments of a
certain integral that will appear in a decomposition of $S^{\varepsilon}_{1}$.
It is important to note that the variational bounds for both moments have the
same structure as an infimum over $x \in A$. While one might consider it
possible that the variational problem for the second moment could require a
pair of parameters (e.g., infimum over $x,y \in A$), the infimum is in fact
achieved on the ``diagonal'' $x=y$. This means that the biggest contribution
to the second moment is likewise due to mass along the ``diagonal.''
\end{remark}

\begin{lemma}
\label{Lem:6.14}Given a compact set $A\subset M,$ a continuous function
$f:M\rightarrow%
\mathbb{R}
,$ $j\in L$ and $\eta>0,$ there exists $\delta_{0}\in(0,1),$ such that for any
$\delta\in(0,\delta_{0})$
\[
\liminf_{\varepsilon\rightarrow0}-\varepsilon\log\left(  \sup_{z\in\partial
B_{\delta}(O_{j})}E_{z}I^{\varepsilon}(\tau_{1};f,A)\right)  \geq\inf_{x\in
A}\left[  f\left(  x\right)  +V\left(  O_{j},x\right)  \right]  -\eta
\]
and%
\[
\liminf_{\varepsilon\rightarrow0}-\varepsilon\log\left(  \sup_{z\in\partial
B_{\delta}(O_{j})}E_{z}I^{\varepsilon}(\tau_{1};f,A)^{2}\right)  \geq
\inf_{x\in A}\left[  2f\left(  x\right)  +V\left(  O_{j},x\right)  \right]
-\eta.
\]

\end{lemma}

\begin{proof}
Since a continuous function is bounded on a compact set, there exists
$m\in(0,\infty)$ such that $-m\leq f(x)\leq m$ for all $x\in A.$ For $n\in%
\mathbb{N}
$ and $k\in\{1,2,\ldots,n\},$ consider the sets
\[
A_{n,k}\doteq\left\{  x\in A:f\left(  x\right)  \in\left[  -m+\frac{2\left(
k-1\right)  m}{n},-m+\frac{2km}{n}\right]  \right\}  .
\]
Note that $A_{n,k}$ is a compact set for any $n,k.$ In addition, for any $n$
fixed, $%
{\textstyle\bigcup_{k=1}^{n}}
A_{n,k}=A.$ With this expression, for any $x\in\partial B_{\delta}(O_{j})$ and
$n\in%
\mathbb{N}
$%
\begin{align*}
E_{x}I^{\varepsilon}(\tau_{1};f,A)   
 \leq\sum\nolimits_{k=1}^{n}E_{x}I^{\varepsilon}(\tau_{1};f,{A_{n,k}})
 \leq\sum\nolimits_{k=1}^{n}E_{x}I^{\varepsilon}(\tau_{1};0,{A_{n,k}})e^{-\frac
{1}{\varepsilon}\left(  F_{n,k}  -2m/n\right)  }.
\end{align*}
The second inequality holds because by definition of $A_{n,k},$ for any $x\in
A_{n,k}$, $f(x)\geq F_{n,k} -2m/n$ with $F_{n,k}\doteq \sup_{y\in A_{n,k}}f\left(  y\right)$.

Next we
first apply (\ref{eqn:sum}) and then Lemma \ref{Lem:6.11} with compact sets
$A_{n,k}$ for $k\in\{1,2,\ldots,n\}$ to get
\begin{align*}
&  \liminf_{\varepsilon\rightarrow0}-\varepsilon\log\left(  \sup_{z\in\partial
B_{\delta}(O_{j})}E_{z}I^{\varepsilon}(\tau_{1};f,A)\right)  \\
&  \geq
  \min_{k\in\left\{  1,\ldots,n\right\}  }\left\{  \liminf_{\varepsilon
\rightarrow0}-\varepsilon\log\left(  \sup\limits_{z\in\partial B_{\delta
}(O_{j})}E_{z}I^{\varepsilon}(\tau_{1};0,{A_{n,k}})e^{-\frac{1}{\varepsilon
}\left(  F_{n,k}-\frac{2m}{n}\right)  }\right)  \right\}  \\
&  =\min_{k\in\left\{  1,\ldots,n\right\}  }\left\{  \liminf_{\varepsilon
\rightarrow0}-\varepsilon\log\left(  \sup_{z\in\partial B_{\delta}(O_{j}%
)}E_{z}I^{\varepsilon}(\tau_{1};0,{A_{n,k}})\right)  +F_{n,k}\right\}
-\frac{2m}{n}\\
&  \geq\min_{k\in\left\{  1,\ldots,n\right\}  }\left\{  \sup_{x\in A_{n,k}%
}f\left(  x\right)  +\inf_{x\in A_{n,k}}V\left(  O_{j},x\right)  \right\}
-\eta-\frac{2m}{n}.
\end{align*}
Finally, we know that $V\left(  O_{j},x\right)  $ is bounded below by $0$, and
then we use the fact that for any two functions $f,g:%
\mathbb{R}
^{d}\rightarrow%
\mathbb{R}
$ with $g$ being bounded below (to ensure that the right hand side is well
defined) and any set $A\subset%
\mathbb{R}
^{d},$ 
$
\inf_{x\in A}\left(  f\left(  x\right)  +g\left(  x\right)  \right)  \leq
\sup_{x\in A}f\left(  x\right)  +\inf_{x\in A}g\left(  x\right)
$
to find that the last minimum in the previous display is greater than or equal to
\begin{align*}
\min_{k\in\left\{  1,\ldots,n\right\}  }\left\{  \inf_{x\in
A_{n,k}}\left[  f\left(  x\right)  +V\left(  O_{j},x\right)  \right]
\right\}  =\inf_{x\in A}\left[  f\left(  x\right)  +V\left(  O_{j},x\right)
\right]  .
\end{align*}
Therefore,%
\[
\liminf_{\varepsilon\rightarrow0}-\varepsilon\log\left(  \sup_{z\in\partial
B_{\delta}(O_{j})}E_{z}I^{\varepsilon}(\tau_{1};f,A)\right)  \geq\inf_{x\in
A}\left[  f\left(  x\right)  +V\left(  O_{j},x\right)  \right]  -\eta
-\frac{2m}{n}.
\]
Since $n$ is arbitrary, sending $n\rightarrow\infty$ completes the proof for
the first part.

Turning to part 2, we follow the same argument as for part 1. For any $n\in%
\mathbb{N}
,$ we use the decomposition of $A$ into $%
{\textstyle\bigcup_{k=1}^{n}}
A_{n,k}$ to have that for any $x\in\partial B_{\delta}(O_{j}),$%
\begin{align*}
&  E_{x}I^{\varepsilon}(\tau_{1};f,A)^{2}
\leq E_{x}\left(  \sum_{k=1}^{n}I^{\varepsilon}(\tau_{1};f,{A_{n,k}%
})\right)  ^{2}=\sum_{k=1}^{n}\sum_{\ell=1}^{n}E_{x}\left[  I^{\varepsilon
}(\tau_{1};f,{A_{n,k}})I^{\varepsilon}(\tau_{1};f,A_{n,\ell})\right]  .
\end{align*}
Recall that $F_{n,k}$ is used to denote $\sup_{y\in A_{n,k}}f\left(  y\right)
$. Using the definition of $A_{n,k}$ gives that for any $k,\ell\in
\{1,\ldots,n\}$
\begin{align*}
&  E_{x}\left[  I^{\varepsilon}(\tau_{1};f,{A_{n,k}})I^{\varepsilon}(\tau
_{1};f,A_{n,\ell})\right] 
\\ & 
\leq\sup_{z\in\partial B_{\delta}(O_{j})}E_{z}\left[  I^{\varepsilon}%
(\tau_{1};0,{A_{n,k}})I^{\varepsilon}(\tau_{1};0,A_{n,\ell})\right]
e^{-\frac{1}{\varepsilon}\left(  F_{n,k}+F_{n,\ell}-\frac{4m}{n}\right)  }.
\end{align*}
Applying (\ref{eqn:sum}) first and then Lemma \ref{Lem:6.12} with compact sets
$A_{n,k}$ and $A_{n,\ell}$ pairwise for all $k,\ell\in\{1,2,\ldots,n\}$ gives
that
\begin{align*}
&  \liminf_{\varepsilon\rightarrow0}-\varepsilon\log\left(  \sup_{z\in\partial
B_{\delta}(O_{j})}E_{z}I^{\varepsilon}(\tau_{1};f,A)^{2}\right)  \\
&  \geq\min_{k,\ell\in\left\{  1,\ldots,n\right\}  }\liminf_{\varepsilon
\rightarrow0}-\varepsilon\log\sup_{z\in\partial B_{\delta}(O_{j})}E_{z}\left[
I^{\varepsilon}(\tau_{1};f,{A_{n,k}})I^{\varepsilon}(\tau_{1};f,A_{n,\ell
})\right]  \\
&  \geq\min_{k,\ell\in\left\{  1,\ldots,n\right\}  }\left\{  \max\left\{
\inf_{x\in A_{n,k}}V\left(  O_{j},x\right)  ,\inf_{x\in A_{n,\ell}}V\left(
O_{j},x\right)  \right\}  +F_{n,k}+F_{n,\ell}\right\}  -\eta-\frac{4m}{n}\\
&  \geq\min_{k\in\left\{  1,\ldots,n\right\}  }\left\{  \sup_{x\in A_{n,k}%
}\left[  2f\left(  x\right)  \right]  +\inf_{x\in A_{n,k}}V\left(
O_{j},x\right)  \right\}  -\eta-\frac{4m}{n}\\
&  \geq\min_{k\in\left\{  1,\ldots,n\right\}  }\left\{  \inf_{x\in A_{n,k}%
}\left[  2f\left(  x\right)  +V\left(  O_{j},x\right)  \right]  \right\}
-\eta-\frac{4m}{n}\\
&  =\inf_{x\in A}\left[  2f\left(  x\right)  +V\left(  O_{j},x\right)
\right]  -\eta-\frac{4m}{n}.
\end{align*}
Sending $n\rightarrow\infty$ completes the proof for the second part.
\end{proof}

\vspace{\baselineskip} Our next interest is to find lower bounds for
\[
\liminf_{\varepsilon\rightarrow0}-\varepsilon\log\left(  \sup_{z\in\partial
B_{\delta}(O_{1})}E_{z}N_{j}\right)  \text{ and } \liminf_{\varepsilon
\rightarrow0}-\varepsilon\log\left(  \sum_{\ell=1}^{\infty}\sup_{z\in\partial
B_{\delta}(O_{j})}P_{z}\left(  \ell\leq N_{j}\right)  \right)  .
\]
We first recall that $N_{j}$ is the number of visits of the embedded Markov
chain $\{Z_{n}\}_{n}=\{X_{\tau_{n}}^{\varepsilon}\}_{n}$ to $\partial
B_{\delta}(O_{j})$ within one loop of regenerative cycle. Also, the
definitions of $G(i)$ and $G(i,j)$ for any $i,j\in L$ with $i\neq j$ are given
in Definition \ref{Def:3.3} and Remark \ref{Rmk:3.1}.

\begin{lemma}
\label{Lem:6.15}For any $\eta>0,$ there exists $\delta_{0}\in(0,1),$ such that
for any $\delta\in(0,\delta_{0})$ and for any $j\in L$
\[
\liminf_{\varepsilon\rightarrow0}-\varepsilon\log\left(  \sup\nolimits_{z\in
\partial B_{\delta}(O_{1})}E_{z}N_{j}\right)  \geq-\min_{\ell\in
L\setminus\{1\}}V\left(  O_{1},O_{\ell}\right)  +W\left(  O_{j}\right)
-W\left(  O_{1}\right)  -\eta,\text{ }%
\]
where
\[
W\left(  O_{j}\right)  \doteq\min_{g\in G\left(  j\right)  }\left[
{\textstyle\sum_{\left(  m\rightarrow n\right)  \in g}}
V\left(  O_{m},O_{n}\right)  \right]  .
\]

\end{lemma}

\begin{proof}
According to Lemma \ref{Lem:3.3} we know that for any $\eta>0,$ there exist
$\delta_{0}\in(0,1)$ and $\varepsilon_{0}\in(0,1),$ such that for any
$\delta\in(0,\delta_{0})$ and $\varepsilon\in(0,\varepsilon_{0}),$ for all
$x\in\partial B_{\delta}(O_{i}),$ the one-step transition probability of the
Markov chain $\{Z_{n}\}_{n}$ on $\partial B_{\delta}(O_{j})$ satisfies the
inequalities%
\begin{equation}
\label{eqn:osb}e^{-\frac{1}{\varepsilon}\left(  V\left(  O_{i},O_{j}\right)
+\eta/4^{l-1}\right)  }\leq p(x,\partial B_{\delta}(O_{j}))\leq e^{-\frac
{1}{\varepsilon}\left(  V\left(  O_{i},O_{j}\right)  -\eta/4^{l-1}\right)  }.
\end{equation}
We can then apply Lemma \ref{Lem:6.5} with $p_{ij}=e^{-\frac{1}{\varepsilon
}V\left(  O_{i},O_{j}\right)  }$ and $a=e^{\frac{1}{\varepsilon}\eta/4^{l-1}}$
to obtain that%
\begin{align*}
\sup_{z\in\partial B_{\delta}(O_{1})}E_{z}N_{j} \leq\frac{e^{\frac
{1}{\varepsilon}\eta}}{\sum_{\ell\in L\setminus\{1\}}e^{-\frac{1}{\varepsilon
}V\left(  O_{1},O_{\ell}\right)  }}\frac{%
{\textstyle\sum_{g\in G\left(  j\right)  }}
\pi\left(  g\right)  }{\sum_{g\in G\left(  1\right)  }\pi\left(  g\right)  }
\leq\frac{e^{\frac{1}{\varepsilon}\eta}}{e^{-\frac{1}{\varepsilon}\min
_{\ell\in L\setminus\{1\}}V\left(  O_{1},O_{\ell}\right)  }}\frac{%
{\textstyle\sum_{g\in G\left(  j\right)  }}
\pi\left(  g\right)  }{\sum_{g\in G\left(  1\right)  }\pi\left(  g\right)  }.
\end{align*}
Thus,%
\begin{align*}
\liminf_{\varepsilon\rightarrow0}-\varepsilon\log\left(  \sup_{z\in\partial
B_{\delta}(O_{1})}E_{z}N_{j}\right)  \geq-\min_{\ell\in L\setminus
\{1\}}V\left(  O_{1},O_{\ell}\right)  -\eta+\liminf_{\varepsilon\rightarrow
0}-\varepsilon\log\left(  \frac{\sum_{g\in G\left(  j\right)  }\pi\left(
g\right)  }{\sum_{g\in G\left(  1\right)  }\pi\left(  g\right)  }\right)  .
\end{align*}
\newline Hence it suffices to show that
\[
\liminf_{\varepsilon\rightarrow0}-\varepsilon\log\left(  \frac{\sum_{g\in
G\left(  j\right)  }\pi\left(  g\right)  }{\sum_{g\in G\left(  1\right)  }%
\pi\left(  g\right)  }\right)  \geq W\left(  O_{j}\right)  -W\left(
O_{1}\right)  .
\]

Observe that by definition for any $j\in L$ and $g\in G\left(  j\right)  $
\begin{align*}
\pi\left(  g\right)  =%
{\textstyle\prod_{\left(  m\rightarrow n\right)  \in g}}
p_{mn}=%
{\textstyle\prod_{\left(  m\rightarrow n\right)  \in g}}
e^{-\frac{1}{\varepsilon}V\left(  O_{m},O_{n}\right)  } =\exp\left\{
-\frac{1}{\varepsilon}%
{\textstyle\sum_{\left(  m\rightarrow n\right)  \in g}}
V\left(  O_{m},O_{n}\right)  \right\}  ,
\end{align*}
which implies that%
\begin{align*}
&  \liminf_{\varepsilon\rightarrow0}-\varepsilon\log\left(  \frac{\sum_{g\in
G\left(  j\right)  }\pi\left(  g\right)  }{\sum_{g\in G\left(  1\right)  }%
\pi\left(  g\right)  }\right) \\
&  \qquad\geq\min_{g\in G\left(  j\right)  }\left[  \liminf_{\varepsilon
\rightarrow0}-\varepsilon\log\left(  \exp\left\{  -\frac{1}{\varepsilon}%
{\textstyle\sum_{\left(  m\rightarrow n\right)  \in g}}
V\left(  O_{m},O_{n}\right)  \right\}  \right)  \right] \\
&  \qquad\qquad-\min_{g\in G\left(  1\right)  }\left[  \limsup_{\varepsilon
\rightarrow0}-\varepsilon\log\left(  \exp\left\{  -\frac{1}{\varepsilon}%
{\textstyle\sum_{\left(  m\rightarrow n\right)  \in g}}
V\left(  O_{m},O_{n}\right)  \right\}  \right)  \right] \\
&  \qquad=\min_{g\in G\left(  j\right)  }\left[
{\textstyle\sum_{\left(  m\rightarrow n\right)  \in g}}
V\left(  O_{m},O_{n}\right)  \right]  -\min_{g\in G\left(  1\right)  }\left[
{\textstyle\sum_{\left(  m\rightarrow n\right)  \in g}}
V\left(  O_{m},O_{n}\right)  \right] \\
&  \qquad=W\left(  O_{j}\right)  -W\left(  O_{1}\right)  .
\end{align*}
The inequality is from Lemma \ref{Lem:6.1}; the last equality holds due the
definition of $W\left(  O_{j}\right)  $.
\end{proof}

\vspace{\baselineskip} Recall the definition of $W(O_{1}\cup O_{j})$ in
\eqref{eqn:Wtwarg_2}. In the next result we obtain bounds on, for example, a
quantity close to the expected number of visits to $B_{\delta}(O_{j})$ before
visiting a neighborhood of $O_{1}$, after starting near $O_{j}$.

\begin{lemma}
\label{Lem:6.16}For any $\eta>0,$ there exists $\delta_{0}\in(0,1),$ such that
for any $\delta\in(0,\delta_{0})$
\[
\liminf_{\varepsilon\rightarrow0}-\varepsilon\log\left(  \sum_{\ell=1}%
^{\infty}\sup_{z\in\partial B_{\delta}(O_{1})}P_{z}\left(  \ell\leq
N_{1}\right)  \right)  \geq-\min_{\ell\in L\setminus\{1\}}V\left(
O_{1},O_{\ell}\right)  -\eta
\]
and for any $j\in L\setminus\{1\}$%
\[
\liminf_{\varepsilon\rightarrow0}-\varepsilon\log\left(  \sum_{\ell=1}%
^{\infty}\sup_{z\in\partial B_{\delta}(O_{j})}P_{z}\left(  \ell\leq
N_{j}\right)  \right)  \geq W(O_{1}\cup O_{j})-W\left(  O_{1}\right)  -\eta.
\]

\end{lemma}

\begin{proof}
We again use that by Lemma \ref{Lem:3.3}, for any $\eta>0$ there exist
$\delta_{0}\in(0,1)$ and $\varepsilon_{0}\in(0,1),$ such that \eqref{eqn:osb}
holds for any $\delta\in(0,\delta_{0})$, $\varepsilon\in(0,\varepsilon_{0})$
and all $x\in\partial B_{\delta}(O_{i}).$
Then by Lemma \ref{Lem:6.5} with $p_{ij}=e^{-\frac{1}{\varepsilon}V\left(
O_{i},O_{j}\right)  }$ and $a=e^{\frac{1}{\varepsilon}\eta/4^{l-1}}$
\[
\sum_{\ell=1}^{\infty}\sup_{x\in\partial B_{\delta}(O_{j})}P_{x}\left(
N_{1}\geq\ell\right)  \leq\frac{e^{\frac{1}{\varepsilon}\eta}}{\sum_{\ell\in
L\setminus\{1\}}e^{-\frac{1}{\varepsilon}V\left(  O_{1},O_{\ell}\right)  }}%
\]
and for any $j\in L\setminus\{1\}$
\[
\sum_{\ell=1}^{\infty}\sup_{x\in\partial B_{\delta}(O_{j})}P_{x}\left(
N_{j}\geq\ell\right)  \leq e^{\frac{1}{\varepsilon}\eta}\frac{\sum_{g\in
G\left(  1,j\right)  }\pi\left(  g\right)  }{\sum_{g\in G\left(  1\right)
}\pi\left(  g\right)  }.
\]
Thus,
\begin{align*}
\liminf_{\varepsilon\rightarrow0}-\varepsilon\log\left(  \sum_{\ell=1}%
^{\infty}\sup_{z\in\partial B_{\delta}(O_{1})}P_{z}\left(  \ell\leq
N_{1}\right)  \right)  \geq-\limsup_{\varepsilon\rightarrow0}-\varepsilon
\log\left(  \sum\nolimits_{\ell\in L\setminus\{1\}}e^{-\frac{1}{\varepsilon
}V\left(  O_{1},O_{\ell}\right)  }\right)  -\eta
\end{align*}
and
\begin{align*}
\liminf_{\varepsilon\rightarrow0}-\varepsilon\log\left(  \sum_{\ell=0}%
^{\infty}\sup_{z\in\partial B_{\delta}(O_{j})}P_{z}\left(  \ell\leq
N_{j}\right)  \right)  \geq\liminf_{\varepsilon\rightarrow0}-\varepsilon
\log\left(  \frac{\sum_{g\in G\left(  1,j\right)  }\pi\left(  g\right)  }%
{\sum_{g\in G\left(  1\right)  }\pi\left(  g\right)  }\right)  -\eta.
\end{align*}
Following the same argument as for the proof of Lemma \ref{Lem:6.15}, we can
use Lemma \ref{Lem:6.1} to obtain that%
\[
-\limsup_{\varepsilon\rightarrow0}-\varepsilon\log\left(  \sum\nolimits_{\ell
\in L\setminus\{1\}}e^{-\frac{1}{\varepsilon}V\left(  O_{1},O_{\ell}\right)
}\right)  \geq-\min_{\ell\in L\setminus\{1\}}V\left(  O_{1},O_{\ell}\right)
\]
and%
\begin{align*}
&  \liminf_{\varepsilon\rightarrow0}-\varepsilon\log\left(  \frac{\sum_{g\in
G\left(  1,j\right)  }\pi\left(  g\right)  }{\sum_{g\in G\left(  1\right)
}\pi\left(  g\right)  }\right) \\
&  \qquad\geq\min_{g\in G\left(  1,j\right)  }\left[  {\textstyle\sum_{\left(
m\rightarrow n\right)  \in g} }V\left(  O_{m},O_{n}\right)  \right]
-\min_{g\in G\left(  1\right)  }\left[  {\textstyle\sum_{\left(  m\rightarrow
n\right)  \in g} }V\left(  O_{m},O_{n}\right)  \right]  .
\end{align*}
Recalling \eqref{eqn:Wtwarg} and \eqref{eqn:Wtwarg_2},
we are done.
\end{proof}

\vspace{0pt} As mentioned at the beginning of this subsection, our main goal
is to provide lower bounds for
\[
\liminf_{\varepsilon\rightarrow0}-\varepsilon\log\left(  \sup_{z\in\partial
B_{\delta}(O_{1})}E_{z}\left(  \int_{0}^{\tau_{1}^{\varepsilon}}e^{-\frac
{1}{\varepsilon}f\left(  X_{s}^{\varepsilon}\right)  }1_{A}\left(
X_{s}^{\varepsilon}\right)  ds\right)  \right)
\]
and
\[
\liminf_{\varepsilon\rightarrow0}-\varepsilon\log\left(  \sup_{z\in\partial
B_{\delta}(O_{1})}E_{z}\left(  \int_{0}^{\tau_{1}^{\varepsilon}}e^{-\frac
{1}{\varepsilon}f\left(  X_{s}^{\varepsilon}\right)  }1_{A}\left(
X_{s}^{\varepsilon}\right)  ds\right)  ^{2}\right)
\]
for a given continuous function $f:M\rightarrow%
\mathbb{R}
$ and compact set $A\subset M.$ We now state the main results of the
subsection. Recall that $h_{1}=\min_{\ell\in L\setminus\{1\}}V\left(
O_{1},O_{\ell}\right)  $, $S_{1}^{\varepsilon}\doteq\int_{0}^{\tau
_{1}^{\varepsilon}}e^{-\frac{1}{\varepsilon}f\left(  X_{s}^{\varepsilon
}\right)  }1_{A}\left(  X_{s}^{\varepsilon}\right)  ds$ and $W\left(
O_{j}\right)  \doteq\min_{g\in G\left(  j\right)  }[\sum_{\left(  m\rightarrow
n\right)  \in g}V\left(  O_{m},O_{n}\right)  ]$ and the definitions
(\ref{eqn:defofI}).

\begin{lemma}
\label{Lem:6.17}Given a compact set $A\subset M,$ a continuous function
$f:M\rightarrow%
\mathbb{R}
$ and $\eta>0,$ there exists $\delta_{0}\in(0,1),$ such that for any
$\delta\in(0,\delta_{0})$%
\[
\liminf_{\varepsilon\rightarrow0}-\varepsilon\log\left[  \sup_{z\in\partial
B_{\delta}(O_{1})}E_{z}S_{1}^{\varepsilon}\right]  \geq\min_{j\in L}\left\{
\inf_{x\in A}\left[  f\left(  x\right)  +V\left(  O_{j},x\right)  \right]
+W\left(  O_{j}\right)  \right\}  -W\left(  O_{1}\right)  -h_{1}-\eta.
\]

\end{lemma}

\begin{proof}
Recall that by Lemma \ref{Lem:6.14}, we have shown that for the given $\eta,$
there exists $\delta_{1}\in(0,1),$ such that for any $\delta\in(0,\delta_{1})$
and $j\in L$
\[
\liminf_{\varepsilon\rightarrow0}-\varepsilon\log\left(  \sup_{z\in\partial
B_{\delta}(O_{j})}E_{z}I^{\varepsilon}(\tau_{1};f,A)\right)  \geq\inf_{x\in
A}\left[  f\left(  x\right)  +V\left(  O_{j},x\right)  \right]  -\frac{\eta
}{2}.
\]
In addition, by Lemma \ref{Lem:6.15}, we know that for the same $\eta,$ there
exists $\delta_{2}\in(0,1),$ such that for any $\delta\in(0,\delta_{2})$
\[
\liminf_{\varepsilon\rightarrow0}-\varepsilon\log\left(  \sup\nolimits_{z\in
\partial B_{\delta}(O_{1})}E_{z}N_{j}\right)  \geq-\min_{\ell\in
L\setminus\{1\}}V\left(  O_{1},O_{\ell}\right)  +W\left(  O_{j}\right)
-W\left(  O_{1}\right)  -{\eta}/{2}.
\]
Hence for any $\delta\in(0,\delta_{0})$ with $\delta_{0}=\delta_{1}%
\wedge\delta_{2},$ we apply Corollary \ref{Cor:6.1} to get
\begin{align*}
&  \liminf_{\varepsilon\rightarrow0}-\varepsilon\log\left[  E_{x}%
I^{\varepsilon}(\tau_{1}^{\varepsilon};f,A)\right]  \\
&  \geq\min_{j\in L}\left\{  \liminf_{\varepsilon\rightarrow0}-\varepsilon
\log\left(  \sup_{z\in\partial B_{\delta}(O_{j})}E_{z}I^{\varepsilon}(\tau
_{1};f,A)\right)  
+\liminf_{\varepsilon
\rightarrow0}-\varepsilon\log\left(  \sup_{z\in\partial B_{\delta}(O_{1}%
)}E_{z}\left(  N_{j}\right)  \right)  \right\}  \\
&  \geq\min_{j\in L}\left\{  \inf_{x\in A}\left[  f\left(  x\right)  +V\left(
O_{j},x\right)  \right]  +W\left(  O_{j}\right)  \right\}  -W\left(
O_{1}\right)  -h_{1}-\eta,
\end{align*}
where $\tau_{1}^{\varepsilon}$ is the time for a regenerative cycle and
$\tau_{1}$ is the first visit time of neighborhoods of equilibrium points
after being a certain distance away from them.
\end{proof}

\begin{remark}
\label{Rmk:6.6}According to Remark \ref{Rmk:6.4} and using the same argument
as in Lemma \ref{Lem:6.17}, we can find that given a compact set $A\subset M,$
a continuous function $f:M\rightarrow%
\mathbb{R}
$ and $\eta>0,$ there exists $\delta_{0}\in(0,1),$ such that for any
$\delta\in(0,\delta_{0})$%
\begin{align*}
&  \liminf_{\varepsilon\rightarrow0}-\varepsilon\log\left[  \sup_{z\in\partial
B_{\delta}(O_{1})}E_{z}I^{\varepsilon}(\sigma_{0}^{\varepsilon},\tau
_{1}^{\varepsilon};f,A)\right]  \\
&  \qquad\geq\min_{j\in L\setminus\{1\}}\left\{  \inf_{x\in A}\left[  f\left(
x\right)  +V\left(  O_{j},x\right)  \right]  +W\left(  O_{j}\right)  \right\}
-W\left(  O_{1}\right)  -h_{1}-\eta.
\end{align*}

\end{remark}

\begin{lemma}
\label{Lem:6.18}Given a compact set $A\subset M,$ a continuous function
$f:M\rightarrow%
\mathbb{R}
$ and $\eta>0,$ there exists $\delta_{0}\in(0,1),$ such that for any
$\delta\in(0,\delta_{0})$%
\[
\liminf_{\varepsilon\rightarrow0}-\varepsilon\log\left[  \sup\nolimits_{z\in
\partial B_{\delta}(O_{1})}E_{z}(S_{1}^{\varepsilon})^{2}\right]  \geq
\min_{j\in L}\left(  R_{j}^{(1)}\wedge R_{j}^{(2)}\right)  -h_{1}-\eta,
\]
where $S_{1}^{\varepsilon}\doteq\int_{0}^{\tau_{1}^{\varepsilon}}e^{-\frac
{1}{\varepsilon}f\left(  X_{s}^{\varepsilon}\right)  }1_{A}\left(
X_{s}^{\varepsilon}\right)  ds$ and $h_{1}=\min_{\ell\in L\setminus
\{1\}}V\left(  O_{1},O_{\ell}\right)  $, and
\[
R_{j}^{(1)}\doteq\inf_{x\in A}\left[  2f\left(  x\right)  +V\left(
O_{j},x\right)  \right]  +W\left(  O_{j}\right)  -W\left(  O_{1}\right)
\]%
\[
R_{1}^{(2)}\doteq2\inf_{x\in A}\left[  f\left(  x\right)  +V\left(
O_{1},x\right)  \right]  -h_{1}%
\]
and for $j\in L\setminus\{1\}$
\[
R_{j}^{(2)}\doteq2\inf_{x\in A}\left[  f\left(  x\right)  +V\left(
O_{j},x\right)  \right]  +W\left(  O_{j}\right)  -2W\left(  O_{1}\right)
+W(O_{1}\cup O_{j}).
\]

\end{lemma}

\begin{proof}
Following a similar argument as for the proof of Lemma \ref{Lem:6.17}, given
any $\eta>0,$ owing to Lemmas \ref{Lem:6.14}, \ref{Lem:6.15} and
\ref{Lem:6.16}, there exists $\delta_{0}\in(0,1)$ such that for any $\delta
\in(0,\delta_{0})$ and for any $j\in L$%
\[
\liminf_{\varepsilon\rightarrow0}-\varepsilon\log\left(  \sup_{z\in\partial
B_{\delta}(O_{j})}E_{z}I^{\varepsilon}(\tau_{1};f,A)\right)  \geq\inf_{x\in
A}\left[  f\left(  x\right)  +V\left(  O_{j},x\right)  \right]  -\frac{\eta
}{4},
\]%
\[
\liminf_{\varepsilon\rightarrow0}-\varepsilon\log\left(  \sup_{z\in\partial
B_{\delta}(O_{j})}E_{z}I^{\varepsilon}(\tau_{1};f,A)^{2}\right)  \geq
\inf_{x\in A}\left[  2f\left(  x\right)  +V\left(  O_{j},x\right)  \right]
-\frac{\eta}{4},
\]
\[
\liminf_{\varepsilon\rightarrow0}-\varepsilon\log\left(  \sup_{z\in\partial
B_{\delta}(O_{1})}E_{z}N_{j}\right)  \geq-h_{1}+W\left(  O_{j}\right)
-W\left(  O_{1}\right)  -\frac{\eta}{4},
\]
\[
\liminf_{\varepsilon\rightarrow0}-\varepsilon\log\left(  \sum_{\ell=1}%
^{\infty}\sup_{z\in\partial B_{\delta}(O_{1})}P_{z}\left(  \ell\leq
N_{1}\right)  \right)  \geq-h_{1}-\frac{\eta}{4},
\]
and for any $j\in L\setminus\{1\}$,%
\[
\liminf_{\varepsilon\rightarrow0}-\varepsilon\log\left(
{\textstyle\sum_{\ell=1}^{\infty}}
\sup_{z\in\partial B_{\delta}(O_{j})}P_{z}\left(  \ell\leq N_{j}\right)
\right)  \geq W(O_{1}\cup O_{j})-W\left(  O_{1}\right)  -\frac{\eta}{4}.
\]
Hence for any $\delta\in(0,\delta_{0})$
we apply Corollary \ref{Cor:6.1} to get%
\[
\liminf_{\varepsilon\rightarrow0}-\varepsilon\log\left(  \sup_{z\in\partial
B_{\delta}(O_{1})}E_{z}\left(  S_{1}^{\varepsilon}\right)  ^{2}\right)
\geq\min_{j\in L}\left(  \hat{R}_{j}^{(1)}\wedge\hat{R}_{j}^{(2)}\right)  ,
\]
where
\begin{align*}
\hat{R}_{j}^{(1)} &  \doteq\liminf_{\varepsilon\rightarrow0}-\varepsilon
\log\left(  \sup_{z\in\partial B_{\delta}(O_{j})}E_{z}I^{\varepsilon}(\tau
_{1};f,A)^{2}\right) 
+\liminf_{\varepsilon\rightarrow0}-\varepsilon\log\left(  \sup
_{z\in\partial B_{\delta}(O_{1})}E_{z}N_{j}\right) 
\\ &
\geq\inf_{x\in A}\left[  2f\left(  x\right)  +V\left(  O_{j},x\right)
\right]  +W\left(  O_{j}\right)  -W\left(  O_{1}\right)  -h_{1}-\eta  =R_{j}^{(1)}-h_{1}-\eta
\end{align*}
and%
\begin{align*}
\hat{R}_{1}^{(2)} &  \doteq2\liminf_{\varepsilon\rightarrow0}-\varepsilon
\log\left(  \sup_{z\in\partial B_{\delta}(O_{1})}E_{z}I^{\varepsilon}(\tau
_{1};f,A)\right)  \\
&  \quad+\liminf_{\varepsilon\rightarrow0}-\varepsilon\log\left(  \sup
_{z\in\partial B_{\delta}(O_{1})}E_{z}N_{1}\right)  +\liminf_{\varepsilon
\rightarrow0}-\varepsilon\log\left(
{\textstyle\sum_{\ell=1}^{\infty}}
\sup_{z\in\partial B_{\delta}(O_{1})}P_{z}\left(  \ell\leq N_{1}\right)
\right)  \\
&  \geq2\left(  \inf_{x\in A}\left[  f\left(  x\right)  +V\left(
O_{1},x\right)  \right]  -\frac{\eta}{4}\right)  +\left(  -h_{1}-\frac{\eta
}{4}\right)  +\left(  -h_{1}-\frac{\eta}{4}\right)  \\
&  =2\inf_{x\in A}\left[  f\left(  x\right)  +V\left(  O_{1},x\right)
\right]  -2h_{1}-\eta=R_{1}^{(2)}-h_{1}-\eta
\end{align*}
and for $j\in L\setminus\{1\}$
\begin{align*}
\hat{R}_{j}^{(2)} &  \doteq2\liminf_{\varepsilon\rightarrow0}-\varepsilon
\log\left(  \sup_{z\in\partial B_{\delta}(O_{j})}E_{z}I^{\varepsilon}(\tau
_{1};f,A)\right)  \\
&  \quad+\liminf_{\varepsilon\rightarrow0}-\varepsilon\log\left(  \sup
_{z\in\partial B_{\delta}(O_{1})}E_{z}N_{j}\right)  +\liminf_{\varepsilon
\rightarrow0}-\varepsilon\log\left(
{\textstyle\sum_{\ell=1}^{\infty}}
\sup_{z\in\partial B_{\delta}(O_{j})}P_{z}\left(  \ell\leq N_{j}\right)
\right)  \\
&  \geq2\left(  \inf_{x\in A}\left[  f\left(  x\right)  +V\left(
O_{j},x\right)  \right]  -\frac{\eta}{4}\right)  +\left(  -h_{1}+W\left(
O_{j}\right)  -W\left(  O_{1}\right)  -\frac{\eta}{4}\right)  \\
&  \quad+\left(  W(O_{1}\cup O_{j})-W\left(  O_{1}\right)  -\frac{\eta}%
{4}\right)  \\
&  =2\inf_{x\in A}\left[  f\left(  x\right)  +V\left(  O_{j},x\right)
\right]  +W\left(  O_{j}\right)  -2W\left(  O_{1}\right)  +W(O_{1}\cup
O_{j})-h_{1}-\eta\\
&  =R_{j}^{(2)}-h_{1}-\eta.
\end{align*}
\end{proof}

\subsection{Asymptotics of moments of $\hat{S}_{1}^{\varepsilon}$}
Recall that 
\[
\hat{S}_{n}^{\varepsilon}\doteq \int_{\hat{\tau}_{n-1}^{\varepsilon}}^{\hat{\tau}_{n}^{\varepsilon}%
}e^{-\frac{1}{\varepsilon}f\left(  X_{t}^{\varepsilon}\right)  }1_{A}\left(
X_{t}^{\varepsilon}\right) dt,
\]
where $\hat{\tau}_{i}^{\varepsilon}$ is a multicycle defined according to
\eqref{eqn:defofMC}
and with $\{\mathbf{M}^{\varepsilon}_i\}_{i\in\mathbb{N}}$ being a sequence of independent and geometrically distributed random variables with parameter $e^{-m/\varepsilon}$ for some $m>0$ such that $m+h_1>w$. Moreover, $\{\mathbf{M}^{\varepsilon}_i\}$ is also independent of $\{\tau^{\varepsilon}_n\}$.
Using the independence of $\{\mathbf{M}^{\varepsilon}_i\}$ and $\{\tau^{\varepsilon}_n\}$, and the fact that $\{\tau^{\varepsilon}_n\}$ and $\{S_{n}^{\varepsilon}\}$ are both iid under $P_{\lambda^{\varepsilon}}$,  we find that $\{\hat{S}_{n}^{\varepsilon}\}$ is also iid under $P_{\lambda^{\varepsilon}}$ and
\begin{align}
    \label{mega_mean_S}
    E_{\lambda^{\varepsilon}}\hat{S}_{1}^{\varepsilon}
    =E_{\lambda^{\varepsilon}}\mathbf{M}^{\varepsilon}_1
    \cdot E_{\lambda^{\varepsilon}}S^\varepsilon_1
\end{align}
and 
\begin{align}
    \label{mega_variance_S}
    \mathrm{Var}_{\lambda^{\varepsilon}}\hat{S}_{1}^{\varepsilon}
    &=E_{\lambda^{\varepsilon}}\mathbf{M}^{\varepsilon}_1
    \cdot \mathrm{Var}_{\lambda^{\varepsilon}}(S^\varepsilon_1)
    +\mathrm{Var}_{\lambda^{\varepsilon}}(\mathbf{M}^{\varepsilon}_1)
    \cdot (E_{\lambda^{\varepsilon}}S^\varepsilon_1)^2\nonumber\\
    &\leq E_{\lambda^{\varepsilon}}\mathbf{M}^{\varepsilon}_1
    \cdot E_{\lambda^{\varepsilon}}(S^\varepsilon_1)^2
    +\mathrm{Var}_{\lambda^{\varepsilon}}(\mathbf{M}^{\varepsilon}_1)
    \cdot (E_{\lambda^{\varepsilon}}S^\varepsilon_1)^2
\end{align}
On the other hand, since $\mathbf{M}^{\varepsilon}_1$ is geometrically distributed with parameter $e^{-m/\varepsilon}$, this gives that
\begin{align}
\label{geometric}
    E_{\lambda^{\varepsilon}}\mathbf{M}^{\varepsilon}_1=e^{\frac{m}{\varepsilon}} \text{ and }\mathrm{Var}_{\lambda^{\varepsilon}}(\mathbf{M}^{\varepsilon}_1) = e^{\frac{2m}{\varepsilon}}(1-e^{\frac{-m}{\varepsilon}}).
\end{align}

Therefore, by combining \eqref{mega_mean_S}, \eqref{mega_variance_S} and \eqref{geometric} with Lemma \ref{Lem:6.17} and Lemma \ref{Lem:6.18}, we have the following two lemmas.
\begin{lemma}
\label{Lem:6.19}
Given a compact set $A\subset M,$ a continuous function
$f:M\rightarrow%
\mathbb{R}
$ and $\eta>0,$ there exists $\delta_{0}\in(0,1),$ such that for any
$\delta\in(0,\delta_{0})$%
\begin{align*}
 &\liminf_{\varepsilon\rightarrow0}-\varepsilon\log  E_{\lambda^{\varepsilon}}\hat{S}_{1}^{\varepsilon}\\
&\qquad \geq\min_{j\in L}\left\{  \inf_{x\in A}\left[  f\left(  x\right)
+V\left(  O_{j},x\right)  \right]  +W\left(  O_{j}\right)  \right\} -W\left(O_{1}\right) -(m+h_1)-\eta.
\end{align*}
\end{lemma}

\begin{lemma}
\label{Lem:6.20}Given a compact set $A\subset M,$ a continuous function
$f:M\rightarrow%
\mathbb{R}
$ and $\eta>0,$ there exists $\delta_{0}\in(0,1),$ such that for any
$\delta\in(0,\delta_{0})$%
\[
\liminf_{\varepsilon\rightarrow0}-\varepsilon\log\mathrm{Var}_{\lambda^{\varepsilon}}(  \hat{S}_{1}^{\varepsilon})  
\geq\min_{j\in L}\left(  R_{j}^{(1)}\wedge R_{j}^{(2)}\wedge R_{j}^{(3,m)}\right) -(m+h_1)-\eta,
\]
where $R_{j}^{(1)}$ and $R_{j}^{(2)}$ are defined as in Lemma \ref{Lem:6.18}, and 
\[
R_{j}^{(3,m)}\doteq2\inf_{x\in A}\left[  f\left(  x\right)  +V\left(O_{j},x\right)  \right]  +2W\left(  O_{j}\right)  -2W\left(  O_{1}\right)-(m+h_1).
\]
\end{lemma}
Later on we will optimize on $m$ to obtain the largest bound from below.
This will require that we consider first $m>w-h_1$,
so that as shown in the next section $N^{\varepsilon}(T^{\varepsilon})$ can be suitably approximated in 
terms of a Poisson distribution,
and then sending $m\downarrow w-h_1$.

\section{Asymptotics of Moments of $N^{\varepsilon}(T^{\varepsilon})$
and $\hat{N}^{\varepsilon}(T^{\varepsilon})$}

\label{sec:moments_of_the_number_of_renewals} 
Recall that the number of
single cycles in the time interval $[0,T^{\varepsilon}]$ plus one is defined as
\[
N^{\varepsilon}\left(  T^{\varepsilon}\right)  \doteq\inf\left\{  n\in%
\mathbb{N}
:\tau_{n}^{\varepsilon}>T^{\varepsilon}\right\}  ,
\]
where the $\tau_{n}^{\varepsilon}$ are the return times to $B_{\delta}(O_{1})$
after ever visiting one of the $\delta$-neighborhood of other equilibrium
points than $O_{1}.$ In addition, $\lambda^{\varepsilon}$ is the unique
invariant measure of $\{Z_{n}^{\varepsilon}\}_{n}=\{X_{\tau_{n}^{\varepsilon}%
}^{\varepsilon}\}_{n}.$ 
The number of
multicycles in the time interval $[0,T^{\varepsilon}]$ plus one is defined as
\[
\hat{N}^{\varepsilon}\left(  T^\varepsilon\right)  \doteq\inf\left\{  n\in \mathbb{N}
:\hat{\tau}_{n}^{\varepsilon}>T^\varepsilon\right\},
\]
where $\hat{\tau}^\varepsilon_i$ are defined as in \eqref{eqn:defofMC}.

In this section, we will find the logarithmic asymptotics of the expected value and the variance of $N^{\varepsilon}\left(  T^{\varepsilon}\right)$
with $T^{\varepsilon}=e^{\frac{1}{\varepsilon}c}$ for some $c>h_1$ in Lemma \ref{Lem:7.1} and  Lemma \ref{Lem:7.2} under the assumption that $h_1>w$ (i.e., single cycle case), and the analogous quantities for   
$\hat{N}^{\varepsilon}\left(  T^\varepsilon\right)$
with $T^{\varepsilon}=e^{\frac{1}{\varepsilon}c}$ for some $c>w$ in Lemma \ref{Lem:7.11} and Lemma \ref{Lem:7.12} under the assumption that $w\geq h_1$ (i.e., multicycle case). 

\begin{remark}
While the proofs of these asymptotic results are quite detailed,
it is essential that we obtain estimates good enough for a relatively precise 
comparison of the expected value and the variance of $N^{\varepsilon}\left(  T^{\varepsilon}\right)$, 
and likewise for $\hat{N}^{\varepsilon}\left(T^\varepsilon\right)$.
For this,
the key result needed is the characterization of $N^{\varepsilon}\left(T^\varepsilon\right)$ (and $\hat{N}^{\varepsilon}\left(T^\varepsilon\right)$) as having an approximately Poisson distribution.
These follow 
by exploiting the asymptotically exponential character of $\tau_{n}^{\varepsilon}$ (and $\hat{\tau}_{n}^{\varepsilon}$), together with some uniform integrability properties. 
\end{remark}

Lemmas \ref{Lem:7.1} and \ref{Lem:7.2} below are proved in Section \ref{subsec:asymptotics_of_moments_of_N}.

\begin{lemma}
\label{Lem:7.1}If $h_1>w$ and $T^{\varepsilon
}=e^{\frac{1}{\varepsilon}c}$ for some $c>h_1$, then there exists $\delta_{0}\in(0,1)$ such that for any
$\delta\in(0,\delta_{0})$
\[
\liminf_{\varepsilon\rightarrow0}-\varepsilon\log\left\vert \frac
{E_{\lambda^{\varepsilon}}\left(  N^{\varepsilon}\left(  T^{\varepsilon
}\right)  \right)  }{T^{\varepsilon}}-\frac{1}{E_{\lambda^{\varepsilon}}%
\tau_{1}^{\varepsilon}}\right\vert \geq c.
\]
\end{lemma}

\begin{corollary}
\label{Cor:7.2}
If $h_1>w$ and $T^{\varepsilon
}=e^{\frac{1}{\varepsilon}c}$ for some $c>h_1$, then there exists $\delta_{0}\in(0,1)$ such that for any
$\delta\in(0,\delta_{0})$
\[
\liminf_{\varepsilon\rightarrow0}-\varepsilon\log\frac{E_{\lambda^{\varepsilon}}\left(  N^{\varepsilon}\left(  T^{\varepsilon}\right)
\right)  }{T^{\varepsilon}}\geq \varkappa_{\delta},
\]
where $\varkappa_{\delta}\doteq \min_{y\in\cup_{k\in L\setminus\{1\}}\partial
B_{\delta}(O_{k})}V(O_{1},y)$.
\end{corollary}

\begin{lemma}
\label{Lem:7.2}If $h_1>w$ and $T^{\varepsilon
}=e^{\frac{1}{\varepsilon}c}$ for some $c>h_1$, then for any $\eta>0,$ there exists $\delta_{0}\in(0,1)$ such that for any
$\delta\in(0,\delta_{0})$
\[
\liminf_{\varepsilon\rightarrow0}-\varepsilon\log\frac{\mathrm{Var}%
_{\lambda^{\varepsilon}}\left(  N^{\varepsilon}\left(  T^{\varepsilon}\right)
\right)  }{T^{\varepsilon}}\geq h_1-\eta.
\]

\end{lemma}

Before proceeding, we mention a result from \cite{fel2} and define some
notation which will be used in this section. Results in Section 5 and Section
10 of \cite[Chapter XI]{fel2} say that for any $t>0,$ the first and second moment of
$N^{\varepsilon}\left(  t\right)  $ can be represented as
\begin{equation}
E_{\lambda^{\varepsilon}}\left(  N^{\varepsilon}\left(  t\right)  \right)
=\sum\nolimits_{n=0}^{\infty}P_{\lambda^{\varepsilon}}\left(  \tau_{n}^{\varepsilon
}\leq t\right)  \text{ and }E_{\lambda^{\varepsilon}}\left(  N^{\varepsilon
}\left(  t\right)  \right)  ^{2}=\sum\nolimits_{n=0}^{\infty}\left(  2n+1\right)
P_{\lambda^{\varepsilon}}\left(  \tau_{n}^{\varepsilon}\leq t\right)  .
\label{eqn:representation}%
\end{equation}
Let $\Gamma^{\varepsilon}\doteq T^{\varepsilon
}/E_{\lambda^{\varepsilon}}\tau_{1}^{\varepsilon}$ and $\gamma^{\varepsilon
}\doteq\left(  \Gamma^{\varepsilon}\right)  ^{-\ell}$ with some $\ell\in(0,1)$
which will be chosen later. Intuitively, $\Gamma^{\varepsilon}$ is the typical
number of regenerative cycles in $[0,T^{\varepsilon}]$ since $E_{\lambda
^{\varepsilon}}\tau_{1}^{\varepsilon}$ is the expected length of one
regenerative cycle. To simplify notation, we pretend that $\left(
1+2\gamma^{\varepsilon}\right)  \Gamma^{\varepsilon}$ and $\left(
1-2\gamma^{\varepsilon}\right)  \Gamma^{\varepsilon}$ are positive integers so
that we can divide $E_{\lambda^{\varepsilon}}\left(  N^{\varepsilon}\left(
T^{\varepsilon}\right)  \right)  $ into three partial sums which are
\[
\mathfrak{P}_{1}\doteq\sum\nolimits_{n=\left(  1+2\gamma^{\varepsilon}\right)
\Gamma^{\varepsilon}+1}^{\infty}P_{\lambda^{\varepsilon}}\left(  \tau
_{n}^{\varepsilon}\leq T^{\varepsilon}\right)  ,\text{ }\mathfrak{P}_{2}%
\doteq\sum\nolimits_{n=\left(  1-2\gamma^{\varepsilon}\right)  \Gamma^{\varepsilon}%
}^{\left(  1+2\gamma^{\varepsilon}\right)  \Gamma^{\varepsilon}}\text{
}P_{\lambda^{\varepsilon}}\left(  \tau_{n}^{\varepsilon}\leq T^{\varepsilon
}\right)  \text{ }%
\]
and%
\begin{equation}
\label{eqn:defofPs}    
\mathfrak{P}_{3}\doteq\sum\nolimits_{n=0}^{\left(  1-2\gamma^{\varepsilon}\right)
\Gamma^{\varepsilon}-1}P_{\lambda^{\varepsilon}}\left(  \tau_{n}^{\varepsilon
}\leq T^{\varepsilon}\right)  .
\end{equation}
Similarly, we divide $E_{\lambda^{\varepsilon}}\left(  N^{\varepsilon}\left(
T^{\varepsilon}\right)  \right)  ^{2}$ into%
\[
\mathfrak{R}_{1}\doteq\sum_{n=\left(  1+2\gamma^{\varepsilon}\right)
\Gamma^{\varepsilon}+1}^{\infty}\left(  2n+1\right)  P_{\lambda^{\varepsilon}%
}\left(  \tau_{n}^{\varepsilon}\leq T^{\varepsilon}\right)  ,
\text{ }
\mathfrak{R}_{2}\doteq\sum_{n=\left(  1-2\gamma^{\varepsilon}\right)
\Gamma^{\varepsilon}}^{\left(  1+2\gamma^{\varepsilon}\right)  \Gamma
^{\varepsilon}}\text{ }\left(  2n+1\right)  P_{\lambda^{\varepsilon}}\left(
\tau_{n}^{\varepsilon}\leq T^{\varepsilon}\right)  \text{ }%
\]
and%
\begin{equation}
\label{eqn:defofRs}
\mathfrak{R}_{3}\doteq\sum\nolimits_{n=0}^{\left(  1-2\gamma^{\varepsilon}\right)
\Gamma^{\varepsilon}-1}\left(  2n+1\right)  P_{\lambda^{\varepsilon}}\left(
\tau_{n}^{\varepsilon}\leq T^{\varepsilon}\right)  .
\end{equation}
The next step is to find upper bounds for these partial sums, and these bounds
will help us to find suitable lower bounds for the logarithmic asymptotics of
$E_{\lambda^{\varepsilon}}\left(  N^{\varepsilon}\left(  T^{\varepsilon
}\right)  \right)  $ and Var$_{\lambda^{\varepsilon}}\left(  N^{\varepsilon
}\left(  T^{\varepsilon}\right)  \right)  $. Before looking into the upper
bound for partial sums, we establish some properties.

\begin{theorem}
\label{Thm:7.1}If $h_1>w$, then for any $\delta>0$ sufficiently small,
\[
\lim_{\varepsilon\rightarrow0}\varepsilon\log E_{\lambda^{\varepsilon}}%
\tau_{1}^{\varepsilon}=\varkappa_{\delta}\text{ and }\tau_{1}^{\varepsilon}/E_{\lambda^{\varepsilon}}\tau_{1}^{\varepsilon}\overset{d}{\rightarrow
}\rm{Exp}(1).
\]
Moreover, there exists $\varepsilon_{0}\in(0,1)$ and a constant $\tilde{c}>0$
such that
\[
P_{\lambda^{\varepsilon}}\left(  \tau_{1}^{\varepsilon}/E_{\lambda
^{\varepsilon}}\tau_{1}^{\varepsilon}>t\right)  \leq e^{-\tilde{c}t}%
\]
for any $t>0$ and any $\varepsilon\in(0,\varepsilon_{0}).$
\end{theorem}
\begin{remark}
    For any $\delta>0,$ $\varkappa_{\delta}\leq h_1.$
\end{remark}

The proof of Theorem \ref{Thm:7.1} will be given in Section
\ref{sec:exponential__returning_law_and_tail_behavior}. In that section, we
will first prove an analogous result for the exit time (or first visiting time
to other equilibrium points to be more precise), and then show how one
can extend those results to the return time. The proof of the following lemma is straightforward and hence omitted.

\begin{lemma}
\label{Lem:7.4}If $h_1>w$ and $T^{\varepsilon
}=e^{\frac{1}{\varepsilon}c}$ for some $c>h_1$, then for any $\eta>0,$ there exists $\delta_{0}\in(0,1)$ such that for any
$\delta\in(0,\delta_{0})$,
\[
h_1-\eta\geq\lim_{\varepsilon\rightarrow0}-\varepsilon\log\Gamma^{\varepsilon}\geq h_1-c-\eta.
\]

\end{lemma}

\begin{lemma}
\label{Lem:7.5}Define $\mathcal{Z}_{1}^{\varepsilon}\doteq \tau_{1}^{\varepsilon
}/E_{\lambda^{\varepsilon}}\tau_{1}^{\varepsilon}.$ Then for any $\delta>0$ sufficiently small,

\begin{itemize}
\item there exists some $\varepsilon_{0}\in(0,1)$ such that $\sup
_{\varepsilon\in(0,\varepsilon_{0})}E_{\lambda^{\varepsilon}}\left(
\mathcal{Z}_{1}^{\varepsilon}\right)  ^{3}<\infty,$

\item there exists some $\varepsilon_{0}\in(0,1)$ such that $\inf
_{\varepsilon\in(0,\varepsilon_{0})}\mathrm{Var}_{\lambda^{\varepsilon}}%
(\mathcal{Z}_{1}^{\varepsilon})>0$ and $E_{\lambda^{\varepsilon}}\left(
\mathcal{Z}_{1}^{\varepsilon}\right)  ^{2}=E_{\lambda^{\varepsilon}}\left(
\tau_{1}^{\varepsilon}\right)  ^{2}/\left(  E_{\lambda^{\varepsilon}}\tau
_{1}^{\varepsilon}\right)  ^{2}\rightarrow2$ as $\varepsilon\rightarrow0.$
\end{itemize}
\end{lemma}

\begin{proof}
For the first part, we use Theorem \ref{Thm:7.1} to find that there exists
$\varepsilon_{0}\in(0,1)$ and a constant $\tilde{c}>0$ such that
\[
P_{\lambda^{\varepsilon}}\left(  \mathcal{Z}_{1}^{\varepsilon}>t\right)
=P_{\lambda^{\varepsilon}}\left(  \tau_{1}^{\varepsilon}/E_{\lambda
^{\varepsilon}}\tau_{1}^{\varepsilon}>t\right)  \leq e^{-\tilde{c}t}%
\]
for any $t>0$ and any $\varepsilon\in(0,\varepsilon_{0}).$ Therefore, for
$\varepsilon\in(0,\varepsilon_{0})$%
\[
E_{\lambda^{\varepsilon}}\left(  \mathcal{Z}_{1}^{\varepsilon}\right)
^{3}=3\int_{0}^{\infty}t^{2}P_{\lambda^{\varepsilon}}(\mathcal{Z}%
_{1}^{\varepsilon}>t)dt\leq3\int_{0}^{\infty}t^{2}e^{-\tilde{c}t}dt<\infty.
\]

For the second assertion, since $\sup_{0<\varepsilon<\varepsilon_{0}%
}E_{\lambda^{\varepsilon}}\left(  \mathcal{Z}_{1}^{\varepsilon}\right)
^{3}<\infty,$ it implies that $\{\left(  \mathcal{Z}_{1}^{\varepsilon}\right)
^{2}\}_{0<\varepsilon<\varepsilon_{0}}$ and $\{\mathcal{Z}_{1}^{\varepsilon
}\}_{0<\varepsilon<\varepsilon_{0}}$ are both uniformly integrable. Moreover,
because $\mathcal{Z}_{1}^{\varepsilon}\overset{d}{\rightarrow}$ Exp$(1)$ as
$\varepsilon\rightarrow0$ from Theorem \ref{Thm:7.1} and since for
$X\overset{d}{=}$ Exp$(1),$ $EX=1$ and $EX^{2}=2,$ we obtain
\[
E_{\lambda^{\varepsilon}}\left(  \tau_{1}^{\varepsilon}/E_{\lambda
^{\varepsilon}}\tau_{1}^{\varepsilon}\right)  ^{2}=E_{\lambda^{\varepsilon}%
}\left(  \mathcal{Z}_{1}^{\varepsilon}\right)  ^{2}\rightarrow2\text{ and
}E_{\lambda^{\varepsilon}}\mathcal{Z}_{1}^{\varepsilon}\rightarrow1.
\]
as $\varepsilon\rightarrow0.$ This implies Var$_{\lambda^{\varepsilon}%
}(\mathcal{Z}_{1}^{\varepsilon})\rightarrow1$ as $\varepsilon\rightarrow0.$
Obviously, there exists some $\varepsilon_{0}\in(0,1)$ such that
$
\inf\nolimits_{\varepsilon\in(0,\varepsilon_{0})}\mathrm{Var}_{\lambda^{\varepsilon}%
}(\mathcal{Z}_{1}^{\varepsilon})\geq 1/2>0.
$
This completes the proof.
\end{proof}

\begin{remark}
Throughout the rest of this section, we will use $C$ to denote a constant in
$(0,\infty)$ which is independent of $\varepsilon$ but whose value may change
from use to use.
\end{remark}

\subsection{Chernoff bound}

\label{subsec:Chernoff_bound}

In this subsection we will provide upper bounds for%
\[
\mathfrak{P}_{1}\doteq\sum_{n=\left(  1+2\gamma^{\varepsilon}\right)
\Gamma^{\varepsilon}+1}^{\infty}P_{\lambda^{\varepsilon}}\left(  \tau
_{n}^{\varepsilon}\leq T^{\varepsilon}\right)
\text{ and } %
\mathfrak{R}_{1}\doteq\sum_{n=\left(  1+2\gamma^{\varepsilon}\right)
\Gamma^{\varepsilon}+1}^{\infty}\left(  2n+1\right)  P_{\lambda^{\varepsilon}%
}\left(  \tau_{n}^{\varepsilon}\leq T^{\varepsilon}\right)
\]
via a Chernoff bound. The following result is well known and its proof is standard.

\begin{lemma}
[Chernoff bound]\label{Lem:7.6}Let $X_{1},\ldots,X_{n}$ be an iid\ sequence of
random variables. For any $a\in%
\mathbb{R}
$ and for any $t\in (0,\infty)$
\[
P\left(  X_{1}+\cdots+X_{n}\leq a\right)  \leq\left(  Ee^{-tX_{1}}\right)
^{n}e^{ta}.
\]

\end{lemma}

\vspace{\baselineskip} Recall that $\Gamma^{\varepsilon}\doteq T^{\varepsilon
}/E_{\lambda^{\varepsilon}}\tau_{1}^{\varepsilon}$ and $\gamma^{\varepsilon
}\doteq\left(  \Gamma^{\varepsilon}\right)  ^{-\ell}$ with some $\ell\in(0,1)$
which will be chosen later.

\begin{lemma}
\label{Lem:7.7}Given any $\delta>0$ and any $\ell>0,$ there exists
$\varepsilon_{0}\in(0,1)$ such that for any $\varepsilon\in(0,\varepsilon
_{0})$%
\[
P_{\lambda^{\varepsilon}}\left(  \tau_{n}^{\varepsilon}\leq T^{\varepsilon
}\right)  \leq e^{-n\left(  \Gamma^{\varepsilon}\right)  ^{-2\ell}}%
\]
for any $n\geq\left(  1+2\gamma^{\varepsilon}\right)  \Gamma^{\varepsilon}.$
In addition,%
\[
\mathfrak{P}_{1}
\leq C\left(  \Gamma^{\varepsilon
}\right)  ^{2\ell}e^{-\left(  \Gamma^{\varepsilon}\right)
^{1-2\ell}}%
\text{ and }%
\mathfrak{R}_{1}  
\leq C\left(  \Gamma^{\varepsilon}\right)  ^{1+2\ell}e^{-\left(  \Gamma^{\varepsilon}\right)  ^{1-2\ell}}+C\left(  \Gamma
^{\varepsilon}\right)  ^{4\ell}e^{-\left(  \Gamma^{\varepsilon
}\right)  ^{1-2\ell}}.
\]

\end{lemma}

\begin{proof}
Given $\delta>0$, $\ell>0$ and $\varepsilon\in(0,1),$ we find that for $n\geq\left(
1+2\gamma^{\varepsilon}\right)  \Gamma^{\varepsilon}$%
\begin{align*}
P_{\lambda^{\varepsilon}}\left(  \tau_{n}^{\varepsilon}\leq T^{\varepsilon
}\right)   
&  =P_{\lambda^{\varepsilon}}\left(  \frac{\tau_{1}^{\varepsilon}+\left(
\tau_{2}^{\varepsilon}-\tau_{1}^{\varepsilon}\right)  +\cdots+\left(  \tau
_{n}^{\varepsilon}-\tau_{n-1}^{\varepsilon}\right)  }{E_{\lambda^{\varepsilon
}}\tau_{1}^{\varepsilon}}\leq\Gamma^{\varepsilon}\right) \\
&  \leq P_{\lambda^{\varepsilon}}\left(  \frac{\tau_{1}^{\varepsilon}+\left(
\tau_{2}^{\varepsilon}-\tau_{1}^{\varepsilon}\right)  +\cdots+\left(  \tau
_{n}^{\varepsilon}-\tau_{n-1}^{\varepsilon}\right)  }{E_{\lambda^{\varepsilon
}}\tau_{1}^{\varepsilon}}\leq\frac{n}{1+2\gamma^{\varepsilon}}\right) \\
&  \leq  \left(  E_{\lambda^{\varepsilon}}e^{-\gamma^{\varepsilon
}\mathcal{Z}_{1}^{\varepsilon}}\right)  e^{\frac{n\gamma^{\varepsilon}%
}{1+2\gamma^{\varepsilon}}},
\end{align*}
where $\mathcal{Z}_{1}^{\varepsilon}=\tau_{1}^{\varepsilon}/E_{\lambda
^{\varepsilon}}\tau_{1}^{\varepsilon}.$ We use the fact that $\{\tau
_{n}^{\varepsilon}-\tau_{n-1}^{\varepsilon}\}_{n\in%
\mathbb{N}
}$ are iid and apply Lemma \ref{Lem:7.6} (Chernoff bound) with $a=n/\left(
1+2\gamma^{\varepsilon}\right)  $ and $t=\gamma^{\varepsilon}$ for the last inequality.
Therefore, in order to verify the first claim, it suffices to show that
\[
\left(  E_{\lambda^{\varepsilon}}e^{-\gamma^{\varepsilon}\mathcal{Z}%
_{1}^{\varepsilon}}\right)  e^{\frac{\gamma^{\varepsilon}}{1+2\gamma
^{\varepsilon}}}\leq e^{-\left(  \gamma^{\varepsilon}\right)  ^{2}%
}=e^{-\left(  \Gamma^{\varepsilon}\right)  ^{-2\ell}}.
\]
We observe that for any $x\geq0,$ $e^{-x}\leq1-x+x^{2}/2,$ and this
gives%
\begin{align*}
E_{\lambda^{\varepsilon}}e^{-\gamma^{\varepsilon}\mathcal{Z}_{1}^{\varepsilon
}}    \leq1-E_{\lambda^{\varepsilon}}\left(  \gamma^{\varepsilon}%
\mathcal{Z}_{1}^{\varepsilon}\right)  +E_{\lambda^{\varepsilon}%
}\left(  \gamma^{\varepsilon}\mathcal{Z}_{1}^{\varepsilon}\right)  ^{2}/2
  =1-\gamma^{\varepsilon}+\left(  \gamma^{\varepsilon}\right)
^{2}E_{\lambda^{\varepsilon}}\left(  \mathcal{Z}_{1}^{\varepsilon}\right)
^{2}/2.
\end{align*}
Moreover, since we can apply Lemma \ref{Lem:7.5} to find $E_{\lambda
^{\varepsilon}}\left(  \mathcal{Z}_{1}^{\varepsilon}\right)  ^{2}\rightarrow2$
as $\varepsilon\rightarrow0,$ there exists $\varepsilon_{0}\in(0,1)$ such that
for any $\varepsilon\in(0,\varepsilon_{0})$, $E_{\lambda^{\varepsilon}}\left(
\mathcal{Z}_{1}^{\varepsilon}\right)  ^{2}\leq9/4.$ Thus, for any
$\varepsilon\in(0,\varepsilon_{0})$%
\[
\left(  E_{\lambda^{\varepsilon}}e^{-\gamma^{\varepsilon}\mathcal{Z}%
_{1}^{\varepsilon}}\right)  e^{\frac{\gamma^{\varepsilon}}{1+2\gamma
^{\varepsilon}}}\leq\exp\left\{  \gamma^{\varepsilon}/(1+2\gamma^{\varepsilon})+\log(  1-\gamma^{\varepsilon}+(9/8)\left(
\gamma^{\varepsilon}\right)  ^{2})  \right\}  .
\]

Using a Taylor series expansion we find that for all $\left\vert x\right\vert
<1$
\[
1/(1+x)=1-x+O\left(  x^{2}\right)  \text{ and }\log\left(
1+x\right)  =x-x^{2}/2+O\left(  x^{3}\right)  ,
\]
which gives%
\begin{align*}
&  \gamma^{\varepsilon}/(1+2\gamma^{\varepsilon})+\log(
1-\gamma^{\varepsilon}+(9/8)\left(  \gamma^{\varepsilon}\right)
^{2}) \\
&  \quad=\gamma^{\varepsilon}-2\left(  \gamma^{\varepsilon}\right)
^{2}+[(  -\gamma^{\varepsilon}+(9/8)\left(  \gamma^{\varepsilon
}\right)  ^{2}]  -[  -\gamma^{\varepsilon}+(9/8)\left(  \gamma^{\varepsilon}\right) ^{2}]  ^{2}/2+O(  (
\gamma^{\varepsilon})  ^{3}) \\
&  \quad=-(11/8)\left(  \gamma^{\varepsilon}\right)  ^{2}+O(
\left(  \gamma^{\varepsilon}\right)  ^{3})  \leq-\left(
\gamma^{\varepsilon}\right)  ^{2},
\end{align*}
for all $\varepsilon\in(0,\varepsilon_{0})$. We are done for part 1.

For part 2, we use the estimate from part 1 and find%
\[
\sum_{n=\left(  1+2\gamma^{\varepsilon}\right)  \Gamma^{\varepsilon}+1}%
^{\infty}P_{\lambda^{\varepsilon}}\left(  \tau_{n}^{\varepsilon}\leq
T^{\varepsilon}\right)  \leq\sum_{n=\left(  1+2\gamma^{\varepsilon}\right)
\Gamma^{\varepsilon}+1}^{\infty}e^{-n\left(  \gamma^{\varepsilon
}\right)  ^{2}}\leq\frac{e^{-\left(  1+2\gamma^{\varepsilon}\right)
\Gamma^{\varepsilon}\left(  \gamma^{\varepsilon}\right)  ^{2}}%
}{1-e^{-\left(  \gamma^{\varepsilon}\right)  ^{2}}}.
\]
Since $e^{-x}\leq1-x+x^{2}/2$ for any $x\in%
\mathbb{R}
,$ we have $1-e^{-x}\geq x-x^{2}/2\geq x-x/2=x/2$ for all $x\in(0,1),$ and
thus $1/(1-e^{-x})\leq2/x$ for all $x\in(0,1).$ As a result%
\begin{align*}
\sum_{n=\left(  1+2\gamma^{\varepsilon}\right)  \Gamma^{\varepsilon}+1}%
^{\infty}P_{\lambda^{\varepsilon}}\left(  \tau_{n}^{\varepsilon}\leq
T^{\varepsilon}\right)    \leq\frac{e^{-\left(  1+2\gamma^{\varepsilon
}\right)  \Gamma^{\varepsilon}\left(  \gamma^{\varepsilon}\right)
^{2}}}{1-e^{-\left(  \gamma^{\varepsilon}\right)  ^{2}}}\leq
\frac{2}{\left(  \gamma^{\varepsilon}\right)  ^{2}}e^{-\left(  1+2\gamma
^{\varepsilon}\right)  \Gamma^{\varepsilon}\left(  \gamma
^{\varepsilon}\right)  ^{2}}
  \leq 2\left(  \Gamma^{\varepsilon}\right)  ^{2\ell}e^{-\left(
\Gamma^{\varepsilon}\right)  ^{1-2\ell}}.
\end{align*}
This completes the proof of part 2.

Finally, for part 3, we use the fact that for $x\in(0,1),$ and for any $k\in%
\mathbb{N}
,$
\[
\sum\nolimits_{n=k}^{\infty}nx^{n}=k x^{k}(1-x)^{-1}+x^{k+1}(
1-x)^{-2}\leq (k(1-x)^{-1} + (1-x)^{-2}) x^{k}.
\]
Using the estimate from part 1 once again, we have%
\begin{align*}
  \sum_{n=\left(  1+2\gamma^{\varepsilon}\right)  \Gamma^{\varepsilon}%
+1}^{\infty}nP_{\lambda^{\varepsilon}}\left(  \tau_{n}^{\varepsilon}\leq
T^{\varepsilon}\right) 
&  \leq\sum_{n=\left(  1+2\gamma^{\varepsilon}\right)  \Gamma
^{\varepsilon}}^{\infty}ne^{-n\left(  \gamma^{\varepsilon}\right)
^{2}}\\
&  \leq\left(  \frac{\left(  1+2\gamma^{\varepsilon}\right)
\Gamma^{\varepsilon}}{1-e^{-\left(  \gamma^{\varepsilon}\right)
^{2}}}+\left(  1-e^{-\left(  \gamma^{\varepsilon}\right)
^{2}}\right)^{-2}\right)  e^{-\left(  1+2\gamma^{\varepsilon}\right)
\Gamma^{\varepsilon}\left(  \gamma^{\varepsilon}\right)  ^{2}}\\
&  \leq\left(  4\left(  \Gamma^{\varepsilon}\right)  ^{1+2\ell
}+4\left(  \Gamma^{\varepsilon}\right)  ^{4\ell}\right)  e^{-\left(  \Gamma^{\varepsilon}\right)  ^{1-2\ell}}.
\end{align*}
We are done.
\end{proof}

\begin{remark}
\label{Rmk:7.1}If $0<\ell<1/2,$ then $\mathfrak{P}_{1}$ and $\mathfrak{R}_{1}$
converge to $0$ doubly exponentially fast as $\varepsilon\rightarrow0$ in the sense that for any $k\in (0,\infty)$%
\[
\liminf_{\varepsilon\rightarrow0}-\varepsilon\log\left[  \left(
\Gamma^{\varepsilon}\right)  ^{k}e^{-\left(  \Gamma^{\varepsilon
}\right)  ^{1-2\ell}}\right]  =\infty.
\]
\end{remark}

\subsection{Berry-Esseen bound}

\label{subsec:Berry-Esseen_bound}

In this subsection we will provide upper bounds for%
\[
\mathfrak{P}_{2}\doteq\sum_{n=\left(  1-2\gamma^{\varepsilon}\right)
\Gamma^{\varepsilon}}^{\left(  1+2\gamma^{\varepsilon}\right)  \Gamma
^{\varepsilon}}P_{\lambda^{\varepsilon}}\left(  \tau_{n}^{\varepsilon}\leq
T^{\varepsilon}\right)
\text{ and }%
\mathfrak{R}_{2}\doteq\sum_{n=\left(  1-2\gamma^{\varepsilon}\right)
\Gamma^{\varepsilon}}^{\left(  1+2\gamma^{\varepsilon}\right)  \Gamma
^{\varepsilon}}\left(  2n+1\right)  P_{\lambda^{\varepsilon}}\left(  \tau
_{n}^{\varepsilon}\leq T^{\varepsilon}\right)
\]
via the Berry-Esseen bound.

We first recall that $\Gamma^{\varepsilon}=T^{\varepsilon}/E_{\lambda
^{\varepsilon}}^{\varepsilon}\tau_{1}^{\varepsilon}$.
The following is Theorem 1 in \cite[Chapter XVI.5]{fel2}.

\begin{theorem}
[Berry-Esseen]Let $\left\{  X_{n}\right\}  _{n\in%
\mathbb{N}
}$ be independent real-valued random variables with a common distribution such
that
\[
E\left(  X_{1}\right)  =0,\text{ }\sigma^{2}\doteq E\left(  X_{1}\right)
^{2}>0,\text{ }\rho\doteq E  \left\vert X_{1}\right\vert ^{3}
<\infty.
\]
Then for all $x\in%
\mathbb{R}
$ and $n\in%
\mathbb{N}
,$%
\[
\left\vert P\left(  \frac{X_{1}+\cdots+X_{n}}{\sigma\sqrt{n}}\leq x\right)
-\Phi\left(  x\right)  \right\vert \leq\frac{3\rho}{\sigma^{3}\sqrt{n}},
\]
where $\Phi\left(  \cdot\right)  $ is the distribution function of $N\left(
0,1\right)  .$
\end{theorem}

\begin{corollary}
\label{Cor:7.1}For any $\varepsilon>0,$ let $\left\{  X_{n}^{\varepsilon
}\right\}  _{n\in%
\mathbb{N}
}$ be independent real-valued random variables with a common distribution such
that
\[
E\left(  X_{1}^{\varepsilon}\right)  =0,\text{ }\left(  \sigma^{\varepsilon
}\right)  ^{2}\doteq E\left(  X_{1}^{\varepsilon}\right)  ^{2}>0,\text{ }%
\rho^{\varepsilon}\doteq E  \left\vert X_{1}^{\varepsilon}\right\vert
^{3} <\infty.
\]
Assume that there exists $\varepsilon_{0}\in(0,1)$ such that
\[
\hat{\rho}\text{ }\doteq\sup\nolimits_{\varepsilon\in(0,\varepsilon_{0})}%
\rho^{\varepsilon}<\infty\text{ and }\hat{\sigma}^{2}\doteq\inf\nolimits_{\varepsilon
\in(0,\varepsilon_{0})}\left(  \sigma^{\varepsilon}\right)  ^{2}>0.
\]
Then for all $x\in%
\mathbb{R}
,n\in%
\mathbb{N}
$ and $\varepsilon\in(0,\varepsilon_{0}),$%
\[
\left\vert P\left(  \frac{X_{1}^{\varepsilon}+\cdots+X_{n}^{\varepsilon}%
}{\sigma^{\varepsilon}\sqrt{n}}\leq x\right)  -\Phi\left(  x\right)
\right\vert \leq\frac{3\rho^{\varepsilon}}{\left(  \sigma^{\varepsilon
}\right)  ^{3}\sqrt{n}}\leq\frac{3\hat{\rho}}{\hat{\sigma}^{3}\sqrt{n}}.
\]
\end{corollary}

\begin{lemma}
\label{Lem:7.8}Given any $\delta>0$ and any $\ell>0,$ there exists
$\varepsilon_{0}\in(0,1)$ such that for any $\varepsilon\in(0,\varepsilon
_{0})$ and $k\in%
\mathbb{N}
_{0}$, $0\leq k\leq 2\gamma^{\varepsilon}\Gamma^{\varepsilon}$%
\[
P_{\lambda^{\varepsilon}}\left(  \tau_{\Gamma^{\varepsilon}+k}^{\varepsilon
}\leq T^{\varepsilon}\right)  \leq1-\Phi\left(  \frac{k}{\sigma^{\varepsilon
}\sqrt{\Gamma^{\varepsilon}+k}}\right)  +\frac{3\hat{\rho}}{\hat{\sigma}%
^{3}\sqrt{\Gamma^{\varepsilon}+k}}%
\]
and%
\[
P_{\lambda^{\varepsilon}}\left(  \tau_{\Gamma^{\varepsilon}-k}^{\varepsilon
}\leq T^{\varepsilon}\right)  \leq\Phi\left(  \frac{k}{\sigma^{\varepsilon
}\sqrt{\Gamma^{\varepsilon}-k}}\right)  +\frac{3\hat{\rho}}{\hat{\sigma}%
^{3}\sqrt{\Gamma^{\varepsilon}-k}},
\]
where $(\sigma^{\varepsilon})^{2}\doteq E_{\lambda^{\varepsilon}}\left(
\mathfrak{X}_{1}^{\varepsilon}\right)  ^{2},$ $\hat{\rho} \doteq \sup_{\varepsilon
\in(0,\varepsilon_{0})}E_{\lambda^{\varepsilon}} \left|  \mathfrak{X}%
_{1}^{\varepsilon}\right|  ^{3}   <\infty$ and $\hat{\sigma}^{2}%
\doteq \inf_{\varepsilon\in(0,\varepsilon_{0})}(\sigma^{\varepsilon})^{2}>0$ with
$\mathfrak{X}_{1}^{\varepsilon}\doteq \tau_{1}^{\varepsilon}/E_{\lambda
^{\varepsilon}}^{\varepsilon}\tau_{1}^{\varepsilon}-1.$
\end{lemma}

\begin{proof}
For any $n\in%
\mathbb{N}
,$ we define $\mathfrak{X}_{n}^{\varepsilon}\doteq \mathcal{Z}_{n}^{\varepsilon
}-E_{\lambda^{\varepsilon}}^{\varepsilon}\mathcal{Z}_{1}^{\varepsilon}$ with
$\mathcal{Z}_{n}^{\varepsilon}\doteq (\tau_{n}^{\varepsilon}-\tau_{n-1}%
^{\varepsilon})/E_{\lambda^{\varepsilon}}^{\varepsilon}\tau_{1}^{\varepsilon
}.$ Obviously, $E_{\lambda^{\varepsilon}}\mathcal{Z}_{n}^{\varepsilon}=1$ and
$E_{\lambda^{\varepsilon}}\mathfrak{X}_{n}^{\varepsilon}=0$ and if we apply
Lemma \ref{Lem:7.5}, then we find that there exists some $\varepsilon_{0}%
\in(0,1)$ such that
$
\sup\nolimits_{\varepsilon\in(0,\varepsilon_{0})}E_{\lambda^{\varepsilon}}\left(
\mathcal{Z}_{1}^{\varepsilon}\right)  ^{3}<\infty\text{ and }\inf\nolimits
_{\varepsilon\in(0,\varepsilon_{0})}\mathrm{Var}_{\lambda^{\varepsilon}%
}(\mathcal{Z}_{1}^{\varepsilon})>0.
$
Since $\mathcal{Z}_{1}^{\varepsilon}\geq0$, Jensen's inequality implies
$\left(  E_{\lambda^{\varepsilon}}\mathcal{Z}_{1}^{\varepsilon}\right)
^{3}\leq E_{\lambda^{\varepsilon}}\left(  \mathcal{Z}_{1}^{\varepsilon
}\right)  ^{3}$, and therefore
\begin{align*}
\hat{\rho}    
 \leq4\sup\nolimits_{\varepsilon\in(0,\varepsilon_{0})}\left(  E_{\lambda
^{\varepsilon}}\left(  \mathcal{Z}_{1}^{\varepsilon}\right)  ^{3}+\left(
E_{\lambda^{\varepsilon}}\mathcal{Z}_{1}^{\varepsilon}\right)  ^{3}\right) 
 \leq8\sup\nolimits_{\varepsilon\in(0,\varepsilon_{0})}E_{\lambda^{\varepsilon}%
}\left(  \mathcal{Z}_{1}^{\varepsilon}\right)  ^{3}<\infty,
\end{align*}
and
\[
\hat{\sigma}^{2}=\inf\nolimits_{\varepsilon\in(0,\varepsilon_{0})}E_{\lambda
^{\varepsilon}}\left(  \mathfrak{X}_{1}^{\varepsilon}\right)  ^{2}%
=\inf\nolimits_{\varepsilon\in(0,\varepsilon_{0})}\mathrm{Var}_{\lambda^{\varepsilon}%
}\left(  \mathcal{Z}_{1}^{\varepsilon}\right)  >0.
\]
Therefore we can use Corollary \ref{Cor:7.1} with the iid sequence $\left\{
\mathfrak{X}_{n}^{\varepsilon}\right\}  _{n\in%
\mathbb{N}
}$ to find that for any $k\in%
\mathbb{N}
_{0}$ and $0\leq k\leq 2\gamma^{\varepsilon}\Gamma^{\varepsilon}$%
\begin{align*}
 P_{\lambda^{\varepsilon}}\left(  \tau_{\Gamma^{\varepsilon}+k}%
^{\varepsilon}\leq T^{\varepsilon}\right) 
&  =P_{\lambda^{\varepsilon}}\left(  \mathcal{Z}_{1}^{\varepsilon}%
+\cdots+\mathcal{Z}_{\Gamma^{\varepsilon}+k}^{\varepsilon}\leq\Gamma
^{\varepsilon}\right) \\
&  =P_{\lambda^{\varepsilon}}\left(  \frac{\mathfrak{X}_{1}^{\varepsilon
}+\cdots+\mathfrak{X}_{\Gamma^{\varepsilon}+k}^{\varepsilon}}{\sigma
^{\varepsilon}\sqrt{\Gamma^{\varepsilon}+k}}\leq\frac{-k}{\sigma^{\varepsilon
}\sqrt{\Gamma^{\varepsilon}+k}}\right) \\
&  \leq 1-\Phi\left(  \frac{k}{\sigma^{\varepsilon}\sqrt{\Gamma^{\varepsilon}+k}%
}\right)  +\frac{3\hat{\rho}}{\hat{\sigma}^{3}\sqrt{\Gamma^{\varepsilon}+k}},
\end{align*}
and similarly%
\begin{align*}
  P_{\lambda^{\varepsilon}}\left(  \tau_{\Gamma^{\varepsilon}-k}%
^{\varepsilon}\leq T^{\varepsilon}\right) 
&   \leq\Phi\left(  \frac{k}{\sigma^{\varepsilon}\sqrt{\Gamma^{\varepsilon}-k}%
}\right)  +\frac{3\hat{\rho}}{\hat{\sigma}^{3}\sqrt{\Gamma^{\varepsilon}-k}}.
\end{align*}
\end{proof}

\begin{lemma}
\label{Lem:7.9}Given any $\delta>0$ and any $\ell\in(0,1/2),$ there exists
$\varepsilon_{0}\in(0,1)$ such that for any $\varepsilon\in(0,\varepsilon_{0})$, 
$
\mathfrak{P}_{2}
\leq C\left(  \Gamma^{\varepsilon}\right)  ^{\frac
{1}{2}-\ell}+2\left(  \Gamma^{\varepsilon}\right)  ^{1-\ell}.
$

\end{lemma}

\begin{proof}
We rewrite $\mathfrak{P}_{2}$ as
\begin{align*}
\mathfrak{P}_{2}  
  =\sum\nolimits_{k=1}^{2\gamma^{\varepsilon}\Gamma^{\varepsilon}}P_{\lambda
^{\varepsilon}}\left(  \tau_{\Gamma^{\varepsilon}-k}^{\varepsilon}\leq
T^{\varepsilon}\right)  +P_{\lambda^{\varepsilon}}\left(  \tau_{\Gamma
^{\varepsilon}}^{\varepsilon}\leq T^{\varepsilon}\right)  +\sum\nolimits_{k=1}%
^{2\gamma^{\varepsilon}\Gamma^{\varepsilon}}P_{\lambda^{\varepsilon}}\left(
\tau_{\Gamma^{\varepsilon}+k}^{\varepsilon}\leq T^{\varepsilon}\right)  .
\end{align*}
Then we use the upper bounds from Lemma \ref{Lem:7.8} to get
\begin{align*}
\mathfrak{P}_{2}  &  \leq\sum_{k=1}^{2\gamma^{\varepsilon}\Gamma^{\varepsilon}%
}\left[  \Phi\left(  \frac{k}{\sigma^{\varepsilon}\sqrt{\Gamma^{\varepsilon
}-k}}\right)  +\frac{3\hat{\rho}}{\hat{\sigma}^{3}\sqrt{\Gamma^{\varepsilon
}-k}}\right]  +1
+\sum_{k=1}^{2\gamma^{\varepsilon}\Gamma^{\varepsilon}}\left[
1-\Phi\left(  \frac{k}{\sigma^{\varepsilon}\sqrt{\Gamma^{\varepsilon}+k}%
}\right)  +\frac{3\hat{\rho}}{\hat{\sigma}^{3}\sqrt{\Gamma^{\varepsilon}+k}%
}\right] \\
&  \leq\frac{24\hat{\rho}}{\hat{\sigma}^{3}}\gamma^{\varepsilon}\sqrt
{\Gamma^{\varepsilon}}+1+2\gamma^{\varepsilon}\Gamma^{\varepsilon}+\sum
_{k=1}^{2\gamma^{\varepsilon}\Gamma^{\varepsilon}}\left[  \Phi\left(  \frac
{k}{\sigma^{\varepsilon}\sqrt{\Gamma^{\varepsilon}-k}}\right)  -\Phi\left(
\frac{k}{\sigma^{\varepsilon}\sqrt{\Gamma^{\varepsilon}+k}}\right)  \right]  .
\end{align*}
The sum of the first three terms is easily bounded above by $C\left(\Gamma^{\varepsilon}\right)  ^{\frac{1}{2}-\ell}+2\left(  \Gamma^{\varepsilon}\right)  ^{1-\ell}$.
We will show that the last term
is bounded above by a constant to complete the proof.

To prove this, we observe that for any $k\leq 2\gamma^{\varepsilon}%
\Gamma^{\varepsilon},$ we may assume $k\leq\Gamma^{\varepsilon}/2$ by taking
$\varepsilon$ sufficiently small. Then we apply the Mean Value Theorem and
find
\begin{align*}
 \left\vert \Phi\left(  \frac{k}{\sigma^{\varepsilon}\sqrt{\Gamma
^{\varepsilon}-k}}\right)  -\Phi\left(  \frac{k}{\sigma^{\varepsilon}%
\sqrt{\Gamma^{\varepsilon}+k}}\right)  \right\vert 
\leq\sup\limits_{x\in\left[
\frac{\sqrt{2/3}k}{\sigma^{\varepsilon}\sqrt{\Gamma^{\varepsilon}}},\frac{\sqrt{2}%
k}{\sigma^{\varepsilon}\sqrt{\Gamma^{\varepsilon}}}\right]  }\phi\left(  x\right)  \cdot\left(  \frac
{k}{\sigma^{\varepsilon}\sqrt{\Gamma^{\varepsilon}-k}}-\frac{k}{\sigma
^{\varepsilon}\sqrt{\Gamma^{\varepsilon}+k}}\right)  ,
\end{align*}
where $\phi\left(  x\right)  \doteq e^{-\frac{x^{2}}{2}}/\sqrt{2\pi}$ and since
$0\leq k\leq\Gamma^{\varepsilon}/2,$ we have%
\[
\left[  \frac{k}{\sigma^{\varepsilon}\sqrt{\Gamma^{\varepsilon}+k}},\frac
{k}{\sigma^{\varepsilon}\sqrt{\Gamma^{\varepsilon}-k}}\right]  \subset\left[
\frac{\sqrt{2/3}k}{\sigma^{\varepsilon}\sqrt{\Gamma^{\varepsilon}}},\frac{\sqrt{2}%
k}{\sigma^{\varepsilon}\sqrt{\Gamma^{\varepsilon}}}\right]  .
\]
Additionally, because $\phi\left(  x\right)  =e^{-\frac{x^{2}}{2}}/\sqrt{2\pi
}$ is a monotone decreasing function on $[0,\infty)$, we find that 
\[
x\in[  (\sqrt{2/3}k)(\sigma^{\varepsilon}\sqrt{\Gamma^{\varepsilon}}),(\sqrt{2}%
k)(\sigma^{\varepsilon}\sqrt{\Gamma^{\varepsilon}})] \quad \mbox{implies} \quad
\phi\left(  x\right)  \leq  e^{-\frac{k^{2}%
}{3\left(  \sigma^{\varepsilon}\right)  ^{2}\Gamma^{\varepsilon}}}/{\sqrt{2\pi}}.
\]
Also,  $\sqrt{1+x}-\sqrt{1-x}\leq2x$ for all $x\in
\lbrack0,1]$ and $k\leq\Gamma^{\varepsilon}/2$ and a little  algebra give
$
{k}/{\sqrt{\Gamma^{\varepsilon}-k}}-{k}/{\sqrt{\Gamma^{\varepsilon}%
+k}}  
\leq {4k^{2}}/{\Gamma^{\varepsilon}\sqrt
{\Gamma^{\varepsilon}}}.$
Therefore we find%
\begin{align*}
& \sum\nolimits_{k=1}^{2\gamma^{\varepsilon}\Gamma^{\varepsilon}}\left[  \Phi\left(
\frac{k}{\sigma^{\varepsilon}\sqrt{\Gamma^{\varepsilon}-k}}\right)
-\Phi\left(  \frac{k}{\sigma^{\varepsilon}\sqrt{\Gamma^{\varepsilon}+k}%
}\right)  \right] \\
& \qquad\leq\sum\nolimits_{k=1}^{2\gamma^{\varepsilon}\Gamma^{\varepsilon}}\frac
{1}{\sqrt{2\pi}}e^{-\frac{k^{2}}{3\left(  \sigma^{\varepsilon}\right)
^{2}\Gamma^{\varepsilon}}}\frac{4k^{2}}{\sigma^{\varepsilon}\Gamma
^{\varepsilon}\sqrt{\Gamma^{\varepsilon}}}
 \leq\frac{4}{\sigma^{\varepsilon}\Gamma^{\varepsilon}}\sum\nolimits
_{k=1}^{2\gamma^{\varepsilon}\Gamma^{\varepsilon}}\int_{k-1}^{k}\frac{\left(
1+x\right)  ^{2}}{\sqrt{2\pi\Gamma^{\varepsilon}}}e^{-\frac{x^{2}}{3\left(
\sigma^{\varepsilon}\right)  ^{2}\Gamma^{\varepsilon}}}dx\\
& \qquad\leq\frac{4}{\Gamma^{\varepsilon}}\sqrt{\frac{3}{2}}\int_{0}^{\infty}\frac{\left(
1+x\right)  ^{2}}{\sqrt{3\pi\left(  \sigma^{\varepsilon}\right)  ^{2}%
\Gamma^{\varepsilon}}}e^{-\frac{x^{2}}{3\left(  \sigma^{\varepsilon}\right)
^{2}\Gamma^{\varepsilon}}}dx
\leq\frac{6}{\Gamma^{\varepsilon}}E\left(  1+X^{+}\right)  ^{2},
\end{align*}
where $X\sim N(0,3\left(  \sigma^{\varepsilon}\right)  ^{2}\Gamma^{\varepsilon
}/2).$ Finally, since $E\left(  1+X^{+}\right)  ^{2}\leq2+2E\left(
X^{2}\right)  =2+3\left(  \sigma^{\varepsilon}\right)  ^{2}\Gamma
^{\varepsilon},$ this implies that%
\begin{align}
\sum_{k=1}^{2\gamma^{\varepsilon}\Gamma^{\varepsilon}}\left[  \Phi\left(
\frac{k}{\sigma^{\varepsilon}\sqrt{\Gamma^{\varepsilon}-k}}\right)
-\Phi\left(  \frac{k}{\sigma^{\varepsilon}\sqrt{\Gamma^{\varepsilon}+k}%
}\right)  \right]   &  \leq\frac{6}{\Gamma^{\varepsilon}}\left(  2+3\left(
\sigma^{\varepsilon}\right)  ^{2}\Gamma^{\varepsilon}\right) \leq 12+18\hat{\rho}^{2/3}, \label{eqn:diff}%
\end{align}
where the last inequality is from
\begin{align*}
\sup\nolimits_{\varepsilon\in(0,\varepsilon_{0})}\sigma^{\varepsilon}  &
=\sup\nolimits_{\varepsilon\in(0,\varepsilon_{0})}\left(  E_{\lambda^{\varepsilon}%
}(  \mathfrak{X}_{1}^{\varepsilon}\right)  ^{2})  ^{1/2}
  \leq\sup\nolimits_{\varepsilon\in(0,\varepsilon_{0})}(  E_{\lambda
^{\varepsilon}}\left\vert \mathfrak{X}_{1}^{\varepsilon}\right\vert
^{3})  ^{1/3}=\hat{\rho}^{1/3}.
\end{align*}
Since according to Lemma \ref{Lem:7.8} $\hat{\rho}^{1/3}$ is finite, we are done.
\end{proof}

\begin{lemma}
\label{Lem:7.10}Given any $\delta>0$ and any $\ell\in(0,1/2),$ there exists
$\varepsilon_{0}\in(0,1)$ and a constant $C<\infty$ such that for any
$\varepsilon\in(0,\varepsilon_{0})$, 
$\mathfrak{R}_{2}  
\leq 4\left(  \Gamma
^{\varepsilon}\right)  ^{2-\ell}+C\left(  \Gamma^{\varepsilon}\right)
^{2\left(  1-\ell\right)  }.
$

\end{lemma}

\begin{proof}
The proof of this lemma is similar to the proof of Lemma \ref{Lem:7.9}.
We rewrite $\mathfrak{R}_{2}$ as
\begin{align*}
\mathfrak{R}_{2}  &  =\sum\nolimits_{k=1}^{2\gamma^{\varepsilon}\Gamma^{\varepsilon}%
}\left(  2\Gamma^{\varepsilon}-2k+1\right)  P_{\lambda^{\varepsilon}}\left(
\tau_{\Gamma^{\varepsilon}-k}%
^{\varepsilon}\leq T^{\varepsilon}\right) 
+\left(  2\Gamma^{\varepsilon}+1\right)  P_{\lambda^{\varepsilon}%
}\left(  \tau_{\Gamma^{\varepsilon}}%
^{\varepsilon}\leq T^{\varepsilon}\right) 
\\
&  \quad
+\sum\nolimits_{k=1}^{2\gamma^{\varepsilon}\Gamma^{\varepsilon}}\left(
2\Gamma^{\varepsilon}+2k+1\right)  P_{\lambda^{\varepsilon}}\left(
\tau_{\Gamma^{\varepsilon}+k}%
^{\varepsilon}\leq T^{\varepsilon}\right)  .
\end{align*}
Then we use the upper bounds from Lemma \ref{Lem:7.8} to get
\begin{align*}
&  \mathfrak{R}_{2}\leq\sum\nolimits_{k=1}^{2\gamma^{\varepsilon}\Gamma^{\varepsilon}%
}\left(  2\Gamma^{\varepsilon}-2k+1\right)  \left[  \Phi\left(  \frac
{k}{\sigma^{\varepsilon}\sqrt{\Gamma^{\varepsilon}-k}}\right)  +\frac
{3\hat{\rho}}{\hat{\sigma}^{3}\sqrt{\Gamma^{\varepsilon}-k}}\right]  +\left(
2\Gamma^{\varepsilon}+1\right) \\
&  \quad\qquad+\sum\nolimits_{k=1}^{2\gamma^{\varepsilon}\Gamma^{\varepsilon}}\left(
2\Gamma^{\varepsilon}+2k+1\right)  \left[  1-\Phi\left(  \frac{k}%
{\sigma^{\varepsilon}\sqrt{\Gamma^{\varepsilon}+k}}\right)  +\frac{3\hat{\rho
}}{\hat{\sigma}^{3}\sqrt{\Gamma^{\varepsilon}+k}}\right]  .
\end{align*}

The next thing is to pair all the terms carefully and bound these pairs
separately. We start with
\begin{align*}
&  \sum\nolimits_{k=1}^{2\gamma^{\varepsilon}\Gamma^{\varepsilon}}\left(  2\Gamma
^{\varepsilon}-2k+1\right)  \Phi\left(  \frac{k}{\sigma^{\varepsilon}%
\sqrt{\Gamma^{\varepsilon}-k}}\right)  -\sum\nolimits_{k=1}^{2\gamma^{\varepsilon}%
\Gamma^{\varepsilon}}\left(  2\Gamma^{\varepsilon}+2k+1\right)  \Phi\left(
\frac{k}{\sigma^{\varepsilon}\sqrt{\Gamma^{\varepsilon}+k}}\right) \\
& \quad \leq\left(  2\Gamma^{\varepsilon}+1\right)  \sum\nolimits_{k=1}^{2\gamma
^{\varepsilon}\Gamma^{\varepsilon}}\left[  \Phi\left(  \frac{k}{\sigma
^{\varepsilon}\sqrt{\Gamma^{\varepsilon}-k}}\right)  -\Phi\left(  \frac
{k}{\sigma^{\varepsilon}\sqrt{\Gamma^{\varepsilon}+k}}\right)  \right]  \leq C\Gamma^{\varepsilon}.
\end{align*}
We use (\ref{eqn:diff}) for the last inequality.
The second pair is
\begin{align*}
&  \sum\nolimits_{k=1}^{2\gamma^{\varepsilon}\Gamma^{\varepsilon}}\left(  2\Gamma
^{\varepsilon}-2k+1\right)  \frac{3\hat{\rho}}{\hat{\sigma}^{3}\sqrt
{\Gamma^{\varepsilon}-k}}+\sum\nolimits_{k=1}^{2\gamma^{\varepsilon}\Gamma^{\varepsilon
}}\left(  2\Gamma^{\varepsilon}+2k+1\right)  \frac{3\hat{\rho}}{\hat{\sigma
}^{3}\sqrt{\Gamma^{\varepsilon}+k}}\\
&  \quad=\frac{6\hat{\rho}}{\hat{\sigma}^{3}}\sum\nolimits_{k=1}^{2\gamma^{\varepsilon}%
\Gamma^{\varepsilon}}\left(  \sqrt{\Gamma^{\varepsilon}-k}+\sqrt
{\Gamma^{\varepsilon}+k}\right)  +\frac{3\hat{\rho}}{\hat{\sigma}^{3}}%
\sum\nolimits_{k=1}^{2\gamma^{\varepsilon}\Gamma^{\varepsilon}}\left(  \frac{1}%
{\sqrt{\Gamma^{\varepsilon}-k}}+\frac{1}{\sqrt{\Gamma^{\varepsilon}+k}%
}\right)  \\
&\quad\leq \frac{6\hat{\rho}}{\hat{\sigma}^{3}}%
\sum\nolimits_{k=1}^{2\gamma^{\varepsilon}\Gamma^{\varepsilon}}2\sqrt{2\Gamma
^{\varepsilon}}+\frac{3\hat{\rho}}{\hat{\sigma
}^{3}}\sum\nolimits_{k=1}^{2\gamma^{\varepsilon}\Gamma^{\varepsilon}}2\leq C\gamma^{\varepsilon}\Gamma^{\varepsilon}\sqrt{\Gamma^{\varepsilon}%
}+C\gamma^{\varepsilon}\Gamma^{\varepsilon}\leq C\left(  \Gamma^{\varepsilon}\right)  ^{\frac{3}{2}-\ell},
\end{align*}
where the first inequality holds due to $k \leq\Gamma^{\varepsilon}/2$.
The third term is
\begin{align*}
  \sum\nolimits_{k=1}^{2\gamma^{\varepsilon}\Gamma^{\varepsilon}}\left(  2\Gamma
^{\varepsilon}+2k+1\right)  +\left(  2\Gamma^{\varepsilon}+1\right) 
&  =4\gamma^{\varepsilon}\left(  \Gamma^{\varepsilon}\right)  ^{2}%
+2\gamma^{\varepsilon}\Gamma^{\varepsilon}+4\left(  \gamma^{\varepsilon}%
\Gamma^{\varepsilon}\right)  ^{2}+2\gamma^{\varepsilon}\Gamma^{\varepsilon
}+\left(  2\Gamma^{\varepsilon}+1\right) \\
& \leq4\gamma^{\varepsilon}\left(  \Gamma^{\varepsilon}\right)
^{2}+C\left(  \gamma^{\varepsilon}\Gamma^{\varepsilon}\right)  ^{2} =4\left(  \Gamma^{\varepsilon}\right)  ^{2-\ell}+C\left(
\Gamma^{\varepsilon}\right)  ^{2\left(  1-\ell\right)  },
\end{align*}
where the inequality holds since for $\ell\in(0,1/2),$ $2-2\ell\geq1$ and this
implies that
$
\left(  2\Gamma^{\varepsilon}+1\right)  \leq C\left(  \gamma^{\varepsilon
}\Gamma^{\varepsilon}\right)  ^{2}.
$
Lastly, combining all the pairs and the corresponding upper bounds, we find
that for any $\ell\in(0,1/2),$
\begin{align*}
\mathfrak{R}_{2}  &  \leq
[  4\left(  \Gamma^{\varepsilon}\right)  ^{2-\ell}+C\left(
\Gamma^{\varepsilon}\right)  ^{2\left(  1-\ell\right)  }]
+C\Gamma^{\varepsilon}+C\left(  \Gamma^{\varepsilon}\right)  ^{\frac{3}%
{2}-\ell}  \leq4\left(  \Gamma^{\varepsilon}\right)  ^{2-\ell}+C\left(  \Gamma
^{\varepsilon}\right)  ^{2\left(  1-\ell\right)  },
\end{align*}
where $C$ is a constant which depends on $\ell$ only (and in particular is
independent of $\varepsilon$).
\end{proof}

\subsection{Asymptotics of moments of $N^{\varepsilon}(T^{\varepsilon})$}

\label{subsec:asymptotics_of_moments_of_N}

In this subsection, we prove Lemma \ref{Lem:7.1} and Lemma \ref{Lem:7.2}.

\begin{proof}
[Proof of Lemma \ref{Lem:7.1}]First, recall that
\[
E_{\lambda^{\varepsilon}}\left(  N^{\varepsilon}\left(  T^{\varepsilon
}\right)  \right)  =\sum\nolimits_{n=0}^{\infty}P_{\lambda^{\varepsilon}}\left(
\tau_{n}^{\varepsilon}\leq T^{\varepsilon}\right)  =\mathfrak{P}%
_{1}+\mathfrak{P}_{2}+\mathfrak{P}_{3},
\]
where the $\mathfrak{P}_{i}$ are defined in \eqref{eqn:defofPs}.
We can simply bound $\mathfrak{P}_{3}$ from above by $\left(  1-2\gamma
^{\varepsilon}\right)  \Gamma^{\varepsilon}$. Applying Lemma \ref{Lem:7.7} and
Lemma \ref{Lem:7.9} for the other terms, we have for any $\ell\in(0,1/2)$
that
\begin{align*}
 E_{\lambda^{\varepsilon}}\left(  N^{\varepsilon}\left(  T^{\varepsilon
}\right)  \right) 
& \leq C\left(  \Gamma^{\varepsilon}\right)  ^{2\ell}e^{-\left(  \Gamma^{\varepsilon}\right)  ^{1-2\ell}}+(  C\left(
\Gamma^{\varepsilon}\right)  ^{\frac{1}{2}-\ell}+2\left(  \Gamma^{\varepsilon
}\right)  ^{1-\ell})  +\left(  1-2\gamma^{\varepsilon}\right)
\Gamma^{\varepsilon}\\
& =T^{\varepsilon}/E_{\lambda^{\varepsilon}}\tau_{1}%
^{\varepsilon}+C\left(  \Gamma^{\varepsilon}\right)  ^{\frac{1}{2}-\ell
}+C\left(  \Gamma^{\varepsilon}\right)  ^{2\ell}e^{-\left(
\Gamma^{\varepsilon}\right)  ^{1-2\ell}}.
\end{align*}
On the other hand, from the definition of $N^{\varepsilon}\left(
T^{\varepsilon}\right)  ,$ $E_{\lambda^{\varepsilon}}\tau_{N^{\varepsilon
}\left(  T^{\varepsilon}\right)  }^{\varepsilon}\geq T^{\varepsilon}.$ Using
Wald's first identity, we find
\[
E_{\lambda^{\varepsilon}}\tau_{N^{\varepsilon}\left(  T^{\varepsilon}\right)
}^{\varepsilon}=E_{\lambda^{\varepsilon}}  \sum\nolimits_{n=1}^{N^{\varepsilon
}\left(  T^{\varepsilon}\right)  }\left(  \tau_{n}^{\varepsilon}-\tau
_{n-1}^{\varepsilon}\right)   =E_{\lambda^{\varepsilon}}\left(
N^{\varepsilon}\left(  T^{\varepsilon}\right)  \right)  \cdot E_{\lambda
^{\varepsilon}}\tau_{1}^{\varepsilon}.
\]
Hence
\[
0\leq\frac{E_{\lambda^{\varepsilon}}\left(  N^{\varepsilon}\left(
T^{\varepsilon}\right)  \right)  }{T^{\varepsilon}}-\frac{1}{E_{\lambda
^{\varepsilon}}\tau_{1}^{\varepsilon}}\leq\frac{1}{T^{\varepsilon}}[
C\left(  \Gamma^{\varepsilon}\right)  ^{\frac{1}{2}-\ell}+C\left(
\Gamma^{\varepsilon}\right)  ^{2\ell}e^{-\left(  \Gamma
^{\varepsilon}\right)  ^{1-2\ell}}]  .
\]
Therefore,%
\begin{align*}
  \liminf_{\varepsilon\rightarrow0}-\varepsilon\log\left\vert \frac
{E_{\lambda^{\varepsilon}}\left(  N^{\varepsilon}\left(  T^{\varepsilon
}\right)  \right)  }{T^{\varepsilon}}-\frac{1}{E_{\lambda^{\varepsilon}}%
\tau_{1}^{\varepsilon}}\right\vert \geq\liminf_{\varepsilon\rightarrow0}-\varepsilon\log\left[  \frac
{1}{T^{\varepsilon}}\left(  C\left(  \Gamma^{\varepsilon}\right)  ^{\frac
{1}{2}-\ell}+\left(  \Gamma^{\varepsilon}\right)  ^{2\ell}e^{-\left(  \Gamma^{\varepsilon}\right)  ^{1-2\ell}}\right)  \right]  .
\end{align*}
It remains to find an appropriate lower bound for the liminf.

We use (\ref{eqn:sum}), Lemma \ref{Lem:7.4} and Remark \ref{Rmk:7.1} to find
that for any $\eta>0,$ there exists $\delta_{0}\in(0,1)$ such that for any
$\delta\in(0,\delta_{0})$ and any $\ell\in(0,1/2)$
\begin{align*}
&  \liminf_{\varepsilon\rightarrow0}-\varepsilon\log\left[  \frac
{1}{T^{\varepsilon}}\left(  C\left(  \Gamma^{\varepsilon}\right)  ^{\frac
{1}{2}-\ell}+\left(  \Gamma^{\varepsilon}\right)  ^{2\ell}e^{-\left(  \Gamma^{\varepsilon}\right)  ^{1-2\ell}}\right)  \right] \\
&  \quad\geq\liminf_{\varepsilon\rightarrow0}\varepsilon\log
T^{\varepsilon} +\min\left\{  \liminf_{\varepsilon\rightarrow0}-\varepsilon
\log\left(  \Gamma^{\varepsilon}\right)  ^{1/2-\ell},\liminf
_{\varepsilon\rightarrow0}-\varepsilon\log\left(  \left(  \Gamma^{\varepsilon
}\right)  ^{2\ell}e^{-\left(  \Gamma^{\varepsilon}\right)
^{1-2\ell}}\right)  \right\} \\
&  \quad\geq c+\min\left\{  \left(  1/2-\ell\right)  \left(  h_1-c-\eta\right)  ,\infty\right\} =c+\left(  1/2-\ell\right)  \left(  h_1-c-\eta\right)  .
\end{align*}
We complete the proof by sending $\ell$ to $1/2$.
\end{proof}

\vspace{\baselineskip}

\begin{proof}
[Proof of Lemma \ref{Lem:7.2}]Recall that
\[
E_{\lambda^{\varepsilon}}\left(  N^{\varepsilon}\left(  T^{\varepsilon
}\right)  \right)  ^{2}=\sum\nolimits_{n=0}^{\infty}\left(  2n+1\right)  P_{\lambda
^{\varepsilon}}\left(  \tau_{n}^{\varepsilon}\leq t\right)  =\mathfrak{R}%
_{1}+\mathfrak{R}_{2}+\mathfrak{R}_{3}%
\]
where the $\mathfrak{R}_{i}$ are defined in \eqref{eqn:defofRs}.
We can bound $\mathfrak{R}_{3}$ from above by
\begin{align*}
&  \sum\nolimits_{n=0}^{\left(  1-2\gamma^{\varepsilon}\right)  \Gamma^{\varepsilon}%
-1}\left(  2n+1\right) 
=(1-4\gamma^{\varepsilon}+4\left(  \gamma^{\varepsilon}\right)
^{2})  \left(  \Gamma^{\varepsilon}\right)  ^{2}.
\end{align*}
Applying Lemma \ref{Lem:7.7} and Lemma \ref{Lem:7.10}, we have for any
$\ell\in(0,1/2)$ that
\begin{align*}
E_{\lambda^{\varepsilon}}(  N^{\varepsilon}\left(  T^{\varepsilon
}\right)  )  ^{2}
&\leq C\left(  \Gamma^{\varepsilon}\right)  ^{1+2\ell}e^{-\left(  \Gamma^{\varepsilon}\right)  ^{1-2\ell}}+  4\left(
\Gamma^{\varepsilon}\right)  ^{2-\ell}+C\left(  \Gamma^{\varepsilon}\right)
^{2\left(  1-\ell\right)  }
+(  1-4\gamma^{\varepsilon}+4\left(  \gamma^{\varepsilon}\right)
^{2})  \left(  \Gamma^{\varepsilon}\right)  ^{2}\\
& \leq\left(  \Gamma^{\varepsilon}\right)  ^{2}+C\left(  \Gamma
^{\varepsilon}\right)  ^{2\left(  1-\ell\right)  }+C\left(  \Gamma
^{\varepsilon}\right)  ^{1+2\ell}e^{-\left(  \Gamma^{\varepsilon
}\right)  ^{1-2\ell}}.
\end{align*}
As in the proof of Lemma \ref{Lem:7.1}
$
E_{\lambda^{\varepsilon}}\left(  N^{\varepsilon}\left(  T^{\varepsilon
}\right)  \right)  \geq
\Gamma^{\varepsilon}.
$
Thus for any $\ell\in(0,1/2)$%
\begin{align*}
 \mathrm{Var}_{\lambda^{\varepsilon}}\left(  N^{\varepsilon}\left(
T^{\varepsilon}\right)  \right) 
& \leq E_{\lambda^{\varepsilon}}\left(  N^{\varepsilon}\left(
T^{\varepsilon}\right)  \right)  ^{2}-\left(  \Gamma^{\varepsilon}\right)
^{2}\\
&  \leq[  \left(  \Gamma^{\varepsilon}\right)  ^{2}+C\left(
\Gamma^{\varepsilon}\right)  ^{2\left(  1-\ell\right)  }+C\left(
\Gamma^{\varepsilon}\right)  ^{1+2\ell}e^{-\left(  \Gamma
^{\varepsilon}\right)  ^{1-2\ell}}]  -\left(  \Gamma^{\varepsilon
}\right)  ^{2}\\
&  =C\left(  \Gamma^{\varepsilon}\right)  ^{2\left(  1-\ell\right)
}+C\left(  \Gamma^{\varepsilon}\right)  ^{1+2\ell}e^{-\left(
\Gamma^{\varepsilon}\right)  ^{1-2\ell}}.
\end{align*}
Again we use (\ref{eqn:sum}), Lemma \ref{Lem:7.4} and Remark \ref{Rmk:7.1} to
find that  for any $\eta>0,$ there exists $\delta_{0}\in(0,1)$ such that for any
$\delta\in(0,\delta_{0})$ and for any $\ell\in(0,1/2)$,%
\begin{align*}
 & \liminf_{\varepsilon\rightarrow0}-\varepsilon\log\frac{\mathrm{Var}%
_{\lambda^{\varepsilon}}\left(  N^{\varepsilon}\left(  T^{\varepsilon}\right)
\right)  }{T^{\varepsilon}}\\
&  \quad\geq\liminf_{\varepsilon\rightarrow0}\varepsilon\log T^{\varepsilon}
+\min\left\{  \liminf_{\varepsilon\rightarrow0}-\varepsilon
\log\left(  \Gamma^{\varepsilon}\right)  ^{2\left(  1-\ell\right)  }%
,\liminf_{\varepsilon\rightarrow0}-\varepsilon\log\left(  \left(
\Gamma^{\varepsilon}\right)  ^{1+2\ell}e^{-\left(  \Gamma
^{\varepsilon}\right)  ^{1-2\ell}}\right)  \right\} \\
&  \quad\geq c+\min\left\{  2\left(  1-\ell\right)  \left(  h_1-c-\eta\right)
,\infty\right\} 
=2\left(  1-\ell\right)(h_1-\eta)+\left(  2\ell-1\right)  c.
\end{align*}
We complete the proof by sending $\ell$ to $1/2$. 
\end{proof}
\subsection{Asymptotics of moments of $\hat{N}^{\varepsilon}(T^{\varepsilon})$}
The proof of the following result is given in Section
\ref{sec:exponential__returning_law_and_tail_behavior}.
\begin{theorem}
\label{Thm:7.2}If $w\geq h_1$, then given any $m>0$ such that $m+h_1>w$ and for any $\delta>0$ sufficiently small,
\[
\lim_{\varepsilon\rightarrow0}\varepsilon\log E_{\lambda^{\varepsilon}}%
\hat{\tau}_{1}^{\varepsilon}=m+\varkappa_{\delta}\text{ and }\hat{\tau}_{1}^{\varepsilon}%
/E_{\lambda^{\varepsilon}} \hat{\tau}_{1}^{\varepsilon}\overset{d}{\rightarrow
}\mathrm{Exp}(1).
\]
Moreover, there exists $\varepsilon_{0}\in(0,1)$ and a constant $\tilde{c}>0$
such that
\[
P_{\lambda^{\varepsilon}}\left(  \hat{\tau}_{1}^{\varepsilon}/E_{\lambda
^{\varepsilon}}\hat{\tau}_{1}^{\varepsilon}>t\right)  \leq e^{-\tilde{c}t}%
\]
for any $t>0$ and any $\varepsilon\in(0,\varepsilon_{0}).$
\end{theorem}

Notice that Theorem \ref{Thm:7.2} is a multicycle version of Theorem \ref{Thm:7.1},
which is the key to the proofs of the asymptotics of moments of $N^{\varepsilon}(T^{\varepsilon})$, namely, Lemma \ref{Lem:7.1} and Lemma \ref{Lem:7.2}. 
Given Theorem \ref{Thm:7.2}, the proofs of the following analogous results 
follow from essentially the same arguments as those of Lemma \ref{Lem:7.1} and Lemma \ref{Lem:7.2}. 
\begin{lemma}
\label{Lem:7.11}Suppose that $w\geq h_1$,  $m+h_1>w$, and that $T^{\varepsilon
}=e^{\frac{1}{\varepsilon}c}$ for some $c>w$.
Then there exists $\delta_{0}\in(0,1)$ such that for any
$\delta\in(0,\delta_{0})$
\[
\liminf_{\varepsilon\rightarrow0}-\varepsilon\log\left\vert \frac
{E_{\lambda^{\varepsilon}}( \hat{N}^{\varepsilon}\left(  T^{\varepsilon
}\right)  )  }{T^{\varepsilon}}-\frac{1}{E_{\lambda^{\varepsilon}}%
\hat{\tau}_{1}^{\varepsilon}}\right\vert \geq c.
\]
\end{lemma}

\begin{corollary}
\label{Cor:7.3}Suppose that $w\geq h_1$, $m+h_1>w$ and that $T^{\varepsilon
}=e^{\frac{1}{\varepsilon}c}$ for some $c>w$. Then there exists $\delta_{0}\in(0,1)$ such that for any
$\delta\in(0,\delta_{0})$ 
\[
\liminf_{\varepsilon\rightarrow0}-\varepsilon\log\frac{E_{\lambda^{\varepsilon}}(  \hat{N}^{\varepsilon}\left(  T^{\varepsilon}\right)
)  }{T^{\varepsilon}}\geq m+\varkappa_{\delta}.
\]
\end{corollary}

\begin{lemma}
\label{Lem:7.12}Suppose that $w\geq h_1$, $m+h_1>w$ and that $T^{\varepsilon
}=e^{\frac{1}{\varepsilon}c}$ for some $c>w$. Then for any $\eta>0,$ there exists $\delta_{0}\in(0,1)$ such that for any
$\delta\in(0,\delta_{0})$ 
\[
\liminf_{\varepsilon\rightarrow0}-\varepsilon\log\frac{\mathrm{Var}%
_{\lambda^{\varepsilon}}(  \hat{N}^{\varepsilon}\left(  T^{\varepsilon}\right)
)  }{T^{\varepsilon}}\geq m+h_1-\eta.
\]
\end{lemma}

\section{Large Deviation Type Upper Bounds}

\label{subsec:lower_bound_for_performance}

In this section we collect results from the previous sections to prove the
main results of the paper, Theorems \ref{Thm:4.1} and \ref{Thm:4.2}, which
give large deviation upper bounds on the bias under the empirical measure and
the variance per unit time. We also give the proof of Theorem \ref{Thm:4.3},
which shows how to simplify some expressions appearing in the large deviation
bounds. Before giving the proof of the first result we 
establish Lemma \ref{Lem:8.1} and Lemma
\ref{Lem:8.2} for the single cycle case, and Lemma \ref{Lem:8.4} and Lemma
\ref{Lem:8.5} for the multicycle case, which are needed in the proof of Theorems \ref{Thm:4.1}.
Recall that for any $n\in%
\mathbb{N}
$
\begin{equation}\label{eqn:defsen}
S_{n}^{\varepsilon}\doteq\int_{\tau_{n-1}^{\varepsilon}}^{\tau_{n}%
^{\varepsilon}}e^{-\frac{1}{\varepsilon}f\left(  X_{t}^{\varepsilon}\right)
}1_{A}\left(  X_{t}^{\varepsilon}\right)  dt.
\end{equation}

\begin{lemma}
\label{Lem:8.1} If $h_1>w$, $A\subset M$ is compact and $T^{\varepsilon
}=e^{\frac{1}{\varepsilon}c}$ for some $c>h_1$, then for any $\eta>0$, there exists $\delta_{0}\in(0,1)$ such that for any
$\delta\in(0,\delta_{0})$
\begin{align*}
&  \liminf_{\varepsilon\rightarrow0}-\varepsilon\log\left\vert \frac
{E_{\lambda^{\varepsilon}}N^{\varepsilon}\left(  T^{\varepsilon}\right)
}{T^{\varepsilon}}E_{\lambda^{\varepsilon}}S_{1}^{\varepsilon}-\int_M
e^{-\frac{1}{\varepsilon}f\left(  x\right)  }1_{A}\left(  x\right)
\mu^{\varepsilon}\left(  dx\right)  \right\vert 
\\ &  \quad
\geq\inf_{x\in A}\left[  f\left(  x\right)  +W\left(  x\right)
\right]  -W\left(  O_{1}\right)  +c
-h_1-\eta.
\end{align*}

\end{lemma}

\begin{proof}
To begin, by Lemma \ref{Lem:4.1} with $g\left(  x\right)  =e^{-\frac
{1}{\varepsilon}f\left(  x\right)  }1_{A}\left(  x\right)  $, we know that for
any $\delta$ sufficiently small and $\varepsilon>0,$
\begin{align*}
E_{\lambda^{\varepsilon}}S_{1}^{\varepsilon}  &  =E_{\lambda^{\varepsilon}%
}\left(  \int_{0}^{\tau_{1}^{\varepsilon}}e^{-\frac{1}{\varepsilon}f\left(
X_{s}^{\varepsilon}\right)  }1_{A}\left(  X_{s}^{\varepsilon}\right)
ds\right) 
=E_{\lambda^{\varepsilon}}\tau_{1}^{\varepsilon}\cdot\int_M e^{-\frac
{1}{\varepsilon}f\left(  x\right)  }1_{A}\left(  x\right)  \mu^{\varepsilon
}\left(  dx\right)  .
\end{align*}
This implies that
\begin{align*}
 \left\vert \frac{E_{\lambda^{\varepsilon}}\left(  N^{\varepsilon}\left(
T^{\varepsilon}\right)  \right)  }{T^{\varepsilon}}E_{\lambda^{\varepsilon}%
}S_{1}^{\varepsilon}-\int_M e^{-\frac{1}{\varepsilon}f\left(  x\right)  }%
1_{A}\left(  x\right)  \mu^{\varepsilon}\left(  dx\right)  \right\vert 
= 
E_{\lambda^{\varepsilon}%
}S_{1}^{\varepsilon}
\cdot\left\vert \frac{E_{\lambda^{\varepsilon}}\left(
N^{\varepsilon}\left(  T^{\varepsilon}\right)  \right)  }{T^{\varepsilon}%
}-\frac{1}{E_{\lambda^{\varepsilon}}\tau_{1}^{\varepsilon}}\right\vert .
\end{align*}
Hence, by (\ref{eqn:product}), Lemma \ref{Lem:6.17} and Lemma \ref{Lem:7.1}, 
we find that given $\eta>0,$
there exists $\delta_{0}\in(0,1)$ such that for any $\delta
\in(0,\delta_{0})$%
\begin{align*}
&  \liminf_{\varepsilon\rightarrow0}-\varepsilon\log\left\vert \frac
{E_{\lambda^{\varepsilon}}\left(  N^{\varepsilon}\left(  T^{\varepsilon
}\right)  \right)  }{T^{\varepsilon}}E_{\lambda^{\varepsilon}}S_{1}%
^{\varepsilon}-\int_M e^{-\frac{1}{\varepsilon}f\left(  x\right)  }1_{A}\left(
x\right)  \mu^{\varepsilon}\left(  dx\right)  \right\vert \\
&  \quad\geq
\liminf_{\varepsilon\rightarrow0}-\varepsilon\log E_{\lambda^{\varepsilon}}S_{1}%
^{\varepsilon}
+\liminf_{\varepsilon\rightarrow0}-\varepsilon\log\left\vert
\frac{E_{\lambda^{\varepsilon}}\left(  N^{\varepsilon}\left(  T^{\varepsilon
}\right)  \right)  }{T^{\varepsilon}}-\frac{1}{E_{\lambda^{\varepsilon}}%
\tau_{1}^{\varepsilon}}\right\vert \\
&\quad
\geq\inf_{x\in A}\left[  f\left(  x\right)  +W\left(  x\right)
\right]  -W\left(  O_{1}\right)  +c
-h_1-\eta.
\end{align*}
\end{proof}

\vspace{\baselineskip} In the application of Wald's identity a difficulty
arises in that, owing to the randomness of $N^{\varepsilon}\left(
T^{\varepsilon}\right)  $, $S_{N^{\varepsilon}\left(  T^{\varepsilon}\right)
}^{\varepsilon}$ need not have the same distribution as $S_{1}^{\varepsilon}$.
Nevertheless, such minor term can be dealt with by using technique in, for example, \cite[Theorem 3.16]{ros4}. The proof of the following lemma can be found in the Appendix.
\begin{lemma}
\label{Lem:8.2} If $h_1>w$, $A\subset M$ is compact and $T^{\varepsilon
}=e^{\frac{1}{\varepsilon}c}$ for some $c>h_1$, then for any $\eta>0$, there exists $\delta_{0}\in(0,1)$ such that for any
$\delta\in(0,\delta_{0})$
\begin{align*}
&  \liminf_{\varepsilon\rightarrow0}-\varepsilon\log\frac{E_{\lambda
^{\varepsilon}}S_{N^{\varepsilon}\left(  T^{\varepsilon}\right)
}^{\varepsilon}}{T^{\varepsilon}}
\geq\inf_{x\in A}\left[  f\left(  x\right)  +W\left(  x\right)
\right]  -W\left(  O_{1}\right)  + c
-h_1-\eta.
\end{align*}
\end{lemma}

We have similar results for multicycles. To be specific, we have the following two lemmas. 
\begin{lemma}
\label{Lem:8.4}
Suppose that $w\geq h_1$, $m+h_1>w$, $A\subset M$ is compact and that $T^{\varepsilon
}=e^{\frac{1}{\varepsilon}c}$ for some $c>w$. Then for any $\eta>0$, there exists $\delta_{0}\in(0,1)$ such that for any
$\delta\in(0,\delta_{0})$ 
\begin{align*}
&  \liminf_{\varepsilon\rightarrow0}-\varepsilon\log\left\vert \frac
{E_{\lambda^{\varepsilon}}\hat{N}^{\varepsilon}\left(  T^{\varepsilon}\right)
}{T^{\varepsilon}}E_{\lambda^{\varepsilon}}\hat{S}_{1}^{\varepsilon}-\int_M
e^{-\frac{1}{\varepsilon}f\left(  x\right)  }1_{A}\left(  x\right)
\mu^{\varepsilon}\left(  dx\right)  \right\vert \\
&  \quad\geq\inf_{x\in A}\left[  f\left(  x\right)  +W\left(  x\right)
\right]  -W\left(  O_{1}\right)  +c-(m+h_1)-\eta.
\end{align*}

\end{lemma}

\begin{lemma}
\label{Lem:8.5}Suppose that $w\geq h_1$, $m+h_1>w$, $A\subset M$ is compact and that $T^{\varepsilon
}=e^{\frac{1}{\varepsilon}c}$ for some $c>w$. Then for any $\eta>0$, there exists $\delta_{0}\in(0,1)$ such that for any
$\delta\in(0,\delta_{0})$ 
\begin{align*}
 \liminf_{\varepsilon\rightarrow0}-\varepsilon\log\frac{E_{\lambda
^{\varepsilon}}\hat{S}_{\hat{N}^{\varepsilon}\left(  T^{\varepsilon}\right)
}^{\varepsilon}}{T^{\varepsilon}}\geq\inf_{x\in A}\left[  f\left(  x\right)  +W\left(  x\right)
\right]  -W\left(  O_{1}\right)  + c-(m+h_1)-\eta.
\end{align*}

\end{lemma}

\begin{proof}
[Proof of Theorem \ref{Thm:4.1}]If $h_1>w$, then recall that
\[
\frac{1}{T^{\varepsilon}}\sum\nolimits_{n=1}^{N^{\varepsilon}\left(  T^{\varepsilon
}\right)  -1}S_{n}^{\varepsilon}\leq\frac{1}{T^{\varepsilon}}\int
_{0}^{T^{\varepsilon}}e^{-\frac{1}{\varepsilon}f\left(  X_{t}^{\varepsilon
}\right)  }1_{A}\left(  X_{t}^{\varepsilon}\right)  dt\leq\frac{1}%
{T^{\varepsilon}}\sum\nolimits_{n=1}^{N^{\varepsilon}\left(  T^{\varepsilon}\right)
}S_{n}^{\varepsilon},
\]
where $S_{n}^{\varepsilon}$ is defined in \eqref{eqn:defsen}.
Then we apply Wald's first identity to obtain
\begin{align*}
E_{\lambda^{\varepsilon}}\left(  \sum\nolimits_{n=1}^{N^{\varepsilon}\left(
T^{\varepsilon}\right)  -1}S_{n}^{\varepsilon}\right)   &  =E_{\lambda
^{\varepsilon}}\left(  \sum\nolimits_{n=1}^{N^{\varepsilon}\left(  T^{\varepsilon
}\right)  }S_{n}^{\varepsilon}\right)  -E_{\lambda^{\varepsilon}%
}S_{N^{\varepsilon}\left(  T^{\varepsilon}\right)  }^{\varepsilon}\\
&  =E_{\lambda^{\varepsilon}}\left(  N^{\varepsilon}\left(  T^{\varepsilon
}\right)  \right)  E_{\lambda^{\varepsilon}}S_{1}^{\varepsilon}-E_{\lambda
^{\varepsilon}}S_{N^{\varepsilon}\left(  T^{\varepsilon}\right)
}^{\varepsilon}.
\end{align*}
Thus
\begin{align*}
&  \left\vert E_{\lambda^{\varepsilon}}\left(  \frac{1}{T^{\varepsilon}}%
\int_{0}^{T^{\varepsilon}}e^{-\frac{1}{\varepsilon}f\left(  X_{t}%
^{\varepsilon}\right)  }1_{A}\left(  X_{t}^{\varepsilon}\right)  dt\right)
-\int_M e^{-\frac{1}{\varepsilon}f\left(  x\right)  }1_{A}\left(  x\right)
\mu^{\varepsilon}\left(  dx\right)  \right\vert \\
&  \quad\leq\left\vert \frac{E_{\lambda^{\varepsilon}}\left(  N^{\varepsilon
}\left(  T^{\varepsilon}\right)  \right)  }{T^{\varepsilon}}E_{\lambda
^{\varepsilon}}S_{1}^{\varepsilon}-\int_M e^{-\frac{1}{\varepsilon}f\left(
x\right)  }1_{A}\left(  x\right)  \mu^{\varepsilon}\left(  dx\right)
\right\vert +\frac{E_{\lambda^{\varepsilon}}S_{N^{\varepsilon}\left(
T^{\varepsilon}\right)  }^{\varepsilon}}{T^{\varepsilon}}.
\end{align*}
Therefore, by Lemma \ref{Lem:8.1} and Lemma \ref{Lem:8.2} we have that for any $\eta>0$, there exists $\delta_{0}\in(0,1)$ such that for any
$\delta\in(0,\delta_{0})$, 
\begin{align*}
&  \liminf_{\varepsilon\rightarrow0}-\varepsilon\log\left\vert E_{\lambda
^{\varepsilon}}\left(  \frac{1}{T^{\varepsilon}}\int_{0}^{T^{\varepsilon}%
}e^{-\frac{1}{\varepsilon}f\left(  X_{t}^{\varepsilon}\right)  }1_{A}\left(
X_{t}^{\varepsilon}\right)  dt\right)  -\int_M e^{-\frac{1}{\varepsilon}f\left(
x\right)  }1_{A}\left(  x\right)  \mu^{\varepsilon}\left(  dx\right)
\right\vert \\
&  \quad\geq\inf_{x\in A}\left[  f\left(  x\right)  +W\left(  x\right)
\right]  -W\left(  O_{1}\right)  + c-h_1-\eta.
\end{align*}

The argument for $h_1\leq w$ is entirely analogous but uses 
by Lemma \ref{Lem:8.4} and Lemma \ref{Lem:8.5}.
\end{proof}

\vspace{\baselineskip}
The following lemma bounds quantities that will arise in the proof of Theorem \ref{Thm:4.2}.
Its proof is given in the Appendix.

\begin{lemma}
\label{Lem:8.3}Recall the definitions 
$
R_{1}^{(2)}\doteq2\inf_{x\in A}\left[  f\left(  x\right)  +V\left(
O_{1},x\right)  \right]  -h_1,
$ 
and for $j\in L\setminus\{1\}$, 
$
R_{j}^{(2)}   \doteq2\inf_{x\in A}\left[  f\left(  x\right)  +V\left(
O_{j},x\right)  \right]  +W\left(  O_{j}\right)  -2W\left(  O_{1}\right) 
+W\left(  O_{1}\cup O_{j}\right)  
$ with $h_1\doteq\min_{\ell\in L\setminus\{1\}}V(O_{1},O_{\ell}).$ 
Then
$
2\inf_{x\in A}\left[  f\left(  x\right)  +W\left(  x\right)  \right]
-2W\left(  O_{1}\right)  -h_1\geq\min_{j\in L}R_{j}^{(2)}.
$
\end{lemma}

\vspace{\baselineskip}
\begin{proof}
[Proof of Theorem \ref{Thm:4.2}]We begin with the observation that is for any
random variables $X,Y$ and $Z$ satisfying $0\leq Y-Z\leq X\leq Y,$
\begin{align*}
\mathrm{Var}\left(  X\right)   &  =EX^{2}-\left(  EX\right)  ^{2}\leq
EY^{2}-\left(  E\left(  Y-Z\right)  \right)  ^{2}\\
&  =\mathrm{Var}\left(  Y\right)  +2EY\cdot EZ-\left(  EZ\right)  ^{2}%
\leq\mathrm{Var}\left(  Y\right)  +2EY\cdot EZ.
\end{align*}
When $h_1>w$, since
\[
0\leq\frac{1}{T^{\varepsilon}}\sum\nolimits_{n=1}^{N^{\varepsilon}\left(
T^{\varepsilon}\right)  }S_{n}^{\varepsilon}-\frac{1}{T^{\varepsilon}%
}S_{N^{\varepsilon}\left(  T^{\varepsilon}\right)  }^{\varepsilon}\leq\frac
{1}{T^{\varepsilon}}\int_{0}^{T^{\varepsilon}}e^{-\frac{1}{\varepsilon
}f\left(  X_{t}^{\varepsilon}\right)  }1_{A}\left(  X_{t}^{\varepsilon
}\right)  dt\leq\frac{1}{T^{\varepsilon}}\sum\nolimits_{n=1}^{N^{\varepsilon}\left(
T^{\varepsilon}\right)  }S_{n}^{\varepsilon},
\]
we have%
\begin{align*}
&  \mathrm{Var}_{\lambda^{\varepsilon}}\left(  \frac{1}{T^{\varepsilon}}\int
_{0}^{T^{\varepsilon}}e^{-\frac{1}{\varepsilon}f\left(  X_{t}^{\varepsilon
}\right)  }1_{A}\left(  X_{t}^{\varepsilon}\right)  dt\right) \\
&  \quad\leq\mathrm{Var}_{\lambda^{\varepsilon}}\left(  \frac{1}{T^{\varepsilon
}}\sum\nolimits_{n=1}^{N^{\varepsilon}\left(  T^{\varepsilon}\right)  }S_{n}%
^{\varepsilon}\right)  +2E_{\lambda^{\varepsilon}}\left(  \frac{1}%
{T^{\varepsilon}}\sum\nolimits_{n=1}^{N^{\varepsilon}\left(  T^{\varepsilon}\right)
}S_{n}^{\varepsilon}\right)  \frac{E_{\lambda^{\varepsilon}}S_{N^{\varepsilon
}\left(  T^{\varepsilon}\right)  }^{\varepsilon}}{T^{\varepsilon}},
\end{align*}
and with the help of (\ref{eqn:sum})%
\begin{align*}
&  \liminf_{\varepsilon\rightarrow0}-\varepsilon\log\left(  \mathrm{Var}%
_{\lambda^{\varepsilon}}\left(  \frac{1}{T^{\varepsilon}}\int_{0}%
^{T^{\varepsilon}}e^{-\frac{1}{\varepsilon}f\left(  X_{t}^{\varepsilon
}\right)  }1_{A}\left(  X_{t}^{\varepsilon}\right)  dt\right)  T^{\varepsilon
}\right) \\
&  \quad\geq\min\left\{  \liminf_{\varepsilon\rightarrow0}-\varepsilon
\log\left[  \mathrm{Var}_{\lambda^{\varepsilon}}\left(  \frac{1}{T^{\varepsilon
}}\sum\nolimits_{n=1}^{N^{\varepsilon}\left(  T^{\varepsilon}\right)  }S_{n}%
^{\varepsilon}\right)  T^{\varepsilon}\right]  ,\right. \\
&  \quad\qquad\left.  \liminf_{\varepsilon\rightarrow0}-\varepsilon\log\left[
E_{\lambda^{\varepsilon}}\left(  \frac{1}{T^{\varepsilon}}\sum\nolimits_{n=1}%
^{N^{\varepsilon}\left(  T^{\varepsilon}\right)  }S_{n}^{\varepsilon}\right)
\frac{E_{\lambda^{\varepsilon}}S_{N^{\varepsilon}\left(  T^{\varepsilon
}\right)  }^{\varepsilon}}{T^{\varepsilon}}T^{\varepsilon}\right]  \right\}  .
\end{align*}
We complete the proof in the case of single cycle by showing both terms are bounded below by $\min\nolimits_{j\in L}(  R_{j}^{(1)}\wedge R_{j}%
^{(2)})  -\eta$, 
where we recall
\[
R_{j}^{(1)}\doteq\inf\nolimits_{x\in A}\left[  2f\left(  x\right)  +V\left(
O_{j},x\right)  \right]  +W\left(  O_{j}\right)  -W\left(  O_{1}\right)  ,
\]%
\[
R_{1}^{(2)}\doteq2\inf\nolimits_{x\in A}\left[  f\left(  x\right)  +V\left(
O_{1},x\right)  \right]  -h_1,
\]
and for $j\in L\setminus\{1\}$%
\begin{align*}
R_{j}^{(2)}  &  \doteq2\inf\nolimits_{x\in A}\left[  f\left(  x\right)  +V\left(
O_{j},x\right)  \right]  +W\left(  O_{j}\right)  -2W\left(  O_{1}\right) 
+W\left(  O_{1}\cup O_{j}\right)  .
\end{align*}

For the second term, we apply Wald's first identity, Lemma \ref{Lem:6.17}, Corollary \ref{Cor:7.2} and
Lemma \ref{Lem:8.2} to find that given $\eta>0,$ there exists $\delta_{0}\in(0,1),$ such that for any
$\delta\in(0,\delta_{0})$
\begin{align*}
&  \liminf_{\varepsilon\rightarrow0}-\varepsilon\log\left[  E_{\lambda
^{\varepsilon}}\left(  \frac{1}{T^{\varepsilon}}\sum\nolimits_{n=1}^{N^{\varepsilon
}\left(  T^{\varepsilon}\right)  }S_{n}^{\varepsilon}\right)  \frac
{E_{\lambda^{\varepsilon}}S_{N^{\varepsilon}\left(  T^{\varepsilon}\right)
}^{\varepsilon}}{T^{\varepsilon}}T^{\varepsilon}\right] \\
&  \quad\geq\liminf_{\varepsilon\rightarrow0}-\varepsilon\log T^{\varepsilon
}+\liminf_{\varepsilon\rightarrow0}-\varepsilon\log E_{\lambda^{\varepsilon}%
}\left(  \frac{1}{T^{\varepsilon}}\sum\nolimits_{n=1}^{N^{\varepsilon}\left(
T^{\varepsilon}\right)  }S_{n}^{\varepsilon}\right) +\liminf_{\varepsilon\rightarrow0}-\varepsilon\log\frac
{E_{\lambda^{\varepsilon}}S_{N^{\varepsilon}\left(  T^{\varepsilon}\right)
}^{\varepsilon}}{T^{\varepsilon}}\\
&  \quad\geq-c+(\inf\nolimits_{x\in A}\left[  f\left(  x\right)  +W\left(  x\right)
\right]  -W\left(  O_{1}\right)-h_1-\eta/3 ) + \varkappa_{\delta}\\
&  \qquad+(  \inf\nolimits_{x\in A}\left[  f\left(  x\right)  +W\left(  x\right)
\right]  -W\left(  O_{1}\right)  +\left(  c-h_1\right)-\eta/3  ) \\
&  \quad\geq 2\inf\nolimits_{x\in A}\left[  f\left(  x\right)  +W\left(  x\right)  \right]
-2W\left(  O_{1}\right)  -h_1-\eta\\
& \quad \geq\min\nolimits_{j\in L}R_{j}^{(2)}-\eta\geq\min\nolimits_{j\in L}(  R_{j}^{(1)}\wedge
R_{j}^{(2)})-\eta .
\end{align*}
The third inequality holds by choosing $\delta$ sufficiently small $h_\delta\geq h_1-\eta/3$. The fourth inequality is from Lemma \ref{Lem:8.3}.

Turning to the first term, we can bound the variance by (\ref{eqn:variance}):%
\begin{align*}
  \mathrm{Var}_{\lambda^{\varepsilon}}\left(  \frac{1}{T^{\varepsilon}}%
\sum\nolimits_{n=1}^{N^{\varepsilon}\left(  T^{\varepsilon}\right)  }S_{n}%
^{\varepsilon}\right)  T^{\varepsilon}
& \leq2\frac{E_{\lambda^{\varepsilon}}\left(  N^{\varepsilon}\left(
T^{\varepsilon}\right)  \right)  }{T^{\varepsilon}}\mathrm{Var}_{\lambda
^{\varepsilon}}S_{1}^{\varepsilon}+2\frac{\mathrm{Var}_{\lambda^{\varepsilon}%
}\left(  N^{\varepsilon}\left(  T^{\varepsilon}\right)  \right)
}{T^{\varepsilon}}\left(  E_{\lambda^{\varepsilon}}S_{1}^{\varepsilon}\right)
^{2}\\
& \leq2\frac{E_{\lambda^{\varepsilon}}\left(  N^{\varepsilon}\left(
T^{\varepsilon}\right)  \right)  }{T^{\varepsilon}}E_{\lambda^{\varepsilon}%
}\left(  S_{1}^{\varepsilon}\right)  ^{2}+2\frac{\mathrm{Var}_{\lambda
^{\varepsilon}}\left(  N^{\varepsilon}\left(  T^{\varepsilon}\right)  \right)
}{T^{\varepsilon}}\left(  E_{\lambda^{\varepsilon}}S_{1}^{\varepsilon}\right)
^{2}.
\end{align*}
If we use Corollary \ref{Cor:7.2} and Lemma \ref{Lem:6.18}, then we know
that given $\eta>0,$ there exists $\delta_{0}\in(0,1),$ such that for any
$\delta\in(0,\delta_{0})$
\begin{align*}
  \liminf_{\varepsilon\rightarrow0}-\varepsilon\log\left[  \frac
{E_{\lambda^{\varepsilon}}\left(  N^{\varepsilon}\left(  T^{\varepsilon
}\right)  \right)  }{T^{\varepsilon}}E_{\lambda^{\varepsilon}}\left(
S_{1}^{\varepsilon}\right)  ^{2}\right] 
& \geq\liminf_{\varepsilon\rightarrow0}-\varepsilon\log\frac
{E_{\lambda^{\varepsilon}}\left(  N^{\varepsilon}\left(  T^{\varepsilon
}\right)  \right)  }{T^{\varepsilon}}+\liminf_{\varepsilon\rightarrow
0}-\varepsilon\log E_{\lambda^{\varepsilon}}\left(  S_{1}^{\varepsilon
}\right)  ^{2}\\
& 
\geq \min_{j\in L}(  R_{j}^{(1)}\wedge R_{j}^{(2)})
-\eta.
\end{align*}
In addition, we can apply Lemma \ref{Lem:6.17} and Lemma \ref{Lem:7.2}  
to show
that given $\eta>0,$ there exists $\delta_{0}\in(0,1),$ such that for any
$\delta\in(0,\delta_{0})$
\begin{align*}
& \liminf_{\varepsilon\rightarrow0}-\varepsilon\log\left[  \frac
{\mathrm{Var}_{\lambda^{\varepsilon}}\left(  N^{\varepsilon}\left(
T^{\varepsilon}\right)  \right)  }{T^{\varepsilon}}\left(  E_{\lambda
^{\varepsilon}}S_{1}^{\varepsilon}\right)  ^{2}\right] \\
&\quad \geq\liminf_{\varepsilon\rightarrow0}-\varepsilon\log\frac
{\mathrm{Var}_{\lambda^{\varepsilon}}\left(  N^{\varepsilon}\left(
T^{\varepsilon}\right)  \right)  }{T^{\varepsilon}}+2\liminf_{\varepsilon
\rightarrow0}-\varepsilon\log E_{\lambda^{\varepsilon}}S_{1}^{\varepsilon}
\\& \quad 
\geq2\inf_{x\in A}\left[  f\left(  x\right)  +W\left(  x\right)  \right]
-2W\left(  O_{1}\right)  -h_1-\eta\\
& \quad \geq\min_{j\in L}R_{j}^{(2)}-\eta 
\geq\min_{j\in L}(
R_{j}^{(1)}\wedge R_{j}^{(2)})-\eta.
\end{align*}
The second last inequality comes from Lemma \ref{Lem:8.3}.

Hence, we find that given $\eta>0,$ there exists $\delta_{0}\in(0,1),$ such that for any
$\delta\in(0,\delta_{0})$
\[
\liminf_{\varepsilon\rightarrow0}-\varepsilon\log\left(  \mathrm{Var}%
_{\lambda^{\varepsilon}}\left(  \frac{1}{T^{\varepsilon}}\sum\nolimits_{n=1}%
^{N^{\varepsilon}\left(  T^{\varepsilon}\right)  }S_{n}^{\varepsilon}\right)
T^{\varepsilon}\right)  \geq\min_{j\in L}(  R_{j}^{(1)}\wedge R_{j}%
^{(2)})-\eta,
\]
and  we are done for the single cycle case.

For multicycle case, by using a similar argument and applying Lemmas \ref{Lem:6.19}, \ref{Lem:6.20}, \ref{Lem:7.12}, \ref{Lem:8.5} and Corollary \ref{Cor:7.3}, we find that 
\begin{align*}
  \liminf_{\varepsilon\rightarrow0}-\varepsilon\log\left(  \mathrm{Var}%
_{\lambda^{\varepsilon}}\left(  \frac{1}{T^{\varepsilon}}\int_{0}%
^{T^{\varepsilon}}e^{-\frac{1}{\varepsilon}f\left(  X_{t}^{\varepsilon
}\right)  }1_{A}\left(  X_{t}^{\varepsilon}\right)  dt\right)  T^{\varepsilon
}\right) 
 \geq\min_{j\in L}(  R_{j}^{(1)}\wedge R_{j}%
^{(2)}\wedge R_{j}^{(3,m)} )  -\eta,
\end{align*}
with 
\[
R_{j}^{(3,m)}\doteq2\inf_{x\in A}\left[  f\left(  x\right)  +V\left(O_{j},x\right)  \right]  +2W\left(  O_{j}\right)  -2W\left(  O_{1}\right)-(m+h_1).
\]
We complete the proof by sending $m\downarrow w-h_1$.
\end{proof}

\vspace{\baselineskip}

\begin{proof}
[Proof of Theorem \ref{Thm:4.3}]Parts 1,
2 and 3 are from Theorem 4.3, Lemma 4.3 (b) and Theorem 6.1 in \cite[Chapter 6]{frewen2}, respectively.

We now turn to part 4. Before giving the proof, we state a result from
\cite{frewen2}. The result is Lemma 4.3 (c) in \cite[Chapter 6]{frewen2},
which says that for any unstable equilibrium point $O_{j},$ there exists a
stable equilibrium point $O_{i}$ such that  $ W(O_{j})=W(O_{i})+V(O_{i},O_{j}).$

Now suppose that 
$
\min\nolimits_{j\in L}(  \inf\nolimits_{x\in A}\left[  f\left(  x\right)  +V\left(
O_{j},x\right)  \right]  +W\left(  O_{j}\right)  )
$
is attained at some $\ell\in L$ such that $O_{\ell}$ is unstable (i.e.,
$\ell\in L\setminus L_{s}$). Then since there exists a stable equilibrium
point $O_{i}$ (i.e., $i\in L_s$) such that $W(O_{\ell})=W(O_{i})+V(O_{i}%
,O_{\ell})$ we find%
\begin{align*}
&\min_{j\in L}\left(  \inf_{x\in A}\left[  f\left(  x\right)  +V\left(
O_{j},x\right)  \right]  +W\left(  O_{j}\right)  \right) \\
&\quad=\inf_{x\in A}\left[  f\left(  x\right)  +V\left(  O_{\ell},x\right)
\right]  +W\left(  O_{\ell}\right) 
  =\inf_{x\in A}\left[  f\left(  x\right)  +V\left(  O_{\ell},x\right)
\right]  +V(O_{i},O_{\ell})+W(O_{i})\\
&\quad  \geq\inf_{x\in A}\left[  f\left(  x\right)  +V\left(  O_{i},x\right)
\right]  +W(O_{i})
 \geq\min_{j\in L_{\rm{s}}}\left(  \inf_{x\in A}\left[  f\left(  x\right)
+V\left(  O_{j},x\right)  \right]  +W\left(  O_{j}\right)  \right) \\
& \quad \geq\min_{j\in L}\left(  \inf_{x\in A}\left[  f\left(  x\right)
+V\left(  O_{j},x\right)  \right]  +W\left(  O_{j}\right)  \right)  .
\end{align*}
The first inequality is from a dynamic programming inequality. Therefore, the
minimum is also attained at $i\in L_{\rm{s}}$ and $\min_{j\in L}R_{j}^{(1)}%
=\min_{j\in L_{\rm{s}}}R_{j}^{(1)}.$
\end{proof}

\section{Exponential Return Law and Tail Behavior}

\label{sec:exponential__returning_law_and_tail_behavior}

In this section we give the proof of Theorem \ref{Thm:7.1}, 
which was the key fact needed to obtain bounds
on the distribution of $N^{\varepsilon}(T^{\varepsilon})$, and the related multicycle analogy. A result of this
type first appears in \cite{day4}, which asserts that the time needed to
escape from an open subset of the domain of attraction of a stable equilibrium
point that contains the equilibrium point has an asymptotically exponential
distribution. \cite{day4} also proves a nonasymptotic bound on the tail of the
probability of escape before a certain time that is also of exponential form.
Theorem \ref{Thm:7.1} is a more complicated statement, in that it asserts the
asymptotically exponential form for the return time to the neighborhood of
$O_{1}$. To prove this we build on the results of \cite{day4}, and decompose
the return time into times of transitions between equilibrium points. This in
turn will require the proof of a number of related results, such as
establishing the independence of certain estimates with respect to initial distributions.

The existence of an exponentially distributed first hitting time is
a central topic in the theory of quasistationary distributions.
For a recent book length treatment of the topic we refer to  \cite{colmarmar}.
However, so far as we can tell the types of situations we encounter are not covered by 
existing results,
and so as noted we  develop what is needed using  \cite{day4}
as the starting point.
See Remark \ref{rk:onconds}.

For any $j\in L,$ define $\upsilon_{j}^{\varepsilon}$ as the hitting
time of $\partial B_{\delta}(O_{k})$ for some $k\in L\setminus\{j\}$, i.e.,
\begin{equation}\label{eqn:defofnu}
\upsilon_{j}^{\varepsilon}\doteq \inf\left\{  t>0:X_{t}^{\varepsilon}\in\cup_{k\in
L\setminus\{j\}}\partial B_{\delta}(O_{k})\right\}  .
\end{equation}
We will prove the following result for first hitting times of another equilibrium point, and later extend
to return times.

\begin{lemma}
\label{Lem:9.0}
For any $j\in L_{\rm{s}}$, there exists $\delta_{0}\in(0,1)$ such that for any $\delta
\in(0,\delta_{0})$ and any distribution $\lambda^{\varepsilon}$ on $\partial
B_{\delta}(O_{j}),$%
\[
\lim_{\varepsilon\rightarrow0}\varepsilon\log E_{\lambda^{\varepsilon}}\upsilon_{j}^{\varepsilon
}=\min_{y\in\cup_{k\in L\setminus\{j\}}\partial B_{\delta}(O_{k})}%
V(O_{j},y)\text{ and }\upsilon_{j}^{\varepsilon}/E_{\lambda^{\varepsilon}}\upsilon
_{j}^{\varepsilon}\overset{d}{\rightarrow}\rm{Exp}(1).
\]
Moreover, there exists $\varepsilon_{0}\in(0,1)$ and a constant $\tilde{c}>0$
such that
\[
P_{\lambda^{\varepsilon}}\left(  \upsilon_{j}^{\varepsilon}/E_{\lambda^{\varepsilon}}\upsilon_{j}^{\varepsilon
}>t\right)  \leq e^{-\tilde{c}t}%
\]
for any $t>0$ and any $\varepsilon\in(0,\varepsilon_{0}).$
\end{lemma}

The organization of this section is as follows.
The first part of Lemma \ref{Lem:9.0} that is concerned with mean
first hitting times is proved in Subsection \ref{subsec:9.1}, while the second part 
that is concerned with an
asymptotically exponential distribution but when starting with a special
distribution is proved in Subsection \ref{subsec:9.2}.
The last part of the lemma,
which focuses on bounds on the tail of the hitting time of another equilibrium point but when starting with a special
distribution 
is proved in Subsection \ref{subsec:9.3}.
We then extend the second and third parts of Lemma \ref{Lem:9.0} to general
initial distributions in Subsection \ref{subsec:9.4} and Subsection \ref{subsec:9.5}.
The last two subsections then extend all of Lemma \ref{Lem:9.0} to return times  for single cycles and multicycles,
respectively.

\subsection{Mean first hitting time}

\label{subsec:9.1}

\begin{lemma}
\label{Lem:9.1} For any $\delta>0$ sufficiently small and $x\in \partial B_{\delta
}(O_{j})$ with $j\in L_{\rm{s}}$
\begin{equation}
\label{eqn:mfht}\lim_{\varepsilon\rightarrow0}\varepsilon\log E_{x}%
\upsilon_{j}^{\varepsilon}=\min_{y\in\cup_{k\in L\setminus\{j\}}\partial
B_{\delta}(O_{k})}V(O_{j},y).
\end{equation}

\end{lemma}

\begin{proof}
For the given $j\in L_{\rm{s}}$ let $D_{j}$ denote the corresponding domain of
attraction. We claim there is $k\in L\setminus L_{\rm{s}}$ such that
\[
q_{j}\doteq\inf_{y\in\partial D_{j}}V(O_{j},y)=V(O_{j},O_{k}).
\]
Since $V(O_{j},\cdot)$ is continuous and $\partial D_{j}$ is compact, there is
a point $y^{\ast}\in\partial D_{j}$ such that $q_{j}=V(O_{j},y^{\ast})$. If
$y^{\ast}\in\cup_{k\in L\setminus L_{\rm{s}}}O_{k}$ then we are done. If this is
not true, then since $y^{\ast}\notin(\cup_{k\in L_{\rm{s}}}D_{k})\cup(\cup_{k\in
L\setminus L_{\rm{s}}}O_{k})$, and since the solution to $\dot{\phi}=b(\phi
),\phi(0)=y^{\ast}$ must converge to $\cup_{k\in L}O_{k}$ as $t\rightarrow
\infty$, it must in fact converge to a point in $\cup_{k\in L\setminus L_{\rm{s}}%
}O_{k}$, say $O_{k}$. Since such trajectories have zero cost, by a standard
argument for any $\varepsilon>0$ we can construct by concatenation a
trajectory that connects $O_{j}$ to $O_{k}$ in finite time and with cost less
than $q_{j}+\varepsilon$. Since $\varepsilon>0$ is arbitrary we have
$q_{j}=V(O_{j},O_{k})$.

There may be more than one $l\in L\setminus L_{\rm{s}}$ such that $O_{l}%
\in\partial D_{j}$ and $q_{j}=V(O_{j},O_{l})$, but we can assume that for some
$k\in L\setminus L_{\rm{s}}$ and $\bar{y}\in\partial B_{\delta}(O_{k})$ we attain
the min in \eqref{eqn:mfht}. Then $\bar{q}_{j}\doteq V(O_{j},\bar{y})\leq
q_{j}$, and we need to show $\lim_{\varepsilon\rightarrow0}\varepsilon\log
E_{x}\upsilon_{j}^{\varepsilon}=\bar{q}_{j}$.

Given $s<\bar{q}_{j}$, let $D_{j}(s)=\{x:V(O_{j},x)\leq s\}$ and assume $s$ is
large enough that $B_{\delta}(O_{j})\subset D_{j}(s)^{\circ}$. Then
$D_{j}(s)\subset D_{j}^{\circ}$ is closed and contained in the open set
$D_{j}\setminus\cup_{l\in L\setminus\{j\}}B_{\delta}(O_{l})$, and thus the
time to reach $\partial D_{j}(s)$ is never greater than $\upsilon
_{j}^{\varepsilon}$. Given $\eta>0$ we can find a set $D_{j}^{\eta}(s)$ that
is contained in $D_{j}(s)$ and satisfies the conditions of \cite[Theorem 4.1,
Chapter 4]{frewen2}, and also $\inf_{z\in\partial D_{j}^{\eta}(s)}%
V(O_{j},z)\geq s-\eta$. This theorem gives the equality in the following display:
\begin{eqnarray*}
\liminf_{\varepsilon\rightarrow0}\varepsilon\log E_{x}\upsilon_{j}%
^{\varepsilon}\geq \liminf_{\varepsilon\rightarrow0}\varepsilon\log E_{x}%
\inf\{t\geq0:X_{t}^{\varepsilon}\in\partial D_{j}^{\eta}(s)\}= \inf_{z\in\partial D_{j}^{\eta}(s)} V(O_{j},z)
\geq s-\eta.
\end{eqnarray*}
Letting $\eta\downarrow0$ and then $s\uparrow\bar{q}_{j}$ gives $\liminf
_{\varepsilon\rightarrow0}\varepsilon\log E_{x}\upsilon_{j}^{\varepsilon}%
\geq\bar{q}_{j}$.

For the reverse inequality we also adapt an argument from the proof of
\cite[Theorem 4.1, Chapter 4]{frewen2}. One can find $T_{1}<\infty$ such that
the probability for $X_{t}^{\varepsilon}$ to reach $\cup_{l\in L}B_{\delta
}(O_{l})$ by time $T_{1}$ from any $x\in M\setminus\cup_{l\in L}B_{\delta
}(O_{l})$ is bounded below by $1/2$. (This follows easily from the law of
large numbers and that all trajectories of the noiseless system reach
$\cup_{l\in L}B_{\delta/2}(O_{l})$ in some finite time that is bounded
uniformly in $x\in M\setminus\cup_{l\in L}B_{\delta}(O_{l})$.) Also, given
$\eta>0$ there is $T_{2}<\infty$ and $\varepsilon_{0}>0$ such that
$P_{x}\{X_{t}^\varepsilon$ reaches $\cup_{k\in L\setminus\{j\}}\partial B_{\delta}(O_{k})$ before $T_{2}\}\geq
\exp-(\bar{q}_{j}+\eta)/\varepsilon$ for all $x\in\partial B_{\delta}(O_{j})$.
It then follows from the strong Markov property that for any $x\in
M\setminus\cup_{l\in L}B_{\delta}(O_{l})$%
\[
P_{x}\{\upsilon_{j}^{\varepsilon}\leq T_{1}+T_{2}\}\geq e^{-\frac
{1}{\varepsilon}(\bar{q}_{j}+\eta)}/2.
\]
Using the ordinary Markov property we have
\begin{align*}
E_{x}\upsilon_{j}^{\varepsilon}  &  \leq\sum\nolimits_{n=0}^{\infty}(n+1)(T_{1}%
+T_{2})P_{x}\{n(T_{1}+T_{2})<\upsilon_{j}^{\varepsilon}\leq(n+1)(T_{1}%
+T_{2})\}\\
&  =(T_{1}+T_{2})\sum\nolimits_{n=0}^{\infty}P_{x}\{\upsilon_{j}^{\varepsilon}%
>n(T_{1}+T_{2})\}\\
&  \leq(T_{1}+T_{2})\sum\nolimits_{n=0}^{\infty}\left[  1-\inf_{x\in M\setminus
\cup_{l\in L}B_{\delta}(O_{l})}P_{x}\{\upsilon_{j}^{\varepsilon}\leq
T_{1}+T_{2}\}\right]  ^{n}\\
&  =(T_{1}+T_{2})\left(  \inf_{x\in M\setminus\cup_{l\in L}B_{\delta}(O_{l}%
)}P_{x}\{\upsilon_{j}^{\varepsilon}\leq T_{1}+T_{2}\}\right)  ^{-1}\\
&  \leq2(T_{1}+T_{2})e^{\frac{1}{\varepsilon}(\bar{q}_{j}+\eta)}.
\end{align*}
Thus $\limsup_{\varepsilon\rightarrow0}\varepsilon\log E_{x}\upsilon
_{j}^{\varepsilon}\leq\bar{q}_{j}+\eta$, and letting $\eta\downarrow0$
completes the proof.
\end{proof}

\begin{remark}
By the standard Freidlin-Wentzell theory, the convergence asserted in Lemma \ref{Lem:9.1} is  uniform on $\partial B_{\delta}(O_j)$. Therefore, we have the first part of Lemma \ref{Lem:9.0}.
\end{remark}

\subsection{Asymptotically exponential distribution}

\label{subsec:9.2} 

\begin{lemma}
\label{Lem:9.2} For each $j\in L_{\rm{s}}$ there is a distribution $u^{\varepsilon
}$ on $\partial B_{2\delta}(O_{j})$ such that 
$
\upsilon_{j}^{\varepsilon}/E_{u^{\varepsilon}}\upsilon_{j}%
^{\varepsilon}\overset{d}{\rightarrow}\rm{Exp}(1).
$

\end{lemma}

\begin{proof}
To simplify notation and since it plays no role, we write $j=1$ throughout the
proof. We call $\partial B_{\delta}\left(  O_{1}\right)  $ and $\partial
B_{2\delta}\left(  O_{1}\right)  $ the inner and outer rings of $O_{1}$.
We can then decompose the hitting time as
\begin{equation}
\upsilon_{1}^{\varepsilon}=\sum\nolimits_{k=1}^{\mathcal{N}^{\varepsilon}-1}\theta
_{k}^{\varepsilon}+\zeta^{\varepsilon}, \label{eqn_2}%
\end{equation}
where $\theta_{k}^{\varepsilon}$ is the $k$-th amount of time that the process
travels from the outer ring to the inner ring and back without visiting
$\cup_{j\in L\setminus\{1\}}\partial B_{\delta}(O_{j})$, $\zeta^{\varepsilon
}$ is the amount of time that the process travels from the outer ring directly
to $\cup_{j\in L\setminus\{1\}}\partial B_{\delta}(O_{j})$ without visiting
the inner ring, and $\mathcal{N}^{\varepsilon}-1$ is the number of times that
the process goes back and forth between the inner ring and outer ring. (It is
assumed that $\delta>0$ is small enough that $B_{2\delta}\left(  O_{1}\right)
\subset M\setminus\cup_{j\in L\setminus\{1\}} B_{2\delta}(O_{j})$.)
Note that $\theta_{k}^{\varepsilon}$ grows exponentially of the order $\delta
$, due to the time taken to travel from the inner ring to the outer ring, and
$\zeta^{\varepsilon}$ is uniformly bounded in expected value.

For any set $A,$ define the first hitting time by 
$
\tau\left(  A\right)  \doteq\inf\left\{  t>0:X_{t}^{\varepsilon}\in A\right\}
.
$
Consider the conditional transition probability from $x\in \partial B_{2\delta}\left(  O_{1}\right)$ to $y\in
\partial B_{\delta}\left(  O_{1}\right)$ given by
\[
\psi_{1}^{\varepsilon}\left(  dy|x\right)  \doteq P\left(  X_{\tau\left(
\partial B_{\delta}\left(  O_{1}\right)\right)  }^{\varepsilon}\in dy|X_{0}^{\varepsilon}=x,\text{ }%
X_{t}^{\varepsilon}\notin\cup_{j\in L\setminus\{1\}}\partial B_{\delta}%
(O_{j}),t\in\lbrack0,\tau(\partial B_{\delta}\left(  O_{1}\right)))]\right)  ,
\]
and the transition probability from $y\in \partial B_{\delta}\left(  O_{1}\right)$ to $x\in \partial B_{2\delta}\left(  O_{1}\right)$ given by
\begin{equation}\label{eqn:defpsi2}
    \psi_{2}^{\varepsilon}\left(  dx|y\right)  \doteq P\left(  X_{\tau\left(
\partial B_{2\delta}\left(  O_{1}\right)\right)  }^{\varepsilon}\in dx|X_{0}^{\varepsilon}=y\right)  .
\end{equation}
Then we can create a transition probability from $x\in \partial B_{2\delta}\left(  O_{1}\right)$ to $y\in \partial B_{2\delta}\left(  O_{1}\right)$
by
\begin{equation}\label{eqn:defofpsi}
  \psi^{\varepsilon}\left(  dy|x\right)  \doteq\int_{\partial B_{\delta}\left(  O_{1}\right)}\psi_{2}%
^{\varepsilon}\left(  dy|z\right)  \psi_{1}^{\varepsilon}\left(  dz|x\right)
.  
\end{equation}
Since $\partial B_{2\delta}\left(  O_{1}\right)$ is compact and $\{X_{t}^{\varepsilon}\}_{t}$ is non-degenerate
and Feller, there exists an invariant measure $u^{\varepsilon}\in
\mathcal{P}\left(  \partial B_{2\delta}\left(  O_{1}\right)\right)  $ with respect to the transition probability
$\psi^{\varepsilon}\left(  dy|x\right)  .$ If we start with the distribution
$u^{\varepsilon}$ on $\partial B_{2\delta}\left(  O_{1}\right)$, then it follows from the definition of
$u^{\varepsilon}$ and the strong Markov property that the $\{\theta
_{k}^{\varepsilon}\}_{k<\mathcal{N}^{\varepsilon}}$ are iid. Moreover, the
indicators of escape (i.e., $1_{\{\tau(\cup_{j\in L\setminus\{1\}}\partial
B_{\delta}(O_{j}))=\tau(\cup_{j\in L}\partial B_{\delta}(O_{j}))\}}$) are iid
Bernoulli, and we write them as $Y_{k}^{\varepsilon}$ with 
$
P_{u^{\varepsilon}}(Y_{k}^{\varepsilon}=1)=e^{-h_{1}^{\varepsilon}%
(\delta)/\varepsilon},
$
where $\delta>0$ is from the construction, $h_{1}^{\varepsilon}(\delta
)\rightarrow h_{1}(\delta)$ as $\varepsilon\rightarrow0$ and $h_{1}%
(\delta)\uparrow h_{1}$ as $\delta\downarrow0$ with $h_{1}=\min_{j\in
L\setminus\{1\}}V(O_{1},O_{j})$. Note that 
$
\mathcal{N}^{\varepsilon}=\inf\left\{  k\in%
\mathbb{N}
:Y_{k}^{\varepsilon}=1\right\}  .
$
We therefore have
\[
P_{u^{\varepsilon}}(\mathcal{N}^{\varepsilon}=k)=(1-e^{-h_{1}^{\varepsilon
}(\delta)/\varepsilon})^{k-1}e^{-h_{1}^{\varepsilon}(\delta)/\varepsilon},
\]
and thus
\[
E_{u^{\varepsilon}}\upsilon_{1}^{\varepsilon}=E_{u^{\varepsilon}}\left[
\sum\nolimits_{j=1}^{\mathcal{N}^{\varepsilon}-1}\theta_{j}^{\varepsilon}\right]
+E_{u^{\varepsilon}}\zeta^{\varepsilon}=E_{u^{\varepsilon}}(\mathcal{N}%
^{\varepsilon}-1)E_{u^{\varepsilon}}\theta_{1}^{\varepsilon}+E_{u^{\varepsilon
}}\zeta^{\varepsilon},
\]
where the second equality comes from Wald's identity.
Using $\sum_{k=1}^{\infty}ka^{k-1}
={1/(1-a)^{2}}$ for $a \in [0,1)$, we also have
\begin{align*}
E_{u^{\varepsilon}}\mathcal{N}^{\varepsilon}    =\sum\nolimits_{k=1}^{\infty
}k(1-e^{-h_{1}^{\varepsilon}(\delta)/\varepsilon})^{k-1}e^{-h_{1}%
^{\varepsilon}(\delta)/\varepsilon}
=e^{-h_{1}^{\varepsilon}%
(\delta)/\varepsilon}e^{2h_{1}^{\varepsilon}(\delta)/\varepsilon}%
=e^{h_{1}^{\varepsilon}(\delta)/\varepsilon},
\end{align*}
and therefore
\begin{equation}
E_{u^{\varepsilon}}\upsilon_{1}^{\varepsilon}=e^{h_{1}^{\varepsilon}%
(\delta)/\varepsilon}E_{u^{\varepsilon}}\theta_{1}^{\varepsilon}%
+(E_{u^{\varepsilon}}\zeta^{\varepsilon}-E_{u^{\varepsilon}}\theta
_{1}^{\varepsilon}). \label{eqn_1}%
\end{equation}

Next consider the characteristic function of $\upsilon_{1}^{\varepsilon
}/E_{u^{\varepsilon}}\upsilon_{1}^{\varepsilon}$
\[
\phi^{\varepsilon}(t)=E_{u^{\varepsilon}}e^{it\upsilon_{1}^{\varepsilon
}/E_{u^{\varepsilon}}\upsilon_{1}^{\varepsilon}}=\phi_{\upsilon}^{\varepsilon
}(t/E_{u^{\varepsilon}}\upsilon_{1}^{\varepsilon}),
\]
where $\phi_{\upsilon}^{\varepsilon}$ is the characteristic function of
$\upsilon_{1}^{\varepsilon}.$ By (\ref{eqn_2}), we have
\begin{align*}
\phi_{\upsilon}^{\varepsilon}(s)  &  =E_{u^{\varepsilon}}e^{is\left(
\sum_{k=1}^{\mathcal{N}^{\varepsilon}-1}\theta_{k}^{\varepsilon}%
+\zeta^{\varepsilon}\right)  }
  =E_{u^{\varepsilon}}e^{is\zeta^{\varepsilon}}E_{u^{\varepsilon}%
}e^{is\left(  \sum\nolimits_{k=1}^{\mathcal{N}^{\varepsilon}-1}\theta_{k}^{\varepsilon
}\right)  }\\
&  =\phi_{\zeta}^{\varepsilon}(s)\sum\nolimits_{k=1}^{\infty}(1-e^{-h^{\varepsilon}%
_{1}(\delta)/\varepsilon})^{k-1}e^{-h^{\varepsilon}_{1}(\delta)/\varepsilon
}\phi_{\theta}^{\varepsilon}(s)^{k-1}\\
&  =\phi_{\zeta}^{\varepsilon}(s)e^{-h^{\varepsilon}_{1}(\delta)/\varepsilon
} (1-[  (1-e^{-h^{\varepsilon}_{1}(\delta)/\varepsilon}%
)\phi_{\theta}^{\varepsilon}(s)]  )^{-1},
\end{align*}
where $\phi_{\theta}^{\varepsilon}$ and $\phi_{\zeta}^{\varepsilon}$ are the
characteristic functions of $\theta_{1}^{\varepsilon}$ and $\zeta
^{\varepsilon}$, respectively. We want to show that for any $t\in\mathbb{R}$
\[
\phi^{\varepsilon}(t)=\phi_{\upsilon}^{\varepsilon}(t/E_{u^{\varepsilon}}\upsilon_{1}^{\varepsilon
})\rightarrow 1/(1-it)\text{ as }\varepsilon\rightarrow0.
\]

We first show that $\phi_{\zeta}^{\varepsilon}(t/E_{u^{\varepsilon}}%
\upsilon_{1}^{\varepsilon})\rightarrow1.$ By definition, 
$
\phi_{\zeta}^{\varepsilon}\left(  t/E_{u^{\varepsilon}\upsilon
_{1}^{\varepsilon}}\right)
=E_{u^{\varepsilon}}\cos\left(  t\zeta^{\varepsilon}/E_{u^{\varepsilon
}}\upsilon_{1}^{\varepsilon}\right)  +iE_{u^{\varepsilon}}\sin\left(
t\zeta^{\varepsilon}/E_{u^{\varepsilon}}\upsilon_{1}^{\varepsilon}%
\right)  .
$ 
According to \cite[Lemma 1.9, Chapter 6]{frewen2}, we know that there exist $T_{0}\in(0,\infty)$ and $\beta>0$ such that for any
$T\in(T_0,\infty)$ and for all $\varepsilon$ sufficiently small
\begin{equation}
\label{eqn:LTFW}P_{u^{\varepsilon}}\left(  \zeta^{\varepsilon}>T\right)  \leq
e^{-\frac{1}{\varepsilon}\beta\left(  T-T_{0}\right)  },
\end{equation}
and therefore for any bounded and continuous function $f:\mathbb{R}
\rightarrow\mathbb{R}$
\begin{align*}
 \left\vert E_{u^{\varepsilon}}f\left(  t\zeta^{\varepsilon}/E_{u^{\varepsilon}}\upsilon_{1}^{\varepsilon}\right)  -f\left(  0\right)
\right\vert 
 \leq2\left\Vert f\right\Vert _{\infty}P_{u^{\varepsilon}}\left(
\zeta^{\varepsilon}>T\right)  +E_{u^{\varepsilon}}\left[  \left\vert f\left(
t\zeta^{\varepsilon}/E_{u^{\varepsilon}}\upsilon_{1}^{\varepsilon}%
\right)  -f\left(  0\right)  \right\vert 1_{\left\{  \zeta^{\varepsilon}\leq
T\right\}  }\right]  .
\end{align*}
The first term in the last display goes to $0$ as $\varepsilon\rightarrow0$.
For any fixed $t,$ $t/E_{u^{\varepsilon}}\upsilon_{1}^{\varepsilon}%
\rightarrow0$ as $\varepsilon\rightarrow0$. Since $f$ is continuous, the
second term in the last display also converges to $0$ as $\varepsilon
\rightarrow0.$
$\phi_{\zeta}^{\varepsilon}(t/E_{u^{\varepsilon}}\upsilon_{1}^{\varepsilon
})\rightarrow1$ follows by taking $f$ to be $\sin x$ and $\cos x.$

It remains to show that for any $t\in\mathbb{R}$
\[
e^{-h^{\varepsilon}_{1}(\delta)/\varepsilon} \left(1-\left[
(1-e^{-h^{\varepsilon}_{1}(\delta)/\varepsilon})\phi_{\theta}^{\varepsilon
}(t/E_{u^{\varepsilon}}\upsilon_{1}^{\varepsilon})\right]  \right)^{-1}\rightarrow
1/(1-it)%
\]
as $\varepsilon\rightarrow0.$ Observe that
\[
e^{-h^{\varepsilon}_{1}(\delta)/\varepsilon} \left(1-\left[
(1-e^{-h^{\varepsilon}_{1}(\delta)/\varepsilon})\phi_{\theta}^{\varepsilon
}(t/E_{u^{\varepsilon}}\upsilon_{1}^{\varepsilon})\right]  \right)^{-1}=\left(
{\frac{1-\phi_{\theta}^{\varepsilon}(t/E_{u^{\varepsilon}}\upsilon
_{1}^{\varepsilon})}{e^{-h^{\varepsilon}_{1}(\delta)/\varepsilon}}%
+\phi_{\theta}^{\varepsilon}(t/E_{u^{\varepsilon}}\upsilon_{1}^{\varepsilon}}%
)\right)^{-1},
\]
so it suffices to show that $\phi_{\theta}^{\varepsilon}(t/E_{u^{\varepsilon}%
}\upsilon_{1}^{\varepsilon})\rightarrow1$ and $[1-\phi_{\theta}^{\varepsilon
}(t/E_{u^{\varepsilon}}\upsilon_{1}^{\varepsilon})]/e^{-h^{\varepsilon}%
_{1}(\delta)/\varepsilon}\rightarrow-it$ as $\varepsilon\rightarrow0.$

For the former, note that by (\ref{eqn_1})
\[
0\leq E_{u^{\varepsilon}}\left(  t\theta_{1}^{\varepsilon}/E_{u^{\varepsilon}%
}\upsilon_{1}^{\varepsilon}\right)  \leq\frac{tE_{u^{\varepsilon}}\theta
_{1}^{\varepsilon}}{\left(  e^{h^{\varepsilon}_{1}(\delta)/\varepsilon
}-1\right)  E_{u^{\varepsilon}}\theta_{1}^{\varepsilon}}\rightarrow0
\]
as $\varepsilon\rightarrow0,$ and so $t\theta_{1}^{\varepsilon}%
/E_{u^{\varepsilon}}\upsilon_{1}^{\varepsilon}$ converges to $0$ in
distribution.
Moreover, since $e^{ix}$ is bounded and continuous, we find
$\phi_{\theta}^{\varepsilon}(t/E_{u^{\varepsilon}}\upsilon_{1}^{\varepsilon
})
\rightarrow1$.
For the second part, using
\[
x-{x^{3}}/{3!}\leq\sin x\leq x\text{ and }1-{x^{2}}/{2}\leq\cos
x\leq1
\]
for $x\in\mathbb{R}$ we find that
\[
0\leq\frac{1-E_{u^{\varepsilon}}\cos\left(  t\theta_{1}^{\varepsilon
}/E_{u^{\varepsilon}}\upsilon_{1}^{\varepsilon}\right)  }{e^{-h_{1}%
^{\varepsilon}(\delta)/\varepsilon}}\leq\frac{E_{u^{\varepsilon}}\left(
t\theta_{1}^{\varepsilon}/E_{u^{\varepsilon}}\upsilon_{1}^{\varepsilon
}\right)  ^{2}}{2e^{-h_{1}^{\varepsilon}(\delta)/\varepsilon}}%
\]
and
\[
\frac{E_{u^{\varepsilon}}\left(  t\theta_{1}^{\varepsilon}/E_{u^{\varepsilon}%
}\upsilon_{1}^{\varepsilon}\right)  }{e^{-h_{1}^{\varepsilon}(\delta
)/\varepsilon}}-\frac{E_{u^{\varepsilon}}\left(  t\theta_{1}^{\varepsilon
}/E_{u^{\varepsilon}}\upsilon_{1}^{\varepsilon}\right)  ^{3}}{3!e^{-h_{1}%
^{\varepsilon}(\delta)/\varepsilon}}\leq\frac{E_{u^{\varepsilon}}\sin\left(
t\theta_{1}^{\varepsilon}/E_{u^{\varepsilon}}\upsilon_{1}^{\varepsilon
}\right)  }{e^{-h_{1}^{\varepsilon}(\delta)/\varepsilon}}\leq\frac
{E_{u^{\varepsilon}}\left(  t\theta_{1}^{\varepsilon}/E_{u^{\varepsilon}%
}\upsilon_{1}^{\varepsilon}\right)  }{e^{-h_{1}^{\varepsilon}(\delta
)/\varepsilon}}.
\]
From our previous observation regarding the distribution of $\zeta
^{\varepsilon}$
and \eqref{eqn_1}
\[
\frac{E_{u^{\varepsilon}}\left(  t\theta_{1}^{\varepsilon}/E_{u^{\varepsilon}%
}\upsilon_{1}^{\varepsilon}\right)  }{e^{-h_{1}^{\varepsilon}(\delta
)/\varepsilon}}\rightarrow t\text{ as }\varepsilon\rightarrow0.
\]
In addition, since $\theta_{1}^{\varepsilon}$ can be viewed as the time from the outer ring to the inner ring without visiting $\cup_{j\in L\setminus\{1\}}\partial B_{\delta}(O_{j})$ plus the time from the inner ring to the outer ring, by applying \eqref{eqn:LTFW} to the former and using \cite[Theorem 4 and Corollary 1]{day4} under Condition
\ref{Con:3.3} to the later, we find that 
\begin{equation}
\label{tail_bound}
P_{u^{\varepsilon}}\left(  \theta_{1}^{\varepsilon}/E_{u^{\varepsilon}%
}\theta_{1}^{\varepsilon}>t\right)  \leq 2e^{-t}%
\end{equation}
for all $t\in\lbrack0,\infty)$ and $\varepsilon$ sufficiently small. This
implies that
\begin{align*}
E_{u^{\varepsilon}}\left(  \theta_{1}^{\varepsilon}/E_{u^{\varepsilon}%
}\theta_{1}^{\varepsilon}\right)  ^{2}    =2\int_{0}^{\infty}t^{2}%
P_{u^{\varepsilon}}\left(  \theta_{1}^{\varepsilon}/E_{u^{\varepsilon}%
}\theta_{1}^{\varepsilon}>t\right)  dt \leq4\int_{0}^{\infty}t^{2}e^{-t}dt=8
\end{align*}
and similarly $E_{u^{\varepsilon}}\left(  \theta_{1}^{\varepsilon}/E_{u^{\varepsilon}%
}\theta_{1}^{\varepsilon}\right)  ^{3}    =3\int_{0}^{\infty}t^{3}\leq 36$.
Then combined with \eqref{eqn_1}, we have
\[
0\leq\frac{E_{u^{\varepsilon}}\left(  t\theta_{1}^{\varepsilon}/E_{u^{\varepsilon}%
}\upsilon_{1}^{\varepsilon}\right)  ^{2}}{2e^{-h_{1}^{\varepsilon}%
(\delta)/\varepsilon}}
\leq 
\frac{t^2 E_{u^{\varepsilon}}\left(  \theta_{1}^{\varepsilon}/E_{u^{\varepsilon}%
}\theta_{1}^{\varepsilon}\right)  ^{2}}{2e^{-h_{1}^{\varepsilon}%
(\delta)/\varepsilon}(e^{h_{1}^{\varepsilon}%
(\delta)/\varepsilon}-1)^2}
\rightarrow0
\]
and 
\[
0\leq\frac{E_{u^{\varepsilon}}\left(  t\theta_{1}^{\varepsilon}/E_{u^{\varepsilon}%
}\upsilon_{1}^{\varepsilon}\right)  ^{3}}{3!e^{-h_{1}^{\varepsilon}%
(\delta)/\varepsilon}}
\leq 
\frac{t^3 E_{u^{\varepsilon}}\left(  \theta_{1}^{\varepsilon}/E_{u^{\varepsilon}%
}\theta_{1}^{\varepsilon}\right)  ^{3}}{3!e^{-h_{1}^{\varepsilon}%
(\delta)/\varepsilon}(e^{h_{1}^{\varepsilon}%
(\delta)/\varepsilon}-1)^3}
\rightarrow0.
\]

Therefore, we have shown that for any $t\in\mathbb{R}$
\[
\frac{1-\phi_{\theta}^{\varepsilon}(t/E_{u^{\varepsilon}}\upsilon
_{1}^{\varepsilon})}{e^{-h^{\varepsilon}_{1}(\delta)/\varepsilon}}%
=\frac{1-E_{u^{\varepsilon}}\cos\left(  t\theta_{1}^{\varepsilon
}/E_{u^{\varepsilon}}\upsilon_{1}^{\varepsilon}\right)  }{e^{-h^{\varepsilon
}_{1}(\delta)/\varepsilon}}-i\frac{E_{u^{\varepsilon}}\sin\left(  t\theta
_{1}^{\varepsilon}/E_{u^{\varepsilon}}\upsilon_{1}^{\varepsilon}\right)
}{e^{-h^{\varepsilon}_{1}(\delta)/\varepsilon}}\rightarrow-it.
\]
\end{proof}
\begin{remark}
\label{Rmk:9.1}
From the proof of Lemma \ref{Lem:9.2}, we actually know that
\[
    \phi_{\upsilon}^{\varepsilon}(t/E_{u^{\varepsilon}}\upsilon_{1}^{\varepsilon
    })\rightarrow 1/(1-it)
\]
uniformly on any compact set in $\mathbb{R}$ as $\varepsilon\rightarrow 0$.
\end{remark}
\vspace{\baselineskip}

\subsection{Tail probability}

\label{subsec:9.3} The goal of this subsection is to prove the following.

\begin{lemma}
\label{Lem:9.3} For each $j\in L_{\rm{s}}$ there is a distribution $u^{\varepsilon
}$ on $\partial B_{2\delta}(O_{j})$ and $\tilde{c}>0$ such that for any
$t\in[0,\infty)$, 
$
P_{u^{\varepsilon}}(  \upsilon_{j}^{\varepsilon}/E_{u^{\varepsilon
}}\upsilon_{j}^{\varepsilon}>t)  \leq e^{-\tilde{c}t}
$
(here $\upsilon_{j}^{\varepsilon}$ and $u^{\varepsilon}$ are defined as in the
last subsection).
\end{lemma}

\begin{proof}
As in the last subsection we give the proof for the case $j=1$. To begin we
note that for any $\alpha>0$ Chebyshev's inequality implies
\[
P_{u^{\varepsilon}}\left(  \upsilon_{1}^{\varepsilon}/E_{u^{\varepsilon
}}\upsilon_{1}^{\varepsilon}>t\right)  =P_{u^{\varepsilon}}(
e^{\alpha\upsilon_{1}^{\varepsilon}/E_{u^{\varepsilon}}\upsilon
_{1}^{\varepsilon}}>e^{\alpha t})  \leq e^{-\alpha t}\cdot E_{u^{\varepsilon}%
}e^{\alpha\upsilon_{1}^{\varepsilon}/E_{u^{\varepsilon}}\upsilon
_{1}^{\varepsilon}}.
\]
By picking $\alpha = \alpha^{\ast}\doteq1/8$, it suffices to show that
$E_{u^{\varepsilon}}e^{{\alpha^{\ast}\upsilon_{1}^{\varepsilon}}/{E_{u^{\varepsilon
}}\upsilon_{1}^{\varepsilon}}}$ is bounded by a constant.
We will do this by showing how the finiteness of $E_{u^{\varepsilon}%
}e^{{\alpha^{\ast}\upsilon_{1}^{\varepsilon}}/{E_{u^{\varepsilon}}\upsilon
_{1}^{\varepsilon}}}$ is implied by the finiteness of $E_{u^{\varepsilon
}}e^{{\alpha^{\ast}\theta_{1}^{\varepsilon}}/{E_{u^{\varepsilon}}\upsilon
_{1}^{\varepsilon}}}$ and $E_{u^{\varepsilon}}e^{{\alpha^{\ast}\zeta^{\varepsilon}%
}/{E_{u^{\varepsilon}}\upsilon_{1}^{\varepsilon}}}.$

Using (\ref{tail_bound}) we find that for any $\alpha>0$
\[
P_{u^{\varepsilon}}(  e^{\alpha\theta_{1}^{\varepsilon}/E_{u^{\varepsilon
}}\theta_{1}^{\varepsilon}}>t)  \leq2e^{-\frac{1}{\alpha}\log
t}=2t^{-\frac{1}{\alpha}}%
\]
for all $t\in\lbrack1,\infty)$ and $\varepsilon$ sufficiently small.
Then (\ref{eqn_1}) implies $E_{u^{\varepsilon}}\upsilon_{1}^{\varepsilon}
\geq\left(  e^{h^{\varepsilon}_{1}\left(  \delta\right)  /\varepsilon
}-1\right)  E_{u^{\varepsilon}}\theta_{1}^{\varepsilon}$ and therefore
\begin{align*}
E_{u^{\varepsilon}}e^{\alpha^{\ast}/E_{u^{\varepsilon}}\upsilon
_{1}^{\varepsilon}}\theta_{1}^{\varepsilon}  &  \leq 
 \int_{0}^{1}P_{u^{\varepsilon}}\left(  \exp\left(  \alpha^{\ast}%
\theta_{1}^{\varepsilon}/[{(e}^{h^{\varepsilon}_{1}\left(  \delta\right)
/\varepsilon}-1{)}E_{u^{\varepsilon}}\theta_{1}^{\varepsilon}]\right)
>t\right)  dt\\
&  \quad +\int_{1}^{\infty}P_{u^{\varepsilon}}\left(  \exp\left( \alpha^{\ast
}\theta_{1}^{\varepsilon}/[{(e}^{h^{\varepsilon}_{1}\left(  \delta\right)
/\varepsilon}-1{)}E_{u^{\varepsilon}}\theta_{1}^{\varepsilon}]\right)
>t\right)  dt\\
&  \leq1+2\int_{1}^{\infty}t^{-{(e}^{h^{\varepsilon}_{1}\left(  \delta\right)
/\varepsilon}-1{)/\alpha}^{\ast}}dt\\
&  =1+2[{(e}^{h^{\varepsilon}_{1}\left(  \delta\right)  /\varepsilon
}-1{)/\alpha}^{\ast}-1]^{-1}=1+2\alpha^{\ast}[{e}^{h^{\varepsilon}_{1}\left(
\delta\right)  /\varepsilon}-\alpha^{\ast}-1]^{-1}.
\end{align*}

To estimate $\zeta^{\varepsilon},$ we use that by (\ref{eqn:LTFW}) there are
$T_{0}\in (0,\infty)$ and $\beta>0$ such that for any $t\in (T_0,\infty)$ and for all $\varepsilon$
sufficiently small
$
P_{u^{\varepsilon}}\left(  \zeta^{\varepsilon}>t\right)  \leq e^{-\frac
{1}{\varepsilon}\beta\left(  t-T_{0}\right)  },
$
 so that for any $\alpha>0$ 
 $
P_{u^{\varepsilon}}\left(  e^{\alpha\zeta^{\varepsilon}}>t\right)  \leq
e^{-\frac{1}{\varepsilon}\beta\left(  \frac{1}{\alpha}\log t-T_{0}\right)  }%
$ 
for any $t\geq e^{\alpha T_{0}}.$ Given $n\in%
\mathbb{N}
,$ for all sufficiently small $\varepsilon$ we have $\alpha^{\ast}%
/E_{u^{\varepsilon}}\upsilon_{1}^{\varepsilon}\leq1/n$, and thus
\[
P_{u^{\varepsilon}}\left(  e^{\alpha^{\ast}\zeta^{\varepsilon}/E_{u^{\varepsilon}}\upsilon
_{1}^{\varepsilon}}>t\right)  \leq P_{u^{\varepsilon}%
}\left(  e^{\zeta^{\varepsilon}/n}>t\right)  \leq e^{-\frac{1}{\varepsilon
}\beta\left(  n\log t-T_{0}\right)  }.
\]
Hence for any $n$ such that $e^{T_{0}/n}\leq3/2$ and $\left(  -\beta n+1\right)
\log\left(  3/2\right)  +\beta T_{0}<0,$ and for $\varepsilon$ small enough that
$\alpha^{\ast}/E_{u^{\varepsilon}}\upsilon_{1}^{\varepsilon}\leq1/n,$ we have%
\begin{align*}
E_{u^{\varepsilon}}e^{\alpha^{\ast}\zeta^{\varepsilon}/E_{u^{\varepsilon}}\upsilon
_{1}^{\varepsilon}}  &  \leq\int_{0}^{\infty
}P_{u^{\varepsilon}}\left(  e^{\alpha^{\ast}\zeta^{\varepsilon}/E_{u^{\varepsilon}}\upsilon
_{1}^{\varepsilon}}>t\right)  dt
 \leq 3/2+\int_{\frac{3}{2}}^{\infty}P_{u^{\varepsilon}}\left(
e^{\alpha^{\ast}\zeta^{\varepsilon}/E_{u^{\varepsilon}}\upsilon
_{1}^{\varepsilon}}>t\right)  dt\\
&  \leq 3/2+\int_{\frac{3}{2}}^{\infty}e^{-\frac{1}{\varepsilon}%
\beta\left(  n\log t-T_{0}\right)  }dt
 =3/2+e^{\frac{1}{\varepsilon}\beta T_{0}}(\beta n/{\varepsilon}-1)^{-1}\left(  3/2\right)  ^{\frac{1}%
{\varepsilon}\left(  -\beta n+\varepsilon\right)  }\\
&  =3/2+(\beta n/{\varepsilon}-1)^{-1}e^{\frac
{1}{\varepsilon}\left[  \left(  -\beta  n+\varepsilon\right)  \log\left(  3/2\right)
+\beta  T_{0}\right]  }
 \leq 3/2+(\beta n/{\varepsilon}-1)^{-1} \leq2.
\end{align*}
We have shown that for such $\alpha^{\ast},$ $E_{u^{\varepsilon}}e^{\alpha^{\ast}\zeta^{\varepsilon}
/E_{u^{\varepsilon}}\upsilon_{1}^{\varepsilon}}$ and
$E_{u^{\varepsilon}}e^{\alpha^{\ast}\theta_{1}^{\varepsilon}/E_{u^{\varepsilon}}\upsilon
_{1}^{\varepsilon}}$ are uniformly bounded for all
$\varepsilon$ sufficiently small. Lastly, using the same calculation as used
for the characteristic function%
\begin{align*}
  E_{u^{\varepsilon}}e^{\alpha^{\ast}\upsilon_{1}^{\varepsilon}%
/E_{u^{\varepsilon}}\upsilon_{1}^{\varepsilon}}
&  =E_{u^{\varepsilon}}e^{\alpha^{\ast}\zeta^{\varepsilon}/E_{u^{\varepsilon}}\upsilon
_{1}^{\varepsilon}}\cdot e^{-h^{\varepsilon}_{1}(\delta)/\varepsilon}%
\left(1-\left[  (1-e^{-h^{\varepsilon}_{1}(\delta)/\varepsilon})E_{u^{\varepsilon}}%
e^{\alpha^{\ast}\theta
_{1}^{\varepsilon}/E_{u^{\varepsilon}}\upsilon_{1}^{\varepsilon}}\right]  \right)^{-1}\\
&  \leq2e^{-h^{\varepsilon}_{1}(\delta)/\varepsilon}\left(1-\left[  (1-e^{-h^{\varepsilon}_{1}(\delta
)/\varepsilon})\left(  1+\frac{2\alpha^{\ast}}{{e}^{h^{\varepsilon}_{1}\left(  \delta\right)
/\varepsilon}-\alpha^{\ast}-1}\right)  \right]  \right)^{-1}\\
&  =2e^{-h^{\varepsilon}_{1}(\delta)/\varepsilon}\left(e^{-h^{\varepsilon}_{1}(\delta)/\varepsilon}%
-\frac{2\alpha^{\ast}}{{e}^{h^{\varepsilon}_{1}\left(  \delta\right)  /\varepsilon}-\alpha^{\ast}-1}%
+\frac{2\alpha^{\ast}}{{e}^{h^{\varepsilon}_{1}\left(  \delta\right)  /\varepsilon}-\alpha^{\ast}%
-1}e^{-h^{\varepsilon}_{1}(\delta)/\varepsilon}\right)^{-1}\\
&  =2\left(1-2\alpha^{\ast}\frac{e^{h^{\varepsilon}_{1}(\delta)/\varepsilon}-1}{{e}^{h^{\varepsilon}_{1}\left(
\delta\right)  /\varepsilon}-\alpha^{\ast}-1}\right)^{-1} \leq 2/(1-4\alpha^{\ast})=4.
\end{align*}
\end{proof}

\subsection{General initial condition}

\label{subsec:9.4} This subsection presents results that will allow us to extend
the results in the previous two subsections to arbitrary initial distribution $\lambda^{\varepsilon}\in \mathcal{P}(\partial B_{\delta}\left(  O_{1}\right))$. Under our
assumptions, for any $j\in L_{\rm{s}}$ we observe that the process model
\begin{equation}
dX_{t}^{\varepsilon}=b\left(  X_{t}^{\varepsilon}\right)  dt+\sqrt
{\varepsilon}\sigma\left(  X_{t}^{\varepsilon}\right)  dW_{t}
\label{eqn_process}%
\end{equation}
has the property that $b(x)=A(x-O_{j})[1+o(1)]$ and $\sigma\left(  x\right)
=\bar{\sigma}[1+o\left(  1\right)  ]$, where $o(1)\rightarrow0$ as $\left\Vert
x-O_{j}\right\Vert \rightarrow0$, $A$ is stable and $\bar{\sigma}$ is
invertible. By an invertible change of variable we can arrange so that
$O_{j}=0$ and $\bar{\sigma}=I$, and to simplify we assume this in the rest of
the section.

Since $A$ is stable there exists a positive definite and symmetric solution
$M$ to the matrix equation $
AM+MA^{T}=-I
$
(we can in fact exhibit the solution in the form $M=\int_{0}^{\infty}%
e^{At}e^{A^{T}t}dt)$. To prove the ergodicity we introduce some additional notation: $ U(x)\doteq \langle x, Mx \rangle$, 
$B_i\doteq \{x:U(x) < b_i^2\}$
and  $\mathcal{S}_{i}(\varepsilon)\doteq \{x:
U(x)< a_{i}^2 \varepsilon\} ,
$
for $i=1,2$, where $0<a_{1}<a_{2}$, $0<b_{0}<b_{1}<b_{2}$. If $\varepsilon_{0}=(b_0^2/a_2^2)/2$ then with cl denoting closure,   
$
{\rm cl}(\mathcal{S}_{2}(\varepsilon_{0})) \subset B_0,
$ 
and we will assume
$\varepsilon\in(0,\varepsilon_{0})$ henceforth. 
For a use later on, we will also assume that 
$
a_1^2 = 2 \sup\nolimits_{x \in B_2}\mbox{tr}[\sigma (x) \sigma(x)^TM].
$

\begin{remark}
The sets $B_1$ and $B_2$ will play the roles that $B_\delta (O_1)$ and $B_{2\delta} (O_1)$ played previously in this section.
Although elsewhere in this paper as well as in the reference \cite{frewen2} these sets  are taken to be balls with respect to the Euclidean norm, 
in this subsection we take them to be level sets of $U(x)$. 
The shape of these sets and the choice of the factor of $2$ relating the radii play no role in the analysis of \cite{frewen2} or in our prior use in this paper.
However in this subsection it is notationally convenient for the sets to be level sets of $U$,
since $U$ is a Lyapunov function for the noiseless dynamics near $0$.
After this subsection we will revert to the $B_\delta (O_1)$ and $B_{2\delta} (O_1)$ notation.
\end{remark}

In addition to the restrictions $a_1<a_2$ and $a_2^2 \varepsilon_0\leq b_0^2$, we also assume that $a_{1},
a_{2}$ and $\varepsilon_{0}>0$ are such that if $\phi^{x}$ is the solution to
the noiseless dynamics $\dot{\phi}=b(\phi)$ with initial condition $x$, then:
(i) for all $x \in\partial\mathcal{S}_{2}(\varepsilon)$, $\phi^{x}$ never
crosses $\partial B_{1}$; (i) for all $x \in\partial\mathcal{S}_{1}(\varepsilon)$,
$\phi^{x}$ never exits $\mathcal{S}_{2}(\varepsilon)$.

The idea that will be used to establish asymptotic independence from the starting distribution is the following. We start the process on $\partial B_{1}$. With some small
probability it will hit $\partial B_{2}$ before hitting $\partial\mathcal{S}_{2}(\varepsilon)$.
This gives a contribution to $\psi_{2}^{\varepsilon}(dz|x)$ defined in 
\eqref{eqn:defpsi2} that will be
relatively unimportant. If instead it hits $\partial\mathcal{S}_{2}(\varepsilon)$
first, then we do a Freidlin-Wentzell type analysis, and decompose the
trajectory into excursions between $\partial\mathcal{S}_{2}(\varepsilon)$ and
$\partial\mathcal{S}_{1}(\varepsilon)$, before a final excursion from
$\partial\mathcal{S}_{2}(\varepsilon)$ to $\partial B_{2}$.

To exhibit the asymptotic independence from $\varepsilon$, we introduce the
scaled process $Y^{\varepsilon}_t=X^{\varepsilon}_t/\sqrt{\varepsilon}$,
which solves the SDE
\[
dY^{\varepsilon}_t=\frac{1}{\sqrt{\varepsilon}}b(\sqrt{\varepsilon
}Y^{\varepsilon}_t)dt+ \sigma(\sqrt{\varepsilon}Y^{\varepsilon}_t)dW_t.
\]
Let 
$
\mathcal{\bar{S}}_{1}=\partial\mathcal{S}_{1}(1) \text{ and }\mathcal{\bar{S}%
}_{2}=\partial\mathcal{S}_{2}(1) .
$ 
Let $\omega^{\varepsilon}(w|x)$ denote the density of the hitting location on
$\mathcal{\bar{S}}_{2}$ by the process $Y^{\varepsilon}$, given
$Y^{\varepsilon}_0=x\in\mathcal{\bar{S}}_{1}$. The following estimate is
essential. 
The density function can be identified with the
normal derivative of a related Green's function, which is bounded from above by the
boundary gradient estimate and bounded below by using the Hopf lemma \cite{giltru}.

\begin{lemma}
\label{Lem:9.4}Given $\varepsilon_{0}>0$, there are $0<c_{1}<c_{2}<\infty$
such that 
$
c_{1}\leq\omega^{\varepsilon}(w|x)\leq c_{2}%
$ 
for all $x\in\mathcal{\bar{S}}_{1}$, $w\in\mathcal{\bar{S}}_{2}$ and
$\varepsilon\in(0,\varepsilon_{0})$.
\end{lemma}

Next let $p^{\varepsilon}(u|w)$ denote the density of the return location for
$Y^{\varepsilon}$ on $\mathcal{\bar{S}}_{2}$, conditioned on visiting
$\mathcal{\bar{S}}_{1}$ before $\partial B_{2}/\sqrt{\varepsilon}$, and starting at $w
\in\mathcal{\bar{S}}_{2}$. The last lemma then directly gives the following.

\begin{lemma}
\label{Lem:9.5}For $\varepsilon_{0}>0$ and $c_{1},c_{2}$ as in the last lemma 
$
c_{1}\leq p^{\varepsilon}(u|w)\leq c_{2}%
$ 
for all $u,w\in\mathcal{\bar{S}}_{2}$ and $\varepsilon\in(0,\varepsilon_{0})$.
\end{lemma}

Let $r^{\varepsilon}(w)$ denote the unique stationary
distribution of $p^{\varepsilon}(u|w)$, and let $p^{\varepsilon,n}(u|w)$
denote the $n$-step transition density. 
The preceding lemma, \cite[Theorem 10.1 Chapter 3]{har7}, and the existence of
a uniform strictly positive lower bound on $r^{\varepsilon}(u)$ for all
sufficiently small $\varepsilon>0$ imply the following.

\begin{lemma}
\label{Lem:9.6}There is $K<\infty$ and
$\alpha\in(0,1)$ such that for all $\varepsilon\in(0,\varepsilon_{0})$
\[
\sup_{w\in\mathcal{\bar{S}}_{2}}\left\vert p^{\varepsilon,n}%
(u|w)-r^{\varepsilon}(u)\right\vert /r^{\varepsilon}(u) \leq K\alpha^{n}.
\]

\end{lemma}

Let $\eta^{\varepsilon}(dx|w)$ denote the
distribution of $X^\varepsilon$ upon first hitting $\partial B_{2}$ given that $X^{\varepsilon}$ reaches $\partial
\mathcal{S}_{1}(\varepsilon)$ before $\partial B_{2}$ and starts at $w \in
\partial\mathcal{S}_{2}(\varepsilon)$.

\begin{lemma}
\label{lem:ble} There is $\kappa>0$ and $\varepsilon_{0}>0$ such that for all
$\varepsilon\in(0,\varepsilon_{0})$ 
\[
\sup_{x \in \partial B_{1}}P_{x}\left\{  X^{\varepsilon}\mbox{ reaches }\partial B_{2}%
\mbox{ before }\mathcal{S}_{2}(\varepsilon)\right\}  \leq e^{-\kappa
/\varepsilon}.
\]

\end{lemma}

\begin{lemma}
\label{lem:blu}
There are $\bar{\eta}^{\varepsilon}(dz)\in$ $\mathcal{P}(\partial B_{2})$, 
$s^{\varepsilon}$ that tends to $0$ as $\varepsilon\rightarrow0$ and $\varepsilon_0>0$, such that
for all $A\in\mathcal{B}(\partial B_{2}),w\in\partial\mathcal{S}_{2}(\varepsilon)$ and $\varepsilon\in(0,\varepsilon_0)$
\[
\bar{\eta}^{\varepsilon}(A)[1-s^{\varepsilon}K/(1-\alpha)]\leq\eta
^{\varepsilon}(A|w)\leq\bar{\eta}^{\varepsilon}(A)[1+s^{\varepsilon
}K/(1-\alpha)],
\]
where $K$ and $\alpha$ are from Lemma \ref{Lem:9.6}.
\end{lemma}

\begin{proof}
[Proof of Lemma \ref{lem:ble}]
Recall that  $a_{1}^{2}=2\sup_{x\in B_{2}}$%
tr$[\sigma(x)\sigma(x)^{T}M]$. We then use that $AM+MA^{T}=-I$ 
to get that with $U(x)\doteq \left\langle x,Mx\right\rangle$,
\begin{equation}
\left\langle DU(x),b(x)\right\rangle \leq-\varepsilon a_{1}^{2}%
\label{eqn:1stbound}%
\end{equation}
for $x\in B_{2}\setminus\mathcal{S}_{2}(\varepsilon)$, and
\begin{equation}
\left\langle DU(x),b(x)\right\rangle \leq-\frac{1}{8}b_{0}^{2}%
\label{eqn:2ndbound}%
\end{equation}
for $B_{2}\setminus(B_{0}/2)$. By It\^{o}'s formula%
\begin{align}
dU(X^{\varepsilon}_t)   =\left\langle DU(X^{\varepsilon}_t%
),b(X^{\varepsilon}_t)\right\rangle dt+\frac{\varepsilon}{2}%
\text{tr}[\sigma(X^{\varepsilon}_t)\sigma(X^{\varepsilon}_t)^{T}%
M]dt  +\sqrt{\varepsilon}\left\langle DU(X^{\varepsilon}_t),\sigma(X^{\varepsilon
}_t)dW_t\right\rangle .\label{eqn:ito}%
\end{align}
Starting at $x\in \partial B_{1}$, we are concerned with the probability
\[
P_{x}\left\{  U(X^{\varepsilon}_t)\text{ reaches }b_{2}^{2}\text{ before
}a_{2}^{2}\varepsilon\right\}  ,
\]
where $U(x)=b_{1}^{2}$. However, according to (\ref{eqn:ito}) and
(\ref{eqn:2ndbound}), reaching $b_{2}^{2}$ before $b_{0}^{2}/4$ is a rare
event, and its probability decays exponentially in the form $e^{-\kappa
/\varepsilon}$ for some $\kappa>0$ and uniformly in $x\in \partial B_{1}$. Once the
process reaches $B_{0}/2$, (\ref{eqn:ito}) and (\ref{eqn:1stbound}) imply
$U(X^{\varepsilon}_t)$ is a supermartingale as long as it is in the interval
$[0,b_{0}^{2}]$, and therefore after $X^{\varepsilon}_t$ reaches $B_{0}/2$, the probability that $U(X^{\varepsilon}_t)$
reaches $a_{2}^{2}\varepsilon$ before $b_{0}^{2}$ is greater than $1/2$.
\end{proof}

\vspace{\baselineskip}

\begin{proof}
[Proof of Lemma \ref{lem:blu}]Consider a starting position $w\in
\partial\mathcal{S}_{2}(\varepsilon)$, and recall that $\eta^{\varepsilon
}(dz|w)$ denotes the hitting distribution on $\partial B_{2}$ after starting at $w$.
Let $\theta_{k}^{\varepsilon}$ denote the return times to $\partial
\mathcal{S}_{2}(\varepsilon)$ after visiting $\partial\mathcal{S}%
_{1}(\varepsilon)$, and let $q_{n}^{\varepsilon}(w)$ denote the probability
that the first $k$ for which $X^{\varepsilon}$ visits $\partial B_{2}$ before visiting
$\partial\mathcal{S}_{1}(\varepsilon)$ during $[\theta_{k}^{\varepsilon
},\theta_{k+1}^{\varepsilon}]$ is $n$. Then by the strong Markov property and
using the rescaled process%
\[
\int_{\partial B_{2}}g(z)\eta^{\varepsilon}(dz|w)=\sum\nolimits_{n=0}^{\infty}\int_{\partial B_{2}%
}g(z)q_{n}^{\varepsilon}(w)\int_{\partial\mathcal{S}_{2}(\varepsilon)}%
\eta^{\varepsilon}(dz|u)J^{\varepsilon}(u)p^{\varepsilon,n}(\sqrt{\varepsilon}u|\sqrt{\varepsilon}w)du,
\]
where $J^{\varepsilon}(u)$ is the Jacobian that accounts for the mapping
between $\partial\mathcal{S}_{2}(\varepsilon)$ and $\partial\mathcal{S}%
_{2}(1)$ and is given by $u/\sqrt{\varepsilon}$.
We next use that uniformly in $w\in\partial\mathcal{S}_{2}(\varepsilon)$%
\[
p^{\varepsilon,n}(\sqrt{\varepsilon}u|\sqrt{\varepsilon}w)\leq r^{\varepsilon
}(\sqrt{\varepsilon}u)[1+K\alpha^{n}]
\]
to get
\begin{align*}
&  \sum\nolimits_{n=0}^{\infty}\int_{\partial B_{2}}g(z)q_{n}^{\varepsilon}(w)\int
_{\partial\mathcal{S}_{2}(\varepsilon)}\eta^{\varepsilon}(dz|u)J^{\varepsilon
}(u)p^{\varepsilon,n}(\sqrt{\varepsilon}u|\sqrt{\varepsilon}w)du\\
&  \quad\leq\left(  \sum\nolimits_{n=0}^{\infty}\int_{\partial B_{2}}g(z)q_{n}^{\varepsilon
}(w)\int_{\partial\mathcal{S}_{2}(\varepsilon)}\eta^{\varepsilon
}(dz|u)J^{\varepsilon}(u)r^{\varepsilon}(\sqrt{\varepsilon}u)du\right)
[1+K\alpha^{n}]\\
&  \quad=\int_{\partial B_{2}}g(z)\int_{\partial\mathcal{S}_{2}(\varepsilon)}%
\eta^{\varepsilon}(dz|u)J^{\varepsilon}(u)r^{\varepsilon}(\sqrt{\varepsilon}u)du\left[  1+K\sum\nolimits_{n=0}^{\infty}q_{n}^{\varepsilon}(w)\alpha
^{n}\right]  .\\
\end{align*}
Now use that $K\sum_{n=0}^{\infty}\alpha^{n}=K/(1-\alpha)<\infty$ and $\sup_{w\in\partial\mathcal{S}_{2}(\varepsilon)}\sup
_{n\in\mathbb{N}_{0}}q_{n}^{\varepsilon}(w)\rightarrow0$
as $\varepsilon\rightarrow0$ to get the
upper bound with
\[
\bar{\eta}^{\varepsilon}(dz)\doteq\int_{\partial\mathcal{S}_{2}(\varepsilon
)}\eta^{\varepsilon}(dz|u)J^{\varepsilon}(u)r^{\varepsilon}(\sqrt{\varepsilon}u)du.
\]
When combined with the lower bound which has an analogous proof, Lemma
\ref{lem:blu} follows.
\end{proof}

\subsection{Times to reach another equilibrium point after starting with general initial distribution}
\label{subsec:9.5}

\begin{lemma}
\label{Lem:9.8}For each $j\in L_{\rm{s}}$, there
exist $\tilde{c}>0$ and $\varepsilon_{0}\in(0,1)$ such that for any distribution $\lambda^{\varepsilon}$ on $\partial B_{\delta}(O_j)$,
\[
P_{\lambda^{\varepsilon}}(\upsilon_{j}^{\varepsilon}/E_{\lambda^{\varepsilon}%
}\upsilon_{j}^{\varepsilon}>t)\leq e^{-\tilde{c}t}%
\]
for all $t>0$ and $\varepsilon\in(0,\varepsilon_{0}).$
\end{lemma}

\begin{proof}
We give the proof for the case $j=1$. We first show for any $r\in(0,1)$ there is $\varepsilon_{0}>0$ such that for
any $\varepsilon\in(0,\varepsilon_{0})$ and $\lambda^{\varepsilon}%
,\theta^{\varepsilon}\in\mathcal{P}(\partial B_{\delta}(O_1))$%
\begin{equation}\label{eqn:timebounds}
    E_{\lambda^{\varepsilon}}\upsilon_{1}^{\varepsilon}/E_{\theta
^{\varepsilon}}\upsilon_{1}^{\varepsilon}\geq r.
\end{equation}
We use that $\upsilon_{1}^{\varepsilon}$ can be decomposed into $\bar
{\upsilon}_1^{\varepsilon}+\hat{\upsilon}_1^{\varepsilon}$, where $\bar{\upsilon
}_1^{\varepsilon}$ is the first hitting time to $\partial B_{2\delta}(O_{1}%
)$. Since by standard large deviation theory the exponential growth rate of
the expected value of $\upsilon_{1}^{\varepsilon}$ is strictly greater than
that of $\bar{\upsilon}_1^{\varepsilon}$ (uniformly in the initial distribution)
$E_{\lambda^{\varepsilon}}\bar{\upsilon}_{1}^{\varepsilon}$ (respectively
$E_{\theta^{\varepsilon}}\bar{\upsilon}_{1}^{\varepsilon}$) is negligible
compared to $E_{\lambda^{\varepsilon}}\upsilon_{1}^{\varepsilon}$
(respectively $E_{\theta^{\varepsilon}}\upsilon_{1}^{\varepsilon}$), and so it
is enough to show 
$
E_{\lambda^{\varepsilon}}\hat{\upsilon}_1^{\varepsilon}/E_{\theta
^{\varepsilon}}\hat{\upsilon}_1^{\varepsilon}\geq r.
$
Owing to Lemma \ref{lem:ble} (and in particular because $\kappa>0$) the
contribution to either $E_{\lambda^{\varepsilon}}\hat{\upsilon}_1^{\varepsilon}$
or $E_{\theta^{\varepsilon}}\hat{\upsilon}_1^{\varepsilon}$ from trajectories
that reach $\partial B_{2\delta}(O_1)$ before $\partial\mathcal{S}_{2}(\varepsilon)$ can be
neglected. Using Lemma \ref{lem:blu} and the strong Markov property gives
\[
\inf_{w_{1},w_{2}\in\partial\mathcal{S}_{2}(\varepsilon)}\frac{E_{w_{1}}%
\hat{\upsilon}_1^{\varepsilon}}{E_{w_{2}}\hat{\upsilon}_1^{\varepsilon}}\geq
\frac{\lbrack1-s^{\varepsilon}K/(1-\alpha)]}{[1+s^{\varepsilon}K/(1-\alpha)]},
\]
and the lower bound follows since $s^{\varepsilon}\rightarrow0$.

We next claim that a suitable bound can be found for $P_{\lambda^{\varepsilon
}}(\hat{\upsilon}_{1}^{\varepsilon}/E_{\lambda^{\varepsilon}}\upsilon
_{1}^{\varepsilon}>t)$. Recall that $u^{\varepsilon}\in\mathcal{P}(\partial B_{2\delta}(O_1))$ is
the stationary probability for $\psi^{\varepsilon}$ defined in \eqref{eqn:defofpsi}. Let
$\beta^{\varepsilon}$ be the probability measure on $\partial B_{\delta}(O_1)$ obtained by
integrating the transition kernel $\psi_{1}^{\varepsilon}$ with respect to
$u^{\varepsilon}$, and note that integrating $\psi_{2}^{\varepsilon}$ with
respect to $\beta^{\varepsilon}$ returns $u^{\varepsilon}$. Since the diffusion matrix is uniformly nondegenerate,
by using well known ``Gaussian type'' bounds on the transition density for the process \cite{aro}
there are $K\in (0,\infty)$ and $p\in (0,\infty)$ such that
\[
P_{x}\left\{  X_{\theta}^{\varepsilon}\in A|X_{\theta}^{\varepsilon}\in
\partial B_{2\delta}(O_1)\right\}  \leq Km(A)/\varepsilon^p
\]
for all $x\in \partial B_{\delta}(O_1)$, where $m$ is the uniform measure on $\partial B_{2\delta}(O_1)$ and
$\theta=\inf\{t>0:X_{t}^{\varepsilon}\in \partial B_{2\delta}(O_1)\cup\mathcal{S}_{2}%
(\varepsilon)\}$. Together with Lemmas \ref{lem:ble} and \ref{lem:blu}, this implies that for all
sufficiently small $\varepsilon>0$ and any bounded measurable function
$h:\partial B_{2\delta}(O_1)\rightarrow\mathbb{R}$,%
\begin{align*}
\int_{\partial B_{2\delta}(O_1)}\int_{\partial B_{\delta}(O_1)}h(y)\psi_{2}^{\varepsilon}(dy|x)\lambda^{\varepsilon
}(dx) &  \leq2\int_{\partial B_{2\delta}(O_1)}\int_{\partial B_{\delta}(O_1)}h(y)\psi_{2}^{\varepsilon}%
(dy|x)\beta^{\varepsilon}(dx)\\
&  \leq2\int_{\partial B_{2\delta}(O_1)}h(y)u^{\varepsilon}(dy).
\end{align*}
Using the last display for the first inequality, \eqref{eqn:timebounds} for the second, that
$\bar{\upsilon}_{1}^{\varepsilon}$ is small compared with $\hat{\upsilon}%
_{1}^{\varepsilon}$ for the third and Lemma \ref{Lem:9.3} for the last, there is
$\varepsilon_{1}>0$ such that
\begin{align*}
P_{\lambda^{\varepsilon}}(\hat{\upsilon}_{1}^{\varepsilon}/E_{\lambda
^{\varepsilon}}\upsilon_{1}^{\varepsilon}>t) &  =E_{\lambda^{\varepsilon}%
}(P_{X_{\bar{\upsilon}_{1}^{\varepsilon}}^{\varepsilon}}(\hat{\upsilon}%
_{1}^{\varepsilon}/E_{\lambda^{\varepsilon}}\upsilon_{1}^{\varepsilon}>t))  \leq2P_{u^{\varepsilon}}(\upsilon_{1}^{\varepsilon}/E_{\lambda
^{\varepsilon}}\upsilon_{1}^{\varepsilon}>t)\\
&  \leq2P_{u^{\varepsilon}}(\upsilon_{1}^{\varepsilon}/E_{\beta^{\varepsilon}%
}\upsilon_{1}^{\varepsilon}>t/2)
 \leq2P_{u^{\varepsilon}}(\upsilon_{1}^{\varepsilon}/E_{u^{\varepsilon}%
}\upsilon_{1}^{\varepsilon}>t/4)
  \leq2e^{-\tilde{c}t/4}%
\end{align*}
for all $\varepsilon\in(0,\varepsilon_{1})$ and $t\geq0$.

Since as
noted previously $E_{\lambda
^{\varepsilon}}\upsilon_{1}^{\varepsilon}\geq E_{\lambda
^{\varepsilon}}\bar{\upsilon}_{1}^{\varepsilon}$
and since by \cite[Theorem 4 and Corollary 1]{day4} there exists
$\varepsilon_{2}\in(0,1)$ such that 
$
P_{\lambda^{\varepsilon}}(\bar{\upsilon}_1^{\varepsilon}/E_{\lambda
^{\varepsilon}}\bar{\upsilon}_1^{\varepsilon}>t)\leq2e^{-t/2}%
$ 
for any $t>0$ and $\varepsilon\in(0,\varepsilon_{2})$, we conclude that for
any $t>0$%
$
P_{\lambda^{\varepsilon}}(\bar{\upsilon}^{\varepsilon}_{1}/E_{\lambda
^{\varepsilon}}\upsilon_{1}^{\varepsilon}    >t/2)
\leq P_{\lambda^{\varepsilon}
}(\bar{\upsilon}^{\varepsilon}_{1}/E_{\lambda^{\varepsilon}}\bar{\upsilon
}^{\varepsilon}_{1}>  t/2)
 \leq2e^{-t/2}.
 $
The conclusion of the lemma follows from these two bounds and 
\begin{align*}
P_{\lambda^{\varepsilon}}(\upsilon_{1}^{\varepsilon}/E_{\lambda^{\varepsilon}%
}\upsilon_{1}^{\varepsilon}  &  >t)\leq P_{\lambda^{\varepsilon}}%
(\bar{\upsilon}^{\varepsilon}_{1}/E_{\lambda^{\varepsilon}}\upsilon
_{1}^{\varepsilon}>t/2)+P_{\lambda^{\varepsilon}}(\hat{\upsilon}^{\varepsilon
}_{1}/E_{\lambda^{\varepsilon}}\upsilon_{1}^{\varepsilon}>t/2).
\end{align*}
\end{proof}

\begin{lemma}
\label{Lem:9.9}For any $j\in L_{\rm{s}}$ and any distribution $\lambda^{\varepsilon}$ on $\partial B_{\delta}(O_j)$,
$\upsilon_{j}^{\varepsilon}/E_{\lambda^{\varepsilon}}\upsilon_{j}%
^{\varepsilon}$ converges in distribution to an {\rm Exp(1)} random variable
under $P_{\lambda^{\varepsilon}}.$ Moreover, $E_{\lambda
^{\varepsilon}}e^{it\upsilon_{j}^{\varepsilon}/E_{\lambda^{\varepsilon}%
}\upsilon_{j}^{\varepsilon}}\rightarrow1/(1-it)$ uniformly on any compact set in $\mathbb{R}$.
\end{lemma}

\begin{proof}
We give the proof for the case $j=1$. Recall that $E_{u^{\varepsilon}}e^{it\upsilon_{1}^{\varepsilon}%
/E_{u^{\varepsilon}}\upsilon_{1}^{\varepsilon}}\rightarrow1/(1-it)$ uniformly on any compact set in $\mathbb{R}$ as
$\varepsilon\rightarrow0$ from Remark \ref{Rmk:9.1}. We would like to show that $E_{\lambda
^{\varepsilon}}e^{it\upsilon_{1}^{\varepsilon}/E_{\lambda^{\varepsilon}%
}\upsilon_{1}^{\varepsilon}}\rightarrow1/(1-it)$ uniformly on any compact set in $\mathbb{R}$. Since $\upsilon
_{1}^{\varepsilon}=\bar{\upsilon}^{\varepsilon}_{1}+\hat{\upsilon
}^{\varepsilon}_{1}$ with $\bar{\upsilon}^{\varepsilon}_{1}$ the first hitting
time to $\partial B_{2\delta}(O_{1}),$ we know that $E_{\lambda
^{\varepsilon}}\bar{\upsilon}^{\varepsilon}_{1}/E_{\lambda^{\varepsilon}%
}\upsilon_{1}^{\varepsilon}\rightarrow0$ and thus $E_{\lambda^{\varepsilon}%
}\hat{\upsilon}^{\varepsilon}_{1}/E_{\lambda^{\varepsilon}}\upsilon
_{1}^{\varepsilon}\rightarrow1.$ Observe that%
\begin{align*}
 E_{\lambda^{\varepsilon}}e^{it\upsilon_{1}^{\varepsilon}/E_{\lambda
^{\varepsilon}}\upsilon_{1}^{\varepsilon}}
  =E_{\lambda^{\varepsilon}}\left[  e^{it\bar{\upsilon}^{\varepsilon}%
_{1}/E_{\lambda^{\varepsilon}}\upsilon_{1}^{\varepsilon}}\cdot
E_{X^{\varepsilon}\left(  \bar{\upsilon}^{\varepsilon}_{1}\right)  }\left(
e^{it\hat{\upsilon}^{\varepsilon}_{1}/E_{\lambda^{\varepsilon}}\upsilon
_{1}^{\varepsilon}}\right)  \right]  ,
\end{align*}
\[
E_{\lambda^{\varepsilon}}\left[  E_{X^{\varepsilon}\left(  \bar{\upsilon
}^{\varepsilon}_{1}\right)  }\left(  e^{it\hat{\upsilon}^{\varepsilon}%
_{1}/E_{\lambda^{\varepsilon}}\upsilon_{1}^{\varepsilon}}\right)  \right]
\leq\frac{[1+s^{\varepsilon}K/(1-\alpha)]}{[1-s^{\varepsilon}K/(1-\alpha
)]}E_{u^{\varepsilon}}e^{it\upsilon_{1}^{\varepsilon}/E_{\lambda^{\varepsilon
}}\upsilon_{1}^{\varepsilon}}\rightarrow1/(1-it)
\]
and
\[
E_{\lambda^{\varepsilon}}\left[  E_{X^{\varepsilon}\left(  \bar{\upsilon
}^{\varepsilon}_{1}\right)  }\left(  e^{it\hat{\upsilon}^{\varepsilon}%
_{1}/E_{\lambda^{\varepsilon}}\upsilon_{1}^{\varepsilon}}\right)  \right]
\geq\frac{[1-s^{\varepsilon}K/(1-\alpha)]}{[1+s^{\varepsilon}K/(1-\alpha
)]}E_{u^{\varepsilon}}e^{it\upsilon_{1}^{\varepsilon}/E_{\lambda^{\varepsilon
}}\upsilon_{1}^{\varepsilon}}\rightarrow1/(1-it).
\]
Since $E_{\lambda^{\varepsilon}}\bar{\upsilon}^{\varepsilon}_{1}%
/E_{\lambda^{\varepsilon}}\upsilon_{1}^{\varepsilon}\rightarrow0$ and $e^{ix}$
is a bounded and continuous function, a conditioning argument gives
\[
\left\vert E_{\lambda^{\varepsilon}}e^{it\upsilon_{1}^{\varepsilon}%
/E_{\lambda^{\varepsilon}}\upsilon_{1}^{\varepsilon}}-E_{\lambda^{\varepsilon
}}\left[  E_{X^{\varepsilon}\left(  \bar{\upsilon}^{\varepsilon}_{1}\right)
}\left(  e^{it\hat{\upsilon}^{\varepsilon}_{1}/E_{\lambda^{\varepsilon}%
}\upsilon_{1}^{\varepsilon}}\right)  \right]  \right\vert \leq E_{\lambda
^{\varepsilon}}\left\vert e^{it\bar{\upsilon}^{\varepsilon}/E_{\lambda
^{\varepsilon}}\upsilon_{1}^{\varepsilon}}-1\right\vert \rightarrow0.
\]
We conclude that $E_{\lambda^{\varepsilon}}e^{it\upsilon_{1}^{\varepsilon
}/E_{\lambda^{\varepsilon}}\upsilon_{1}^{\varepsilon}}\rightarrow1/(1-it)$ uniformly on any compact set in $\mathbb{R}$.
\end{proof}

\subsection{Return times (single cycles)}

In this subsection, we will extend all the three results to return times for the single cycle case (i.e. when $h_1>w$).

\begin{lemma}
\label{Lem:9.10}There exists $\delta_{0}\in(0,1)$ such that for any $\delta
\in(0,\delta_{0})$ and any distribution $\lambda^{\varepsilon}$ on $\partial
B_{\delta}(O_{1}),$%
\[
\lim_{\varepsilon\rightarrow0}\varepsilon\log E_{\lambda^{\varepsilon}}%
\tau_{1}^{\varepsilon}=\min_{y\in\cup_{k\in L\setminus\{1\}}\partial
B_{\delta}(O_{k})}V(O_{1},y).
\]

\end{lemma}

\begin{proof}
We have $E_{\lambda^{\varepsilon}}\tau_{1}^{\varepsilon}=E_{\lambda
^{\varepsilon}}\upsilon_{1}^{\varepsilon}+E_{\lambda^{\varepsilon}}%
(\tau_{1}^{\varepsilon}-\upsilon_{1}^{\varepsilon})$, and by Lemma \ref{Lem:9.1} we know that
\[
\lim_{\varepsilon\rightarrow0}\varepsilon\log E_{\lambda^{\varepsilon}%
}\upsilon_{1}^{\varepsilon}=\min_{y\in\cup_{k\in L\setminus\{1\}}\partial
B_{\delta}(O_{k})}V(O_{1},y).
\]
Moreover, observe that $W(O_{j})>W(O_{1})$ for any $j\in L\setminus\{1\}$ due
to  
Remark \ref{Rmk:4.1}. 
Note that $\upsilon_{1}^{\varepsilon}$ as defined in \eqref{eqn:defofnu}
 coincides with $\sigma_{0}^{\varepsilon}$ defined in \eqref{eqn:sigma}.
We can therefore apply
Remark
\ref{Rmk:6.6} with $f=0$, $A=M$ and $\eta=[\min_{j\in L\setminus\{1\}}%
W(O_{j})-W(O_{1})]/3,$ we find that there exists $\delta_{1}\in(0,1)$ such that for
any $\delta\in(0,\delta_{1})$
\begin{align*}
 \liminf_{\varepsilon\rightarrow0}-\varepsilon\log\left(  \sup
_{z\in\partial B_{\delta}(O_{1})}E_{z}\left(  \tau_{1}^{\varepsilon}%
-\upsilon_{1}^{\varepsilon}\right)  \right)  &\geq\min_{j\in L\setminus\{1\}}W(O_{j})-W(O_{1})-\min_{j\in L\setminus
\{1\}}V(O_{1},O_{j})-\eta\\
&  =-\min_{j\in L\setminus\{1\}}V(O_{1},O_{j})+2\eta.
\end{align*}
On the other hand, by continuity of $V(O_{1},\cdot),$ for this given $\eta,$
there exists $\delta_{2}\in(0,1)$ such that for any $\delta\in(0,\delta_{2})$%
\[
\min_{y\in\cup_{k\in L\setminus\{1\}}\partial B_{\delta}(O_{k})}%
V(O_{1},y)\geq\min_{j\in L\setminus\{1\}}V(O_{1},O_{j})-\eta.
\]

Thus, for any $\delta\in(0,\delta_{0})$ with $\delta_{0}\doteq\delta_{1}%
\wedge\delta_{2}$
\begin{align*}
\limsup_{\varepsilon\rightarrow0}\varepsilon\log E_{\lambda^{\varepsilon}%
}(\tau_{1}^{\varepsilon}-\upsilon_{1}^{\varepsilon})  &  \leq\limsup_{\varepsilon\rightarrow
0}\varepsilon\log\left(  \sup_{z\in\partial B_{\delta}(O_{1})}E_{z}%
(\tau_{1}^{\varepsilon}-\upsilon_{1}^{\varepsilon})\right) \\
&  \leq\min_{j\in L\setminus\{1\}}V(O_{1},O_{j})-2\eta \leq\min_{y\in\cup_{k\in L\setminus\{1\}}\partial B_{\delta}(O_{k})}%
V(O_{1},y)-\eta\\
&  =\lim_{\varepsilon\rightarrow0}\varepsilon\log E_{\lambda^{\varepsilon}%
}\upsilon_{1}^{\varepsilon}-\eta
\end{align*}
and%
\begin{align*}
\lim_{\varepsilon\rightarrow0}\varepsilon\log E_{\lambda^{\varepsilon}}%
\tau_{1}^{\varepsilon}  &  =\lim_{\varepsilon\rightarrow0}\varepsilon\log
E_{\lambda^{\varepsilon}}\upsilon_{1}^{\varepsilon} =\min_{y\in\cup_{k\in
L\setminus\{1\}}\partial B_{\delta}(O_{k})}V(O_{1},y).
\end{align*}
\end{proof}

\begin{lemma}
\label{Lem:9.11}Given $\delta>0$ sufficiently small, and for any distribution $\lambda^{\varepsilon}$
on $\partial B_{\delta}(O_{1}),$ there exist $\tilde{c}>0$ and $\varepsilon
_{0}\in(0,1)$ such that
\[
P_{\lambda^{\varepsilon}}(\tau_{1}^{\varepsilon}/E_{\lambda^{\varepsilon}}%
\tau_{1}^{\varepsilon}>t)\leq e^{-\tilde{c}t}%
\]
for all $t\geq1$ and $\varepsilon\in(0,\varepsilon_{0}).$
\end{lemma}

\begin{proof}
For any $t>0,$ $P_{\lambda^{\varepsilon}}(\tau_{1}^{\varepsilon}%
/E_{\lambda^{\varepsilon}}\tau_{1}^{\varepsilon}>t)\leq P_{\lambda
^{\varepsilon}}(\upsilon_{1}^{\varepsilon}/E_{\lambda^{\varepsilon}}\tau
_{1}^{\varepsilon}>t/2)+P_{\lambda^{\varepsilon}}((\tau_{1}^{\varepsilon}-\upsilon_{1}^{\varepsilon})/E_{\lambda^{\varepsilon}}\tau_{1}^{\varepsilon}>t/2)$. It is easy to see that
the first term has this sort of bound due to Lemma \ref{Lem:9.8} and
$E_{\lambda^{\varepsilon}}\tau_{1}^{\varepsilon}\geq E_{\lambda^{\varepsilon}%
}\upsilon_{1}^{\varepsilon}.$

It suffices to show that this sort of bound holds for the second term,
namely, there exists a constant $\tilde{c}>0$ such that 
\[
P_{\lambda^{\varepsilon}}\left((\tau_{1}^{\varepsilon}-\upsilon_{1}^{\varepsilon})/E_{\lambda^{\varepsilon}}\tau_{1}^{\varepsilon}>t\right)\leq e^{-\tilde{c} t}
\]
for all $t\in[0,\infty)$ and $\varepsilon$ sufficiently small.
By Chebyshev's inequality$,$ 
\[
P_{\lambda^{\varepsilon}}\left((\tau_{1}^{\varepsilon}-\upsilon_{1}^{\varepsilon})/E_{\lambda^{\varepsilon}}\tau_{1}^{\varepsilon}>t\right)=P_{\lambda^{\varepsilon}}(e^{(\tau_{1}^{\varepsilon}-\upsilon_{1}^{\varepsilon})/E_{\lambda^{\varepsilon}}\tau_{1}^{\varepsilon}}>e^{t})\leq e^{-t}E_{\lambda^{\varepsilon}}e^{(\tau_{1}^{\varepsilon}-\upsilon_{1}^{\varepsilon})/E_{\lambda^{\varepsilon}}\tau_{1}^{\varepsilon}},
\]
and it therefore suffices to prove that $E_{\lambda^{\varepsilon}}e^{(\tau_{1}^{\varepsilon}-\upsilon_{1}^{\varepsilon})/E_{\lambda^{\varepsilon}}\tau_{1}^{\varepsilon}}$
is less than a constant for all $\varepsilon$ sufficiently small.
Observe that 
\[
\tau_{1}^{\varepsilon}-\upsilon_{1}^{\varepsilon}=\sum\nolimits_{j\in L\setminus\{1\}}\sum\nolimits_{k=1}^{N_{j}}\upsilon_{j}^{\varepsilon}(k),
\]
where $N_{j}$ is the number of visits of 
$\partial B_{\delta}(O_{j})$, and $\upsilon_{j}^{\varepsilon}(k)$
is the $k$-th copy of the first hitting time to $\cup_{k\in L\setminus\{j\}}\partial B_\delta(O_{k})$
after starting from $\partial B_{\delta}(O_{j}).$ 

If we consider $\partial B_{\delta}(O_{j})$ as the starting location of a
regenerative cycle, as was done previously in the paper for $\partial
B_{\delta}(O_{1})$, then there will be a unique stationary distribution, and
if the process starts with that as the initial distribution then the times
$\upsilon_{j}^{\varepsilon}(k)$ are independent from each other and from the
number of returns to $\partial B_{\delta}(O_{j})$ before first visiting
$\partial B_{\delta}(O_{1})$. While these random times as used here do not
arise from starting with such a distribution, we can use Lemma \ref{lem:blu}
to bound the error in terms of a multiplicative factor that is independent of
$\varepsilon$ for small $\varepsilon>0$, and thereby justify treating $N_{j}$
as though it is independent of the $\upsilon_{j}^{\varepsilon}(k)$.

Recalling that $l\doteq |L|$,
\begin{align}
\label{eqn:holder}
E_{\lambda^{\varepsilon}}e^{(\tau_{1}^{\varepsilon}-\upsilon_{1}^{\varepsilon})/E_{\lambda^{\varepsilon}}\tau_{1}^{\varepsilon}}
&=  E_{\lambda^{\varepsilon}}\prod\nolimits_{j\in L\setminus\{1\}}e^{\left(\sum_{k=1}^{N_{j}}\upsilon_{j}^{\varepsilon}(k)\right)/E_{\lambda^{\varepsilon}}\tau_{1}^{\varepsilon}}\nonumber\\
&\leq  \prod\nolimits_{j\in L\setminus\{1\}}\left(E_{\lambda^{\varepsilon}}\left[e^{\left(\sum_{k=1}^{N_{j}}\upsilon_{j}^{\varepsilon}(k)\right)(l-1)/E_{\lambda^{\varepsilon}}\tau_{1}^{\varepsilon}}\right]\right)^{1/(l-1)},\nonumber
\end{align}
where we use the generalized H\"{o}lder's inequality for the last line. Thus,
if we can show for each $j\in L\setminus\{1\}$ that $E_{\lambda^{\varepsilon}}\exp [{(\sum_{k=1}^{N_{j}}\upsilon_{j}^{\varepsilon}(k))(l-1)/E_{\lambda^{\varepsilon}}\tau_{1}^{\varepsilon}}]$
is less than a constant for all $\varepsilon$ sufficiently small, then we are done.

Such an estimate is straightforward for the case of an unstable equilibrium, i.e., for $j\in L\backslash L_{\rm{s}}$, and so we focus on the case $j\in L_{\rm{s}}\backslash \{1\}$. For this case, we apply Lemma \ref{Lem:9.8} to find that there exists $\tilde{{c}}>0$
and $\varepsilon_{0}\in(0,1)$ such that for any $j\in L$ and any
distribution $\tilde{\lambda}^{\varepsilon}$ on $\partial B(O_{j}),$
\begin{equation}
P_{\tilde{\lambda}^{\varepsilon}}(\upsilon_{j}^{\varepsilon}/E_{\tilde{\lambda}^{\varepsilon}}\upsilon_{j}^{\varepsilon}>t)\leq e^{-\tilde{{c}}t}\label{tail}
\end{equation}
for any $t>0$ and $\varepsilon\in(0,\varepsilon_{0}).$ Hence, given
any $\eta>0$, there is $\bar{\varepsilon}_0 \in (0,\varepsilon_0)$ such that for all $\varepsilon \in (0,\bar{\varepsilon}_0)$ and any $j\in L\setminus\{1\}$ 
\begin{align*}
E_{\lambda^{\varepsilon}}\left[e^{\upsilon_{j}^{\varepsilon}(l-1)/E_{\lambda^{\varepsilon}}\tau_{1}^{\varepsilon}}\right] 
 & \leq1+\int_{1}^{\infty}P_{\lambda^{\varepsilon}}(e^{(l-1)\upsilon_{j}^{\varepsilon}/E_{\lambda^{\varepsilon}}\tau_{1}^{\varepsilon}}>t)dt\\
 & \leq 1+\int_{1}^{\infty}P_{\lambda^{\varepsilon}}\left(\upsilon_{j}^{\varepsilon}/E_{\lambda^{\varepsilon}}\upsilon_{j}^{\varepsilon}>\log t\cdot E_{\lambda^{\varepsilon}}\tau_{1}^{\varepsilon}/((l-1)E_{\lambda^{\varepsilon}}\upsilon_{j}^{\varepsilon})\right)dt\\
 & \leq1+\int_{1}^{\infty}t^{-\tilde{{c}}E_{\lambda^{\varepsilon}}\tau_{1}^{\varepsilon}/((l-1)E_{\lambda^{\varepsilon}}\upsilon_{j}^{\varepsilon})}dt\\
 & =1+\left(\tilde{{c}}E_{\lambda^{\varepsilon}}\tau_{1}^{\varepsilon}/((l-1)E_{\lambda^{\varepsilon}}\upsilon_{j}^{\varepsilon})-1\right)^{-1}
 \leq1+e^{-\frac{1}{\varepsilon}(h_{1}-h_{j}-\eta)},
\end{align*}
where the last inequality comes from Lemma \ref{Lem:9.1} and Lemma \ref{Lem:9.10}, and by picking the range of $\varepsilon$ small if it needs to be.

By using induction and a conditioning argument, it follows that for any $\eta>0$, for any $j\in L\setminus\{1\}$ and
for any $n\in\mathbb{N},$
\[
E_{\lambda^{\varepsilon}}\left[e^{\left(\sum_{k=1}^{n}\upsilon_{j}^{\varepsilon}(k)\right)(l-1)/E_{\lambda^{\varepsilon}}\tau_{1}^{\varepsilon}}\right]\leq\left(1+e^{-\frac{1}{\varepsilon}(h_{1}-h_{j}-\eta)}\right)^{n}.
\]
This implies that 
\[
E_{\lambda^{\varepsilon}}\left[e^{\left(\sum_{k=1}^{N_{j}}\upsilon_{j}^{\varepsilon}(k)\right)(l-1)/E_{\lambda^{\varepsilon}}\tau_{1}^{\varepsilon}}\right]\leq E_{\lambda^{\varepsilon}}\left[\left(1+e^{-\frac{1}{\varepsilon}(h_{1}-h_{j}-\eta)}\right)^{N_{j}}\right].
\]

The next thing we need to know is the distribution of $N_{j}$,
i.e., $P_{\lambda^{\varepsilon}}(N_{j}=n)$ for $n\in\mathbb{N}$. Following a similar
argument as in the proof of Lemma \ref{Lem:6.3} and the proof of Lemma \ref{Lem:6.5},
for sufficiently small $\varepsilon>0$ we find 
\begin{align*}
P_{\lambda^{\varepsilon}}(N_{j}=n) &\leq \left(1\wedge e^{-\frac{1}{\varepsilon}(W(O_{j})-W(O_{1}\cup O_{j})-h_{1}-\eta)}\right)(1-q_{j})^{n-1}q_{j},
\end{align*}
where 
\begin{align}
\label{eqn:q_j}
q_{j} \doteq
  \frac{\inf_{x\in\partial B_{\delta}(O_{j})}P_{x}(\tilde{{T_{1}}}<{\tilde{T}}_{j}^{+})}{1-\sup_{y\in\partial B_{\delta}(O_{j})}p(y,\partial B_{\delta}(O_{j}))}\geq e^{-\frac{1}{\varepsilon}(W(O_{1})-W(O_{1}\cup O_{j})-h_{j}+\eta)}.
\end{align}
Therefore,
\begin{align*}
 & E_{\lambda^{\varepsilon}}\left[e^{\left(\sum_{k=1}^{N_{j}}\upsilon_{j}^{\varepsilon}(k)\right)(l-1)/E_{\lambda^{\varepsilon}}\tau_{1}^{\varepsilon}}\right]\\
 & \quad\leq E_{\lambda^{\varepsilon}}\left[\left(1+e^{-\frac{1}{\varepsilon}(h_{1}-h_{j}-\eta)}\right)^{N_{j}}\right]
 =\sum\nolimits_{n=1}^{\infty}\left(1+e^{-\frac{1}{\varepsilon}(h_{1}-h_{j}-\eta)}\right)^{n}P_{\lambda^{\varepsilon}}(N_{j}=n)\\
 & \quad\leq\sum_{n=1}^{\infty}\left(1\wedge e^{-\frac{1}{\varepsilon}(W(O_{j})-W(O_{1}\cup O_{j})-h_{1}-\eta)}\right)\left(1+e^{-\frac{1}{\varepsilon}(h_{1}-h_{j}-\eta)}\right)^{n}(1-q_{j})^{n-1}q_{j}\\
 & \quad=\frac{\left(1\wedge e^{-\frac{1}{\varepsilon}(W(O_{j})-W(O_{1}\cup O_{j})-h_{1}-\eta)}\right)q_{j}\left(1+e^{-\frac{1}{\varepsilon}(h_{1}-h_{j}-\eta)}\right)}{1-\left(1+e^{-\frac{1}{\varepsilon}(h_{1}-h_{j}-\eta)}\right)(1-q_{j})}\\
 &\quad \leq\frac{\left(1\wedge e^{-\frac{1}{\varepsilon}(W(O_{j})-W(O_{1}\cup O_{j})-h_{1}-\eta)}\right)\left(1+e^{-\frac{1}{\varepsilon}(h_{1}-h_{j}-\eta)}\right)}{-e^{-\frac{1}{\varepsilon}(h_{1}-h_{j}-\eta)}/q_{j}+1}\\
 & \quad\leq\frac{\left(1\wedge e^{-\frac{1}{\varepsilon}(W(O_{j})-W(O_{1}\cup O_{j})-h_{1}-\eta)}\right)\left(1+e^{-\frac{1}{\varepsilon}(h_{1}-h_{j}-\eta)}\right)}{-e^{-\frac{1}{\varepsilon}(h_{1}+W(O_1\cup O_j)-W(O_1)-2\eta)}+1}.
\end{align*}
The second equality holds since $h_1>w\geq h_j$ and \eqref{eqn:q_j} imply $(1-q_{j})(1+e^{-\frac{1}{\varepsilon}(h_{1}-h_{j}-
\eta)})<1$ for all $\varepsilon$ sufficiently small; the last inequality is from \eqref{eqn:q_j}.

Then we use the fact that for $x\in (0,1/2)$, $1/(1-x)\leq 1+2x$ 
to find that
\begin{align}
\label{eqn:mgf_bound}
 &  E_{\lambda^{\varepsilon}}\left[e^{\left(\sum_{k=1}^{N_{j}}\upsilon_{j}^{\varepsilon}(k)\right)(l-1)/E_{\lambda^{\varepsilon}}\tau_{1}^{\varepsilon}}\right]\nonumber\\
  & \leq \left(1\wedge e^{-\frac{1}{\varepsilon}(W(O_{j})-W(O_{1}\cup O_{j})-h_{1}-\eta)}\right)\left(1+e^{-\frac{1}{\varepsilon}(h_{1}-h_{j}-\eta)}\right)
   \left(1+2e^{-\frac{1}{\varepsilon}(h_{1}+W(O_1\cup O_j)-W(O_1)-2\eta)}\right)\nonumber\\
 &\leq \left(1\wedge e^{-\frac{1}{\varepsilon}(W(O_{j})-W(O_{1}\cup O_{j})-h_{1}-\eta)}\right) \left(1+5e^{-\frac{1}{\varepsilon}(h_{1}+W(O_1\cup O_j)-W(O_1)-2\eta)}\right)\\
  &\leq 1\cdot 6=6.\nonumber
 \end{align}
The third inequality holds due to the fact that $W(O_1)\geq W(O_1\cup O_j)+h_j$ and the last inequality comes from the assumption that $h_1>w$ and by picking $\eta$ to be smaller than $(h_1-w)/2$. This completes the proof.
\end{proof}

\begin{lemma}
\label{Lem:9.12}Given $\delta>0$ sufficiently small, and for any distribution $\lambda^{\varepsilon}$
on $\partial B_{\delta}(O_{1})$, $\tau_{1}^{\varepsilon}/E_{\lambda
^{\varepsilon}}\tau_{1}^{\varepsilon}$ converges in distribution to an
\rm{Exp(1)} random variable under $P_{\lambda^{\varepsilon}}.$ Moreover, $E_{\lambda^{\varepsilon}}(  e^{it\left(  \tau_{1}^{\varepsilon
}/E_{\lambda^{\varepsilon}}\tau_{1}^{\varepsilon}\right)  })
\rightarrow1/(1-it)$ 
uniformly on any compact set in $\mathbb{R}$.
\end{lemma}

\begin{proof}
Note that%
\[
E_{\lambda^{\varepsilon}}\left(  e^{it\left(  \tau_{1}^{\varepsilon
}/E_{\lambda^{\varepsilon}}\tau_{1}^{\varepsilon}\right)  }\right)
=E_{\lambda^{\varepsilon}}\left(  e^{it\left(  \upsilon_{1}^{\varepsilon
}/E_{\lambda^{\varepsilon}}\tau_{1}^{\varepsilon}\right)  }E_{X^{\varepsilon
}(  \upsilon_{1}^{\varepsilon})  }\left(  e^{it\left(  (\tau_{1}^{\varepsilon}-\upsilon_{1}^{\varepsilon})/E_{\lambda^{\varepsilon}}\tau_{1}^{\varepsilon}\right)
}\right)  \right)  .
\]
Since
\[
E_{\lambda^{\varepsilon}}\left(  e^{it\left(  \upsilon_{1}^{\varepsilon
}/E_{\lambda^{\varepsilon}}\tau_{1}^{\varepsilon}\right)  }\right)
=E_{\lambda^{\varepsilon}}\left(  e^{it\left(  E_{\lambda^{\varepsilon}%
}\upsilon_{1}^{\varepsilon}/E_{\lambda^{\varepsilon}}\tau_{1}^{\varepsilon
}\right)  \left(  \upsilon_{1}^{\varepsilon}/E_{\lambda^{\varepsilon}}%
\upsilon_{1}^{\varepsilon}\right)  }\right)
\]
and we know that $E_{\lambda^{\varepsilon}}\upsilon_{1}^{\varepsilon
}/E_{\lambda^{\varepsilon}}\tau_{1}^{\varepsilon}\rightarrow1$ from the proof
of Lemma \ref{Lem:9.10}, by applying Lemma \ref{Lem:9.9} we have
$E_{\lambda^{\varepsilon}}(  e^{it\left(  \upsilon_{1}^{\varepsilon
}/E_{\lambda^{\varepsilon}}\tau_{1}^{\varepsilon}\right)  })
\rightarrow1/(1-it)$ uniformly on any compact set in $\mathbb{R}$.
Also
\begin{align*}
\left\vert E_{\lambda^{\varepsilon}}\left(  e^{it\left(  \tau_{1}%
^{\varepsilon}/E_{\lambda^{\varepsilon}}\tau_{1}^{\varepsilon}\right)
}\right)  -E_{\lambda^{\varepsilon}}\left(  e^{it\left(  \upsilon
_{1}^{\varepsilon}/E_{\lambda^{\varepsilon}}\tau_{1}^{\varepsilon}\right)
}\right)  \right\vert 
 \leq E_{\lambda^{\varepsilon}}\left\vert
E_{X^{\varepsilon}(  \upsilon_{1}^{\varepsilon})  }\left(
e^{it\left(  (\tau_{1}^{\varepsilon}-\upsilon_{1}^{\varepsilon})/E_{\lambda^{\varepsilon}}\tau
_{1}^{\varepsilon}\right)  }\right)  -1\right\vert ,
\end{align*}
where the right hand side  converges to $0$ using $E_{\lambda^{\varepsilon}}(\tau_{1}^{\varepsilon}-\upsilon_{1}^{\varepsilon})/E_{\lambda^{\varepsilon}}\tau_{1}^{\varepsilon}%
\rightarrow0$ and the dominated convergence theorem.
The convergence of $\tau_{1}^{\varepsilon}/E_{\lambda
^{\varepsilon}}\tau_{1}^{\varepsilon}$ to an \rm{Exp(1)} random variable
under $P_{\lambda^{\varepsilon}}$ and uniform convergence of $E_{\lambda^{\varepsilon}}(  e^{it\left(  \tau_{1}^{\varepsilon
}/E_{\lambda^{\varepsilon}}\tau_{1}^{\varepsilon}\right)  })$ to $1/(1-it)$ on compact set in $\mathbb{R}$ follows.
\end{proof}
\subsection{Return times (multicycles)}
\label{subsec:9.6}
In this subsection, we will extend all the three results to multi regenerative cycles (when $w\geq h_1$). Recall that the multicycle
times $\hat{\tau}^\varepsilon_i$ are defined according to \eqref{eqn:defofMC}
where $\{\mathbf{M}^{\varepsilon}_i\}_{i\in\mathbb{N}}$ is a sequence of independent and geometrically distributed random variables with parameter $e^{-m/\varepsilon}$ for some $m>0$ such that $m+h_1>w$. In addition, $\{\mathbf{M}^{\varepsilon}_i\}$ is independent of $\{\tau^\varepsilon_n\}$.

\begin{lemma}
\label{Lem:9.13}There exists $\delta_{0}\in(0,1)$ such that for any $\delta
\in(0,\delta_{0})$ and any distribution $\lambda^{\varepsilon}$ on $\partial
B_{\delta}(O_{1}),$%
\[
\lim_{\varepsilon\rightarrow0}\varepsilon\log E_{\lambda^{\varepsilon}}%
\hat{\tau}_{1}^{\varepsilon}=m+\min_{y\in\cup_{k\in L\setminus\{1\}}\partial
B_{\delta}(O_{k})}V(O_{1},y).
\]
\end{lemma}
\begin{proof}
  Since $\{\mathbf{M}^{\varepsilon}_i\}$ is independent of $\{\tau^\varepsilon_n\}$ and $E_{\lambda^{\varepsilon}}\mathbf{M}^{\varepsilon}_i= e^{m/\varepsilon}$, we apply Lemma \ref{Lem:9.10} to complete the proof. 
\end{proof}

\begin{lemma}
\label{Lem:9.14} Given $\delta>0,$ for any distribution $\lambda^{\varepsilon}$
on $\partial B_{\delta}(O_{1}),$ there exist $\tilde{c}>0$ and $\varepsilon
_{0}\in(0,1)$ such that
\[
P_{\lambda^{\varepsilon}}(\hat{\tau}_{1}^{\varepsilon}/E_{\lambda^{\varepsilon}}%
\hat{\tau}_{1}^{\varepsilon}>t)\leq e^{-\tilde{c}t}%
\]
for all $t\geq1$ and $\varepsilon\in(0,\varepsilon_{0}).$
\end{lemma}
\begin{proof}
We divide the multicycle into a sum of two terms.
The first term is the sum of all the hitting times to $\cup_{j\in L\setminus \{1\}}\partial B_{\delta}(O_j)$, and the second term
is the sum of all residual times. That is, $\hat{\tau}_1^{\varepsilon}=\hat{\upsilon}_{1}^{\varepsilon}+(\hat{\tau}_1^{\varepsilon}-\hat{\upsilon}_{1}^{\varepsilon})$,
where 
\[
\hat{\upsilon}_{1}^{\varepsilon}=\sum\nolimits_{i=1}^{\mathbf{M}^{\varepsilon}_1}\upsilon_{1}^{\varepsilon}(i)\text{ and }\hat{\tau}_1^{\varepsilon}-\hat{\upsilon}_{1}^{\varepsilon}=\sum\nolimits_{i=1}^{\mathbf{M}^{\varepsilon}_1}\left(\sum\nolimits_{j\in L\setminus\{1\}}\sum\nolimits_{k=1}^{N_{j}}\upsilon_{j}^{\varepsilon}(i,k)\right).
\]
As discussed many times, it suffices to show that there exist $\tilde{c}>0$ and $\varepsilon_{0}\in(0,1)$ such that
\[
P_{\lambda^{\varepsilon}}\left(\hat{\upsilon}_{1}^{\varepsilon}/E_{\lambda^{\varepsilon}}\hat{\tau}_1^{\varepsilon}>t\right)\leq e^{-\tilde{c}t}\text{ and }P_{\lambda^{\varepsilon}}\left((\hat{\tau}_1^{\varepsilon}-\hat{\upsilon}_{1}^{\varepsilon})/E_{\lambda^{\varepsilon}}\hat{\tau}_1^{\varepsilon}>t\right)\leq e^{-\tilde{c}t}
\]
for all $t\geq 1$ and $\varepsilon\in(0,\varepsilon_{0}).$

The first bound is relatively easy since $\upsilon_{1}^{\varepsilon}(i)$
is a sum of approximate exponentials  with a tail bound of the given sort, and since the sum of geometrically many
independent and identically distributed exponentials is again an exponential
distribution. 

For the second bound, we use Chebyshev's inequality again as in the proof of Lemma \ref{Lem:9.11} to find that it suffices to prove that $E_{\lambda^{\varepsilon}}e^{(\hat{\tau}_1^{\varepsilon}-\hat{\upsilon}_{1}^{\varepsilon})/E_{\lambda^{\varepsilon}}\hat{\tau}_1^{\varepsilon}}$
is less than a constant for all $\varepsilon$ sufficiently small.
Now due to the independence of $\mathbf{M}^{\varepsilon}_1$ and $\{\upsilon_{j}^{\varepsilon}(i,k)\}$,
we have
\begin{align}
\label{eqn_bound}
&E_{\lambda^{\varepsilon}}e^{(\hat{\tau}_1^{\varepsilon}-\hat{\upsilon}_{1}^{\varepsilon})/E_{\lambda^{\varepsilon}}\hat{\tau}_1^{\varepsilon}}\nonumber\\
&\qquad=  \sum\nolimits_{i=1}^{\infty}\left(E_{\lambda^{\varepsilon}}\left[\prod\nolimits_{j\in L\setminus\{1\}}e^{\left(\sum_{k=1}^{N_{j}}\upsilon_{j}^{\varepsilon}(k)\right)/E_{\lambda^{\varepsilon}}\hat{\tau}_1^{\varepsilon}}\right]\right)^{i}\cdot P_{\lambda^{\varepsilon}}(\mathbf{M}^{\varepsilon}_1=i)\nonumber\\
&\qquad=e^{-m/\varepsilon}\cdot E_{\lambda^{\varepsilon}}\left[\prod\nolimits_{j\in L\setminus\{1\}}e^{\left(\sum_{k=1}^{N_{j}}\upsilon_{j}^{\varepsilon}(k)\right)/E_{\lambda^{\varepsilon}}\hat{\tau}_1^{\varepsilon}}\right]\nonumber\\
&\qquad\qquad
\cdot \sum\nolimits_{i=1}^{\infty}\left( E_{\lambda^{\varepsilon}}\left[\prod\nolimits_{j\in L\setminus\{1\}}e^{\left(\sum_{k=1}^{N_{j}}\upsilon_{j}^{\varepsilon}(k)\right)/E_{\lambda^{\varepsilon}}\hat{\tau}_1^{\varepsilon}}\right](1-e^{-m/\varepsilon})\right)^{i-1}.
\end{align}
To do a further computation, we have to at least make sure that 
\begin{equation}\label{eqn:lessone}
     E_{\lambda^{\varepsilon}}\left[\prod\nolimits_{j\in L\setminus\{1\}}e^{\left(\sum_{k=1}^{N_{j}}\upsilon_{j}^{\varepsilon}(k)\right)/E_{\lambda^{\varepsilon}}\hat{\tau}_1^{\varepsilon}}\right](1-e^{-m/\varepsilon})<1.
\end{equation}
To see this, we first use the generalized H\"{o}lder's inequality
to find
\begin{align*}
E_{\lambda^{\varepsilon}}\left[\prod\nolimits_{j\in L\setminus\{1\}}e^{\left(\sum_{k=1}^{N_{j}}\upsilon_{j}^{\varepsilon}(k)\right)/E_{\lambda^{\varepsilon}}\hat{\tau}_1^{\varepsilon}}\right]
\leq
 \prod\nolimits_{j\in L\setminus\{1\}}\left(E_{\lambda^{\varepsilon}}\left[e^{\left(\sum_{k=1}^{N_{j}}\upsilon_{j}^{\varepsilon}(k)\right)(l-1)/E_{\lambda^{\varepsilon}}\hat{\tau}_{1}^{\varepsilon}}\right]\right)^{1/(l-1)}.
\end{align*}
Moreover, since $m+h_1>w$ and $E_{\lambda^{\varepsilon}}\hat{\tau}_1^{\varepsilon}=E_{\lambda^{\varepsilon}}\tau_{1}^{\varepsilon}\cdot E_{\lambda^{\varepsilon}}\mathbf{M}^{\varepsilon}_1=e^{m/\varepsilon}E_{\lambda^{\varepsilon}}\tau_{1}^{\varepsilon}$, by the same argument that gives \eqref{eqn:mgf_bound}, for any $\eta>0$ and $j\in L\setminus\{1\}$
\begin{align*}
&E_{\lambda^{\varepsilon}}\left[e^{\left(\sum_{k=1}^{N_{j}}\upsilon_{j}^{\varepsilon}(k)\right)(l-1)/E_{\lambda^{\varepsilon}}\hat{\tau}_{1}^{\varepsilon}}\right]\\
&\quad  \leq \left(1\wedge e^{-\frac{1}{\varepsilon}(W(O_{j})-W(O_{1}\cup O_{j})-h_{1}-\eta)}\right) \left(1+5e^{-\frac{1}{\varepsilon}(m+h_{1}+W(O_1\cup O_j)-W(O_1)-2\eta)}\right).
\end{align*}
Therefore,
\begin{align*}
  E_{\lambda^{\varepsilon}}\left[\prod\nolimits_{j\in L\setminus\{1\}}e^{\left(\sum_{k=1}^{N_{j}}\upsilon_{j}^{\varepsilon}(k)\right)/E_{\lambda^{\varepsilon}}\hat{\tau}_1^{\varepsilon}}\right](1-e^{-m/\varepsilon})
 \leq\prod\nolimits_{j\in L\setminus\{1\}}s_j^{1/(l-1)},
\end{align*}
with 
\begin{align*}
    s_j\doteq \left(1\wedge e^{-\frac{1}{\varepsilon}(W(O_{j})-W(O_{1}\cup O_{j})-h_{1}-\eta)}\right) \left(1+5e^{-\frac{1}{\varepsilon}(m+h_{1}+W(O_1\cup O_j)-W(O_1)-2\eta)}\right)
    (1-e^{-m/\varepsilon}).
\end{align*}
Using $(a\wedge b)(c+d)\leq ac +bd$ for positive numbers $a,b,c,d$,
\[
s_j \leq  \left(1+5e^{-\frac{1}{\varepsilon}(m+W( O_j)-W(O_1)-3\eta)}\right)(1-e^{-m/\varepsilon}) \leq 1-e^{-m/\varepsilon}/2,
\]
where we use $W(O_j)>W(O_1)$ for the second inequality and pick the range of $\varepsilon$ small if it needs to be.
Thus \eqref{eqn:lessone} holds, 
and by \eqref{eqn_bound}
\[
E_{\lambda^{\varepsilon}}e^{(\hat{\tau}_1^{\varepsilon}-\hat{\upsilon}_{1}^{\varepsilon})/E_{\lambda^{\varepsilon}}\hat{\tau}_1^{\varepsilon}}\leq e^{-m/\varepsilon}\cdot2\sum_{i=1}^{\infty}\left(1-e^{-m/\varepsilon}/2\right)^{i-1}=\frac{2e^{-m/\varepsilon}}{1-\left(1-e^{-m/\varepsilon}/2\right)}=4.
\]
We complete the proof.
\end{proof}

\begin{lemma}
\label{Lem:9.15}Given $\delta>0,$ for any distribution $\lambda^{\varepsilon}$
on $\partial B_{\delta}(O_{1})$, $\hat{\tau}_{1}^{\varepsilon}/E_{\lambda
^{\varepsilon}}\hat{\tau}_{1}^{\varepsilon}$ converges in distribution to an
\rm{Exp(1)} random variable under $P_{\lambda^{\varepsilon}}.$
\end{lemma}

\begin{proof}
Let $\mathbf{M}$ be a geometrically distributed random variables with parameter $p\in(0,1)$ and assume that it is independent of $\{\tau^\varepsilon_n\}$. Then 
$
    E_{\lambda^{\varepsilon}}\left(\sum\nolimits_{n=1}^{\mathbf{M}}\tau^\varepsilon_n\right)
    =E_{\lambda^{\varepsilon}}\tau^\varepsilon_1/p
$
and
\begin{align*}
    E_{\lambda^{\varepsilon}}e^{it
    \left( \left(p\sum_{n=1}^{\mathbf{M}}\tau^\varepsilon_n\right)/E_{\lambda^{\varepsilon}}\tau^\varepsilon_1\right)}
    =\sum_{k=1}^{\infty} \left(E_{\lambda^{\varepsilon}}e^{it
    \left( p\tau^\varepsilon_1/E_{\lambda^{\varepsilon}}\tau^\varepsilon_1\right)}\right)^k
    (1-p)^{k-1}p
    =\frac{pE_{\lambda^{\varepsilon}}e^{it
    \left( p\tau^\varepsilon_1/E_{\lambda^{\varepsilon}}\tau^\varepsilon_1\right)}}{1-(1-p)E_{\lambda^{\varepsilon}}e^{it
    \left( p\tau^\varepsilon_1/E_{\lambda^{\varepsilon}}\tau^\varepsilon_1\right)}}.
\end{align*}
Given any fixed $t\in\mathbb{R}$, consider 
\[
    f_{\varepsilon}(p)=\frac{pE_{\lambda^{\varepsilon}}e^{it
    \left( p\tau^\varepsilon_1/E_{\lambda^{\varepsilon}}\tau^\varepsilon_1\right)}}{1-(1-p)E_{\lambda^{\varepsilon}}e^{it
    \left( p\tau^\varepsilon_1/E_{\lambda^{\varepsilon}}\tau^\varepsilon_1\right)}}
\text{ and } 
    f(p)=\frac{1}{1-it}.
\]
According to Lemma \ref{Lem:9.12}, 
$
    E_{\lambda^{\varepsilon}}e^{it
    \left( \tau^\varepsilon_1/E_{\lambda^{\varepsilon}}\tau^\varepsilon_1\right)}
    \rightarrow 1/(1-it)
$
uniformly on any compact set in $\mathbb{R}$. This implies that 
\[
    f_{\varepsilon}(p)=\frac{pE_{\lambda^{\varepsilon}}e^{it
    \left( p\tau^\varepsilon_1/E_{\lambda^{\varepsilon}}\tau^\varepsilon_1\right)}}{1-(1-p)E_{\lambda^{\varepsilon}}e^{it
    \left( p\tau^\varepsilon_1/E_{\lambda^{\varepsilon}}\tau^\varepsilon_1\right)}}
    \rightarrow 
    \frac{p/(1-itp)}{1-(1-p)/(1-itp)}=\frac{1}{1-it}=f(p)
\] uniformly on $p\in (0,1)$.
Therefore, if we consider $p^{\varepsilon}\doteq e^{-m/{\varepsilon}}\rightarrow 0$, 
it follows from the uniform (in $p$) convergence that
\[
    E_{\lambda^{\varepsilon}}e^{it(\hat{\tau}_{1}^{\varepsilon}/E_{\lambda
    ^{\varepsilon}}\hat{\tau}_{1}^{\varepsilon})}
    =f_{\varepsilon}(p^{\varepsilon})
    \rightarrow f(0)=\frac{1}{1-it}.
\]
We complete the proof.
\end{proof}

\section{Sketch of the Proof of Conjecture \ref{Conj:4.1} for a Special Case}
\label{sec:upper_bound_for_performance}

In this section we outline the proof
of the upper bound on the decay rate (giving a lower bound on the variance per unit time) that complements Theorem \ref{Thm:4.2} for a special case.
Consider $U:\mathbb{R}\rightarrow\mathbb{R}$ shown as in Figure \ref{fig:3}. 
\begin{figure}[h]
    \centering
    \includegraphics[width=0.9\textwidth]{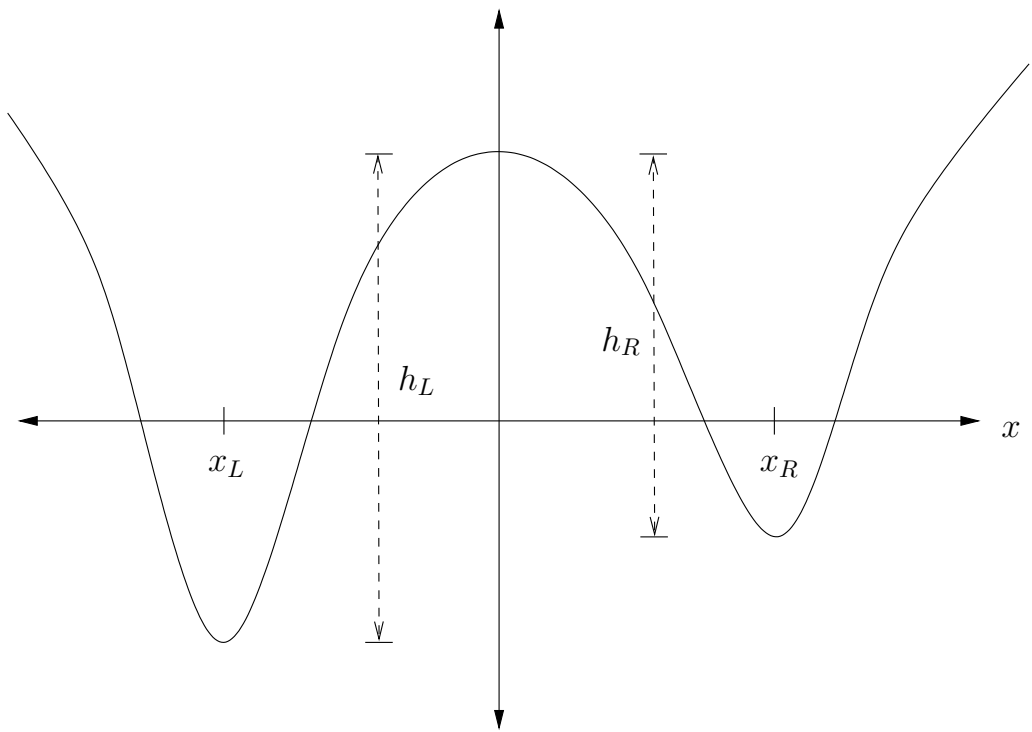}
    \caption{Asymmetric well for lower bound}
    \label{fig:3}
\end{figure}

In particular, assume $U$ is a bounded $C^{2}$ function satisfying the
following conditions:

\begin{condition}
\begin{itemize}
\item $U$ is defined on a compact interval $D\doteq [\bar{x}_L,\bar{x}_R] \subset\mathbb{R}$ and extends
periodically as a $C^{2}$ function.

\item $U$ has two local minima at $x_{L}$ and $x_{R}$ with values
$U(x_{L})<U(x_{R})$ and $[x_L-\delta,x_R+\delta]\subset D$ for some $\delta>0$.

\item $U$ has one local maximum at $0\in(x_{L},x_{R})$.

\item $U(x_{L})=0,$ $U(0)=h_{L}$ and $U(x_{R})=h_{L}-h_{R}>0.$

\item $\inf_{x\in\partial D}U(x)>h_{L}.$
\end{itemize}
\end{condition}

Consider the diffusion process $\{X^{\varepsilon}_t\}_{t\geq0}$ satisfying
the stochastic differential equation%
\begin{equation}
dX^{\varepsilon}_t  =-\nabla U\left(  X^{\varepsilon}_t  \right)  dt+\sqrt{2\varepsilon}dW_t  ,
\label{eqn:dym_original}%
\end{equation}
where $W$ is a $1$-dimensional standard Wiener process. Then there are just
two stable equilibrium points $O_{1}=x_{L}$ and $O_{2}=x_{R}$, and one unstable equilibrium point $O_3=0.$ Moreover, one
easily finds that $V(O_{1},O_{2})=h_{L}$ and $V(O_{2},O_{1})=h_{R},$ and these
give that $W(O_{1})=V(O_{2},O_{1})$, $W(O_{2})=V(O_{1},O_{2})$ and $W\left(  O_{1}\cup O_2\right)=0$ (since
$L_{\rm{s}}=\{1,2\},$ this implies that $G_{\rm{s}}(1)=\{(2\rightarrow1)\}$ and
$G_{\rm{s}}(2)=\{(1\rightarrow2)\}$). Another observation is that $h_1\doteq\min_{\ell\in\mathcal{M}%
\setminus\{1\}}V\left(  O_{1},O_{\ell}\right) =V\left(  O_{1},O_{3}\right)  =h_{L}$ in this model.

If $f\equiv 0,$ then one obtains
\begin{align*}
R_{1}^{(1)}  &  \doteq \inf_{y\in A}V(O_{1},y)+W(O_{1})-W(O_{1})=\inf_{y\in
A}V(O_{1},y);\\
R_{1}^{(2)}  &  \doteq 2\inf_{y\in A}V\left(  O_{1},y\right)  -h_1
=2\inf_{y\in A}V\left(  O_{1},y\right)  -h_{L};
\end{align*}%
\begin{align*}
R_{2}^{(1)}  &  \doteq \inf_{y\in A}V(O_{2},y)+W(O_{2})-W(O_{1})
  =\inf_{y\in A}V(O_{2},y)+h_{L}-h_{R};\\
R_{2}^{(2)}  &  \doteq2\inf_{y\in A}V\left(  O_{2},y\right)  +W\left(
O_{2}\right)  -2W\left(  O_{1}\right)  +0-W\left(  O_{1}\cup O_2\right) \\
 & =2\inf_{y\in A}V\left(  O_{2},y\right)  +h_{L}-2h_{R}.
\end{align*}

Let $A\subset [0,\bar{x}_R]$ and assume that it contains a nonempty open interval, so that we are computing approximations to probabilities that
are small under the stationary distribution (the case of bounded and continuous $f$ 
can be dealt with by approximation,
as in the case of the upper bound on the decay rate).
We first compute the bounds one would obtain from Theorem \ref{Thm:4.2}.

\vspace{\baselineskip}\noindent\textbf{Case I.} If $x_{R}\in A,$ then
$\inf_{y\in A}V(O_{1},y)=h_{L}$ and $\inf_{y\in A}V\left(  O_{2},y\right)
=0.$ Thus the decay rate of variance per unit time is bounded below by
\begin{align*}
\min_{j=1,2}\left[  R_{j}^{(1)}\wedge R_{j}^{(2)}\right]  
  =\min\left\{  h_{L},h_{L}-2h_{R}\right\} 
  =h_{L}-2h_{R}.
\end{align*}

\vspace{\baselineskip}\noindent\textbf{Case II.} If $A\subset\lbrack
0,x_{R}-\delta]$ for some $\delta>0$ and $\delta<x_{R},$ then $\inf_{y\in
A}V(O_{1},y)=h_{L}$ and $\inf_{y\in A}V\left(  O_{2},y\right)  >0$ (we denote
it by $b\in(0,h_{R}]$). Thus the decay rate of variance per unit time is
bounded below by
\begin{align*}
\min_{j=1,2}\left[  R_{j}^{(1)}\wedge R_{j}^{(2)}\right]   
 =\min\left\{  h_{L},h_{L}+2\left(  b-h_{R}\right)  \right\} 
 =h_{L}+2\left(  b-h_{R}\right)  .
\end{align*}

\vspace{\baselineskip}\noindent\textbf{Case III.} If $A\subset\lbrack
x_{R}+\delta,x^{\ast}]$ with $U(x^{\ast})=h_{L}$ for some $\delta>0$ and
$\delta<x^{\ast}-x_{R},$ then $\inf_{y\in A}V(O_{1},y)=h_{L}+\inf_{y\in
A}V\left(  O_{2},y\right)  $ and $\inf_{y\in A}V\left(  O_{2},y\right)  >0$
(we denote it by $b\in(0,h_{R}]$). Thus the decay rate of variance per unit
time is bounded below by
\begin{align*}
\min_{j=1,2}\left[  R_{j}^{(1)}\wedge R_{j}^{(2)}\right]   
  =\min\left\{  h_{L}+b,h_{L}+2\left(  b-h_{R}\right)  \right\}
  =h_{L}+2\left(  b-h_{R}\right)  .
\end{align*}

\vspace{\baselineskip}\noindent\textbf{Case IV.} If $A\subset\lbrack x^{\ast
}+\delta,\bar{x}_R]$ with $U(x^{\ast})=h_{L}$ for some $\delta>0$ and $x^{\ast}>x_{R},$ then
$\inf_{y\in A}V(O_{1},y)=h_{L}+\inf_{y\in A}V\left(  O_{2},y\right)  $ and
$\inf_{y\in A}V\left(  O_{2},y\right)  >0$ (we denote it by $\bar{b}>h_{R}$).
Thus the decay rate of variance per unit time is bounded below by
\begin{align*}
\min_{j=1,2}\left[  R_{j}^{(1)}\wedge R_{j}^{(2)}\right]   
  =\min\left\{  h_{L}+\bar{b},h_{L}+\left(  \bar{b}-h_{R}\right)  \right\} 
  =h_{L}+\left(  \bar{b}-h_{R}\right)  .
\end{align*}

To find an upper bound for the decay rate of variance per unit time, we recall
that
\[
\frac{1}{T^{\varepsilon}}\sum_{j=1}%
^{N^{\varepsilon}\left(  T^{\varepsilon}\right)-1  }\int_{\tau_{j-1}%
^{\varepsilon}}^{\tau_{j}^{\varepsilon}}1_{A}\left(  X_{t}^{\varepsilon
}\right)  dt \leq
\frac{1}{T^{\varepsilon}}\int_{0}^{T^{\varepsilon}}1_{A}\left(  X_{t}%
^{\varepsilon}\right)  dt
\leq \frac{1}{T^{\varepsilon}}\sum_{j=1}%
^{N^{\varepsilon}\left(  T^{\varepsilon}\right) }\int_{\tau_{j-1}%
^{\varepsilon}}^{\tau_{j}^{\varepsilon}}1_{A}\left(  X_{t}^{\varepsilon
}\right)  dt 
\]
with $\tau_{j}^{\varepsilon}$ being the $j$-th regenerative cycle. In Case I,
one might guess that
\begin{equation}
\int_{\tau_{j-1}^{\varepsilon}}^{\tau_{j}^{\varepsilon}}1_{A}\left(
X_{t}^{\varepsilon}\right)  dt \label{eqn:justpart}%
\end{equation}
has approximately the same distribution as the exit time from the shallow
well, which has been shown to asymptotically have an exponential distribution
with parameter $\exp(-h_{R}/\varepsilon).$ Additionally, since the exit time
from the shallower well is exponentially smaller than $\tau
_{j}^{\varepsilon},$ it suggests that the random variables (\ref{eqn:justpart}%
) can be taken as independent of $N^{\varepsilon}\left(  T^{\varepsilon
}\right)  $ when $\varepsilon$ is small. We also know that
\[
EN^{\varepsilon}\left(  T^{\varepsilon}\right)  /T^{\varepsilon%
}\approx 1/E\tau_{1}^{\varepsilon}\approx\exp\left(  -h_{L}(\delta)%
/\varepsilon\right)  ,
\]
where $h_{L}(\delta)\uparrow h_{L}$ as $\delta \downarrow 0$ and 
$\approx$ means that quantities on either side have the same 
exponential decay rate.
Using Jensen's inequality to find that
$E[N^\varepsilon(T^\varepsilon)]^2 \geq [EN^\varepsilon(T^\varepsilon)]^2$ and then applying Wald's identity, we obtain
\begin{align}
\label{eqn:variance_inequality}
&  T^{\varepsilon}\mathrm{Var}\left(  \frac{1}{T^{\varepsilon}}\int
_{0}^{T^{\varepsilon}}1_{A}\left(  X_{t}^{\varepsilon}\right)  dt\right) \nonumber\\
&  \approx\frac{1}{T^{\varepsilon}}E\left[  \sum\nolimits_{j=1}^{N^{\varepsilon}\left(
T^{\varepsilon}\right)  }\int_{\tau_{j-1}^{\varepsilon}}^{\tau_{j}%
^{\varepsilon}}1_{A}\left(  X_{t}^{\varepsilon}\right)  dt-EN^{\varepsilon
}\left(  T^{\varepsilon}\right)  E\left(  \int_{\tau_{j-1}^{\varepsilon}%
}^{\tau_{j}^{\varepsilon}}1_{A}\left(  X_{t}^{\varepsilon}\right)  dt\right)
\right]  ^{2}\nonumber\\
&  =\frac{1}{T^{\varepsilon}}E\left(  \sum\nolimits_{j=1}^{N^{\varepsilon}\left(
T^{\varepsilon}\right)  }  \int_{\tau_{j-1}^{\varepsilon}}^{\tau
_{j}^{\varepsilon}}1_{A}\left(  X_{t}^{\varepsilon}\right)  dt
\right)^{2}  - \frac{1}{T^{\varepsilon}} (E(N^{\varepsilon}(T^{\varepsilon
})))^{2} \left(  E\left(  \int_{\tau_{j-1}^{\varepsilon}}^{\tau_{j}%
^{\varepsilon}}1_{A}\left(  X_{t}^{\varepsilon}\right)  dt\right)  \right)
^{2}\nonumber\\
&  = \frac{1}{T^{\varepsilon}} EN^{\varepsilon}\left(  T^{\varepsilon}\right)
E\left(  \int_{\tau_{j-1}^{\varepsilon}}^{\tau_{j}^{\varepsilon}}1_{A}\left(
X_{t}^{\varepsilon}\right)  dt\right)  ^{2}-\frac{1}{T^{\varepsilon}}\left[  EN^{\varepsilon}\left(
T^{\varepsilon}\right)  \right]  ^{2}\left(  E\left(  \int_{\tau
_{j-1}^{\varepsilon}}^{\tau_{j}^{\varepsilon}}1_{A}\left(  X_{t}^{\varepsilon
}\right)  dt\right)  \right)  ^{2}\nonumber\\
&  \quad+\frac{1}{T^{\varepsilon}}\left(  E\left[ N^{\varepsilon}\left(
T^{\varepsilon}\right)  \right]   ^{2}-EN^{\varepsilon}\left(  T^{\varepsilon
}\right)  \right)  \left(  E\left(  \int_{\tau_{j-1}^{\varepsilon}}^{\tau
_{j}^{\varepsilon}}1_{A}\left(  X_{t}^{\varepsilon}\right)  dt\right)
\right)  ^{2}\nonumber\\
&  \geq\frac{EN^{\varepsilon}\left(  T^{\varepsilon}\right)  }{T^{\varepsilon
}}\mathrm{Var}\left(  \int_{\tau_{j-1}^{\varepsilon}}^{\tau_{j}^{\varepsilon}%
}1_{A}\left(  X_{t}^{\varepsilon}\right)  dt\right) \\
&  \approx\exp\left(  -h_{L}(\delta)/\varepsilon\right)  \cdot\exp(2h_{R}%
/\varepsilon)
  =\exp(\left(  2h_{R}-h_{L}(\delta)\right)  /\varepsilon).\nonumber
\end{align}
Letting $\delta \rightarrow 0$, we see that the decay rate of variance per unit time is bounded above by
$h_{L}-2h_{R}$, which is the same as lower bound found for Case I.

For the other three Cases II, III and IV,
the process spends only a very small fraction of the time while in the shallower well in the set $A$.
In fact,
using the stopping time arguments of the sort that appear in \cite[Chapter 4]{frewen2},
the event that the process enters $A$ during an excursion away from
the neighborhood of $x_R$ can be accurately approximated (as far as large deviation 
behavior goes) using independent Bernoulli random variables $\{B^\varepsilon_i\}$ with success parameter $e^{-b/\varepsilon}$,
and when this occurs the process spends an order one amount of time in $A$
before returning to the neighborhood of $x_R$.
There is however another sequence of independent Bernoulli random variables with success parameter $e^{-h_R/\varepsilon}$,
and the process accumulates time in $A$ only up till the
time of first success of this sequence.

Then 
$
\mathrm{Var}(  \int_{\tau_{j-1}^{\varepsilon}}^{\tau_{j}^{\varepsilon}%
}1_{A}\left(  X_{t}^{\varepsilon}\right)  dt)
$ 
has the same logarithmic asymptotics as 
$
\mathrm{Var}(
\sum\nolimits_{i=1}^{R^\varepsilon}1_{\{B^\varepsilon_i=1\}}
),
$
where $R^\varepsilon$ is geometric with success parameter $e^{-h_R/\varepsilon}$ and independent of the $\{B^\varepsilon_i\}$.
Straightforward calculation using Wald's identity then gives the exponential rate of decay $2h_R-2b$ for Cases II, III and $h_R-\bar{b}$ for Case IV,
so according to \eqref{eqn:variance_inequality} we obtain 
\begin{align*}
T^{\varepsilon}\mathrm{Var}\left(  \frac{1}{T^{\varepsilon}}\int_{0}%
^{T^{\varepsilon}}1_{A}\left(  X_{t}^{\varepsilon}\right)  dt\right)   
\geq\frac{EN^{\varepsilon}\left(  T^{\varepsilon}\right)  }{T^{\varepsilon}%
}\mathrm{Var}\left(  \int_{\tau_{j-1}^{\varepsilon}}^{\tau_{j}^{\varepsilon}%
}1_{A}\left(  X_{t}^{\varepsilon}\right)  dt\right) 
 \approx
 e^{\left[  \left(  2\left(  h_{R}-b\right)-h_{L}(\delta)\right) /\varepsilon\right]}  
\end{align*}
for Cases II and III and
\begin{align*}
T^{\varepsilon}\mathrm{Var}\left( \frac{1}{T^{\varepsilon}}\int_{0}^{T^{\varepsilon}}1_{A}\left(  X_{t}^{\varepsilon}\right)  dt\right)   \geq\frac{EN^{\varepsilon}\left(  T^{\varepsilon}\right)  }{T^{\varepsilon}}\mathrm{Var}\left(  \int_{\tau_{j-1}^{\varepsilon}}^{\tau_{j}^{\varepsilon}%
}1_{A}\left(  X_{t}^{\varepsilon}\right)  dt\right) 
 \approx
e^{\left[\left(( h_{R}-\bar{b}) -h_{L}(\delta)\right)/\varepsilon\right]}  
\end{align*} for Case IV.

Letting $\delta\rightarrow 0$, this means that the decay rate of variance per unit time is bounded above by
$h_{L}+2\left(  b-h_{R}\right)$ for Case II and III, and by $h_L+(\bar{b}-h_R)$ for Case IV which is again the same as the corresponding lower bound.

\vspace{\baselineskip}
\noindent
{\bf Acknowledgement.}
We thank the referee for corrections and suggestions that improved this paper.

\bibliographystyle{plain}
\bibliography{main}

\section{Appendix}

\begin{proof}
[Proof of Lemma \ref{Lem:6.6}]Given a function $g$, we define the notation
\[
    I(t_1,t_2;g) \doteq \int_{t_1}^{t_2}g(X^{\varepsilon}_s)ds,
\]
for any $0\leq t_1\leq t_2$. By definition, $\tau_{1}^{\varepsilon}=\tau_{N}$
and observe that%
\begin{align*}
I(0,\tau_{N};g)  &  =\sum\nolimits_{\ell
=1}^{N}I(\tau_{\ell-1},\tau_{\ell};g)=\sum\nolimits_{\ell=1}^{\infty}I(\tau_{\ell-1},\tau_{\ell};g)\cdot1_{\left\{  \ell\leq N\right\}  }\\
&  =\sum\nolimits_{\ell=1}^{\infty}I(\tau_{\ell-1},\tau_{\ell};g)\cdot1_{\left\{  \ell\leq\hat{N}\right\}  }%
+\sum\nolimits_{\ell=1}^{\infty}I(\tau_{\ell-1},\tau_{\ell};g)\cdot1_{\left\{  \hat{N}+1\leq\ell\leq
N\right\}  }\\
&  \quad+\sum\nolimits_{j\in L\setminus\left\{  1\right\}  }\sum\nolimits_{\ell=1}^{\infty}\left(
I(\tau_{\ell-1},\tau_{\ell};g)\cdot1_{\left\{  \hat{N}+1\leq\ell\leq N,Z_{\ell-1}\in\partial B_{\delta
}(O_{j})\right\}  }\right)  .
\end{align*}
Since $\hat{N}$ and $N$ are stopping times with respect to the filtration
$\{\mathcal{G}_{n}\}_{n},$ it implies that 
$
\{  \ell\leq\hat{N}\}  =\{  \hat{N}\leq\ell-1\}^{c}%
\in\mathcal{G}_{\ell-1}%
$ 
and 
$
\{  \hat{N}+1\leq\ell\leq N,Z_{\ell-1}\in\partial B_{\delta}%
(O_{j})\}  \in\mathcal{G}_{\ell-1}.
$ 
Let
\[
\mathfrak{S}_{1}=\sum_{\ell=1}^{\infty}I(\tau_{\ell-1},\tau_{\ell};g)\cdot1_{\left\{  \ell\leq\hat
{N}\right\}  }
\text{ and }
\mathfrak{S}_{j}=\sum_{\ell=1}^{\infty}\left(  I(\tau_{\ell-1},\tau_{\ell};g)\cdot1_{\left\{  \hat
{N}+1\leq\ell\leq N,Z_{\ell-1}\in\partial B_{\delta}(O_{j})\right\}  }\right)
\]
for all $j\in L\setminus\left\{  1\right\}  .$ We find%
\begin{align*}
E_{x}\left(  \mathfrak{S}_{1}\right)   
 & =\sum\nolimits_{\ell=1}^{\infty}E_{x}\left(  E_{x}\left[  \left. I(\tau_{\ell-1},\tau_{\ell};g)\cdot1_{\left\{  \ell
\leq\hat{N}\right\}  }\right|\mathcal{G}_{\ell-1}\right]  \right) \\
  &=\sum\nolimits_{\ell=1}^{\infty}E_{x}\left(  1_{\left\{  \ell\leq\hat{N}\right\}
}E_{Z_{\ell-1}}\left[  I(0,\tau_{1};g)\right]  \right) \\
&  \leq\sup\nolimits_{y\in\partial B_{\delta}(O_{1})}E_{y}\left[  I(0,\tau_{1};g)\right]  \cdot\left(  \sum\nolimits_{\ell
=1}^{\infty}P_{x}(  \hat{N}\geq\ell)  \right)  .
\end{align*}
In addition, for $j\in L\setminus\left\{  1\right\}  ,$
\begin{align*}
E_{x}\left(  \mathfrak{S}_{j}\right)   &  =\sum\nolimits_{\ell=1}^{\infty}E_{x}\left(
I(\tau_{\ell-1},\tau_{\ell};g)\cdot1_{\left\{  \hat{N}+1\leq\ell\leq N,Z_{\ell-1}\in\partial B_{\delta
}(O_{j})\right\}  }\right) \\
&  =\sum\nolimits_{\ell=1}^{\infty}E_{x}\left(  E_{x}\left[  \left. I(\tau_{\ell-1},\tau_{\ell};g)\cdot1_{\left\{  \hat
{N}+1\leq\ell\leq N,Z_{\ell-1}\in\partial B_{\delta}(O_{j})\right\}
}\right|\mathcal{G}_{\ell-1}\right]  \right) \\
&  =\sum\nolimits_{\ell=1}^{\infty}E_{x}\left(  1_{\left\{  \hat{N}+1\leq\ell\leq
N,Z_{\ell-1}\in\partial B_{\delta}(O_{j})\right\}  }E_{Z_{\ell-1}}\left[
I(0,\tau_{1};g)\right]  \right) \\
&  \leq\sup\nolimits_{y\in\partial B_{\delta}(O_{j})}E_{y}\left[  I(0,\tau_{1};g)\right]  \cdot\left(  \sum\nolimits_{\ell
=1}^{\infty}E_{x}\left(  1_{\left\{  \hat{N}+1\leq\ell\leq N,Z_{\ell-1}%
\in\partial B_{\delta}(O_{j})\right\}  }\right)  \right)  .
\end{align*}
It is straightforward to see that $\hat{N}=N_{1}$. This implies that 
$
\sum\nolimits_{\ell=1}^{\infty}P_{x}(  \hat{N}\geq\ell)  =E_{x}\hat{N}%
=E_{x}N_{1}.
$ 
Moreover, observe that for any $j\in L\setminus\left\{  1\right\}  $
$
\sum_{\ell=1}^{\infty}1_{\left\{  \hat{N}+1\leq\ell\leq N,Z_{\ell-1}%
\in\partial B_{\delta}(O_{j})\right\}  }=N_{j},
$
which gives that 
$
\sum_{\ell=1}^{\infty}E_{x}(  1_{\{  \hat{N}+1\leq\ell\leq
N,Z_{\ell-1}\in\partial B_{\delta}(O_{j})\}  })  =E_{x}N_{j}.
$ 
Hence,%
\begin{align*}
E_{x}\left(  I(0,\tau_{N};g)\right)    =\sum\nolimits_{j\in L}E_{x}\left(  \mathfrak{S}_{j}\right) 
  \leq\sum\nolimits_{j\in L}\left[  \sup\nolimits_{y\in\partial B_{\delta}(O_{j})}E_{y}\left(
I(0,\tau_{1};g)\right)  \right]
\cdot E_{x}N_{j}.
\end{align*}
\end{proof}

\bigskip

\begin{proof}
[Proof of Lemma \ref{Lem:6.7}] Let $l=|L|$. For any $j\in L$ and $n\in%
\mathbb{N}
,$ $\xi_{1}^{(j)}=\inf\{k\in%
\mathbb{N}
_{0}:$ $Z_{k}\in\partial B_{\delta}(O_{j})\},$ $\xi_{n}^{(j)}=\inf\{k\in%
\mathbb{N}
:k>\xi_{n-1}^{(j)}$ and $Z_{k}\in\partial B_{\delta}(O_{j})\}$ i.e. $\xi
_{n}^{(j)}$ is the $n$-th time of hitting $\partial B_{\delta}(O_{j}).$
Moreover, we define $N^{(j)}=\inf\{n\in%
\mathbb{N}
:\xi_{n}^{(j)}\geq N\},$ recalling that $N\doteq \inf\{n\geq\hat{N}:Z_{n}\in\partial
B_{\delta}(O_{1})\}$ and $\hat{N}\doteq \inf\{n\in
\mathbb{N}
:Z_{n}\in
{\textstyle\cup
_{j\in L\setminus\{1\}}}
\partial B_{\delta}(O_{j})\}.$ Since $\xi_{n}^{(j)}$ is a stopping time with
respect to $\{\mathcal{G}_{n}\}_{n},$ we can define the filtration
$\{\mathcal{G}_{\xi_{n}^{(j)}}\},$ and one can verify that $N^{(j)}$ is a
stopping time with respect to $\{\mathcal{G}_{\xi_{n}^{(j)}}\}_{n}.$ 
As in the proof just given, for any function $g$ and for any $0\leq t_1\leq t_2$ we define 
$I(t_1,t_2;g) \doteq \int_{t_1}^{t_2}g(X^{\varepsilon}_s)ds$.
With this  notation and since by definition $\tau_{1}^{\varepsilon}=\tau_{N}$, we can write%
\[
I(0,\tau_N;g)=\sum\nolimits_{j\in L}\sum\nolimits_{\ell
=1}^{\infty}I(\tau_{\xi_{\ell}^{(j)}},\tau_{\xi_{\ell}^{(j)}+1};g)\cdot1_{\left\{  \ell\leq N^{(j)}-1\right\}  }.
\]
Since $(x_{1}+\cdots+x_{l})^{2}\leq l(x_{1}^{2}+\cdots+x_{l}^{2})$ for any
$(x_{1},\ldots,x_{l})\in%
\mathbb{R}
^{l}$ and $l\in%
\mathbb{N}
,$%
\begin{align*}
 I(0,\tau_{N};g)^2
&  \leq l\sum\nolimits_{j\in L}\left(  \sum\nolimits_{\ell=1}^{\infty}I(\tau_{\xi_{\ell}^{(j)}},\tau_{\xi_{\ell}^{(j)}+1};g)\cdot1_{\left\{
\ell\leq N^{(j)}-1\right\}  }\right)  ^{2}.
\end{align*}
Now for any $j\in L$, each square term from the right can be written an addition of two sums, where the first sum is summation of $I(\tau_{\xi_{\ell}^{(j)}},\tau_{\xi_{\ell}^{(j)}+1};g)^2\cdot1_{\{
\ell\leq N^{(j)}-1\} }$ over all $\ell$, and the second sum is twice of summation of $I(\tau_{\xi_{\ell}^{(j)}},\tau_{\xi_{\ell}^{(j)}+1};g)\cdot1_{\{
\ell\leq N^{(j)}-1\} }I(\tau_{\xi_{k}^{(j)}},\tau_{\xi_{k}^{(j)}+1};g)\cdot1_{\{
\ell\leq N^{(j)}-1\} }$ over $k,\ell$ with $k<\ell$.
For the expected value of the first sum, note that $\{  \ell\leq
N^{(j)}-1\}  =\{  N^{(j)}\leq\ell\}  ^{c}\in\mathcal{G}%
_{\xi_{\ell}^{(j)}},$ we have%
\begin{align*}
  \sum_{\ell=1}^{\infty}E_{x}\left[   I(\tau_{\xi_{\ell}^{(j)}},\tau_{\xi_{\ell}^{(j)}+1};g)^2 1_{\left\{
\ell\leq N^{(j)}-1\right\}  }\right] 
 &=\sum_{\ell=1}^{\infty}E_{x}\left[  1_{\left\{  \ell\leq N^{(j)}-1\right\}
}E_{x}\left[  \left.  I(\tau_{\xi_{\ell}^{(j)}},\tau_{\xi_{\ell}^{(j)}+1};g)^2\right|\mathcal{G}_{\xi_{\ell}%
^{(j)}}\right]  \right] \\
&  \leq\sup_{y\in\partial B_{\delta}(O_{j})}E_{y}  I(0,\tau_{1};g)^2  \sum\nolimits_{\ell=1}^{\infty}P_{x}(
N^{(j)}-1\geq\ell)  \\
&=\sup_{y\in\partial B_{\delta}%
(O_{j})}E_{y}  I(0,\tau_{1};g)^2
E_{x}(  N_{j}).
\end{align*}
The last equality holds since $N^{(j)}-1=N_{j}$ (recall that $N_{j}$ is the number of
visits of $\{Z_{n}\}_{n\in\mathbb{N}_{0}}$ to $\partial B_{\delta}(O_{j})$ before $N$ including the initial
position) this implies that 
$
\sum_{\ell=1}^{\infty}P_{x}(  N^{(j)}-1\geq\ell)  =\sum_{\ell
=1}^{\infty}P_{x}(  N_{j}\geq\ell)  =E_{x}(  N_{j})  .
$ 

Turning to the expected value of the second sum, by conditioning on $\mathcal{G}_{\xi_{\ell}^{(j)}}$ gives
\begin{align*}
&  \sum_{\ell=2}^{\infty}\sum_{k=1}^{\ell-1}E_{x}\left[  I(\tau_{\xi_{\ell}^{(j)}},\tau_{\xi_{\ell}^{(j)}+1};g)\cdot1_{\left\{  \ell\leq N^{(j)}-1\right\}  }I(\tau_{\xi_{k}^{(j)}},\tau_{\xi_{k}^{(j)}+1};g)\cdot1_{\left\{  k\leq
N^{(j)}-1\right\}  }\right] \\
&  \leq\sup_{y\in\partial B_{\delta}(O_{j})}E_{y}\left(  I(0,\tau_{1};g)\right)  \sum_{\ell=2}^{\infty}\sum_{k=1}^{\ell
-1}E_{x}\left[  1_{\left\{  \ell\leq N^{(j)}-1\right\}  }I(\tau_{\xi_{k}^{(j)}},\tau_{\xi_{k}^{(j)}+1};g)\cdot1_{\left\{
k\leq N^{(j)}-1\right\}  }\right]  .
\end{align*}
Now since for any $k\leq\ell-1,$ i.e. $k+1\leq\ell$, 
$
I(\tau_{\xi_{k}^{(j)}},\tau_{\xi_{k}^{(j)}+1};g)\in\mathcal{G}_{\xi_{k}^{(j)}+1}\text{ and }1_{\left\{  \ell\leq
N^{(j)}-1\right\}  }\in\mathcal{G}_{\xi_{\ell}^{(j)}},
$
we have
\begin{align*}
&  E_{x}\left[  1_{\left\{  \ell\leq N^{(j)}-1\right\}  }I(\tau_{\xi_{k}^{(j)}},\tau_{\xi_{k}^{(j)}+1};g)\cdot1_{\left\{
k\leq N^{(j)}-1\right\}  }\right] \\
&  =E_{x}\left[  I(\tau_{\xi_{k}^{(j)}},\tau_{\xi_{k}^{(j)}+1};g)\cdot1_{\{\tau_{\xi_{1}^{(j)}}<N,\ldots,\tau_{\xi_{\ell}%
^{(j)}}<N\}}\right] \\
&  =E_{x}\left[  E_{Z_{\xi_{k+1}^{(j)}}}\left[  1_{\{\tau_{\xi_{1}^{(j)}%
}<N,\ldots,\tau_{\xi_{\ell-k}^{(j)}}<N\}}\right]  1_{\{\tau_{\xi_{k+1}^{(j)}%
}<N\}}I(\tau_{\xi_{k}^{(j)}},\tau_{\xi_{k}^{(j)}+1};g)\cdot1_{\{\tau_{\xi_{1}^{(j)}}<N,\ldots,\tau_{\xi_{k}^{(j)}%
}<N\}}\right] \\
&  =E_{x}\left[  E_{Z_{\xi_{k+1}^{(j)}}}\left[  1_{\left\{  \ell-k\leq
N^{(j)}-1\right\}  }\right]  1_{\{\tau_{\xi_{k+1}^{(j)}}<N\}}I(\tau_{\xi_{k}^{(j)}},\tau_{\xi_{k}^{(j)}+1};g)\cdot1_{\left\{  k\leq N^{(j)}-1\right\}  }\right] \\
&  \leq\sup\nolimits_{y\in\partial B_{\delta}(O_{j})}P_{y}(  \ell-k\leq
N^{(j)}-1)  E_{x}\left[  I(\tau_{\xi_{k}^{(j)}},\tau_{\xi_{k}^{(j)}+1};g)\cdot1_{\left\{  k\leq N^{(j)}%
-1\right\}  }\right] \\
&  =\sup\nolimits_{y\in\partial B_{\delta}(O_{j})}P_{y}\left(  \ell-k\leq N_{j}\right)
E_{x}\left[  E_{x}\left[ \left. I(\tau_{\xi_{k}^{(j)}},\tau_{\xi_{k}^{(j)}+1};g)\right|\mathcal{G}_{\xi_{k}^{(j)}}\right]  \cdot
1_{\left\{  k\leq N^{(j)}-1\right\}  }\right] \\
&  \leq\sup\nolimits_{y\in\partial B_{\delta}(O_{j})}E_{y}\left( I(0,\tau_{1};g)\right)  \cdot\sup\nolimits_{y\in\partial B_{\delta}(O_{j}%
)}P_{y}\left(  \ell-k\leq N_{j}\right)  \cdot P_{x}\left(  k\leq N_{j}\right)
.
\end{align*}
This gives that the expected value of the second sum is less than or equal to 
\begin{align*}
&\left(  \sup\nolimits_{y\in\partial B_{\delta}(O_{j})}E_{y} I(0,\tau_{1};g)  \right)  ^{2}\sum\nolimits_{\ell
=2}^{\infty}\sum\nolimits_{k=1}^{\ell-1}\sup\nolimits_{y\in\partial B_{\delta}(O_{j})}%
P_{y}\left(  \ell-k\leq N_{j}\right)  \cdot P_{x}\left(  k\leq N_{j}\right) \\
&  \qquad=\left(  \sup\nolimits_{y\in\partial B_{\delta}(O_{j})}E_{y} I(0,\tau_{1};g)  \right)  ^{2}\sum\nolimits_{k
=1}^{\infty}\sup\nolimits_{y\in\partial B_{\delta}(O_{j})}P_{y}\left(  k\leq
N_{j}\right)  \cdot E_{x}N_{j}.
\end{align*}
Therefore, putting the estimates together gives%
\begin{align*}
  E_{x}  I(0,\tau_{1}^{\varepsilon};g)^{2}
&\leq 2l\sum_{j\in L}\left[  \sup_{y\in\partial B_{\delta}(O_{j})}%
E_{y}  I(0,\tau_{1};g)  \right]  ^{2}\cdot E_{x}N_{j}\cdot\sum_{\ell=1}^{\infty}\sup
_{y\in\partial B_{\delta}(O_{j})}P_{y}\left(  \ell\leq N_{j}\right) \\
 &\quad+ l\sum_{j\in L}\left[  \sup_{y\in\partial B_{\delta}(O_{j})}%
E_{y}  I(0,\tau_{1};g)^2\right]  \cdot E_{x}N_{j}.
\end{align*}
\end{proof}

\begin{proof}
[Proof of Lemma \ref{Lem:8.2}]
The main idea of the proof comes from \cite[Theorem 3.16]{ros4}.

Given any $\varepsilon>0,$ we define $g^{\varepsilon}\left(  t\right)  \doteq
E_{\lambda^{\varepsilon}}S_{N^{\varepsilon}\left(  t\right)  }^{\varepsilon}$
for any $t\geq0.$ Conditioning on $\tau_{1}^{\varepsilon}$ yields%
\[
g^{\varepsilon}\left(  t\right)  =\int_{0}^{\infty}E_{\lambda^{\varepsilon}%
}[  S_{N^{\varepsilon}\left(  t\right)  }^{\varepsilon}|\tau
_{1}^{\varepsilon}=x]  dF^{\varepsilon}\left(  x\right)  ,
\]
where $F^{\varepsilon}\left(  \cdot\right)  $ is the distribution function of
$\tau_{1}^{\varepsilon}.$
Note that
\[
E_{\lambda^{\varepsilon}}\left[  S_{N^{\varepsilon}\left(  t\right)
}^{\varepsilon}|\tau_{1}^{\varepsilon}=x\right]  =\left\{
\begin{array}
[c]{c}%
g^{\varepsilon}\left(  t-x\right)  \text{ if }x\leq t\\
E_{\lambda^{\varepsilon}}\left[  S_{1}^{\varepsilon}|\tau_{1}^{\varepsilon
}=x\right]  \text{ if }x>t
\end{array}
\right.  ,
\]
which implies
\[
g^{\varepsilon}\left(  t\right)  =\int_{0}^{t}g^{\varepsilon}\left(
t-x\right)  dF^{\varepsilon}\left(  x\right)  +h^{\varepsilon}\left(
t\right)  ,
\]
with
\[
h^{\varepsilon}\left(  t\right)  =\int_{t}^{\infty}E_{\lambda^{\varepsilon}%
}\left[  S_{1}^{\varepsilon}|\tau_{1}^{\varepsilon}=x\right]  dF^{\varepsilon
}\left(  x\right)  .
\]
Since $E_{\lambda^{\varepsilon}}S_{1}^{\varepsilon}=\int_{0}^{\infty
}E_{\lambda^{\varepsilon}}\left[  S_{1}^{\varepsilon}|\tau_{1}^{\varepsilon
}=x\right]  dF^{\varepsilon}\left(  x\right)  <\infty,$ we have
$h^{\varepsilon}\left(  t\right)  \leq E_{\lambda^{\varepsilon}}%
S_{1}^{\varepsilon}$ for all $t\geq0.$ Moreover, if we apply H\"{o}lder's
inequality first and then the conditional Jensen's inequality, we find that
for all $t\geq0,$%
\begin{align*}
h^{\varepsilon}\left(  t\right)   &  
\leq\left(  \int_{t}^{\infty}\left(  E_{\lambda^{\varepsilon}}\left[
S_{1}^{\varepsilon}|\tau_{1}^{\varepsilon}=x\right]  \right)  ^{2}%
dF^{\varepsilon}\left(  x\right)  \right)  ^{\frac{1}{2}}\left(  \int
_{t}^{\infty}1^{2}dF^{\varepsilon}\left(  x\right)  \right)  ^{\frac{1}{2}}\\
&  \leq\left(  1-F^{\varepsilon}\left(  t\right)  \right)  ^{\frac{1}{2}%
}\left(  \int_{t}^{\infty}E_{\lambda^{\varepsilon}}[  \left(
S_{1}^{\varepsilon}\right)  ^{2}|\tau_{1}^{\varepsilon}=x]
dF^{\varepsilon}\left(  x\right)  \right)  ^{\frac{1}{2}}  \leq\left(  1-F^{\varepsilon}\left(  t\right)  \right)  ^{\frac{1}{2}%
}(  E_{\lambda^{\varepsilon}}\left(  S_{1}^{\varepsilon}\right)
^{2})  ^{\frac{1}{2}}.
\end{align*}

Given $\ell\in(0,c-h_1)$  let $U^{\varepsilon}\doteq e^{{\ell}/{\varepsilon}
}E_{\lambda^{\varepsilon}}\tau_{1}^{\varepsilon}$. According to Theorem
\ref{Thm:7.1}, there exists $\varepsilon_{0}\in(0,1)$ and a constant
$\tilde{c}>0$ such that
\[
1-F^{\varepsilon}\left(  U^{\varepsilon}\right)  =P_{\lambda^{\varepsilon}%
}(  \tau_{1}^{\varepsilon}/E_{\lambda^{\varepsilon}}\tau
_{1}^{\varepsilon}>e^{{\ell}/{\varepsilon}})  \leq e^{-\tilde
{c}e^{{\ell}/{\varepsilon}}}%
\]
for any $\varepsilon\in(0,\varepsilon_{0}).$
Also by Theorem \ref{Thm:7.1}, $U^{\varepsilon}<T^{\varepsilon}$ for all
$\varepsilon$ small enough. Hence for any $t\geq U^{\varepsilon}$,
\[
1-F^{\varepsilon}\left(  t\right)  \leq1-F^{\varepsilon}\left(  U^{\varepsilon
}\right)  \leq e^{-\tilde{c}e^{{\ell}/{\varepsilon}}}\text{ and
}h^{\varepsilon}\left(  t\right)  \leq e^{-\tilde{c}e^{{\ell}/{\varepsilon}}/2}(  E_{\lambda^{\varepsilon}}\left(  S_{1}^{\varepsilon}\right)
^{2})  ^{\frac{1}{2}}.
\]
By Proposition 3.4 in \cite{ros4}, we know that for any $\varepsilon>0$, for
$t\in\lbrack0,\infty)$
\[
g^{\varepsilon}\left(  t\right)  =h^{\varepsilon}\left(  t\right)  +\int
_{0}^{t}h^{\varepsilon}\left(  t-x\right)  da^{\varepsilon}\left(  x\right)
,
\]
where
\[
a^{\varepsilon}\left(  t\right)  \doteq\int_{0}^{\infty}E_{\lambda
^{\varepsilon}}\left[  N^{\varepsilon}\left(  t\right)  |\tau_{1}%
^{\varepsilon}=x\right]  dF^{\varepsilon}\left(  x\right)  =E_{\lambda
^{\varepsilon}}\left(  N^{\varepsilon}\left(  t\right)  \right)  .
\]
This implies
\begin{align*}
\frac{E_{\lambda^{\varepsilon}}S_{N^{\varepsilon}\left(  T^{\varepsilon
}\right)  }^{\varepsilon}}{T^{\varepsilon}}  
&  =\frac{h^{\varepsilon}\left(  T^{\varepsilon}\right)  }{T^{\varepsilon}%
}+\frac{1}{T^{\varepsilon}}\int_{0}^{T^{\varepsilon}-U^{\varepsilon}%
}h^{\varepsilon}\left(  T^{\varepsilon}-x\right)  da^{\varepsilon}\left(
x\right) +\frac{1}{T^{\varepsilon}}\int_{T^{\varepsilon}-U^{\varepsilon}%
}^{T^{\varepsilon}}h^{\varepsilon}\left(  T^{\varepsilon}-x\right)
da^{\varepsilon}\left(  x\right)  ,\\
&  \leq\frac{E_{\lambda^{\varepsilon}}S_{1}^{\varepsilon}}{T^{\varepsilon}%
}+ e^{-\tilde{c}e^{{\ell}/{\varepsilon}}/2}   (  E_{\lambda^{\varepsilon}}\left(  S_{1}^{\varepsilon}\right)
^{2})  ^{\frac{1}{2}}\frac{a^{\varepsilon}\left(  T^{\varepsilon
}-U^{\varepsilon}\right)  }{T^{\varepsilon}}+E_{\lambda^{\varepsilon}}S_{1}^{\varepsilon}\frac{a^{\varepsilon
}\left(  T^{\varepsilon}\right)  -a^{\varepsilon}\left(  T^{\varepsilon
}-U^{\varepsilon}\right)  }{T^{\varepsilon}},
\end{align*}
where we use $h^{\varepsilon}\left(  t\right)  \leq E_{\lambda^{\varepsilon}%
}S_{1}^{\varepsilon}$ to bound the first term and the third term, and
$h^{\varepsilon}\left(  t\right)  \leq e^{-\tilde{c}e^{{\ell}/{\varepsilon}}/2}(E_{\lambda^{\varepsilon}}\left(  S_{1}^{\varepsilon}\right)
^{2})^{1/2}$ for any $t\geq U^{\varepsilon}$ for the second term.

To calculate the decay rate of the first term, we apply Lemma \ref{Lem:6.17} to
find that for any $\eta>0$, there exists $\delta_{0}\in(0,1)$ such that for any
$\delta\in(0,\delta_{0})$
\begin{align}
&  \liminf_{\varepsilon\rightarrow0}-\varepsilon\log\frac{E_{\lambda
^{\varepsilon}}S_{1}^{\varepsilon}}{T^{\varepsilon}}\label{eqn:1of3}
\geq\inf_{x\in A}\left[  f\left(  x\right)  +W\left(  x\right)  \right]
-W\left(  O_{1}\right)  +  c-h_1-\eta.
\end{align}
For the decay rate of the second term, given any $\delta>0$
\begin{align}
&  \liminf_{\varepsilon\rightarrow0}-\varepsilon\log\left(  e^{-\tilde
{c}e^{{\ell}/{\varepsilon}}/4}  (E_{\lambda^{\varepsilon}}\left(
S_{1}^{\varepsilon}\right)  ^{2})  ^{\frac{1}{2}}\frac{a^{\varepsilon
}\left(  T^{\varepsilon}-U^{\varepsilon}\right)  }{T^{\varepsilon}}\right)
\label{eqn:2of3}\\
&  \quad=\frac{\tilde{c}}{4}\liminf_{\varepsilon\rightarrow0}\varepsilon e^{{\ell}/{\varepsilon}}+\liminf_{\varepsilon\rightarrow0}-\varepsilon\log\left(
(  E_{\lambda^{\varepsilon}}\left(  S_{1}^{\varepsilon}\right)
^{2})  ^{\frac{1}{2}}\frac{a^{\varepsilon}\left(  T^{\varepsilon
}-U^{\varepsilon}\right)  }{T^{\varepsilon}}\right)  =\infty,\nonumber
\end{align}
where the last equality holds since $\ell>0$ implies $\liminf_{\varepsilon
\rightarrow0}\varepsilon e^{{\ell}/{\varepsilon}}=\infty$ and also
because Lemma \ref{Lem:6.18} and Corollary \ref{Cor:7.2} ensure that
\newline$\liminf_{\varepsilon\rightarrow0}-\varepsilon\log((E_{\lambda
^{\varepsilon}}\left(  S_{1}^{\varepsilon}\right)  ^{2})^{1/2}a^{\varepsilon
}(T^{\varepsilon}-U^{\varepsilon})/T^{\varepsilon})$ is bounded below by a constant.

For the last term, note that for any $\varepsilon$ fixed, the renewal function
$a^{\varepsilon}\left(  t\right)  $ is subadditive in $t$ (see for example
Lemma 1.2 in \cite{limopr}), so we have $a^{\varepsilon}\left(  T^{\varepsilon
}\right)  -a^{\varepsilon}\left(  T^{\varepsilon}-U^{\varepsilon}\right)  \leq
a^{\varepsilon}\left(  U^{\varepsilon}\right)  .$ Thus we apply by Lemma
\ref{Lem:6.17}, Corollary \ref{Cor:7.2} and Theorem \ref{Thm:7.1} to find that for any $\eta>0$, there exists $\delta_{0}\in(0,1)$ such that for any $\delta\in(0,\delta_{0})$,
\begin{align}
&  \liminf_{\varepsilon\rightarrow0}-\varepsilon\log\left(  E_{\lambda
^{\varepsilon}}S_{1}^{\varepsilon}\frac{a^{\varepsilon}\left(  T^{\varepsilon
}\right)  -a^{\varepsilon}\left(  T^{\varepsilon}-U^{\varepsilon}\right)
}{T^{\varepsilon}}\right) \nonumber\\
&  \quad\geq\liminf_{\varepsilon\rightarrow0}-\varepsilon\log E_{\lambda
^{\varepsilon}}S_{1}^{\varepsilon}+\liminf_{\varepsilon\rightarrow
0}-\varepsilon\log\frac{a^{\varepsilon}\left(  U^{\varepsilon}\right)
}{U^{\varepsilon}}+\liminf_{\varepsilon\rightarrow0}-\varepsilon\log
\frac{U^{\varepsilon}}{T^{\varepsilon}}\nonumber\\
&  \quad\geq\inf_{x\in A}\left[  f\left(  x\right)  +W\left(  x\right)  \right]
-W\left(  O_{1}\right)  +\left(  c-h_1-\ell\right)-\eta  .\label{eqn:3of3}
\end{align}
Since \eqref{eqn:3of3} holds for all $\ell>0$, by sending $\ell$ to 0, we know that \eqref{eqn:3of3} holds with $\ell =0$.

Putting the bounds \eqref{eqn:1of3}, \eqref{eqn:2of3} and \eqref{eqn:3of3} with $\ell=0$
together gives that for any $\eta>0$, there exists $\delta_{0}\in(0,1)$ such that for any $\delta\in(0,\delta_{0})$,
\begin{align*}
&  \liminf_{\varepsilon\rightarrow0}-\varepsilon\log\frac{E_{\lambda
^{\varepsilon}}S_{N^{\varepsilon}\left(  T^{\varepsilon}\right)
}^{\varepsilon}}{T^{\varepsilon}}
\geq\inf_{x\in A}\left[  f\left(  x\right)  +W\left(  x\right)
\right]  -W\left(  O_{1}\right)  +  c-h_1-\eta .
\end{align*}
\end{proof}

\begin{proof}
[Proof of Lemma \ref{Lem:8.3}]
By the definition of $W(x),$%
\begin{align*}
&  2\inf_{x\in A}[  f\left(  x\right)  +W\left(  x\right)  ]
-2W\left(  O_{1}\right)  -h_1\\
&  =2\inf_{x\in A}[  f\left(  x\right)  +\min_{j\in L}\left(
V(O_{j},x)+W\left(  O_{j}\right)  \right)  ]  -2W\left(  O_{1}\right)
-h_1\\
&  =\min_{j\in L}\{  2\inf_{x\in A}\left[  f\left(  x\right)  +V\left(
O_{j},x\right)  \right]  +2W\left(  O_{j}\right)  -2W\left(  O_{1}\right)
-h_1\}  .
\end{align*}
Define $Q_{j}\doteq2\inf_{x\in A}\left[  f\left(  x\right)  +V\left(
O_{j},x\right)  \right]  +2W\left(  O_{j}\right)  -2W\left(  O_{1}\right)
-h_1$. Then it suffices to show that $Q_{j}\geq R_{j}^{(2)}$ for all $j\in L.$

For $j=1,$ $Q_{1}=2\inf_{x\in A}\left[  f\left(  x\right)  +V\left(
O_{1},x\right)  \right]  -h_1=R_{1}^{(2)}.$ For $j\in L\setminus\{1\},$
$Q_{j}\geq R_{j}^{(2)}$ if and only if $W\left(  O_{j}\right)  -h_1\geq W\left(
O_{1}\cup O_{j}\right)  .$ Recall that
\[
W\left(  O_{j}\right)  =\min_{g\in G\left(  j\right)  }\left[  {\textstyle\sum_{\left(
m\rightarrow n\right)  \in g}}V\left(  O_{m},O_{n}\right)  \right]
\text{ and }
W\left(  O_{1}\cup O_{j}\right)  =\min_{g\in G\left(  1,j\right)  }\left[
{\textstyle\sum_{\left(  m\rightarrow n\right)  \in g}}V\left(  O_{m},O_{n}\right)
\right]  .
\]
Therefore, for any $\tilde{g}\in G\left(  j\right)  $ such that $
W\left(  O_{j}\right)  =\sum\nolimits_{\left(  m\rightarrow n\right)  \in\tilde{g}%
}V\left(  O_{m},O_{n}\right)  ,$
if we remove the arrow starting from $1$, and assume that it goes to $i,$ then
it is easy to see that $\hat{g}\doteq\tilde{g}\setminus\{(1,i)\}\in G(1,j)$.
Since $V(O_{1},$ $O_{j})\geq h_1,$ we find that
\begin{align*}
W\left(  O_{j}\right)  -h_1  &  =\sum\nolimits_{\left(  m\rightarrow n\right)  \in
\tilde{g}}V\left(  O_{m},O_{n}\right)  -h_1
  =\sum\nolimits_{\left(  m\rightarrow n\right)  \in\hat{g}}V\left(  O_{m}%
,O_{n}\right)  +V(O_{1},O_{j})-h_1\\
&  \geq\min_{g\in G\left(  1,j\right)  }\left[  {\textstyle\sum_{\left(  m\rightarrow
n\right)  \in g}}V\left(  O_{m},O_{n}\right)  \right]   =W\left(  O_{1}\cup O_{j}\right)  .
\end{align*}
\end{proof}

\end{document}